\documentclass[11pt]{article}
\usepackage{amsbsy,amsfonts,amsmath,amssymb,amsthm,bbm,euscript,rotating}
\usepackage{graphicx,enumerate}
\usepackage{eepic}
\usepackage{centernot}
\usepackage[usenames]{color}

\newcommand{\ch}[1]{{#1}}
\newcommand{\chg}[1]{{#1}}

\def\UseSection{%%
        \numberwithin{equation}{section}
	\theoremstyle{plain}% default theorem style 
        \newtheorem{theorem}    {Theorem}[section]
        \DefineTheorems % Use this to define other environments to be 
}

\def\DefineTheorems{%%
	
	\newtheorem{lemma}      [theorem] {Lemma}
	
	\newtheorem{prop}       [theorem] {Proposition}
	
	\newtheorem{cor}        [theorem] {Corollary}

	\theoremstyle{definition}% ``defn'' theorem style 
	\newtheorem{defn}       [theorem] {Definition}

	\theoremstyle{definition}% ``remark'' theorem style 
	\newtheorem*{rem}	{Remark}
	
}

\newcommand{\bt}   {\begin{theorem}}
\newcommand{\et}   {\end  {theorem}}
\newcommand{\bl}   {\begin{lemma}}
\newcommand{\el}   {\end  {lemma}}
\newcommand{\bp}   {\begin{prop}}
\newcommand{\ep}   {\end  {prop}}
\newcommand{\bc}   {\begin{cor}}
\newcommand{\ec}   {\end  {cor}}
\newcommand{\bd}   {\begin{defn}}
\newcommand{\ed}   {\end  {defn}}

\newcommand{\ba}   {\begin{array}}
\newcommand{\ea}   {\end  {array}}
\newcommand{\be}   {\begin{enumerate}}
\newcommand{\ee}   {\end  {enumerate}}
\newcommand{\bi}   {\begin{itemize}}
\newcommand{\ei}   {\end  {itemize}}

\def\eq#1\en{\begin{equation}#1\end{equation}}  
\def\eqsplit#1\ensplit{
	\begin{equation}\begin{split}#1\end{split}\end{equation}
	}
\def\eqalign#1\enalign{
	\begin{align}#1\end{align}
	}
\def\eqmul#1\enmul{
	\begin{multline}#1\end{multline}
	}
\newcommand{\eqarrstar} {\begin{eqnarray*}} 
\newcommand{\enarrstar} {\end{eqnarray*}} 
\newcommand{\eqarray}   {\begin{eqnarray}} 
\newcommand{\enarray}   {\end{eqnarray}} 
\newcommand{\nnb}	{\nonumber \\} 

\newcommand{\lbeq}[1]  {\label{e:#1}}
\newcommand{\refeq}[1] {\eqref{e:#1}}    % AMS-LaTeX trick!
\makeatletter
\newcommand{\labelcounter}[2]{{%
	\stepcounter{#1}%	First, increase the ``countC'' by one.
	\protected@write\@auxout{}%
	{\string\newlabel{#2}{{\csname the#1\endcsname}{\thepage}}}%
	{\ref{#2}}%	Finally, make sure to refer to this label, 
	}}
\makeatother
\newcommand{\sss}   { \scriptscriptstyle } 

\newcommand{\Ebold} {{\mathbb E}}

\newcommand{\Nbold} {{\mathbb N}}

\newcommand{\Pbold} {{\mathbb P}}

\newcommand{\Rbold} {{\mathbb R}}

\newcommand{\Zbold} {{\mathbb Z}}

\newcommand{\avec}  {\boldsymbol{a}}

\newcommand{\cvec}  {\boldsymbol{c}}

\newcommand{\ovec}  {\boldsymbol{o}}

\newcommand{\uvec}  {\boldsymbol{u}}
\newcommand{\vvec}  {\boldsymbol{v}}
\newcommand{\wvec}  {\boldsymbol{w}}
\newcommand{\xvec}  {\boldsymbol{x}}
\newcommand{\yvec}  {\boldsymbol{y}}
\newcommand{\zvec}  {\boldsymbol{z}}

\newcommand{\Dcal}   {\mathcal{D}} 
\newcommand{\Ecal}   {\mathcal{E}}

\newcommand{\Pcal}   {\mathcal{P}}

\newcommand{\Scal}   {\mathcal{S}}

\newcommand{\Zd}    {{ {\Zbold}^d }}

\newcommand{\spose}[1] {{\hbox to 0pt{#1\hss}} }
\newcommand{\ltapprox} {\mathrel{\spose{\lower 3pt\hbox{$\mathchar"218$}}
 \raise 2.0pt\hbox{$\mathchar"13C$}}}
\newcommand{\gtapprox} {\mathrel{\spose{\lower 3pt\hbox{$\mathchar"218$}}
 \raise 2.0pt\hbox{$\mathchar"13E$}}}

\newcommand{\nin}  {{ \, \not\in \,}}

\UseSection  %Necessary to define Numbering scheme for Theorem, etc.
\theoremstyle{plain}
\newtheorem{lem}[theorem]{Lemma}

\newcommand{\bA}{{\bf A}}
\newcommand{\bB}{{\bf B}}
\newcommand{\bb}{\underline{b}}

\newcommand{\bC}{{\bf C}}

\newcommand{\cL}{{\cal L}}

\newcommand{\daw}{\downarrow}

\newcommand{\BDcup}[1]{~\underset{#1}{\Dot{\bigcup}}~}
\newcommand{\Dcup}[2]{~\Dot{\cup}_{#1}^{#2}~}
\newcommand{\DDcup}{~\Dot{\cup}}

\newcommand{\ddsum}{\sideset{_{}^{}}{_{}^{\bullet}}\sum}

\newcommand{\dsum}{\sum^\bullet}

\newcommand{\dpst}{\displaystyle}

\newcommand{\indic}{\mathbbm{1}}
\newcommand{\ind}[1]{\indic{\scriptstyle\{#1\}}}
\newcommand{\lamb}{\lambda}
\newcommand{\lambc}{\lamb_{\rm c}}
\newcommand{\lambce}{\lambc^{\scscst(\vep)}}

\newcommand{\mE}{{\mathbb E}}
\newcommand{\mN}{{\mathbb N}}
\newcommand{\mP}{{\mathbb P}}

\newcommand{\mR}{{\mathbb R}}
\newcommand{\mZ}{{\mathbb Z}}

\newcommand{\nnmb}{\nonumber}

\newcommand{\scscst}{\scriptscriptstyle}

\newcommand{\tb}{\overline{b}}
\newcommand{\tbp}
 {\tb^{\raisebox{-2pt}{\scriptsize$\prime$}}}

\newcommand{\tbsp}
 {\tb^{\raisebox{-2pt}{$\sss\prime$}}}

\newcommand{\uaw}{\uparrow}

\newcommand{\vep}{\varepsilon}
\newcommand{\vk}{\vec{k}}

\newcommand{\vt}{\vec{t}}

\newcommand{\wD}{\hat{D}}
\newcommand{\wM}{\hat{M}}

\newcommand{\wtau}{\hat{\tau}}

\newcommand{\Zp}{\mZ_+}

\newcommand{\smallsup}[1] {{\scriptscriptstyle{({#1}})}}

\newcommand{\R}{\Rbold}
\newcommand{\Z}{\Zbold}
\newcommand{\N}{\Nbold}
\newcommand{\conn}{\longrightarrow}

\newcommand{\dbc}{\Longrightarrow}
\newcommand{\ct}[1]     { \stackrel{#1}{\conn} }

\newcommand{\ctx}[1]     { \xrightarrow{#1} }

\newcommand{\prob}[2][]   {  {  {\mathbb P}_{#1} ( #2 ) }  }

\newcommand{\lupa}{\begin{rotate}{-10}$\scriptstyle\Longleftarrow$
\end{rotate}}
\newcommand{\rupa}{\begin{rotate}{10}$\scriptstyle\Longrightarrow$
\end{rotate}}
\newcommand{\btr}{$\scriptstyle\blacktriangle$}
\newcommand{\wtr}{$\scriptstyle\vartriangle$}

\newcommand{\bspat}{B_{\tt spat}}
\newcommand{\btemp}{B_{\tt temp}}
\newcommand{\vhi}{\varphi}
\newcommand{\sstar}{{\scriptstyle\,\star\,}}
\newcommand{\sT}{{\sss T}}
\newcommand{\nn}{\nonumber}

\newcommand{\twop}{\eta}

\newcommand{\kt}{k^{\sss (t)}}
\newcommand{\veckt}{\vec{k}^{\sss (t)}}
\newcommand{\veckT}{\vec{k}^{\sss (T)}}
\newcommand{\tildeb}{\tilde{b}}
\newcommand{\zetav}{\zeta}
\newcommand{\alphamin}{\underline{\alpha}}
\newcommand{\alphaminn}{\kappa}

\title{Convergence of the critical finite-range contact process\\
to super-Brownian motion above the upper critical dimension:\\
I. The higher-point functions}

\author{
    Remco van der Hofstad\thanks{Department of Mathematics and
        Computer Science, Eindhoven University of Technology,
        5600 MB Eindhoven, The Netherlands.
        {\tt rhofstad@win.tue.nl}}
    \and
    Akira Sakai\thanks{\chg{Creative Research Initiative ``Sousei",
        Hokkaido University, North 21, West 10, Kita-ku,
        Sapporo 001-0021, Japan.
        {\tt sakai@cris.hokudai.ac.jp}}}
    }

\begin{document}
\maketitle

%\begin{center}
%{\large \bf IN PROGRESS: NOT FOR DISTRIBUTION}
%\end{center}

\begin{abstract}
We consider the critical spread-out contact process in $\Zd$ with
$d\geq 1$, whose infection range is denoted by $L\geq1$.
%The
%two-point function $\tau_t(x)$ is the probability that $x\in\Zd$
%is infected at time $t$ by individual located at the origin $o\in\Zd$ at
%time 0.  In a previous paper, we have proved Gaussian behavior
%for the two-point function $\tau_t(x)$ with $L$ sufficiently large and
%$d>4$.
%, and, when $1\leq d\leq 4$, by a local mean-field limit for
%$\tau_{tT}(x)$ with $0<t\leq \log{T}$, when the infection range depends
%on $T$ and $L_T=LT^b$ for any $b>(4-d)/2d$.
In this paper, we investigate the
higher-point functions $\tau_{\vec t}^{\sss(r)}(\vec x)$ for $r\geq 3$, where $\tau_{\vec t}^{\sss(r)}(\vec x)$ is the probability that, for all $i=1,\dots,r-1$, the individual located at $x_i\in\Zd$ is infected at time $t_i$ by the individual at the origin $o\in\Zd$ at time 0.  Together with the results of the 2-point function in \cite{hsa04}, on which our proofs crucially rely, we prove that the $r$-point functions converge to the moment measures of the canonical measure of super-Brownian motion above the upper critical dimension 4.  We also prove partial results for $d\le4$ in a local mean-field setting.

The proof is based on the lace expansion for the time-discretized contact process, which is a version of oriented percolation in $\Zd\times\vep\Zp$, where $\vep\in(0,1]$ is the time unit.  For ordinary oriented percolation (i.e., $\vep=1$), we thus reprove the results of \cite{hs01}.  The lace expansion coefficients are shown to obey bounds uniformly in $\vep\in(0,1]$, which allows us to establish the scaling results also for the contact process (i.e., $\vep\downarrow0$).
%The continuum
%limit is shown to exist using a simplified argument, which
%requires no pointwise convergence of the lace expansion
%coefficients. This argument can also be used to simplify
%the extension of the results to the contact process in
%our previous paper on the two-point function.
We also show that the main term of the vertex factor $V$, which is one of the non-universal constants in the scaling limit, depends explicitly on the time unit as $2-\vep$, while the main terms of the other non-universal constants are independent of $\vep$.

The lace expansion we develop in this paper is adapted to both the $r$-point function and the survival probability.
%, and the idea of deriving the expansion was already used in \cite{HHS05b}
%for oriented percolation.
This unified approach makes it easier to relate the expansion coefficients derived in this paper and the expansion coefficients for the survival probability, which will be reported in Part II \cite{hsa06}.
\end{abstract}

\tableofcontents

%\input{rpt1}
% rpt1.tex

% June 1, 2007, RvdH
% March 9th, 2005: RvdH adapted after a long pause
% March 15, 2004, RvdH: adapted comments by Akira.
% January 5, 2004, RvdH
% July  17, 2003, RvdH
\section{Introduction and results}

\subsection{Introduction}\label{ss:intro}
The contact process is a model for the spread of an infection among individuals in the $d$-dimensional integer lattice $\mZ^d$.  Suppose that the origin $o\in\Zd$ is the only infected individual at time 0, and assume for now that every infected individual may infect a healthy individual at a distance less than $L\geq1$.  We refer to this type of model as the spread-out contact process.  The rate of infection is denoted by $\lambda$, and it is well known
that there is a phase transition in $\lambda$ at a critical value
$\lambc\in (0,\infty)$ (see, e.g., \cite{Ligg99}).

In the previous paper \cite{hsa04}, and following the idea of \cite{s01}, we proved the 2-point function results for the contact process for $d>4$ via a time discretization, as well as a partial extension to $d\le4$.
The discretized contact process is a version of oriented percolation in
$\Zd\times\vep \Zp$, where $\vep\in(0,1]$ is the time unit.  The proof
is based on the strategy for ordinary oriented percolation
($\vep=1$), i.e., on the application of the lace expansion
and an adaptation of the inductive method so as to deal with
the time discretization.

In this paper, we use the 2-point function results in \cite{hsa04} as a key ingredient to show that, for any $r\ge3$, the $r$-point functions of the critical contact process for $d>4$ converge to those of the canonical measure of super-Brownian motion, as was proved in \cite{hs01} for ordinary oriented percolation.  We follow the strategy in \cite{hs01} to analyze the lace expansion, but derive an expansion which is different from the expansion used in \cite{hs01}.
%and where possible simplify the argument.
%For instance, the expansion in the current paper is simpler \ch{than} the
%one in \cite{hs01}, which also allows us to present simpler bounds
%on the lace expansion \ch{coefficients}.
The lace expansion used in this paper is closely related to the expansion in \cite{HHS05b} for the oriented-percolation survival probability.  The latter was used in \cite{HHS05a} to show that the probability that the oriented-percolation cluster survives up to time $n$ decays proportionally to $1/n$.  Due to this close relation, we can reprove an identity relating the constants arising in the scaling limit of the 3-point function and the survival probability, as was stated in \cite[Theorem~1.5]{hhs01} for oriented percolation.

The main selling points of this paper in comparison to other papers on the topic are the following:

\begin{enumerate}
\item
Our proof yields a simplification of the expansion argument, which is still inherently difficult, but has been simplified as much as possible, making use of and extending the combined insights of \cite{HS90a,HHS05b,hsa04,hs01}.
\item
The expansion for the higher-point functions yields similar expansion coefficients to those for the survival probability in \cite{HHS05b}, thus making the investigation of the contact-process survival probability more efficient and allowing for a direct comparison of the various constants arising in the 2- and 3-point functions and the survival probability.   This was proved for oriented percolation in \cite[Theorem 1.5]{hhs01}, which, on the basis of the expansion in \cite{hs02}, was \emph{not} directly possible.
\item
The extension of the results to certain local mean-field limit type results in low dimensions, as was initiated in \cite{dp99} and taken up again in \cite{hsa04}.
\item
A simplified argument for the continuum limit of the discretized model, which was performed in \cite{hsa04} through an intricate weak convergence argument, and which in the current paper is replaced by a soft argument on the basis of subsequential limits and uniformity of our bounds.
\end{enumerate}

The investigation of the contact-process survival probability is deferred to Part~II of this paper \cite{hsa06}, in which we also discuss the implications of our results for the convergence of the critical spread-out contact process towards super-Brownian motion, in the sense of convergence of finite-dimensional distributions \cite{HolPer07}.  See also \cite{Hofs05} and \cite{Slad02} for more expository discussions of the various results for oriented percolation and the contact process for $d>4$, and \cite{Slad05} for a detailed discussion of
the applications of the lace expansion.

\subsection{Main results}\label{ss:results}
We define the spread-out contact process as follows.
Let $\bC_t\subseteq\mZ^d$ be the set of infected individuals
at time $t\in\mR_+$, and let $\bC_0=\{o\}.$ An infected
site $x$ recovers in a small time interval $[t,t+\vep]$
with probability $\vep+o(\vep)$ independently of $t$,
where $o(\vep)$ is a function that satisfies
$\lim_{\vep\downarrow 0}o(\vep)/\vep=0$. In other words, $x\in\bC_t$
recovers at rate 1.  A healthy site $x$ gets infected, depending
on the status of its neighboring sites, at rate $\lamb\sum_{y\in\bC_t}D(x-y)$,
where $\lamb\geq0$ is the infection rate.  We denote
the associated probability measure by $\mP^\lamb$.  We assume that
the function $D:\Zd\to[0,1]$ is a probability distribution which
is symmetric with respect to the lattice symmetries. Further assumptions
on $D$ involve a parameter $L\geq 1$ which serves to spread out the
infections, and will be taken to be large. In particular, we require
that $D(o)=0$ and $\|D\|_\infty\equiv\sup_{x\in\Zd}D(x)\le CL^{-d}$.  Moreover, with
$\sigma$ defined as
    \eq\lbeq{sgm-def}
    \sigma^2=\sum_x|x|^2D(x),
    \en
where $|\cdot|$ denotes the Euclidean norm on $\R^d$, we require
that $C_1L\leq \sigma \leq C_2L$ and that there exists a
$\Delta>0$ such that
    \eq
    \sum_x|x|^{2+2\Delta} D(x)\leq CL^{2+2\Delta}.
    \lbeq{Deltadef}
    \en
See \cite[Section 5]{hsa04} for the precise assumptions on $D$.
A simple example of $D$ is
\begin{align}\lbeq{Dexample}
D(x)=\frac{\ind{0<\|x\|_\infty\le L}}{(2L+1)^d-1},
\end{align}
which is the uniform distribution on the cube of radius $L$.

For $r\ge2$, $\vec{t}=(t_1,\dots,t_{r-1})\in\mR_+^{r-1}$ and
$\vec{x}=(x_1,\dots,x_{r-1})\in\mZ^{(r-1)d}$, we define the $r$-point function as
\begin{align}\lbeq{2pt-def}
\tau_{\vec t}^\lamb(\vec x)=\mP^\lamb(x_i\in\bC_{t_i}~\forall i=1,\dots,r-1).
\end{align}
%In words, $\tau_t^\lamb(x)$ is the probability that at time
%$t$, the individual located at $x\in \Z^d$ is infected due to
%the infection at time 0 located at $o\in \Z^d$.
%Fourier analysis plays an important role throughout the paper.
For a summable function $f\colon \Z^d \to \R$, we define its Fourier
transform for $k\in[-\pi,\pi]^d$ by
    \eq
    \hat{f}(k) = \sum_{x\in \Z^d} e^{ik\cdot x} f(x).
    \en

By the results in \cite{gh01} and the extension in \cite{bg91} to the spread-out model, there exists a unique critical point $\lambc\in(0,\infty)$ such that
\begin{align}\lbeq{lambc-def}
\int_0^\infty dt\;\wtau_t^\lamb(0)
 \begin{cases}
 <\infty,&\mbox{if }\lamb<\lambc,\\
 =\infty,&\mbox{otherwise},
 \end{cases}&&
\lim_{t\uaw\infty}\mP^\lamb(\bC_t\ne\varnothing)
 \begin{cases}
 =0,&\mbox{if }\lamb\leq\lambc,\\
 >0,&\mbox{otherwise}.
 \end{cases}
\end{align}
We will next investigate the sufficiently spread-out contact process at the critical value $\lambc$ for $d>4$,
as well as a local mean-field limit when $d\leq 4$.

\subsection{Previous results for the 2-point function}
%\subsubsection{Previous results for the two-point function above four dimensions}
%\label{sss:resultsd>4}
We first state the results for the 2-point function proved in \cite{hsa04}. Those results will be crucial for the current paper.  In the statements, $\sigma$ is defined in \refeq{sgm-def} and $\Delta$ in \refeq{Deltadef}.

Besides the high-dimensional setting for $d>4$, we also consider a low-dimensional setting, i.e., $d\le4$.  In this case,  the contact process is {\it not} believed to be in the mean-field regime, and Gaussian asymptotics are thus not expected to hold as long as $L$ remains finite.  However, following the rescaling of Durrett and Perkins in \cite{dp99}, we have proved Gaussian asymptotics when range and time grow simultaneously \cite{hsa04}.  We suppose that the infection range grows as
\begin{align}\lbeq{Lt-def}
L_{\sT}=L_1\,T^b,
\end{align}
where $L_1\ge1$ is the initial infection range and $T\ge1$.  We denote by $\sigma_{\sT}^2$ the variance of $\ch{D=D_{\sT}}$ in this situation.  We will assume that
\begin{align}\lbeq{alphadef}
\alpha=bd+\frac{d-4}2>0.
\end{align}

\begin{theorem}[{\bf Gaussian asymptotics two-point function}]
\label{thm:2pt}
\begin{enumerate}[(i)]
\item
Let $d>4$, $\delta\in(0,1\wedge\Delta\wedge\frac{d-4}2)$ and $L\gg1$.  There
exist positive and finite constants $v,A$ (depending on $d$ and $L$) and
$C_1,C_2$ (depending only on $d$) such that
\begin{gather}
\wtau_t^{\lambc}\big(\tfrac{k}{\sqrt{v\sigma^2t}}\big)=A\,e^{-\frac{|k|^2}
 {2d}}\Big(1+O\big(|k|^2(1+t)^{-\delta}\big)+O\big((1+t)^{-(d-4)/2}\big)
 \Big),\lbeq{tauasy}\\
\frac1{\wtau_t^{\lambc}(0)}\sum_x|x|^2\tau_t^{\lambc}(x)=v\sigma^2t\Big(1+
 O\big((1+t)^{-\delta}\big)\Big),\lbeq{taugyr}\\
C_1L^{-d}(1+t)^{-d/2}\le\|\tau_t^{\lambc}\|_{\sss\infty}\le e^{-t}+C_2L^{-d}(1+
 t)^{-d/2},\lbeq{tausup}
\end{gather}
with the error estimate in \refeq{tauasy} uniform in $k\in\mR^d$
with $\chg{|k|^2/\log(2+t)}$ sufficiently small. Moreover,
\begin{align}\lbeq{estimates}
\lambc=1+O(L^{-d}),&& A=1+O(L^{-d}),&& v=1+O(L^{-d}).
\end{align}
\item
Let $d\le4$, $\delta\in(0,1\wedge\Delta\wedge\alpha)$ and $L_1\gg1$.  There
exist $\lamb_{\sT}=1+O(T^{-\mu})$ for some $\mu\in(0,\alpha-\delta)$ and
$C_1,C_2\in(0,\infty)$ (depending only on $d$) such that, for every
$0<t\le\log T$,
\begin{gather}
\wtau_{Tt}^{\lamb_T}\big(\tfrac{k}{\sqrt{\sigma_{\sT}^2Tt}}\big)=e^{-
 \frac{|k|^2}{2d}}\Big(1+O(T^{-\mu})+O\big(|k|^2(1+Tt)^{-\delta}\big)
 \Big),\lbeq{tauasy-lowdim}\\
\frac1{\wtau_{Tt}^{\lamb_{\sT}}(0)}\sum_x|x|^2\tau_{Tt}^{\lamb_{\sT}}(x)=
 \sigma_{\sT}^2Tt\Big(1+O(T^{-\mu})+O\big((1+Tt)^{-\delta}\big)\Big),
 \lbeq{taugyr-lowdim}\\
C_1L_{\sT}^{-d}(1+Tt)^{-d/2}\le\|\tau_{Tt}^{\lamb_T}\|_{\sss\infty}\le e^{-Tt}+
 C_2L_{\sT}^{-d}(1+Tt)^{-d/2},\lbeq{tausup-lowdim}
\end{gather}
with the error estimate in \refeq{tauasy-lowdim} uniform in $k\in\mR^d$
with $\chg{|k|^2/\log(2+Tt)}$ sufficiently small.
\end{enumerate}
\end{theorem}

In the rest of the paper, we will always work at the critical value, i.e., we take $\lambda=\lambda_{\rm c}$ for $d>4$ and $\lamb=\lamb_\sT$ as in
Theorem~\ref{thm:2pt}(ii) for $d\le4$.  We will often omit the $\lamb$-dependence and write $\tau_{\vec t}^{\smallsup{r}}(\vec x)=\tau_{\vec t}^{\lamb}(\vec x)$ to emphasize the number of arguments of $\tau_{\vec t}^{\lamb}(\vec x)$.

While $\tau_t^{\lambc}(x)$ tells us what paths in a critical cluster look like, the critical $r$-point functions give us information about the branching structure of critical clusters. Our goal in this paper is to prove that the suitably scaled critical $r$-point functions converge to those of the canonical measure of super-Brownian motion (SBM).

\subsection{The $r$-point function for $r\geq3$}
To state the result for the $r$-point function for $r\geq3$,
we begin by describing the Fourier transforms of the
moment measures of SBM.
These are most easily defined recursively, and will serve
as the limits of the $r$-point functions. We define
    \eq
    \lbeq{M1-def}
    \wM_t^{\smallsup{1}}(k)=e^{-\frac{|k|^2}{2d}t},\qquad
    k\in\mR^d,\,t\in\mR_+,
    \en
and define recursively, for $l\geq2$,
    \eq
    \lbeq{Mr-def}
    \wM_{\vt}^{\smallsup{l}}(\vk)=\int_0^{\underline{t}}dt\,
    \wM_t^{\smallsup{1}}(k_1+\cdots+k_l)\sum_{I\subset J_1:|I|
    \geq1}\wM_{\vt_I-t}^{\smallsup{|I|}}(\vk_I)\,\wM_{\vt_{J
    \setminus I}-t}^{\smallsup{l-|I|}}(\vk_{J\setminus I}),
    \qquad\vk\in\mR^{dl},\,\vt\in\mR_+^l,
    \en
where $J=\{1,\dots,l\}$, $J_1=J\setminus\{1\}$, $\underline
{t}=\min_it_i$, $\vt_I$ is the vector consisting of $t_i$
with $i\in I$, and $\vt_I-t$ is subtraction of $t$ from
each component of $\vt_I$.  The quantity
$\wM_{\vt}^{\smallsup{l}}(\vk)$ is the Fourier transform
of the $l^{\rm th}$ moment measure of the canonical measure
of SBM (see \cite[Sections~1.2.3 and 2.3]{hs01} for more details
on the moment measures of SBM).

The following is the result for the $r$-point function for
$r\geq3$ linking the critical contact process and the canonical measure of
SBM:

\begin{theorem}[{\bf Convergence of $r$-point functions to SBM moment measures}]\label{thm:rpt}
\begin{enumerate}[(i)]
\item
Let $d>4$, $\lamb=\lambc$, $\vk\in\mR^{d(r-1)}$, $\vt\in(0,\infty)^{r-1}$ and $\delta,L,v,A$ be the same as in Theorem~\ref{thm:2pt}(i).  There exists $V=V(d,L)$ such that, for every $r\ge2$ and as $T\uparrow\infty$,
\begin{align}\lbeq{taurasy}
\wtau_{T\vt}^{\sss(r)}\big(\tfrac{\vk}{\sqrt{v\sigma^2T}}\big)=A\,(A^2V
 T)^{r-2}\Big(\wM_{\vt}^{\sss(r-1)}(\vk)+O(T^{-\delta})\Big),
\end{align}
uniformly in $\vk\in\mR^{d(r-1)}$ with $\sum_{i=1}^{r-1}|k_i|^2$ bounded. Moreover, $V=2 +O(L^{-d})$.
\item
Let $d\le4$, $\vk\in\mR^{d(r-1)}$, $\vt\in(0,\infty)^{r-1}$ and let $\delta,L_1,\lamb_{\sT},\mu$ be the same as in Theorem~\ref{thm:2pt}(ii).  For every $r\ge2$, $0<\max_is_i\le\log T$ and as $T\uaw\infty$,
\begin{align}\lbeq{taurasylowd}
\wtau_{T\vt}^{\smallsup{r}}\big(\tfrac{\vk}{\sqrt{\sigma_{\sT}^2T}}\big)=
 (2T)^{r-2}\Big(\wM_{\vt}^{\smallsup{r-1}}(\vk)+O(T^{-\mu\wedge\delta})\Big),
\end{align}
uniformly in $\vk\in\mR^{d(r-1)}$ with $\sum_{i=1}^{r-1}
|k_i|^2$ bounded.
\end{enumerate}
\end{theorem}

Since the statements for $r=2$ in Theorem~\ref{thm:rpt} follow from Theorem~\ref{thm:2pt}, we only need prove Theorem~\ref{thm:rpt} for $r\ge3$.  As described in more detail in Part~II \cite{hsa06}, Theorems~\ref{thm:2pt}--\ref{thm:rpt} can be rephrased to say that, under their hypotheses, the moment measures of the rescaled critical contact process converge to those of the canonical measure of SBM.  The consequences of this result for the convergence of the critical contact process towards SBM will be
deferred to \cite{hsa06}.

Theorem~\ref{thm:rpt} will be proved using the \emph{lace expansion}, which perturbs the $r$-point functions for the critical contact process around those for critical branching random walk.  To derive the lace expansion, we use a
time-discretization. The time-discretized contact process has a parameter $\vep\in (0,1]$. The boundary case $\vep=1$ corresponds to ordinary oriented percolation, while the limit $\vep\downarrow0$ yields the contact process.  We will prove Theorem~\ref{thm:rpt} for the time-discretized contact process and prove that the error terms are uniform in the discretization parameter $\vep$. As a consequence, we will reprove Theorem~\ref{thm:rpt} for oriented percolation.  The first proof of Theorem~\ref{thm:rpt} for oriented percolation appeared in \cite{hs01}.

In \cite{CheSak08a, CheSak08b}, spread-out oriented percolation is investigated in the setting where the finite variance condition \refeq{Deltadef} fails, and it was shown that for certain infinite variance step distributions $D$ in the domain of attraction of an $\alpha$-stable distribution, the Fourier transform
of two-point function converges to the one of an $\alpha$-stable random variable, when $d>2\alpha$ and $\alpha\in (0,2)$. We conjecture that, in this case, the limits of the $r$-point functions satisfy a limiting result similarly to \refeq{taurasy} when the argument in
the $r$-point function in \refeq{taurasy} is replaced by $\tfrac{\vk}{vT^{1/\alpha}}$ for some $v>0$,
and where the limit corresponds to the moment measures of a super-process where the motion is $\alpha$-stable and the branching has finite variance (in the terminology of \cite[Definition 1.33, p.\ 22]{Ethe00}, this corresponds to the
$(\alpha,d,1)$-superprocess and SBM corresponds to $\alpha=2$). These limiting moment measures should satisfy \refeq{Mr-def}, but \refeq{M1-def} is replaced by $e^{-|k|^\alpha t}$, which is the Fourier transform of an $\alpha$-stable motion at time $t$.

\subsection{Organization}
The paper is organised as follows. In Section \ref{sec-outline},
we will describe the time-discretization, state the results for
the time-discretized contact process and give an outline of the proof.
In this outline, the proof of Theorem~\ref{thm:rpt}
will be reduced to Propositions \ref{thm-psivphibd} and \ref{prop-disc}.
In Proposition \ref{thm-psivphibd}, we state the bounds on the
expansion coefficients arising in the expansion for the $r$-point function.
In Proposition \ref{prop-disc}, we state and prove that the sum of these
coefficients converges, when appropriately scaled and as $\vep\downarrow 0$.
The rest of the paper is devoted to the proof of Propositions
\ref{thm-psivphibd} and \ref{prop-disc}. In Sections \ref{s:lace}--\ref{s:exp-applexp},
we derive the lace expansion for the $r$-point function, thus
identifying the lace-expansion coefficients. In Sections \ref{s:bounds}--\ref{ss:ebd},
we prove the bounds on the coefficients and thus prove
Proposition \ref{thm-psivphibd}.

%%%%%%%%%%%%%%%%%%%%%%%%%%%%%%%%%%%%%%%%%%%%%%%%%%%%%%%%%%%%%%%%%%%%%%%%%%%%%%%%%%%%%
%%%%%%%%%%%%%%%%%%%%%%%%%%%%%%%%%%%%%%%%%%%%%%%%%%%%%%%%%%%%%%%%%%%%%%%%%%%%%%%%%%%%%
%%%%%%%%%%%%%%%%%%%%%%%%%%%%%%%%%%%%%%%%%%%%%%%%%%%%%%%%%%%%%%%%%%%%%%%%%%%%%%%%%%%%%
%%%%%%%%%%%%%%%%%%%%%%%%%%%%%%%%%%%%%%%%%%%%%%%%%%%%%%%%%%%%%%%%%%%%%%%%%%%%%%%%%%%%%
%\input{rpt2}
% rpt2.tex

% August 22, 2008, RvdH
% August 20, 2008, RvdH

% June 1, 2007, RvdH
% January  5, 2004, RvdH
% July  17, 2003, RvdH
% July  16, 2003, RvdH

%%%%%%%%%%%%%%%%%%%%%%%%%%%%%%%%%%%%%%%%%%%%%%%%%%%%%%%%%%%%%%%%%%%%%%%%
%%%%%%%%%%%%%%%%%%%%%%%%%%%%%%%%%%%%%%%%%%%%%%%%%%%%%%%%%%%%%%%%%%%%%%%%
%%%%%%%%%%%%%%%%%%%%%%%%%%%%%%%%%%%%%%%%%%%%%%%%%%%%%%%%%%%%%%%%%%%%%%%%
%%%%%%%%%%%%%%%%%%%%%%%%%%%%%%%%%%%%%%%%%%%%%%%%%%%%%%%%%%%%%%%%%%%%%%%%
\section{Outline of the proof}
\label{sec-outline}
In this section, we give an outline of the proof of
Theorem~\ref{thm:rpt}, and reduce this proof to
Propositions~\ref{thm-psivphibd} and \ref{prop-disc}. This section is organized
as follows.  In Section~\ref{ss:discretize}, we describe the
time-discretized contact process.  In Section~\ref{ss:expans}, we outline
the lace expansion for the $r$-point functions and state the bounds on the
coefficients in Proposition~\ref{thm-psivphibd}.  In Section~\ref{sec-proofr=3},
we prove Theorem~\ref{thm:rpt} for the time-discretized
contact process subject to
Propositions~\ref{thm-psivphibd}.
Finally, in Section~\ref{sec-outlinedisc}, we prove
Proposition~\ref{prop-disc}, and complete the proof of Theorem~\ref{thm:rpt}
for the contact process.

\subsection{Discretization}\label{ss:discretize}
In this section, we introduce the discretized contact
process, which is an interpolation between oriented
percolation on the \ch{one} hand, and the contact process on the other.
This section contains the same material as \cite[\ch{Section 2.1}]{hsa04}.
\ch{We shall also use the notation $\N=\{1,2,\ldots\}$,
$\Z_+=\{0\}\Dcup{}{}\N$ and $\R_+=[0,\infty)$.}

The contact process can be constructed using a graphical
representation as follows.  We consider $\Zd\times\mR_+$ as
space-time.  Along each time line $\{x\}\times\mR_+$, we place
points according to a Poisson process with intensity 1,
independently of the other time lines.  For each ordered pair of
distinct time lines from $\{x\}\times\mR_+$ to $\{y\}\times\mR_+$,
we place directed bonds $((x,\,t),(y,\,t))$, $t\geq0$, according to
a Poisson process with intensity $\lamb\,D(y-x)$,
independently of the other Poisson processes.  A site $(x,s)$ is
said to be {\it connected to} $(y,t)$ if either $(x,s)= (y,t)$ or
there is a non-zero path in $\Zd\times\mR_+$ from $(x,s)$ to
$(y,t)$ using the Poisson bonds and time line segments traversed
in the increasing time direction without traversing the Poisson
points.  The law of $\{\bC_t\}_{t\in\R_+}$ defined in Section~\ref{ss:results} is
equal to that of $\big\{\{x\in \Zd:(o,\,0)$ is connected to
$(x,\,t)\}\big\}_{t\in\R_+}$.

We follow \cite{s01} and consider
an oriented percolation process in $\Zd\times\vep\Zp$ with
$\vep \in(0,1]$ being a discretization parameter as follows.
A directed pair $b=((x,t),(y,t+\vep))$ of sites in $\Zd\times\vep\Zp$ is called a
{\it bond}.  In particular, $b$ is said to be {\it temporal} if
$x=y$, otherwise {\it spatial}.  Each bond is either {\it
occupied} or {\it vacant} independently of the other bonds, and a
bond $b=((x,t),(y,t+\vep))$ is occupied with probability
    \eq
    \lbeq{bprob} p_\vep(y-x)=\begin{cases}
    1-\vep,&\mbox{if }x=y,\\
    \lamb\vep\,D(y-x),&\mbox{otherwise},
    \end{cases}
    \en
provided that $\sup_x p_{\vep}(x)\leq1$.  We denote the
associated probability measure by $\mP_\vep^\lamb$.  It has been proved in
\cite{bg91} that $\mP_\vep^\lamb$ weakly converges to $\mP^\lamb$ as
$\vep\daw0$. See Figure \ref{fig-1} for a graphical representation
of the contact process and the discretized contact process.
As explained in more detail in Section \ref{ss:expans},
we prove our main results by proving the results first for the
discretized contact process, and then taking the continuum limit
\ch{$\vep\downarrow 0$}.

%%%%%FIGFIGFIGFIGFIGFIGFIGFIGFIGFIGFIGFIGFIGFIGFIGFIGFIGFIGFIG
%%%%%FIGFIGFIGFIGFIGFIGFIGFIGFIGFIGFIGFIGFIGFIGFIGFIGFIGFIGFIG
\begin{figure}
%\vskip-3.5cm
\begin{center}
\end{center}
\begin{center}
\includegraphics[scale = 0.45]{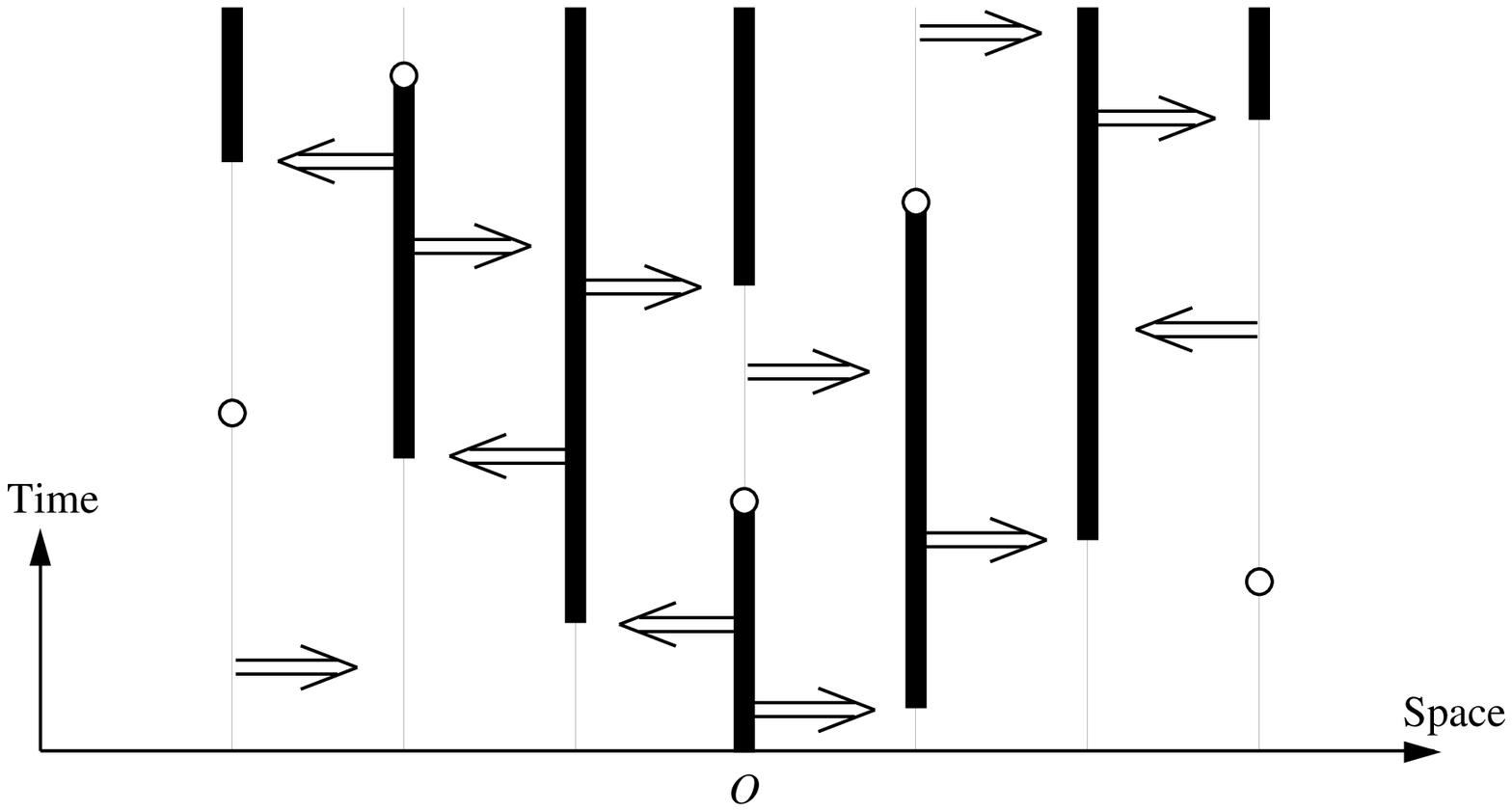}\hskip 1cm
\includegraphics[scale = 0.45]{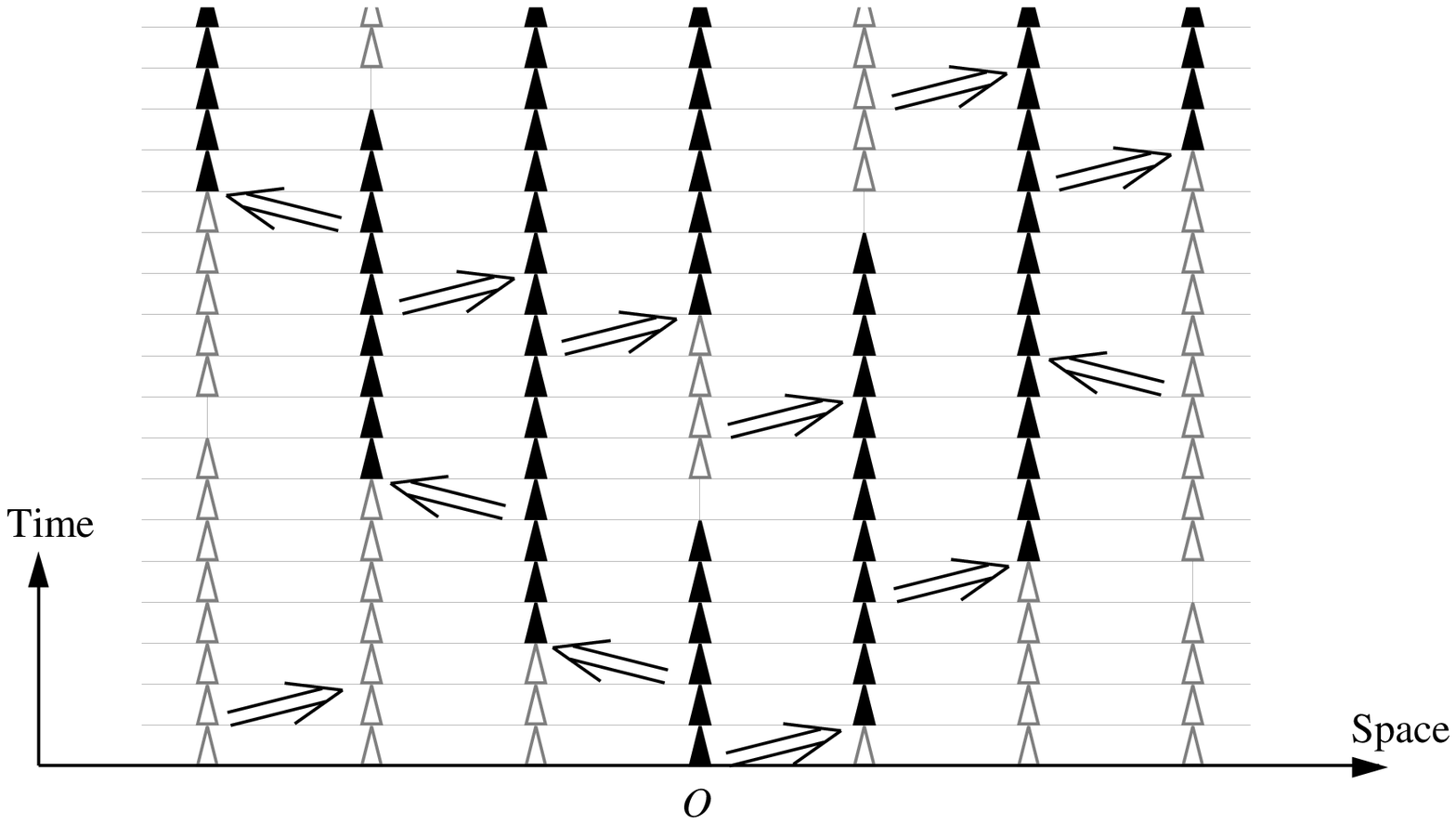}
\label{fig-1}
\caption{Graphical representation of the contact process and the discretized
contact process.}
\end{center}
\end{figure}
%%%%%FIGFIGFIGFIGFIGFIGFIGFIGFIGFIGFIGFIGFIGFIGFIGFIGFIGFIGFIG
%%%%%FIGFIGFIGFIGFIGFIGFIGFIGFIGFIGFIGFIGFIGFIGFIGFIGFIGFIGFIG

We denote by $(x,s)\conn(y,t)$ the event that $(x,s)$ is {\it
connected to} $(y,t)$, i.e., either $(x,s)=(y,t)$ or there is a
non-zero path in $\Zd\times\vep\Zp$ from $(x,s)$ to $(y,t)$
consisting of occupied bonds.  The {\it $r$-point functions},
for $r\ge2$, $\vec t=(t_1,\ldots,t_{r-1})\in\vep\Z_+^{r-1}$ and
$\vec x=(x_1,\ldots,x_{r-1})\in\mZ^{d(r-1)}$, are defined as
\begin{align}\lbeq{disc2pt-def}
\tau_{\vec t;\vep}^{\sss(r)}(\vec x)=\mP_\vep^\lamb\big((o,0)\conn(x_i,t_i)
 ~~\forall i=1,\ldots,r-1\big).
\end{align}
Similarly to \refeq{lambc-def}, the discretized contact process
has a critical value $\lambda_{\rm c}^{\smallsup{\vep}}$ satisfying
\begin{align}\lbeq{lambcvep-def} %\int_0^\infty
\vep\sum_{t\in\vep\Zp}\wtau_{t;\vep}^\lamb(0)
 \begin{cases}
 <\infty,&\mbox{if }\lamb<\lambc^{\sss(\vep)},\\
 =\infty,&\mbox{otherwise},
 \end{cases}&&
\lim_{t\uaw\infty}\mP_\vep^\lamb(\bC_t\ne\varnothing)
\begin{cases}
    =0,&\mbox{if }\lamb\le\lambc^{\sss(\vep)},\\
    >0,&\mbox{otherwise}.
\end{cases}
\end{align}
The discretization procedure will be essential in order to derive the
lace expansion for the $r$-point functions for $r\geq 3$, as it was
for the 2-point function in \cite{hsa04}.

Note that for $\vep=1$ the discretized contact process is simply oriented
percolation. Our main result for the discretized contact process is the following
theorem, similar to Theorem \ref{thm:rpt}:

\begin{theorem}[{\bf Convergence of time-discretized $r$-point functions to SBM moment measures}]
\label{thm:rptvep}
~
\begin{enumerate}[(i)]
\item
Let $d>4$, $\lamb=\lambce$, $\vk\in\mR^{d(r-1)}$, $\vt\in(0,\infty)^{r-1}$,
$\delta\in(0,1\wedge\Delta\wedge\frac{d-4}2)$ and $L\gg1$, as in
Theorem~\ref{thm:2pt}(i).  There exist $A^{\sss(\vep)}=A^{\sss(\vep)}(d,L)$,
$v^{\sss(\vep)}=v^{\sss(\vep)}(d,L)$, $V^{\sss(\vep)}=V^{\sss(\vep)}(d,L)$
such that, for every $r\geq 2$ and as $T\uparrow\infty$,
\begin{align}\lbeq{taurasyvep}
\wtau_{T\vt}^{\sss(r)}\big(\tfrac{\vk}{\sqrt{v\sigma^2T}}\big)=A^{\sss(\vep)}
 \big((A^{\sss(\vep)})^2V^{\sss(\vep)}T\big)^{r-2}\Big(\wM_{\vt}^{\sss(r-1)}
 (\vk)+O(T^{-\delta})\Big),
\end{align}
uniformly in $\vep\in(0,1]$ and $\vk\in\mR^{d(r-1)}$ with
$\sum_{i=1}^{r-1}|k_i|^2$ bounded.  Moreover, for any $\vep\in(0,1]$,
\begin{align}\lbeq{estimates-discr}
\lambce=1+O(L^{-d}),&&&&
A^{\sss(\vep)}=1+O(L^{-d}),&&&&
v^{\sss(\vep)}=1+O(L^{-d}),&&&&
V^{\sss(\vep)}=2-\vep+O(L^{-d}).
\end{align}
\item
Let $d\le4$, $\vk\in\mR^{d(r-1)}$, $\vt\in(0,\infty)^{r-1}$ and
let $\delta,L_1,\lamb_\sT,\mu$ be as in Theorem~\ref{thm:2pt}(ii).
For every $r\ge2$, $0<\max_is_i\le\log T$ and as $T\uparrow\infty$,
    \begin{align}\lbeq{taurasyveplowd}
    \wtau_{T\vt}^{\sss(r)}\big(\tfrac{\vk}{\sqrt{\sigma_{\sT}^2T}}\big)=\big(
     (2-\vep)T\big)^{r-2}\Big(\wM_{\vt}^{\sss(r-1)}(\vk)+O(T^{-\mu\wedge\delta})
     \Big),
    \end{align}
uniformly in $\vep\in(0,1]$ and $\vk\in\mR^{d(r-1)}$ with
$\sum_{i=1}^{r-1}|k_i|^2$ bounded.
\end{enumerate}
\end{theorem}

For $r=2$, the claims in Theorem~\ref{thm:rptvep} were proved in
\cite[Propositions~2.1--2.2]{hsa04}.  We will only prove the
statements for $r\ge3$.

For oriented percolation for which $\vep=1$, Theorem
\ref{thm:rptvep}(i) reproves \cite[Theorem 1.2]{hs02}.
The uniformity in $\vep$ in Theorem \ref{thm:rptvep} is crucial in order for
the continuum limit \ch{$\vep\downarrow 0$}
to be performed, and to extend the results to the contact
process.

%%%%%%%%%%%%%%%%%%%%%%%%%%%%%%%%%%%%%%%%%%%%%%%%%%%%%%%%%%%%%%%%%%%%%%%%
\subsection{Overview of the expansion for the higher-point functions}
\label{ss:expans}

In this section, we give an introduction to the expansion
methods of Sections~\ref{s:lace}--\ref{s:exp-applexp}.  For this, it will be
convenient to introduce new notation for sites in $\Z^d \times
\varepsilon\Z_+$. We write
\begin{align}
\Lambda=\Zd\times\vep\Zp,
\end{align}
and we write a typical element of $\Lambda$ as $\xvec$ rather than
$(x,t)$ as was used until now. We fix $\lambda = \lambce$
throughout Section~\ref{ss:expans} for simplicity,
though the discussion also applies without change when
$\lambda<\lambce$.  We begin by discussing the underlying
philosophy of the expansion. This philosophy is identical to the one
described in \cite[Section 2.2.1]{hs01}.

As explained in more detail in \cite{hsa04}, the basic picture underlying the
expansion for the 2-point function is that a cluster connecting $\ovec$ and
$\xvec$ can be viewed as a string of sausages.  In this picture, the strings
joining sausages are the occupied pivotal bonds for the connection from
$\ovec$ to $\xvec$.  Pivotal bonds are the essential bonds for the connection
from $\ovec$ to $\xvec$, in the sense that each occupied path from $\ovec$ to
$\xvec$ must use all the pivotal bonds.  Naturally, these pivotal bonds are
ordered in time.  Each sausage corresponds to an occupied cluster from the
endpoint of a pivotal bond, containing the starting point of the next pivotal
bond.  Moreover, a sausage consists of two parts: the backbone, which is the
set of sites that are along occupied paths from the top of the lower pivotal
bond to the bottom of the upper pivotal bond, and the hairs, which are the
parts of the cluster that are not connected to the bottom of the upper pivotal
bond.  The backbone may consist of a single site, but may also consist of
sites on at least two bond-disjoint connections.  We say that both these cases
correspond to double connections.  We now extend this picture to the
higher-point functions.

For connections from the origin to multiple points
$\vec\xvec=(\xvec_1,\ldots,\xvec_{r-1})$, the corresponding picture is a
``tree of sausages'' as depicted in Figure~\ref{fig-rpt1}.  In the tree of
sausages, the strings represent the union over $i=1,\ldots,r-1$ of the
occupied pivotal bonds for the connections $\ovec\conn\xvec_i$, and the
sausages are again parts of the cluster between successive pivotal bonds.
Some of them may be pivotal for $\{\ovec\conn\xvec_j~\forall j\in J\}$,
while others are pivotal only for $\{\ovec\conn\xvec_j\}$ for some $j\in J$.

We regard this picture as corresponding to a kind of branching random walk.
In this correspondence, the steps of the walk are the pivotal bonds, while
the sites of the walk are the backbones between subsequent pivotal bonds.
Of course, the pivotal bonds introduce an avoidance interaction on the
branching random walk. Indeed, the sausages are not allowed to share sites
with the later backbones (since otherwise the pivotal bonds in between would
not be pivotal).

When $d>4$ or when $d\le4$ and the range of the contact process is
sufficiently large as described in \refeq{Lt-def}--\refeq{alphadef}, the
interaction is weak and, in particular, the different parts of the backbone in
between different pivotal bonds are small and the steps of the walk are
effectively independent.  Thus, we can think of the higher-point functions of
the critical time-discretized contact process as ``small perturbations" of the
higher-point functions of critical branching random walk.  We will use this
picture now to give an informal overview of the expansions we will derive in
Sections~\ref{s:lace}--\ref{s:exp-applexp}.

%%%%%FIGFIGFIGFIGFIGFIGFIGFIGFIGFIGFIGFIGFIGFIGFIGFIGFIGFIGFIGFIG
%%%%%FIGFIGFIGFIGFIGFIGFIGFIGFIGFIGFIGFIGFIGFIGFIGFIGFIGFIGFIGFIG
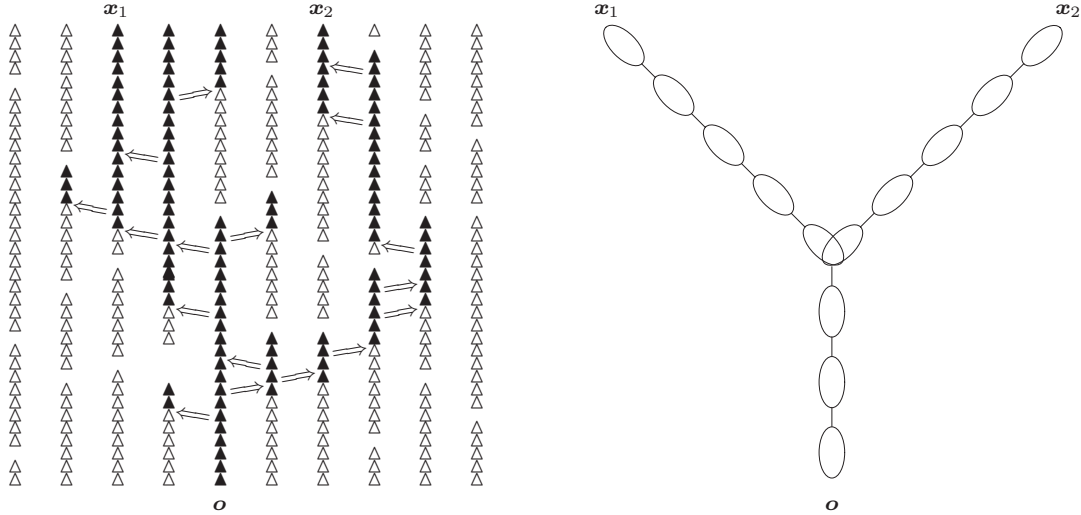
\begin{figure}[t]
\begin{center}

\setlength{\unitlength}{0.00067in}
{
\begin{picture}(-6400,-2300)(1000,2000)
{\put(-6400,-1700){\makebox(0,0)[lb]{\wtr}}
\put(-6400,-1600){\makebox(0,0)[lb]{\wtr}}
%\put(-6400,-1500){\makebox(0,0)[lb]{\wtr}}
\put(-6400,-1400){\makebox(0,0)[lb]{\wtr}}
\put(-6400,-1300){\makebox(0,0)[lb]{\wtr}}
\put(-6400,-1200){\makebox(0,0)[lb]{\wtr}}
\put(-6400,-1100){\makebox(0,0)[lb]{\wtr}}
\put(-6400,-1000){\makebox(0,0)[lb]{\wtr}}
\put(-6400,-900){\makebox(0,0)[lb]{\wtr}}
\put(-6400,-800){\makebox(0,0)[lb]{\wtr}}
\put(-6400,-700){\makebox(0,0)[lb]{\wtr}}
%\put(-6400,-600){\makebox(0,0)[lb]{\wtr}}
\put(-6400,-500){\makebox(0,0)[lb]{\wtr}}
\put(-6400,-400){\makebox(0,0)[lb]{\wtr}}
\put(-6400,-300){\makebox(0,0)[lb]{\wtr}}
\put(-6400,-200){\makebox(0,0)[lb]{\wtr}}
\put(-6400,-100){\makebox(0,0)[lb]{\wtr}}
\put(-6400,0){\makebox(0,0)[lb]{\wtr}}
\put(-6400,100){\makebox(0,0)[lb]{\wtr}}
\put(-6400,200){\makebox(0,0)[lb]{\wtr}}
\put(-6400,300){\makebox(0,0)[lb]{\wtr}}
\put(-6400,400){\makebox(0,0)[lb]{\wtr}}
\put(-6400,500){\makebox(0,0)[lb]{\wtr}}
\put(-6400,600){\makebox(0,0)[lb]{\wtr}}
\put(-6400,700){\makebox(0,0)[lb]{\wtr}}
\put(-6400,800){\makebox(0,0)[lb]{\wtr}}
\put(-6400,900){\makebox(0,0)[lb]{\wtr}}
\put(-6400,1000){\makebox(0,0)[lb]{\wtr}}
\put(-6400,1100){\makebox(0,0)[lb]{\wtr}}
\put(-6400,1200){\makebox(0,0)[lb]{\wtr}}
\put(-6400,1300){\makebox(0,0)[lb]{\wtr}}
%\put(-6400,1400){\makebox(0,0)[lb]{\wtr}}
\put(-6400,1500){\makebox(0,0)[lb]{\wtr}}
\put(-6400,1600){\makebox(0,0)[lb]{\wtr}}
\put(-6400,1700){\makebox(0,0)[lb]{\wtr}}
\put(-6400,1800){\makebox(0,0)[lb]{\wtr}}

\put(-6000,-1700){\makebox(0,0)[lb]{\wtr}}
\put(-6000,-1600){\makebox(0,0)[lb]{\wtr}}
\put(-6000,-1500){\makebox(0,0)[lb]{\wtr}}
\put(-6000,-1400){\makebox(0,0)[lb]{\wtr}}
\put(-6000,-1300){\makebox(0,0)[lb]{\wtr}}
\put(-6000,-1200){\makebox(0,0)[lb]{\wtr}}
\put(-6000,-1100){\makebox(0,0)[lb]{\wtr}}
\put(-6000,-1000){\makebox(0,0)[lb]{\wtr}}
%\put(-6000,-900){\makebox(0,0)[lb]{\wtr}}
\put(-6000,-800){\makebox(0,0)[lb]{\wtr}}
\put(-6000,-700){\makebox(0,0)[lb]{\wtr}}
\put(-6000,-600){\makebox(0,0)[lb]{\wtr}}
\put(-6000,-500){\makebox(0,0)[lb]{\wtr}}
\put(-6000,-400){\makebox(0,0)[lb]{\wtr}}
\put(-6000,-300){\makebox(0,0)[lb]{\wtr}}
%\put(-6000,-200){\makebox(0,0)[lb]{\wtr}}
\put(-6000,-100){\makebox(0,0)[lb]{\wtr}}
\put(-6000,0){\makebox(0,0)[lb]{\wtr}}
\put(-6000,100){\makebox(0,0)[lb]{\wtr}}
\put(-6000,200){\makebox(0,0)[lb]{\wtr}}
\put(-6000,300){\makebox(0,0)[lb]{\wtr}}
\put(-6000,400){\makebox(0,0)[lb]{\wtr}}
\put(-6000,500){\makebox(0,0)[lb]{\btr}}
\put(-6000,600){\makebox(0,0)[lb]{\btr}}
\put(-6000,700){\makebox(0,0)[lb]{\btr}}
%\put(-6000,800){\makebox(0,0)[lb]{\wtr}}
\put(-6000,900){\makebox(0,0)[lb]{\wtr}}
\put(-6000,1000){\makebox(0,0)[lb]{\wtr}}
\put(-6000,1100){\makebox(0,0)[lb]{\wtr}}
\put(-6000,1200){\makebox(0,0)[lb]{\wtr}}
\put(-6000,1300){\makebox(0,0)[lb]{\wtr}}
\put(-6000,1400){\makebox(0,0)[lb]{\wtr}}
\put(-6000,1500){\makebox(0,0)[lb]{\wtr}}
\put(-6000,1600){\makebox(0,0)[lb]{\wtr}}
\put(-6000,1700){\makebox(0,0)[lb]{\wtr}}
\put(-6000,1800){\makebox(0,0)[lb]{\wtr}}

\put(-5600,-1700){\makebox(0,0)[lb]{\wtr}}
\put(-5600,-1600){\makebox(0,0)[lb]{\wtr}}
\put(-5600,-1500){\makebox(0,0)[lb]{\wtr}}
\put(-5600,-1400){\makebox(0,0)[lb]{\wtr}}
\put(-5600,-1300){\makebox(0,0)[lb]{\wtr}}
\put(-5600,-1200){\makebox(0,0)[lb]{\wtr}}
\put(-5600,-1100){\makebox(0,0)[lb]{\wtr}}
\put(-5600,-1000){\makebox(0,0)[lb]{\wtr}}
\put(-5600,-900){\makebox(0,0)[lb]{\wtr}}
%\put(-5600,-800){\makebox(0,0)[lb]{\wtr}}
\put(-5600,-700){\makebox(0,0)[lb]{\wtr}}
\put(-5600,-600){\makebox(0,0)[lb]{\wtr}}
\put(-5600,-500){\makebox(0,0)[lb]{\wtr}}
\put(-5600,-400){\makebox(0,0)[lb]{\wtr}}
\put(-5600,-300){\makebox(0,0)[lb]{\wtr}}
\put(-5600,-200){\makebox(0,0)[lb]{\wtr}}
\put(-5600,-100){\makebox(0,0)[lb]{\wtr}}
%\put(-5600,0){\makebox(0,0)[lb]{\wtr}}
\put(-5600,100){\makebox(0,0)[lb]{\wtr}}
\put(-5600,200){\makebox(0,0)[lb]{\wtr}}
\put(-5600,300){\makebox(0,0)[lb]{\btr}}
\put(-5600,400){\makebox(0,0)[lb]{\btr}}
\put(-5600,500){\makebox(0,0)[lb]{\btr}}
\put(-5600,600){\makebox(0,0)[lb]{\btr}}
\put(-5600,700){\makebox(0,0)[lb]{\btr}}
\put(-5600,800){\makebox(0,0)[lb]{\btr}}
\put(-5600,900){\makebox(0,0)[lb]{\btr}}
\put(-5600,1000){\makebox(0,0)[lb]{\btr}}
\put(-5600,1100){\makebox(0,0)[lb]{\btr}}
\put(-5600,1200){\makebox(0,0)[lb]{\btr}}
\put(-5600,1300){\makebox(0,0)[lb]{\btr}}
\put(-5600,1400){\makebox(0,0)[lb]{\btr}}
\put(-5600,1500){\makebox(0,0)[lb]{\btr}}
\put(-5600,1600){\makebox(0,0)[lb]{\btr}}
\put(-5600,1700){\makebox(0,0)[lb]{\btr}}
\put(-5600,1800){\makebox(0,0)[lb]{\btr}}

\put(-5200,-1700){\makebox(0,0)[lb]{\wtr}}
\put(-5200,-1600){\makebox(0,0)[lb]{\wtr}}
\put(-5200,-1500){\makebox(0,0)[lb]{\wtr}}
\put(-5200,-1400){\makebox(0,0)[lb]{\wtr}}
\put(-5200,-1300){\makebox(0,0)[lb]{\wtr}}
\put(-5200,-1200){\makebox(0,0)[lb]{\wtr}}
\put(-5200,-1100){\makebox(0,0)[lb]{\btr}}
\put(-5200,-1000){\makebox(0,0)[lb]{\btr}}
\put(-5200,-90){\makebox(0,0)[lb]{\btr}}
\put(-5200,-80){\makebox(0,0)[lb]{\btr}}
%\put(-5200,-700){\makebox(0,0)[lb]{\wtr}}
\put(-5200,-600){\makebox(0,0)[lb]{\wtr}}
\put(-5200,-500){\makebox(0,0)[lb]{\wtr}}
\put(-5200,-400){\makebox(0,0)[lb]{\wtr}}
\put(-5200,-300){\makebox(0,0)[lb]{\btr}}
\put(-5200,-200){\makebox(0,0)[lb]{\btr}}
\put(-5200,-100){\makebox(0,0)[lb]{\btr}}
\put(-5200,0){\makebox(0,0)[lb]{\btr}}
\put(-5200,100){\makebox(0,0)[lb]{\btr}}
\put(-5200,200){\makebox(0,0)[lb]{\btr}}
\put(-5200,300){\makebox(0,0)[lb]{\btr}}
\put(-5200,400){\makebox(0,0)[lb]{\btr}}
\put(-5200,500){\makebox(0,0)[lb]{\btr}}
\put(-5200,600){\makebox(0,0)[lb]{\btr}}
\put(-5200,700){\makebox(0,0)[lb]{\btr}}
\put(-5200,800){\makebox(0,0)[lb]{\btr}}
\put(-5200,900){\makebox(0,0)[lb]{\btr}}
\put(-5200,1000){\makebox(0,0)[lb]{\btr}}
\put(-5200,1100){\makebox(0,0)[lb]{\btr}}
\put(-5200,1200){\makebox(0,0)[lb]{\btr}}
\put(-5200,1300){\makebox(0,0)[lb]{\btr}}
\put(-5200,1400){\makebox(0,0)[lb]{\btr}}
\put(-5200,1500){\makebox(0,0)[lb]{\btr}}
\put(-5200,1600){\makebox(0,0)[lb]{\btr}}
\put(-5200,1700){\makebox(0,0)[lb]{\btr}}
\put(-5200,1800){\makebox(0,0)[lb]{\btr}}

\put(-4800,-1700){\makebox(0,0)[lb]{\btr}}
\put(-4800,-1600){\makebox(0,0)[lb]{\btr}}
\put(-4800,-1500){\makebox(0,0)[lb]{\btr}}
\put(-4800,-1400){\makebox(0,0)[lb]{\btr}}
\put(-4800,-1300){\makebox(0,0)[lb]{\btr}}
\put(-4800,-1200){\makebox(0,0)[lb]{\btr}}
\put(-4800,-1100){\makebox(0,0)[lb]{\btr}}
\put(-4800,-1000){\makebox(0,0)[lb]{\btr}}
\put(-4800,-900){\makebox(0,0)[lb]{\btr}}
\put(-4800,-800){\makebox(0,0)[lb]{\btr}}
\put(-4800,-700){\makebox(0,0)[lb]{\btr}}
\put(-4800,-600){\makebox(0,0)[lb]{\btr}}
\put(-4800,-500){\makebox(0,0)[lb]{\btr}}
\put(-4800,-400){\makebox(0,0)[lb]{\btr}}
\put(-4800,-300){\makebox(0,0)[lb]{\btr}}
\put(-4800,-200){\makebox(0,0)[lb]{\btr}}
\put(-4800,-100){\makebox(0,0)[lb]{\btr}}
\put(-4800,0){\makebox(0,0)[lb]{\btr}}
\put(-4800,100){\makebox(0,0)[lb]{\btr}}
\put(-4800,200){\makebox(0,0)[lb]{\btr}}
\put(-4800,300){\makebox(0,0)[lb]{\btr}}
%\put(-4800,400){\makebox(0,0)[lb]{\wtr}}
\put(-4800,500){\makebox(0,0)[lb]{\wtr}}
\put(-4800,600){\makebox(0,0)[lb]{\wtr}}
\put(-4800,700){\makebox(0,0)[lb]{\wtr}}
\put(-4800,800){\makebox(0,0)[lb]{\wtr}}
\put(-4800,900){\makebox(0,0)[lb]{\wtr}}
\put(-4800,1000){\makebox(0,0)[lb]{\wtr}}
\put(-4800,1100){\makebox(0,0)[lb]{\wtr}}
\put(-4800,1200){\makebox(0,0)[lb]{\wtr}}
\put(-4800,1300){\makebox(0,0)[lb]{\wtr}}
\put(-4800,1400){\makebox(0,0)[lb]{\btr}}
\put(-4800,1500){\makebox(0,0)[lb]{\btr}}
\put(-4800,1600){\makebox(0,0)[lb]{\btr}}
\put(-4800,1700){\makebox(0,0)[lb]{\btr}}
\put(-4800,1800){\makebox(0,0)[lb]{\btr}}

\put(-4400,-1700){\makebox(0,0)[lb]{\wtr}}
\put(-4400,-1600){\makebox(0,0)[lb]{\wtr}}
\put(-4400,-1500){\makebox(0,0)[lb]{\wtr}}
\put(-4400,-1400){\makebox(0,0)[lb]{\wtr}}
\put(-4400,-1300){\makebox(0,0)[lb]{\wtr}}
\put(-4400,-1200){\makebox(0,0)[lb]{\wtr}}
\put(-4400,-1100){\makebox(0,0)[lb]{\wtr}}
\put(-4400,-1000){\makebox(0,0)[lb]{\btr}}
\put(-4400,-900){\makebox(0,0)[lb]{\btr}}
\put(-4400,-800){\makebox(0,0)[lb]{\btr}}
\put(-4400,-700){\makebox(0,0)[lb]{\btr}}
\put(-4400,-600){\makebox(0,0)[lb]{\btr}}
%\put(-4400,-500){\makebox(0,0)[lb]{\wtr}}
\put(-4400,-400){\makebox(0,0)[lb]{\wtr}}
\put(-4400,-300){\makebox(0,0)[lb]{\wtr}}
\put(-4400,-200){\makebox(0,0)[lb]{\wtr}}
\put(-4400,-100){\makebox(0,0)[lb]{\wtr}}
\put(-4400,0){\makebox(0,0)[lb]{\wtr}}
\put(-4400,100){\makebox(0,0)[lb]{\wtr}}
\put(-4400,200){\makebox(0,0)[lb]{\wtr}}
\put(-4400,300){\makebox(0,0)[lb]{\btr}}
\put(-4400,400){\makebox(0,0)[lb]{\btr}}
\put(-4400,500){\makebox(0,0)[lb]{\btr}}
%\put(-4400,600){\makebox(0,0)[lb]{\wtr}}
\put(-4400,700){\makebox(0,0)[lb]{\wtr}}
\put(-4400,800){\makebox(0,0)[lb]{\wtr}}
\put(-4400,900){\makebox(0,0)[lb]{\wtr}}
\put(-4400,1000){\makebox(0,0)[lb]{\wtr}}
\put(-4400,1100){\makebox(0,0)[lb]{\wtr}}
\put(-4400,1200){\makebox(0,0)[lb]{\wtr}}
\put(-4400,1300){\makebox(0,0)[lb]{\wtr}}
\put(-4400,1400){\makebox(0,0)[lb]{\wtr}}
%\put(-4400,1500){\makebox(0,0)[lb]{\wtr}}
\put(-4400,1600){\makebox(0,0)[lb]{\wtr}}
\put(-4400,1700){\makebox(0,0)[lb]{\wtr}}
\put(-4400,1800){\makebox(0,0)[lb]{\wtr}}

\put(-4000,-1700){\makebox(0,0)[lb]{\wtr}}
\put(-4000,-1600){\makebox(0,0)[lb]{\wtr}}
\put(-4000,-1500){\makebox(0,0)[lb]{\wtr}}
\put(-4000,-1400){\makebox(0,0)[lb]{\wtr}}
\put(-4000,-1300){\makebox(0,0)[lb]{\wtr}}
\put(-4000,-1200){\makebox(0,0)[lb]{\wtr}}
\put(-4000,-1100){\makebox(0,0)[lb]{\wtr}}
\put(-4000,-1000){\makebox(0,0)[lb]{\wtr}}
\put(-4000,-900){\makebox(0,0)[lb]{\btr}}
\put(-4000,-800){\makebox(0,0)[lb]{\btr}}
\put(-4000,-700){\makebox(0,0)[lb]{\btr}}
\put(-4000,-600){\makebox(0,0)[lb]{\btr}}
%\put(-4000,-500){\makebox(0,0)[lb]{\wtr}}
\put(-4000,-400){\makebox(0,0)[lb]{\wtr}}
\put(-4000,-300){\makebox(0,0)[lb]{\wtr}}
\put(-4000,-200){\makebox(0,0)[lb]{\wtr}}
\put(-4000,-100){\makebox(0,0)[lb]{\wtr}}
\put(-4000,0){\makebox(0,0)[lb]{\wtr}}
%\put(-4000,100){\makebox(0,0)[lb]{\wtr}}
\put(-4000,200){\makebox(0,0)[lb]{\wtr}}
\put(-4000,300){\makebox(0,0)[lb]{\wtr}}
\put(-4000,400){\makebox(0,0)[lb]{\wtr}}
\put(-4000,500){\makebox(0,0)[lb]{\wtr}}
\put(-4000,600){\makebox(0,0)[lb]{\wtr}}
\put(-4000,700){\makebox(0,0)[lb]{\wtr}}
\put(-4000,800){\makebox(0,0)[lb]{\wtr}}
\put(-4000,900){\makebox(0,0)[lb]{\wtr}}
\put(-4000,1000){\makebox(0,0)[lb]{\wtr}}
\put(-4000,1100){\makebox(0,0)[lb]{\wtr}}
\put(-4000,1200){\makebox(0,0)[lb]{\btr}}
\put(-4000,1300){\makebox(0,0)[lb]{\btr}}
\put(-4000,1400){\makebox(0,0)[lb]{\btr}}
\put(-4000,1500){\makebox(0,0)[lb]{\btr}}
\put(-4000,1600){\makebox(0,0)[lb]{\btr}}
\put(-4000,1700){\makebox(0,0)[lb]{\btr}}
\put(-4000,1800){\makebox(0,0)[lb]{\btr}}

\put(-3600,-1700){\makebox(0,0)[lb]{\wtr}}
\put(-3600,-1600){\makebox(0,0)[lb]{\wtr}}
\put(-3600,-1500){\makebox(0,0)[lb]{\wtr}}
%\put(-3600,-1400){\makebox(0,0)[lb]{\wtr}}
\put(-3600,-1300){\makebox(0,0)[lb]{\wtr}}
\put(-3600,-1200){\makebox(0,0)[lb]{\wtr}}
\put(-3600,-1100){\makebox(0,0)[lb]{\wtr}}
\put(-3600,-1000){\makebox(0,0)[lb]{\wtr}}
\put(-3600,-900){\makebox(0,0)[lb]{\wtr}}
\put(-3600,-800){\makebox(0,0)[lb]{\wtr}}
\put(-3600,-700){\makebox(0,0)[lb]{\wtr}}
\put(-3600,-600){\makebox(0,0)[lb]{\btr}}
\put(-3600,-500){\makebox(0,0)[lb]{\btr}}
\put(-3600,-400){\makebox(0,0)[lb]{\btr}}
\put(-3600,-300){\makebox(0,0)[lb]{\btr}}
\put(-3600,-200){\makebox(0,0)[lb]{\btr}}
\put(-3600,-100){\makebox(0,0)[lb]{\btr}}
%\put(-3600,0){\makebox(0,0)[lb]{\btr}}
\put(-3600,100){\makebox(0,0)[lb]{\wtr}}
\put(-3600,200){\makebox(0,0)[lb]{\btr}}
\put(-3600,300){\makebox(0,0)[lb]{\btr}}
\put(-3600,400){\makebox(0,0)[lb]{\btr}}
\put(-3600,500){\makebox(0,0)[lb]{\btr}}
\put(-3600,600){\makebox(0,0)[lb]{\btr}}
\put(-3600,700){\makebox(0,0)[lb]{\btr}}
\put(-3600,800){\makebox(0,0)[lb]{\btr}}
\put(-3600,900){\makebox(0,0)[lb]{\btr}}
\put(-3600,1000){\makebox(0,0)[lb]{\btr}}
\put(-3600,1100){\makebox(0,0)[lb]{\btr}}
\put(-3600,1200){\makebox(0,0)[lb]{\btr}}
\put(-3600,1300){\makebox(0,0)[lb]{\btr}}
\put(-3600,1400){\makebox(0,0)[lb]{\btr}}
\put(-3600,1500){\makebox(0,0)[lb]{\btr}}
\put(-3600,1600){\makebox(0,0)[lb]{\btr}}
%\put(-3600,1700){\makebox(0,0)[lb]{\wtr}}
\put(-3600,1800){\makebox(0,0)[lb]{\wtr}}

\put(-3200,-1700){\makebox(0,0)[lb]{\wtr}}
\put(-3200,-1600){\makebox(0,0)[lb]{\wtr}}
\put(-3200,-1500){\makebox(0,0)[lb]{\wtr}}
\put(-3200,-1400){\makebox(0,0)[lb]{\wtr}}
\put(-3200,-1300){\makebox(0,0)[lb]{\wtr}}
\put(-3200,-1200){\makebox(0,0)[lb]{\wtr}}
\put(-3200,-1100){\makebox(0,0)[lb]{\wtr}}
\put(-3200,-1000){\makebox(0,0)[lb]{\wtr}}
%\put(-3200,-900){\makebox(0,0)[lb]{\wtr}}
\put(-3200,-800){\makebox(0,0)[lb]{\wtr}}
\put(-3200,-700){\makebox(0,0)[lb]{\wtr}}
\put(-3200,-600){\makebox(0,0)[lb]{\wtr}}
\put(-3200,-500){\makebox(0,0)[lb]{\wtr}}
\put(-3200,-400){\makebox(0,0)[lb]{\wtr}}
\put(-3200,-300){\makebox(0,0)[lb]{\btr}}
\put(-3200,-200){\makebox(0,0)[lb]{\btr}}
\put(-3200,-100){\makebox(0,0)[lb]{\btr}}
\put(-3200,0){\makebox(0,0)[lb]{\btr}}
\put(-3200,100){\makebox(0,0)[lb]{\btr}}
\put(-3200,200){\makebox(0,0)[lb]{\btr}}
\put(-3200,300){\makebox(0,0)[lb]{\btr}}
%\put(-3200,400){\makebox(0,0)[lb]{\wtr}}
\put(-3200,500){\makebox(0,0)[lb]{\wtr}}
\put(-3200,600){\makebox(0,0)[lb]{\wtr}}
\put(-3200,700){\makebox(0,0)[lb]{\wtr}}
%\put(-3200,800){\makebox(0,0)[lb]{\wtr}}
\put(-3200,900){\makebox(0,0)[lb]{\wtr}}
\put(-3200,1000){\makebox(0,0)[lb]{\wtr}}
\put(-3200,1100){\makebox(0,0)[lb]{\wtr}}
%\put(-3200,1200){\makebox(0,0)[lb]{\wtr}}
\put(-3200,1300){\makebox(0,0)[lb]{\wtr}}
\put(-3200,1400){\makebox(0,0)[lb]{\wtr}}
\put(-3200,1500){\makebox(0,0)[lb]{\wtr}}
\put(-3200,1600){\makebox(0,0)[lb]{\wtr}}
\put(-3200,1700){\makebox(0,0)[lb]{\wtr}}
\put(-3200,1800){\makebox(0,0)[lb]{\wtr}}

\put(-2800,-1700){\makebox(0,0)[lb]{\wtr}}
\put(-2800,-1600){\makebox(0,0)[lb]{\wtr}}
\put(-2800,-1500){\makebox(0,0)[lb]{\wtr}}
\put(-2800,-1400){\makebox(0,0)[lb]{\wtr}}
\put(-2800,-1300){\makebox(0,0)[lb]{\wtr}}
%\put(-2800,-1200){\makebox(0,0)[lb]{\wtr}}
\put(-2800,-1100){\makebox(0,0)[lb]{\wtr}}
\put(-2800,-1000){\makebox(0,0)[lb]{\wtr}}
\put(-2800,-900){\makebox(0,0)[lb]{\wtr}}
\put(-2800,-800){\makebox(0,0)[lb]{\wtr}}
\put(-2800,-700){\makebox(0,0)[lb]{\wtr}}
\put(-2800,-600){\makebox(0,0)[lb]{\wtr}}
\put(-2800,-500){\makebox(0,0)[lb]{\wtr}}
\put(-2800,-400){\makebox(0,0)[lb]{\wtr}}
\put(-2800,-300){\makebox(0,0)[lb]{\wtr}}
\put(-2800,-200){\makebox(0,0)[lb]{\wtr}}
\put(-2800,-100){\makebox(0,0)[lb]{\wtr}}
\put(-2800,0){\makebox(0,0)[lb]{\wtr}}
\put(-2800,100){\makebox(0,0)[lb]{\wtr}}
\put(-2800,200){\makebox(0,0)[lb]{\wtr}}
\put(-2800,300){\makebox(0,0)[lb]{\wtr}}
%\put(-2800,400){\makebox(0,0)[lb]{\wtr}}
\put(-2800,500){\makebox(0,0)[lb]{\wtr}}
\put(-2800,600){\makebox(0,0)[lb]{\wtr}}
\put(-2800,700){\makebox(0,0)[lb]{\wtr}}
\put(-2800,800){\makebox(0,0)[lb]{\wtr}}
\put(-2800,900){\makebox(0,0)[lb]{\wtr}}
%\put(-2800,1000){\makebox(0,0)[lb]{\wtr}}
\put(-2800,1100){\makebox(0,0)[lb]{\wtr}}
\put(-2800,1200){\makebox(0,0)[lb]{\wtr}}
\put(-2800,1300){\makebox(0,0)[lb]{\wtr}}
\put(-2800,1400){\makebox(0,0)[lb]{\wtr}}
\put(-2800,1500){\makebox(0,0)[lb]{\wtr}}
\put(-2800,1600){\makebox(0,0)[lb]{\wtr}}
\put(-2800,1700){\makebox(0,0)[lb]{\wtr}}
\put(-2800,1800){\makebox(0,0)[lb]{\wtr}}

\put(-5100,-370){\makebox(0,0)[lb]{\lupa}}
\put(-5100,-1170){\makebox(0,0)[lb]{\lupa}}
\put(-5100,130){\makebox(0,0)[lb]{\lupa}}
\put(-4660,-1020){\makebox(0,0)[lb]{\rupa}}
\put(-4260,-920){\makebox(0,0)[lb]{\rupa}}
\put(-5060,1280){\makebox(0,0)[lb]{\rupa}}
\put(-3860,-720){\makebox(0,0)[lb]{\rupa}}
\put(-3460,-420){\makebox(0,0)[lb]{\rupa}}
\put(-3460,-220){\makebox(0,0)[lb]{\rupa}}
\put(-5500,240){\makebox(0,0)[lb]{\lupa}}
\put(-5900,440){\makebox(0,0)[lb]{\lupa}}
\put(-5500,840){\makebox(0,0)[lb]{\lupa}}
\put(-4700,-770){\makebox(0,0)[lb]{\lupa}}
\put(-4660,180){\makebox(0,0)[lb]{\rupa}}

\put(-3500,130){\makebox(0,0)[lb]{\lupa}}
\put(-3900,1130){\makebox(0,0)[lb]{\lupa}}
\put(-3900,1530){\makebox(0,0)[lb]{\lupa}}

\put(-4800,-1900){\makebox(0,0)[lb]{$\scriptstyle \ovec$}}
\put(-5650,1950){\makebox(0,0)[lb]{$\scriptstyle \xvec_1$}}
\put(-4050,1950){\makebox(0,0)[lb]{$\scriptstyle \xvec_2$}}

\put(30, -1450){\ellipse{200}{400}}
\put(30, -900){\ellipse{200}{400}}
\put(30, -350){\ellipse{200}{400}}
\path(30,-1250)(30,-1100)
\path(30,-700)(30,-550)
\path(30,-150)(30,0)

%\begin{rotate}{0}
%\put(-200,0){\makebox(0,0)[lb]{$~$}}
%\put(0, 0){\ellipse{400}{200}}
%\path(200,0)(350,0)
%\put(550, 0){\ellipse{400}{200}}
%\path(750,0)(900,0)
%\put(1100, 0){\ellipse{400}{200}}
%\path(1300,0)(1450,0)
%\put(1650, 0){\ellipse{400}{200}}
%\put(1850,0){\makebox(0,0)[lb]{$~$}}
%\put(1950,-100){\makebox(0,0)[lb]{0}}
%\end{rotate}

\begin{rotate}{45}
\put(0,0){\makebox(0,0)[lb]{$~$}}
\put(200, 40){\ellipse{400}{200}}
\path(400,40)(550,40)
\put(750, 40){\ellipse{400}{200}}
\path(950,40)(1100,40)
\put(1300, 40){\ellipse{400}{200}}
\path(1500,40)(1650,40)
\put(1850,40){\ellipse{400}{200}}
\path(2050,40)(2200,40)
\put(2400,40){\ellipse{400}{200}}
\put(2050,40){\makebox(0,0)[lb]{$~$}}
%\put(2150,40){\makebox(0,0)[lb]{1}}
\end{rotate}

\begin{rotate}{135}
\put(0,-200){\makebox(0,0)[lb]{$~$}}
\put(200, -40){\ellipse{400}{200}}
\path(400,-40)(550,-40)
\put(750, -40){\ellipse{400}{200}}
\path(950,-40)(1100,-40)
\put(1300, -40){\ellipse{400}{200}}
\path(1500,-40)(1650,-40)
\put(1850,-40){\ellipse{400}{200}}
\path(2050,-40)(2200,-40)
\put(2400,-40){\ellipse{400}{200}}
\put(2050,-40){\makebox(0,0)[lb]{$~$}}
%\put(2150,-40){\makebox(0,0)[lb]{2}}
\end{rotate}

\put(-100,-1900){\makebox(0,0)[lb]{$\scriptstyle \ovec$}}
\put(-1900,1950){\makebox(0,0)[lb]{$\scriptstyle \xvec_1$}}
\put(1700,1950){\makebox(0,0)[lb]{$\scriptstyle \xvec_2$}}
}
\end{picture}
}
\end{center}

\vskip 7cm

\caption{\label{fig-rpt1}
(a) A configuration for the discretized contact process.
Both $\blacktriangle$ and $\vartriangle$ denote occupied temporal bonds;
$\blacktriangle$ is connected from $\ovec$, while $\vartriangle$ is not.
The arrows are occupied spatial bonds, representing the spread of an
infection to neighbours.
(b) Schematic depiction of the configuration as a ``string of sausages.''}
\end{figure}
%%%%%FIGFIGFIGFIGFIGFIGFIGFIGFIGFIGFIGFIGFIGFIGFIGFIGFIGFIGFIG
%%%%%FIGFIGFIGFIGFIGFIGFIGFIGFIGFIGFIGFIGFIGFIGFIGFIGFIGFIGFIG

We start by introducing some notation. For $r \geq 3$, let
\begin{align}\lbeq{JJ1}
J=\{1,2,\ldots,r-1\},&& J_j=J\setminus\{j\}\qquad(j\in J).
\end{align}
For $I=\{i_1,\ldots,i_s\}\subset J$, we write
$\vec\xvec_I=\{\xvec_{i_1},\ldots,\xvec_{i_s}\}$ and
$\vec\xvec_I-\yvec=\{\xvec_{i_1}-\yvec,\ldots,\xvec_{i_s}-\yvec\}$,
and abuse notation by writing
\begin{align}\lbeq{pabuse}
p_\vep((x,t))=p_\vep(x)\,\delta_{t,\vep}.
\end{align}

From the sausage at the origin, there may be anywhere from zero to $r-1$
pivotal bonds for $\{\ovec\conn\vec\xvec_J\}$ emerging, where we let
\begin{align}\lbeq{conn-all}
\{\ovec\conn\vec\xvec_J\}=\{\vvec\conn\xvec_j~~\forall j\in J\}.
\end{align}
Configurations with zero or more than two pivotal bonds will turn out to
constitute an error term.  Indeed, when there are zero pivotal bonds, this
means that $\ovec\dbc\xvec_i$ for some $i$, which constitutes an error term.
When there are more than two pivotal bonds, the sausage at the origin has at
least \emph{three} disjoint connections to different $\xvec_i$'s, which also
turns out to constitute an error term.  Therefore, we are left with
configurations which have one or two branches emerging from the sausage at the
origin.  When there is one branch, then this branch contains $\vec\xvec_J$.
When there are two branches, one branch will contain $\vec\xvec_I$ for some
nonempty $I\subseteq J_1$ and the other branch will contain
$\vec\xvec_{J\setminus I}$, where we require $1\in J\setminus I$ to make the
identification unique.

The first expansion deals with the case where there is a single branch
from the origin.  It serves to decouple the interaction between that single
branch and the branches of the tree of sausages leading to $\vec\xvec_J$.
The expansion writes $\tau(\vec{\xvec}_J)$ in the form
    \begin{align}\lbeq{taunxvecr2}
    \tau(\vec\xvec_J)=A(\vec\xvec_J)+(B\sstar \tau)(\vec\xvec_J)=A(\vec\xvec_J)+
    \sum_{\vvec\in \Lambda}B(\vvec)~\tau(\vec\xvec_J-\vvec),
    \end{align}
where $(f\sstar g)(\xvec)$ represents the space-time convolution of two function
$f, g\colon \Lambda \to \R$ given by
    \eq
    \lbeq{conv}
    (f\sstar g)(\xvec)=\sum_{\yvec\in \Lambda} f(\yvec)g(\xvec-\yvec).
    \en
For details, see
Section~\ref{s:lace}, where \refeq{taunxvecr2} is derived.  We have that
    \begin{align}\lbeq{Bpieq}
    B(\xvec)=(\pi\sstar p_\vep)(\xvec),
    \end{align}
where $\pi(\xvec)$ is the expansion coefficient for the 2-point function
as derived in \cite[Section 3]{hsa04}. Moreover, for $r=2$,
    \begin{align}\lbeq{A/B}
    A(\xvec)=\pi(\xvec),
    \end{align}
so that \refeq{taunxvecr2} becomes
    \begin{align}\lbeq{lace-exp}
    \tau(\xvec)=\pi(\xvec)+(\pi\sstar p_\vep\sstar \tau)(\xvec).
    \end{align}
This is the lace expansion for the 2-point function, which
serves as the key ingredient in the analysis of the 2-point
function in \cite{hsa04}.\footnote{In this paper, we will use a
different expansion for the 2-point function than the one used in
\cite{hsa04}. However, the resulting $\pi(\xvec)$ {\it is} the same,
as $\pi(\xvec)$ is uniquely defined by the equation \refeq{lace-exp}.}

The next step is to write $A(\vec\xvec_J)$ as
    \begin{align}\lbeq{Adec}
    A(\vec\xvec_J)=\sum_{I\subset J_1:I\ne\varnothing}\sum_{\yvec_1}
    B(\yvec_1,\vec\xvec_I)\tau(\vec\xvec_{J\backslash I}-\yvec_1)+a(\vec\xvec_J;1),
    \end{align}
where, to leading order, $J\setminus I$ consists of those $j$ for which the first pivotal of
$\xvec_j$ is the same as the one for $\xvec_1$, while for $i\in I$, this
first pivotal is different.  The equality \refeq{Adec} is the result
of the {\it first expansion for } $A(\vec\xvec_J)$. In this expansion,
we wish to treat the connections from the top of the first pivotal to
$\vec\xvec_{J\backslash I}$ as being independent from the connections from
$\ovec$ to $\vec\xvec_I$ that do not use the first pivotal bond.
In the {\it second expansion} for $A(\vec\xvec_J)$, we wish to
extract a factor $\tau(\vec\xvec_I-\yvec_2)$ for some $\yvec_2$
from the connection from $\ovec$ to $\vec\xvec_I$ that is still present in
$B(\yvec_1,\vec\xvec_I)$. This leads to a result of the form
    \begin{align}\lbeq{BCE}
    \sum_{\yvec_1}
    B(\yvec_1,\vec\xvec_I)\tau(\vec\xvec_{J\backslash I}-\yvec_1)
    =\sum_{\yvec_1, \yvec_2}C(\yvec_1,
    \yvec_2)~\tau(\vec\xvec_{J\setminus I}-\yvec_1)~\tau(\vec\xvec_I-\yvec_2)
    +a(\vec\xvec_{J\backslash I},\vec\xvec_I),
    \end{align}
where $a(\vec\xvec_{J\backslash I}, \vec{\xvec}_I)$
is an error term, and, to first approximation,
$C(\yvec_1,\yvec_2)$ represents the sausage at $\ovec$
together with the
pivotal bonds ending at $\yvec_1$ and $\yvec_2$, with
the two branches removed.
In particular, $C(\yvec_1,\yvec_2)$ is independent of $I$.
The leading contribution to $C(\yvec_1,\yvec_2)$ is
$p_\vep(\yvec_1)\,p_\vep(\yvec_2)$ with $\yvec_1\neq \yvec_2$, corresponding
to the case where the sausage at $\ovec$ is the single point $\ovec$.
For details, see Section~\ref{s:exp-applexp}, where
\refeq{BCE} is derived.

We will use a new expansion for the higher-point functions, which is
a simplification of the expansion for oriented percolation in $\Zd\times\Zp$
in \cite{hs01}. The difference resides mainly in the second expansion,
i.e., the expansion of $A(\vec\xvec_J)$.

In the course of the expansion, in Section \ref{sec-Cvepvep},
we shall also describe a close relation between the expansion coefficients
for the $r$-point functions derived in this paper and the ones
for the survival probability of the descritized contact process
derived in \cite{HHS05b}. See \cite{hsa06} for a more detailed discussion of this
relation and its consequences.

%which can be seen
%as an expansion for the two-point function. In fact, \refeq{lace-exp} will follow
%immediately from this first expansion. In \cite{hs01}, the Hara-Slade expansion
%was used for this first expansion, while we use the simpler inclusion-exclusion
%expansion by Sakai \cite{s01,s02}. As we will see in the course of the expansion,
%these expansions are very close in spirit. The Sakai expansion uses the Markov property,
%which is valid for general oriented percolation models, while the Hara-Slade expansion
%uses a factorization lemma, which is based on the independence of bonds, and is valid
%also for non-oriented percolation.
%
%After the first expansion, we will use Hara-Slade expansion steps to make
%the disjoint connections to $\vec{\xvec}_{J\setminus I}$ and to $\vec{\xvec}_I$
%effectively independent. The expansion is completed with a final Sakai expansion
%for a quantity which is a generalised version of an $r$-point function, and
%which arises in the Hara-Slade expansion steps. We believe that the current
%set-up is simpler than the Hara-Slade expansion as in \cite{hs01},
%and the fact that the expansion coefficients in the different expansions arise
%separately will make the bounds on the expansion in Section \ref{s:bounds} simpler
%as well.

\subsection{The main identity and estimates}
\label{sec-convolutionalgebra}
In this section, we solve the recursion \refeq{taunxvecr2} by iteration,
so that on the right-hand side no $r$-point function appears. Instead,
only $s$-point functions with $s<r$ appear, which opens up the possibility
for an inductive analysis in $r$. The argument in
this section is virtually identical to the argument in
\cite[Section 2.3]{hs02}, and we add it to make the paper
self-contained.

We define
    \begin{align}\lbeq{nudef}
    \nu(\xvec)=\sum_{n=0}^\infty B^{\sstar n}(\xvec),
    \end{align}
where $B^{\sstar n}$ denotes the $n$-fold space-time convolution of $B$ with
itself, with $B^{\sstar 0}(\xvec)=\delta_{\ovec,\xvec}$.  The sum over $n$ in
\refeq{nudef} terminates after finitely many terms, since by definition
\ch{$B((x,t))\neq 0$} only if $t\in\vep\N$, so that in particular $B((x,0))=0$.
Therefore, $B^{\sstar n}(\xvec)=0$ if $n>t_{\xvec}/\vep$, \ch{where, for $\xvec=(x,t)\in \Lambda$,
$t_{\xvec}=t$ denotes the time coordinate of $\xvec$.}
Then \refeq{taunxvecr2} can be
solved to give
    \begin{align}\lbeq{taunuA}
    \tau(\vec\xvec_J)=(\nu\sstar A)(\vec\xvec_J).
    \end{align}
The function $\nu$ can be identified as follows.  We note that
\refeq{taunuA} for $r=2$ yields that
    \begin{align}\lbeq{taunuA2}
    \tau(\xvec)=(\nu\sstar A)(\xvec).
    \end{align}
Thus, extracting the $n=0$ term from \refeq{nudef}, using \refeq{A/B}
to write one factor of $B$ as $A\sstar p_\vep$ (cf., \refeq{Bpieq}) for the
terms with $n\geq1$, it follows from \refeq{taunuA2} that
    \begin{align}\lbeq{nuform}
    \nu(\xvec)=\delta_{\ovec,\xvec}+(\nu\sstar B)(\xvec)
    =\delta_{\ovec,\xvec}+(\nu\sstar A\sstar p_\vep)(\xvec)
    =\delta_{\ovec,\xvec}+(\tau\sstar p_\vep)(\xvec).
    \end{align}
Substituting \refeq{nuform} into \refeq{taunuA}, the solution to
\refeq{taunxvecr2} is then given by
    \begin{align}\lbeq{X3.4}
    \tau(\vec\xvec_J)=A(\vec\xvec_J)+(\tau\sstar p_\vep\sstar A)(\vec\xvec_J),
    \end{align}
which recovers \refeq{lace-exp} when $r=2$, using \refeq{A/B}.  For
$r\geq3$, we further substitute \refeq{Adec}--\refeq{BCE} into
\refeq{X3.4}.  Let
    \begin{align}
    \psi(\yvec_1,\yvec_2)&=\sum_{\vvec}p_\vep(\vvec)~C(\yvec_1-\vvec,
    \yvec_2-\vvec),\lbeq{psidef}\\
    \zetav^{\sss(r)}(\vec\xvec_J)&=A(\vec\xvec_J)+(\tau\sstar p_\vep\sstar a)
    (\vec\xvec_J),\lbeq{phidef}
    \end{align}
where
    \begin{align}\lbeq{edef}
    a(\vec\xvec_J)=a(\vec\xvec_J;1)+\sum_{I\subset J_1:I\ne\varnothing}
    a(\vec\xvec_{J\setminus I},\vec\xvec_I).
    \end{align}
Then, \refeq{X3.4} becomes
    \begin{align}\lbeq{id3}
    \tau^{\sss(r)}(\vec\xvec_J)=\sum_{\vvec,\yvec_1,\yvec_2}\tau^{\sss(2)}
    (\vvec)~\psi(\yvec_1-\vvec,\yvec_2-\vvec)\sum_{I\subset J_1:I\ne
    \varnothing}\tau^{\sss(r_1)}(\vec\xvec_{J\setminus I}-\yvec_1)~
    \tau^{\sss(r_2)}(\vec\xvec_I-\yvec_2)+\zetav^{\sss(r)}(\vec\xvec_J),
    \end{align}
where we recall that $r_1=|J\setminus I|+1$ and $r_2=|I|+1$, and we
write the superscripts of the $r$-point functions explicitly.  Since
$1\leq|I|\leq r-2$, we have that $r_1, r_2\leq r-1$, which opens up
the possibility for induction in $r$.

The first term on the right side of \refeq{id3} is the main term.
% and is depicted schematically in Figure~\ref{fig-tsexp}.
The leading contribution to $\psi(\yvec_1,\yvec_2)$ is
    \begin{align}\lbeq{psimain}
    \psi_{2\vep,2\vep}(y_1,y_2)=\sum_up_\vep(u)\,p_\vep(y_1-u)\,p_\vep(y_2
     -u)\,(1-\delta_{y_1,y_2}),
    \end{align}
using the leading
contribution to $C$ described \ch{below} \refeq{BCE}.  Here, we are writing
$\psi_{s_1,s_2}(y_1,y_2)$ for $\psi((y_1,s_1),(y_2,s_2))$.

We will analyse \refeq{id3} using the Fourier transform.
For brevity, we write $\vec{t}= (t_1,\ldots,t_{r-1})$ and
$\vec{k}=(k_1, \ldots, k_{r-1})$.
For $I \subset \{1,2,\ldots, r-1\}$, we also write
$\vec{k}_I=(k_i)_{i\in I}$, $k_I=\sum_{i\in I} k_i$ and
$k=\sum_{i=1}^{r-1} k_i$.
For $I\subseteq J$, we further write $\underline{t}_I = \min_{i\in I} t_i$
and $\underline{t} = \underline{t}_J$.  With this notation, the Fourier
transform of \refeq{id3} becomes
    \begin{align}\lbeq{tau3exp}
    \hat\tau_{\vec t}^{\sss(r)}(\vec k)=\ddsum_{s_0=0}^{\underline t-2\vep}\hat
     \tau_{s_0}^{\sss(2)}(k)\sum_{\varnothing\ne I\subset J_1}~\ddsum_{s_1=2\vep}
     ^{\underline t_{J\setminus I}-s_0}~\ddsum_{s_2=2\vep}^{\underline t_I-s_0}
     \hat\psi_{s_1,s_2}(k_{J\setminus I},k_I)~\hat\tau_{\vec t_{J\setminus I}
     -s_1-s_0}^{\sss(r_1)}(\vec k_{J\setminus I})~\hat\tau_{\vec t_I-s_2-s_0}
     ^{\sss(r_2)}(\vec k_I)+\hat\zetav_{\vec t}^{\sss(r)}(\vec k),
    \end{align}
where $\dsum_{t\le s\le t'}$ is an abbreviation for
$\sum_{s\in[t,t']\cap\vep\Zp}$.  The identity \refeq{tau3exp} is our main
identity and will be our point of departure for analysing
the $r$-point functions for $r \geq 3$.
Apart from $\psi$ and $\zetav^\smallsup{r}$, the right-hand side of \refeq{id3}
involves the $s$-point functions with $s=2,r_1,r_2$.
As discussed below \refeq{id3}, we can use an inductive analysis,
with the $r=2$ case given by the result of Theorem~\ref{thm:2pt}
proved in \cite{hsa04}.  The term involving $\psi$ is the main term,
whereas $\zetav^\smallsup{r}$ will turn out to be an error term.

The analysis will be based on the following important proposition, whose
proof is deferred to \ch{Sections~\ref{s:bounds}--\ref{ss:ebd}}. In its statement, we use the
notation
    \begin{align}
    \lbeq{bs-def}
    b_{s_1,s_2}^{\sss(\vep)}=
    \frac{\vep^{n_{s_1,s_2}}\,\ind{s_1\leq s_2}}
    {(1+s_1)^{(d-2)/2}}\times
     \begin{cases}
     (1+s_2-s_1)^{-(d-2)/2}\quad&(d>2),\\[1pt]
     \log(1+s_2)&(d=2),\\[2pt]
     (1+s_2)^{(2-d)/2}&(d<2),
     \end{cases}
    \end{align}
where
    \eq
    \lbeq{ns-def}
    n_{s_1,s_2} = 3-\delta_{s_1,s_2}-\delta_{s_1,2\vep}\delta_{s_2,2\vep}.
    \en
We note that the number of powers of $\vep$ is precisely such that,
for $d>4$,
    \eq
    \lbeq{bs-sum}
    \ddsum_{s_1,s_2=2\vep}^{\infty}b_{s_1,s_2}^{\smallsup{\vep}} =O(\vep).
    \en
We also rely on the notation
    \eq
    \lbeq{beta-def}
    \beta=L^{-d},
    \en
and, for $d\leq 4$, we write $\beta_{\sT}=L_{\sT}^{-d}$.
Then, the main bounds on the lace-expansion coefficients are as follows:

\begin{prop}[{\bf Bounds on the lace-expansion coefficients}]
\label{thm-psivphibd}
The lace-expansion coefficients satisfy the following properties:
    \eq
    \lbeq{psimainrep}
    \psi_{2\vep,2\vep}(y_1,y_2) =
    \sum_up_\vep(u)\,p_\vep(y_1-u)\,p_\vep(y_2-u)\,(1-\delta_{y_1,y_2}).
    \en

\begin{enumerate}
\item[(i)] Let $d>4$, $\alphaminn\in(0,1\wedge\Delta\wedge\frac{d-4}2)$ and
$\lamb=\lambc^{\sss(\vep)}$.  Let $\bar t$ denote the second-largest element of
$\{t_1,\dots,t_{r-1}\}$.  There exist $C_\psi,C_\zetav^{\sss(r)}>0$
(independent of $L$) and $L_0=L_0(d)$ such that, for all $L\ge L_0$,
$q\in\{0,2\}$, $s_i\ge0$, $\vec t$, $r\ge3$ and $k_i\in[-\pi,\pi]^d$, and
uniformly in $\vep\in(0,1]$, the following bounds hold:
    \begin{align}
    |\nabla_{k_i}^q\hat\psi_{s_1,s_2}(k_1,k_2)|&\leq C_\psi\,\sigma^q
    (1+s_i)^{q/2} (\delta_{s_1,s_2}+\beta)\beta
    (b_{s_1,s_2}^{\sss(\vep)}+b_{s_2,s_1}^{\sss(\vep)}),
    \lbeq{psibd}\\
    |\hat\zetav_{\vec t}^{\sss(r)}(\vec k)|&\leq C_\zetav^{\sss(r)}\,(1+\bar
    t)^{r-2-\alphaminn}.\lbeq{vphibd}
    \end{align}
\item[(ii)] Let $d\le4$ with $\alpha\equiv bd-\frac{4-d}2>0$ and $\alphaminn\in(0,\alpha)$,
and let $\beta_{\sT}=\beta_1T^{-bd}$.
Let $\lamb_{\sT}=1+O(L_{\sT}^{-\mu})$ and $\mu\in(0,\alpha-\delta)$
be as in Theorem \ref{thm:2pt}(i). There exist
\chg{$C_\psi,C_\zetav^{\sss(r)}>0$ (independent of $L$) and $L_0=L_0(d)$} such
that, for \chg{$L_1\ge L_0$ with $L_{\sT}$ defined as in \refeq{Lt-def}}, $r\ge2$,
$0<s\le\log{T}$, as $T\uaw\infty$, and uniformly in $\vep\in (0,1]$,
the following bounds hold:
    \begin{align}
    |\nabla_{k_i}^q\hat\psi_{s_1,s_2}(k_1,k_2)|&\leq C_\psi\,\sigma^q
    (1+s_i)^{q/2} (\delta_{s_1,s_2}+\beta_{\sT})\beta_{\sT}
    (b_{s_1,s_2}^{\sss(\vep)}+b_{s_2,s_1}^{\sss(\vep)}),
    \lbeq{psibd-low}\\
    |\hat\zetav_{\vec t}^{\sss(r)}(\vec k)|&\leq C_\zetav^{\sss(r)}
    T^{r-2-\alphaminn}.\lbeq{vphibd-low}
    \end{align}
\end{enumerate}
\end{prop}

It follows from \refeq{psibd} and \refeq{bs-sum} that for $d>4$,
the constant $V^{\smallsup{\vep}}$ defined
by
    \eq
    \lbeq{Vpsi}
    V^{\smallsup{\vep}}=\frac{1}{\vep}\ddsum_{s_1,s_2=2\vep}^\infty
    \hat{\psi}_{s_1,s_2}(0,0),
    \en
with $\lambda=\lambda_{\rm c}^{\smallsup{\vep}}$, is finite uniformly in $\vep>0$.  The
constant $V$ of Theorem~\ref{thm:rpt} should then be given by
$\lim_{\vep \downarrow 0}V^{\smallsup{\vep}}$. In Proposition \ref{prop-disc}
below, we will prove the existence of the limit
$\lim_{\vep \downarrow 0}V^{\smallsup{\vep}}$. Since
$\hat{\psi}_{2\vep,2\vep}(0,0)\approx 2\vep\lambda_{\rm c}(1-\vep\lambc)=
\vep[2-\vep+O(\beta)]$ by \refeq{psimain}, it follows from
Proposition~\ref{thm-psivphibd} that uniformly in $\vep>0$,
    \eq
    \lbeq{V1}
    V^{\smallsup{\vep}} = 2-\vep +O(\beta).
    \en
This establishes the claim on $V$ of Theorem~\ref{thm:rpt}(i).
For $d\leq 4$, on the other hand, $\beta=\beta_{\sT}$ converges to zero
as $T\uparrow \infty$, so that $V^{\smallsup{\vep}}$ is replaced
by $2-\vep$ in \ch{Theorem~\ref{thm:rptvep}(ii)}.

%%%%%%%%%%%%%%%%%%%%%%%%%%%%%%%%%%%%%%%%%%%%%%%%%%%%%%%%%%%%%%%%%%%%%%%%%%%%%%%
%%%%%%%%%%%%%%%%%%%%%%%%%%%%%%%%%%%%%%%%%%%%%%%%%%%%%%%%%%%%%%%%%%%%%%%%%%%%%%%
\subsection{Induction in $r$}
\label{sec-proofr=3}
In this section, we prove \ch{Theorem~\ref{thm:rptvep}} for $\vep\in (0,1]$
fixed, assuming \refeq{tau3exp} and Proposition~\ref{thm-psivphibd}.
We fix $\lambda=\lambda_{\rm c}^{\smallsup{\vep}}$ throughout this
section. The argument in this section is an adaptation of the
argument in \cite[Section 2.4]{hs02}, adapted so as to deal with
the uniformity in the time discretization. In particular, in this section,
we prove \ch{Theorem~\ref{thm:rptvep}} for oriented percolation for which $\vep=1$.

We start by giving the proof for $d>4$.
Let $\bar t$ denote the second-largest element of $\{t_1,\ldots,t_{r-1}\}$.
We will prove that for $d>4$ there are positive constants $L_0=L_0(d)$ and
$V^{\sss(\vep)}=V^{\sss(\vep)}(d,L)$ such that for $\lamb=\lambc^{\sss(\vep)}$,
$L\geq L_0$ and $\alphaminn\in(0,1\wedge\Delta\wedge\frac{d-4}2)$, we have
    \eq
    \hat{\tau}_{\vec{t}}^\smallsup{r}(\tfrac{\vec{k}}{\sqrt{v^{\sss(\vep)}\sigma^2 t}})
    = \ch{A^{\sss(\vep)}\big((A^{\sss(\vep)})^2V^{\sss(\vep)}t)^{r-2}}
    \left[
    \hat{M}_{\vec{t}/t}^{\smallsup{r-1}}(\vec{k})
    + O((\bar{t}+1)^{-\alphaminn})
        \right]
    \quad
    (r \geq 3)
    \lbeq{aim3pt}
    \en
uniformly in $t \geq \bar{t}$ and in
$\vec{k} \in \Rbold^{(r-1)d}$ with $\sum_{i=1}^{r-1}|k_i|^2$
bounded, and uniformly in $\vep>0$.  Since the $\hat{M}_{\vec{t}}^\smallsup{r-1}(\vec{k})$
are smooth functions of $\vec{t}$ (cf., \cite[(2.51)]{hs01}),
proving the above is sufficient to prove Theorem~\ref{thm:rptvep}(i).

We will prove \refeq{aim3pt} by induction in $r$,
with the case $r=2$ given by Theorem~\ref{thm:rptvep}(i) for $r=2$.
Indeed, Theorem~\ref{thm:rptvep}(i) for $r=2$ gives
    \eq
    \lbeq{aimr2}
    \ch{\hat{\tau}_{t_1}(\tfrac{k}{\sqrt{v^{\sss(\vep)}\sigma^2 t}})}
    =
    \hat{\tau}_{t_1}\Big((\tfrac{t_1}t)^{\sss1/2}\tfrac{k}{\sqrt{\ch{v^{\sss(\vep)}}\sigma^2 t_1}}\Big)
    =
    \ch{A^{\sss(\vep)}} \left[
    e^{-\tfrac{|k|^2t_1}{2dt}} + O((t_1+1)^{-\alphaminn})
    \right],
    \en
using the facts that $|k|^2$ is bounded, $t_1 \leq t$,
and $\alphaminn < \frac{d-4}{2}$.

\medskip \noindent
{\em Proof of Theorem~\ref{thm:rptvep}(i) assuming Proposition~\ref{thm-psivphibd}.}
Let $r \geq 3$.  The proof is by induction in $r$, with the induction
hypothesis that \refeq{aim3pt} holds for $\tau^{\sss(s)}$ with $2\leq s<r$.
We have seen in \refeq{aimr2} that \refeq{aim3pt} does hold for $r=2$.
The induction will be advanced using \refeq{tau3exp}.
By \refeq{vphibd}, $\hat{\varphi}^\smallsup{r}_{\vec{n}}(\vec{k})$
is an error term. Thus, we are
left to determine the asymptotic behaviour of the
first term on the right side of
\refeq{tau3exp}.

Fix $\vec{k}$ with $\sum_{i=1}^{r-1}|k_i|^2$ bounded.
To abbreviate the notation, we write
$\veckt = \vec{k}/\sqrt{\ch{v^{\sss(\vep)}}\sigma^2 t}$.
Recall the notation $\underline{t}=\min\{t_1,\ldots,t_{r-1}\}$.
Given $0 \leq s_0 \leq \underline{t}$, let
$\underline{t}_0 = \min \{s_0, \underline{t} -s_0\}$.
We will show that for every nonempty subset $I \subset J_1$,
    \begin{gather}
    \Bigg|~\ddsum_{s_1=2\vep}^{\underline{t}_{J\setminus I}-s_0}~\ddsum_{s_2
    =2\vep}^{\underline{t}_I-s_0}\hat\psi_{s_1,s_2}(\kt_{J\setminus I},
    \kt_I)~\hat\tau_{\vec t_{J\setminus I}-s_1-s_0}^{\sss(r_1)}(\vec\kt
    _{J\setminus I})~\hat\tau_{\vec t_I-s_2-s_0}^{\sss(r_2)}(\vec\kt_I)-V^{\sss(\vep)}
    \,\hat\tau_{\vec t_{J\setminus I}-s_0}^{\sss(r_1)}(\vec\kt_{J\setminus
    I})~\hat\tau_{\vec t_I-s_0}^{\sss(r_2)}(\vec\kt_I)\Bigg|\nnmb\\
    \leq C\vep t^{r-3}(\underline{t}_0+1)^{-\alphaminn}.\lbeq{conv3}
    \end{gather}

Before establishing \refeq{conv3}, we first show that it implies
\refeq{aim3pt}.  Since $|\hat{\tau}_{s_0}(\kt)|$ is uniformly bounded
by \ch{Theorem~\ref{thm:rptvep}} for $r=2$,
inserting \refeq{conv3} into \refeq{tau3exp}
and applying \refeq{vphibd} gives
\eqalign
\lbeq{taur.1}
    \hat{\tau}^\smallsup{r}_{\vec{t}}(\veckt)
    & =
    \ch{V^{\sss(\vep)}}\vep \ddsum_{s_0 =0}^{\underline{t}}
    \hat{\tau}_{s_0}(\kt)\sum_{I \subset J_1 : |I| \geq 1}
    \hat{\tau}_{\vec{t}_{J\setminus I}-s_0}^{\smallsup{r_1}}(\veckt_{J\setminus I})
    \hat{\tau}_{\vec{t}_{I}-s_0}^{\smallsup{r_2}}(\veckt_{I}) +
    O(t^{r-3})\vep \ddsum_{s_0 =0}^{\underline{t}}(\underline{t}_0+1)^{-\alphaminn}
    + O(t^{r-2-\alphaminn}).
\enalign
Using the fact that $\alphaminn <1$, the summation in the error term can
be seen to be bounded by a multiple of $\underline{t}^{1-\alphaminn}
\leq t^{1-\alphaminn}$.  With the
induction hypothesis and the identity $r_1+r_2=r+1$,
\refeq{taur.1} then implies that
    \eqalign
    \lbeq{taur.7}
    \hat{\tau}^\smallsup{r}_{\vec{t}}(\veckt)
    & =
    A^{\sss(\vep)}\big((A^{\sss(\vep)})^2V^{\sss(\vep)}t\big)^{r-2}
     \vep \ddsum_{s_0 =0}^{\underline{t}}
    \hat{M}_{s_0/t}^{\smallsup{1}}(k)\sum_{I \subset J_1 : |I| \geq 1}
    \hat{M}_{(\vec{t}_{J\setminus I}-s_0)/t}^{\smallsup{r_1-1}}
    (\vec{k}_{J\setminus I})~
    \hat{M}_{(\vec{t}_{I}-s_0)/t}^{\smallsup{r_2-1}}(\vec{k}_{I})
    +
    O(t^{r-2-\alphaminn}),
    \enalign
where the error arising from the error terms in the induction hypothesis again contributes
an amount
$O(t^{r-3})\vep \ddsum_{s_0=0}^{\underline{t}}(\underline{t}_0+1)^{-\alphaminn}
\leq O(t^{r-2-\alphaminn})$.
The summation on the right-hand side of \refeq{taur.7}, divided by $t$, is the Riemann sum
approximation to an integral.  The error in approximating the integral by this Riemann
sum is $O(\vep t^{-1})$.  Therefore, using \refeq{Mr-def}, we obtain
    \eqalign
    \hat{\tau}^\smallsup{r}_{\vec{t}}(\veckt)
    & =
    \ch{A^{\sss(\vep)}\big((A^{\sss(\vep)})^2V^{\sss(\vep)}t\big)^{r-2}}
    \int_0^{\underline{t}/t} ds_0 \; \hat{M}_{s_0}^{\smallsup{1}}(k)
    \sum_{I \subset J_1 : |I| \geq 1}
    \hat{M}_{t^{-1}\vec{t}_{J\setminus I}-s_0}^{\smallsup{r_1-1}}(\vec{k}_{J\setminus I})
    \hat{M}_{t^{-1}\vec{t}_{I}-s_0}^{\smallsup{r_2-1}}(\vec{k}_{I})  +
    O(t^{r-2-\alphaminn})
    \nnb
    \lbeq{3lim}
    & =  \ch{A^{\sss(\vep)}\big((A^{\sss(\vep)})^2V^{\sss(\vep)}t\big)^{r-2}}
    \hat{M}_{\vec{t}/t}^{\smallsup{r-1}}(\vec{k}) +
    O(t^{r-2-\alphaminn}).
    \enalign
Since $t \geq \bar{t}$, it follows that
$t^{r-2-\alphaminn} \leq Ct^{r-2}(\bar{t}+1)^{-\alphaminn}$.  Thus, it suffices to
establish \refeq{conv3}.

To prove \refeq{conv3}, we write the quantity inside the absolute
value signs on the left-hand side as
\eqalign
    \lbeq{3ptdecomp}
    &
    \left[~
    \ddsum_{s_1=2\vep}^{\underline{t}_{J\setminus I}-s_0}
    \ddsum_{s_2=2\vep}^{\underline{t}_{I} - {s}_0}
    \hat{\psi}_{s_1,s_2}(\kt_{J\setminus I},\kt_{I})
    \hat{\tau}_{\vec{t}_{J\setminus I}-s_1-s_0}^{\smallsup{r_1}}
    (\veckt_{J\setminus I})
    \hat{\tau}_{\vec{t}_{I}-s_2-s_0}^{\smallsup{r_2}}(\veckt_{I})
    \right]
    -  V^{\sss(\vep)}\hat{\tau}_{\vec{t}_{J\setminus I}-s_0}^{\smallsup{r_1}}
    (\veckt_{J\setminus I})
    \hat{\tau}_{\vec{t}_{I}-s_0}^{\smallsup{r_2}}(\veckt_{I})
    \nnb
    & \hspace{10mm}  = T_1 + T_2 + T_3,
    \enalign
with
    \eqalign
    \lbeq{T1def}
    T_1 & =   \left[~
    \ddsum_{s_1=2\vep}^{\underline{t}_{J\setminus I}-s_0}
    \ddsum_{s_2=2\vep}^{\underline{t}_{I} - {s}_0}
    \hat{\psi}_{s_1,s_2} (0,0) -
    \ch{V^{\sss(\vep)}}\right]
    \hat{\tau}_{\vec{t}_{J\setminus I}-s_0}^{\smallsup{r_1}}
    (\veckt_{J\setminus I})
    \hat{\tau}_{\vec{t}_{I}-s_0}^{\smallsup{r_2}}(\veckt_{I}),
    \\
    \lbeq{T2def}
    T_2 & =
    \ddsum_{s_1=2\vep}^{\underline{t}_{J\setminus I}-s_0}
    \ddsum_{s_2 =2\vep}^{\underline{t}_{I}-{s}_0} \left[
    \hat{\psi}_{s_1,s_2}
    (\kt_{J\setminus I},\kt_{I})-
    \hat{\psi}_{s_1,s_2} (0,0)\right]
    %\nnb && \hspace{10mm} \times
    \hat{\tau}_{\vec{t}_{J\setminus I}-s_0}^{\smallsup{r_1}}(\veckt_{J\setminus I})
    \hat{\tau}_{\vec{t}_{I}-s_0}^{\smallsup{r_2}}(\veckt_{I}),
    \\
    \lbeq{T3def}
    T_3 & =
    \ddsum_{s_1=2\vep}^{\underline{t}_{J\setminus I}-s_0}
    \ddsum_{s_2=2\vep}^{\underline{t}_{I} - {s}_0}
    \hat{\psi}_{s_1,s_2} (\kt_{J\setminus I},\kt_{I})
    \nnb & \hspace{10mm} \times
    \left[
    \hat{\tau}^{\smallsup{r_1}}_{\vec{t}_{J\setminus I}-s_1-s_0}
    (\veckt_{J\setminus I})
    \hat{\tau}^{\smallsup{r_2}}_{\vec{t}_{I}-s_2-s_0}(\veckt_{I})
    -\hat{\tau}_{\vec{t}_{J\setminus I}-s_0}^{\smallsup{r_1}}(\veckt_{J\backslash I})
    \hat{\tau}_{\vec{t}_{I}-s_0}^{\smallsup{r_2}}(\veckt_{I})
    \right].
    \enalign
To complete the proof, it suffices to show that
for each nonempty $I\subset J_1$, the absolute value
of each $T_i$ is bounded above by the right-hand side of
\refeq{conv3}.

In the course of the proof, we will make use of some bounds on sums
involving $b_{s_1,s_2}^{\smallsup{\vep}}$:

\begin{lemma}[{\bf Bounds on sums involving $b_{s_1,s_2}^{\smallsup{\vep}}$}]
\label{lem-sumbds}
\begin{enumerate}
\item[(i)] Let $d>4$. For every $\alphaminn\in [0,1\wedge \frac{d-4}{2})$, there exists
a constant $C=C(\alphaminn, d)$ such that the following bounds hold
uniformly in $\vep\in (0,1]$
    \eq
    \lbeq{bdsumd12}
    \ddsum_{\stackrel{s_1,s_2=2\vep}{s_1\vee s_2\leq s}}s_i (b_{s_1,s_2}^{\smallsup{\vep}}
    +b_{s_2,s_1}^{\smallsup{\vep}})
    \leq C \vep s^{1-\alphaminn}, \qquad
    \ddsum_{\stackrel{s_1,s_2=2\vep}{s_1\vee s_2\geq s}}^{\infty} b_{s_1,s_2}^{\smallsup{\vep}}
    \leq C \vep s^{-\alphaminn}.
    \en
\item[(ii)] Let $d\le4$ with $\alpha\equiv bd-\frac{4-d}2>0$ and fix $\alphamin\in(0,\alpha)$,
recall that $\beta_{\sT}=\beta_1T^{-bd}$ and let $\hat\beta_{\sT}=\beta_1T^{-\alphamin}$. There exists
a constant $C=C(\alphaminn,d)$ such that the following bound holds
uniformly in $\vep\in (0,1]$
    \eq
    \lbeq{bdsumd12lowdim}
    \beta_{\sT}\ddsum_{\stackrel{s_1,s_2=2\vep}{s_1\vee s_2>2\vep}}^{T\log{T}} (\delta_{s_1,s_2}+\beta_{\sT})(b_{s_1,s_2}^{\smallsup{\vep}}
    +b_{s_2,s_1}^{\smallsup{\vep}})
    \leq C \hat{\beta}_{\sT} \vep.
    \en
\end{enumerate}
\end{lemma}

\proof $(i)$ This is straightforward from \refeq{bs-def}, when we pay
special attention to the number of powers of $\vep$ present
in $b_{s_1,s_2}^{\smallsup{\vep}}$ and use the fact that the power of $(1+s_1)$
and of $(1+s_2-s_1)$ is $(d-2)/2>1$.\\
$(ii)$ We shall only perform the proof for $d\leq 4$ with $d\neq 2$,
the proof for $d=2$ being a slight modification of the argument below.
Using \refeq{bs-def}, we can perform the sum to obtain
    \eqalign
    {\rm LHS}~\refeq{bdsumd12lowdim}&\leq C\beta_{\sT}\ddsum_{s_1>2\vep}^{T\log{T}} \vep^2 (1+s_1)^{(2-d)/2}
    +C\beta_{\sT}^2\ddsum_{\stackrel{2\vep\leq s_1<s_2\leq T\log{T}}{s_2>2\vep}}
    \vep^3 (1+s_1)^{(2-d)/2}(1+s_2-s_1)^{(2-d)/2}\nonumber\\
    &\leq C\vep \beta_{\sT} \big(T\log{T}\big)^{(4-d)/2}\big(1+\beta_{\sT} \big(T\log{T}\big)^{(4-d)/2}\big)
    \leq C\vep \hat\beta_{\sT}(1+\hat\beta_{\sT}),
    \enalign
as long as $\alphamin\in(0,\alpha)$. Using that $\hat\beta_{\sT}$ converges to 0
as $T\uaw\infty$, this proves \refeq{bdsumd12lowdim}.
\qed
\vskip0.5cm

\noindent
By the induction hypothesis and the fact that $\bar{t}_{I_i} \leq t$,
it follows that
$|\hat{\tau}_{\vec{t}_{I_i}}^{\smallsup{r_i}}(\vec{k}_{I_i})| \leq O(t^{r_i-2})$,
uniformly in $\vec{t}_{I_i}$ and $\vec{k}_{I_i}$.  Therefore, it follows from
\refeq{psibd} and the definition of \ch{$V^{\sss(\vep)}$} in \refeq{Vpsi} that
    \eq
    \lbeq{T1bd}
    |T_1| \le \ddsum_{\substack{s_1\ge\underline{t}_{J\setminus I}-s_0\\
        \text{or }s_2\ge\underline{t}_I-s_0}}
    O(t^{r-3})b_{s_1,s_2}^{\smallsup{\vep}}
    \leq O(\vep t^{r-3}(\underline{t}_0+1)^{-(d-4)/2}),
    \en
where the final bound follows from the second bound in \refeq{bdsumd12}.

Similarly, by \refeq{psibd} with $q=2$, now using the first bound in
\refeq{bdsumd12},
    \eq
    \lbeq{T2bd}
    | T_2 |
    \leq
    \ddsum_{s_1=2\vep}^{\underline{t}_{J\setminus I}-s_0}
    \ddsum_{s_2=2\vep}^{\underline{t}_{I} - {s}_0}
    (s_1 |\kt_{J\setminus I}|^2+s_2|\kt_{I}|^2)
    O(t^{r-3})b_{s_1,s_2}^{\smallsup{\vep}}
    \leq
    O(\vep t^{r-3}(\underline{t}_0+1)^{-\alphaminn}),
    \en
\ch{using that $t^{r-4}(\underline{t}_0+1)^{1-\alphaminn}\leq t^{r-3}(\underline{t}_0+1)^{-\alphaminn}$
since $t\geq \underline{t}_0$.}
It remains to prove that
    \eq
    \lbeq{conv31}
    |T_3| \leq
    O( \vep t^{r-3} (\underline{t}_0+1)^{-\alphaminn}).
    \en

To begin the proof of \refeq{conv31}, we note that the domain of
summation over $s_1,s_2$ in \refeq{T3def} is contained in
$\cup_{j=0}^2 \Scal_j(\vec{t})$, where
    \eqalign
    \Scal_0(\vec{t}) & =
    [0 , {\textstyle \frac{1}{2}}(\underline{t}_{J\setminus I}-s_0)] \times
    [0 ,{\textstyle \frac{1}{2}}(\underline{t}_{I}-s_0)],
    \nonumber \\ \nonumber
    \Scal_1(\vec{t}) & =
    [{\textstyle \frac{1}{2}}(\underline{t}_{J\setminus I} - s_0) ,
    \underline{t}_{J\setminus I} - s_0 ] \times
    [0 ,  \underline{t}_{I} -s_0]
    ,\\
    \Scal_2(\vec{t}) & =
    [0 ,  \underline{t}_{J\setminus I} -s_0] \times
    [{\textstyle \frac{1}{2}}(\underline{t}_{I} - s_0)
    , \underline{t}_{I} - s_0
    ].\nonumber
    \enalign
Therefore, $|T_3|$ is bounded by
    \eq
    \lbeq{T3decomp}
    \sum_{j=0}^2  \ddsum_{\vec{s} \in \Scal_j(\vec{t})}
    \left| \hat{\psi}_{s_1,s_2} (\kt_{J\setminus I},\kt_{I})
    \right| \,
    \left|\hat{\tau}^{\smallsup{r_1}}_{\vec{t}_{J\setminus I}-s_1-s_0}
    (\veckt_{J\setminus I})
    \hat{\tau}^{\smallsup{r_2}}_{\vec{t}_{I}-s_2-s_0}(\veckt_{I})
    -\hat{\tau}_{\vec{t}_{J\setminus I}-s_0}^{\smallsup{r_1}}(\veckt_{J\backslash I})
    \hat{\tau}_{\vec{t}_{I}-s_0}^{\smallsup{r_2}}(\veckt_{I})\right|.
    \en
The terms with $j=1,2$ in \refeq{T3decomp}
can be estimated as in the
bound \refeq{T1bd} on $T_1$, after using the triangle inequality
and bounding the $r_i$-point functions by $O(t^{r_i-2})$.

For the $j=0$ term of \refeq{T3decomp}, we write
    \eqalign
    \lbeq{tauni-mi1}
    \hat{\tau}^{\smallsup{r_1}}_{\vec{t}_{J\setminus I}-s_1-s_0}
    (\veckt_{J\setminus I})
    & =
    \hat{\tau}^{\smallsup{r_1}}_{\vec{t}_{J\setminus I}-s_0}(\veckt_{J\setminus I})
    +
    \left[ \hat{\tau}^{\smallsup{r_1}}_{\vec{t}_{J\setminus I}-s_1-s_0}
    (\veckt_{J\setminus I})
    - \hat{\tau}^{\smallsup{r_1}}_{\vec{t}_{J\setminus I}-s_0}(\veckt_{J\setminus I})
    \right],\\
    \lbeq{tauni-mi2}
    \hat{\tau}^{\smallsup{r_2}}_{\vec{t}_{I}-s_2-s_0}(\veckt_{I}) & =
    \hat{\tau}^{\smallsup{r_2}}_{\vec{t}_{I}-s_0}(\veckt_{I}) +
    \left[ \hat{\tau}^{\smallsup{r_2}}_{\vec{t}_{I}-s_2-s_0}(\veckt_{I})
    - \hat{\tau}^{\smallsup{r_2}}_{\vec{t}_{I}-s_0}(\veckt_{I})
    \right].
    \enalign
We expand the product of \refeq{tauni-mi1} and \refeq{tauni-mi2}. This gives
four terms, one of which is cancelled by
$\hat{\tau}_{\vec{t}_{J\setminus I}-s_0}^{\smallsup{r_1}}
(\veckt_{J\setminus I})
\hat{\tau}_{\vec{t}_{I}-s_0}^{\smallsup{r_2}}(\veckt_{I})$
in \refeq{T3decomp}.
Three terms remain,
each of which contains at least one factor from the second
terms in \refeq{tauni-mi1}--\refeq{tauni-mi2}.
In each term we retain one such factor and bound the
other factor by a power of $t$, and we estimate $\hat{\psi}$
using \refeq{psibd}.  This
gives a bound for the $j=0$ contribution to \refeq{T3decomp} equal to the sum
of
    \eq
    \lbeq{T3bd1}
    \ddsum_{ (s_1,s_2) \in \Scal_0(\vec{n})}
    O(t^{r_2-2})b_{s_1,s_2}^{\smallsup{\vep}}
     \left|
     \hat{\tau}^{\smallsup{r_1}}_{\vec{t}_{J\setminus I}-s_1-s_0}
     (\veckt_{J\setminus I})
    - \hat{\tau}^{\smallsup{r_1}}_{\vec{t}_{J\setminus I}-s_0}
    (\veckt_{J\setminus I})
    \right|
    \en
plus a similar term with $J\setminus I$ replaced by $I$.

By the induction hypothesis, the difference of $r_1$-point functions
in \refeq{T3bd1} is equal to
\eq
\lbeq{T3bd1.5}
    \ch{A^{\sss(\vep)} \big((A^{\sss(\vep)})^2V^{\sss(\vep)}t\big)^{r_1-2}}
    \left[
    f\big((\vec{t}_{J\setminus I}-s_1-s_0)/t\big)
     - f\big( (\vec{t}_{J\setminus I}-s_0)/t \big)
        \right]
        + O(t^{r_1-2}(\underline{t}_0+1)^{-\alphaminn})
\en
with $f(\vec{t})
= \hat{M}_{\vec t}^{\smallsup{r_1-1}}(\vec{k}_{J\setminus I})$.
Using \refeq{Mr-def}, the difference in \refeq{T3bd1.5} can
be seen to be at most $O(s_1 t^{-1})$.
Therefore, \refeq{T3bd1} is bounded above, using \refeq{bdsumd12}, by
    \eq
    \lbeq{T3bd2}
    \ddsum_{ (s_1,s_2) \in \Scal_0(\vec{t})}
    \big(O(s_1 t^{r-4})+O(t^{r-3}(\underline{t}_0+1)^{-\alphaminn})\big)
    (b_{s_1,s_2}^{\smallsup{\vep}}+b_{s_2,s_1}^{\smallsup{\vep}})
    \leq O(\vep t^{r-3}(\underline{t}_0+1)^{-\alphaminn}).
   \en
This establishes \refeq{conv31}.

Combining \refeq{T1bd}, \refeq{T2bd}, and \refeq{conv31} yields
\refeq{conv3}.  This completes the proof of Theorem~\ref{thm:rptvep}{\rm (i)},
assuming Proposition~\ref{thm-psivphibd}{\rm (i)}.

The proof of Theorem~\ref{thm:rptvep}{\rm (ii)} is similar, now using
Proposition~\ref{thm-psivphibd}{\rm (ii)} instead of Proposition~\ref{thm-psivphibd}{\rm (i)}
and Lemma \ref{lem-sumbds}{\rm (ii)} instead of Lemma \ref{lem-sumbds}{\rm (i)}.
For $d\leq 4$, we will prove that for there are positive constants $L_0=L_0(d)$ and
such that for $\lamb_{\sT}$ and $\mu$ as in Theorem \ref{thm:2pt}(ii),
$L_1\geq L_0$, with $L_{\sT}$ defined as in \refeq{Lt-def},
and $\delta\in(0,1\wedge\Delta\wedge\alpha)$, we have
    \eq
    \hat{\tau}_{\vec{t}}^\smallsup{r}(\tfrac{\vec{k}}{\sqrt{\sigma^2_{\sT} T}})
    = \big((2-\vep)T)^{r-2}
    \left[
    \hat{M}_{\vec{t}/T}^{\smallsup{r-1}}(\vec{k})
    + O(T^{-\mu\wedge\delta})
        \right]
    \quad
    (r \geq 3)
    \lbeq{aim3ptlowdim}
    \en
uniformly in $T \geq \bar{t}$, in $\vec{t}$ such that $\max_{i=1}^{r-1} t_i\leq T\log{T}$,
and in $\vec{k} \in \Rbold^{(r-1)d}$ with $\sum_{i=1}^{r-1}|k_i|^2$
bounded, and uniformly in $\vep>0$.

We will again prove \refeq{aim3ptlowdim} by induction in $r$,
with the case $r=2$ given by Theorem~\ref{thm:rptvep}(ii) for $r=2$.
This part is a straightforward adaptation of the argument in
\refeq{aimr2}, and is omitted.

We now advance the induction hypothesis. By \refeq{tau3exp}
and \refeq{vphibd-low},
    \begin{align}\lbeq{tau3explowdim}
    \hat\tau_{\vec t}^{\sss(r)}(\vec k)=\ddsum_{s_0=0}^{\underline t-2\vep}\hat
     \tau_{s_0}^{\sss(2)}(k)\sum_{\varnothing\ne I\subset J_1}~\ddsum_{s_1=2\vep}
     ^{\underline t_{J\setminus I}-s_0}~\ddsum_{s_2=2\vep}^{\underline t_I-s_0}
     \hat\psi_{s_1,s_2}(k_{J\setminus I},k_I)~\hat\tau_{\vec t_{J\setminus I}
     -s_1-s_0}^{\sss(r_1)}(\vec k_{J\setminus I})~\hat\tau_{\vec t_I-s_2-s_0}
     ^{\sss(r_2)}(\vec k_I)+O(T^{r-2-\alphaminn}),
    \end{align}
where, since, by Theorem \ref{thm:2pt}(i), $\mu\in (0,\alpha-\delta)$
and since, by Proposition \ref{thm-psivphibd}(ii), $\alphaminn\in (0,\alpha)$
is arbitrary, without loss of generality, we may assume that $\alphaminn\geq \mu$.

By Lemma \ref{lem-sumbds}{\rm (ii)}, using the fact that $\hat{\beta}_{\sT}=\beta_1T^{-\mu}$
and the tree-graph inequality, we can bound
    \eqalign
    &\ddsum_{s_0=0}^{\underline t-2\vep}\hat
     \tau_{s_0}^{\sss(2)}(k)\sum_{\varnothing\ne I\subset J_1}~\ddsum_{s_1=2\vep}
     ^{\underline t_{J\setminus I}-s_0}~\ddsum_{s_2=2\vep}^{\underline t_I-s_0}
     \ind{(s_1,s_2)\neq (2\vep,2\vep)}\hat\psi_{s_1,s_2}(k_{J\setminus I},k_I)~\hat\tau_{\vec t_{J\setminus I}
     -s_1-s_0}^{\sss(r_1)}(\vec k_{J\setminus I})~\hat\tau_{\vec t_I-s_2-s_0}
     ^{\sss(r_2)}(\vec k_I^)\\
     &\qquad \leq C_{\psi}(\bar{t}+1)^{r-3} \ddsum_{s_0=0}^{\underline t-2\vep}
     \ddsum_{\stackrel{s_1,s_2=2\vep}{s_1\vee s_2>2\vep}}^{T\log{T}} (\delta_{s_1,s_2}+\beta_{\sT})(b_{s_1,s_2}^{\smallsup{\vep}}
    +b_{s_2,s_1}^{\smallsup{\vep}})\leq O(T^{r-2-\alphaminn}).
    \nonumber
    \enalign
Fix $\vec{k}$ with $\sum_{i=1}^{r-1}|k_i|^2$ bounded.
To abbreviate the notation, we now write
$\veckT = \vec{k}/\sqrt{\sigma^2_{\sT} T}$.
By \refeq{psimain} and \refeq{Deltadef},
    \eq
    \hat\psi_{2\vep,2\vep}(\veckT_{J\setminus I},\veckT_I)-\hat{\psi}_{2\vep,2\vep}(0,0)
    = O(\vep |k|^2 T^{-1}),
    \en
and, by \refeq{psimain} and the fact that $\lamb_{\sT}=1+O(T^{-\mu})$,
    \eq
    \hat{\psi}_{2\vep,2\vep}(0,0)=\lambda_{\sT}\vep(2-\vep)
    =\vep(2-\vep)+O(\vep T^{-\mu}).
    \en
As a result, we obtain that
    \eqalign
    \hat\tau_{\vec t}^{\sss(r)}(\veckT)
    &=\vep(2-\vep)\ddsum_{s_0=0}^{\underline t-2\vep}\hat
     \tau_{s_0}^{\sss(2)}(k^{\sss(T)})\sum_{\varnothing\ne I\subset J_1}~
     \hat\tau_{\vec t_{J\setminus I}-2\vep-s_0}^{\sss(r_1)}
     (\veckT_{J\setminus I})~\hat\tau_{\vec t_I-2\vep-s_0}^{\sss(r_2)}(\veckT_I)
     +O(T^{r-2-\mu})+O(|k|^2 T^{r-3}).
     \enalign
The remainder of the argument can now be completed as in \refeq{taur.1}--\refeq{3lim},
using the induction hypothesis in \refeq{aim3ptlowdim} instead of the one in
\refeq{aim3pt}.
\qed

\subsection{The continuum limit}
\label{sec-outlinedisc}
In this section we state the results concerning the
continuum limit when $\vep \downarrow 0$.
This proof will crucially rely on the convergence
of $A^{\smallsup{\vep}}, V^{\smallsup{\vep}}$ and $v^{\smallsup{\vep}}$
when $\vep \downarrow 0$.
The convergence of $A^{\smallsup{\vep}}$ and $v^{\smallsup{\vep}}$
was proved in \cite[Proposition 2.6]{hsa04}, so we are left to study
$V^{\smallsup{\vep}}$.
When $1\leq d\leq 4$, we have that the role of
$A^{\smallsup{\vep}}, V^{\smallsup{\vep}}$ and $v^{\smallsup{\vep}}$ are taken by
$A^{\smallsup{\vep}}=1, V^{\smallsup{\vep}}=2-\vep$ and $v^{\smallsup{\vep}}=1$,
so there is nothing to prove. Thus, we are left to study the convergence of
$V^{\smallsup{\vep}}$ when $\vep \downarrow 0$ for
$d>4$.

%%%%%%%%%%%PROPOSITION%%%%%%%%%
\begin{prop}[{\bf Continuum limit}]\label{prop-disc}
Fix $d>4$. Suppose that $\lambda^{\smallsup{\vep}}\rightarrow \lambda$ and
$\lambda^{\smallsup{\vep}} \leq \lambda^{\smallsup{\vep}}_{\rm c}$
for $\vep$ sufficiently small. Then,
there exists a finite and positive constant $V=2+O(\beta)$ such that
    \eq\lbeq{convV}
    \lim_{\varepsilon \downarrow 0} V^{\smallsup{\vep}}=V.
    \en
\end{prop}
%%%%%%%%%%%PROPOSITION%%%%%%%%%
\noindent
Before proving Proposition \ref{prop-disc}, we first
complete the proof of Theorem~\ref{thm:rpt}.
\vskip0.3cm

\noindent
{\it Proof of Theorem~\ref{thm:rpt}.} We start by proving Theorem~\ref{thm:rpt}{\rm (i)}.
We first claim that
$\lim_{\vep \downarrow 0} \hat\tau_{\vec t;\vep}^{\lambc^{\smallsup{\vep}}}(\vec k)
=\hat \tau_{\vec t}^{\lambc}(\vec k)$. For this, the argument in
\cite[Section 2.5]{hsa04} can easily be adapted from the 2-point function to
the higher-point functions.

Using the convergence of $\hat\tau_{\vec t;\vep}^{\lambc^{\smallsup{\vep}}}(\vec k)$,
together with Theorem \ref{thm:rptvep}(i) and the uniformity of the
error term in \refeq{taurasyvep} in $\vep\in (0,1]$, to obtain
\begin{align}
\hat\tau_{T \vec t}^{\lambc}\big(\tfrac{\vec k}{\sqrt{v\sigma^2T}}\big)&=\lim_{\vep\daw0}
 \hat\tau_{T\vec t;\vep}^{\lambc^{(\vep)}}\big(\tfrac{\vec k}{\sqrt{v\sigma^2T}}\big)
=\lim_{\vep\daw0}\hat\tau_{T\vec t;\vep}^{\lambc^{(\vep)}}\Big(\tfrac{\sqrt{v^{
 (\vep)}}}{\sqrt{v}}\tfrac{\vec k}{\sqrt{v^{(\vep)}\sigma^2T}}\Big)\nn\\
&=\lim_{\vep\daw0}  A^{\smallsup{\vep}}((A^{\smallsup{\vep}})^2V^{\smallsup{\vep}}T)^{r-2}
\big[\wM_{\vec{t}}^{\smallsup{r-1}}\big(\tfrac{\sqrt{v^{
 (\vep)}}}{\sqrt{v}}\vk\big)
    +O(T^{-\delta})\big] \nn\\
&=A(A^2Vt)^{r-2}
\big[\wM_{\vec{t}}^{\smallsup{r-1}}(\vk)
    +O(T^{-\delta})\big],
    \lbeq{continlimproof}
\end{align}
where we have made use of the convergence of $v^{\smallsup{\vep}}$ to $v$,
and the fact that $\vk\mapsto \wM_{\vec{t}}^{\smallsup{r-1}}(\vk)$ is
continuous. This proves \refeq{taurasy}.

The proof of Theorem \ref{thm:rptvep}(ii) is similar, where on the right-hand side
of \refeq{continlimproof} we need to replace \ch{$A, A^{\smallsup{\vep}}, v$ and
$v^{\smallsup{\vep}}$ by 1,
$V^{\smallsup{\vep}}$ by $2-\vep$, $V$ by 2 and $\delta$ by $\mu\wedge\delta$.}
\qed

\vskip0.3cm

\noindent
{\it Proof of Proposition \ref{prop-disc}.}
The proof of the continuum limit is substantially different from the
proof used in \cite{hsa04}, where, among other things, it was shown that
$A^{\smallsup{\vep}}$ and $v^{\smallsup{\vep}}$ converge as $\vep\downarrow
0$. The main idea behind the argument in this paper also applies to
the convergence of $A^{\smallsup{\vep}}$ and $v^{\smallsup{\vep}}$,
as we first show. This simpler argument leads to an alternative proof
of the convergence of $A^{\smallsup{\vep}}$ and $v^{\smallsup{\vep}}$.

For this proof, we use \cite[Proposition 2.1]{hsa04}, which states that,
uniformly in $\vep\in (0,1]$,
    \eq
    \hat\tau_{t;\vep}^{\lambc^{(\vep)}}(0)=A^{\smallsup{\vep}}
    [1+O(t^{-(d-4)/2})].
    \en
The uniformity of the error term can be reformulated by saying that
    \eq
    \lbeq{taueq}
    \hat\tau_{t;\vep}^{\lambc^{(\vep)}}(0)=A^{\smallsup{\vep}}
    [1+\gamma_{\vep}(t)],
    \en
where
    \eq
    \lbeq{gammabd}
    \gamma(t) = \sup_{\vep\in (0,1]}|\gamma_{\vep}(t)|=O((t+1)^{-(d-4)/2}).
    \en
Therefore, we obtain that, uniformly in $\vep\in (0,1]$ and $t\geq 0$,
    \eq
    \hat\tau_{t;\vep}^{\lambc^{(\vep)}}(0) [1+\gamma(t)]^{-1}\leq
    A^{\smallsup{\vep}}
    \leq \hat\tau_{t;\vep}^{\lambc^{(\vep)}}(0) [1-\gamma(t)]^{-1}.
    \en
Now we take the limit $\vep\downarrow 0$, and use that,
as proved in \cite[Section 2.4]{hsa04}, we have $\lim_{\vep \downarrow 0} \hat\tau_{t;\vep}^{\lambc^{\smallsup{\vep}}}(0)
=\hat \tau_{t}^{\lambc}(0)$, to obtain that
    \eq
    \lbeq{sandwich}
    \hat \tau_{t}^{\lambc}(0) \frac{1}{1+\gamma(t)}
    \leq
    \liminf_{\vep\downarrow 0} A^{\smallsup{\vep}}\leq \limsup_{\vep\downarrow 0} A^{\smallsup{\vep}}
    \leq \hat \tau_{t}^{\lambc}(0)\frac{1}{1-\gamma(t)}.
    \en

Since $A^{\smallsup{\vep}}=1+O(\beta)$, uniformly in $\vep\in (0,1]$, we see from
\refeq{taueq} that $t\mapsto \hat\tau_{t;\vep}^{\lambc}(0)$ is a bounded sequence.
Therefore,
%Since $\lim_{\vep \downarrow 0} \hat\tau_{t;\vep}^{\lambc^{\smallsup{\vep}}}(0)
%=\hat \tau_{t}^{\lambc}(0)$,
we conclude that also $\hat \tau_{t}^{\lambc}(0)$
is uniformly bounded in $t\geq 0$. Therefore, \ch{there exists a subsequence of
times $\{t_l\}_{l=1}^{\infty}$ satisfying $t_l\rightarrow \infty$ such that
$\hat \tau_{t_l}^{\lambc}(0)$ converges
as $l\rightarrow \infty$.} Denote the limit of $\hat \tau_{t_l}^{\lambc}(0)$
by $A$. Then we obtain from \refeq{gammabd} and \refeq{sandwich} that
    \eq
    A
    \leq
    \liminf_{\vep\downarrow 0} A^{\smallsup{\vep}}\leq \limsup_{\vep\downarrow 0} A^{\smallsup{\vep}}
    \leq A,
    \en
so that $\lim_{\vep\downarrow 0} A^{\smallsup{\vep}}=A$. This completes the proof
of convergence of $A^{\smallsup{\vep}}$. A similar proof can also be used to
prove that the limit $\lim_{\vep\downarrow 0} v^{\smallsup{\vep}}=v$ exists.

On the other hand, the proof in \cite{hsa04} was based on the explicit
formula for $A^{\sss(\vep)}$, which reads
    \eq
    \lbeq{Avep-def}
    A^{\sss(\vep)}=
    \frac{\dpst1+\ddsum_{s=2\vep}^\infty\hat\pi_{s;\vep}^{\lambc^{(\vep)}}\!(0)}
    {\dpst1+\frac1\vep\ddsum_{s=2\vep}^\infty s\;\hat\pi^{
    \lambc^{\smallsup{\vep}}}_{s;\vep}\!(0)\,\hat p_\vep^{\lambc^{(\vep)}}\!
    (0)},
    \en
where $\hat p_\vep^\lamb(k)=1-\vep+\lamb\vep\hat D(k)$, and on the fact
that $\frac1{\vep^2}\hat\pi_{s;\vep}^{\lambc^{(\vep)}}(0)$ converges as
$\vep\daw0$ for every $s>0$.  This proof was much more involved, but also
allowed us to give a formula for $A$ in terms of the pointwise limits of
$\frac1{\vep^2}\hat\pi_{s;\vep}^{\lambc^{(\vep)}}(0)$ as $\vep\daw0$.

For the convergence of $V^{\smallsup{\vep}}$, we adapt the above simple argument
proving convergence of $A^{\sss(\vep)}$.
We use \refeq{taurasyvep} for $r=3$, $\vec t=(t,t)$ and $\vec k=0$, rewritten in the following way:
    \eq
    \hat{\tau}_{(t,t);\vep}^\smallsup{3}(0,0)
    = \big(A^{\smallsup{\vep}}\big)^3V^{\smallsup{\vep}} t
    \big(1
    + \gamma_{\vep}(t)\big),
    \lbeq{res3pt}
    \en
where the error term satisfies $\gamma_{\vep}(t)=O((t+1)^{-\delta})$
uniformly in $\vep$. Therefore,
    \eq
    \gamma(t) = \sup_{\vep\in (0,1]}|\gamma_{\vep}(t)|
    =O\big((t+1)^{-\delta}\big).
    \en
We conclude that
    \begin{align}
    \frac{(A^{\sss(\vep)})^{-3}}{(1+\gamma(t))\,t}\,\hat{\tau}_{(t,t);
     \vep}^{\sss(3)}(0,0)\leq V^{\sss(\vep)}\leq\frac{(A^{\sss(\vep)})
     ^{-3}}{(1-\gamma(t))\,t}\,\hat{\tau}_{(t,t);\vep}^{\sss(3)}(0,0)
    \end{align}
We next let $\vep\daw0$, and use that the limits
    \begin{align}
    \lim_{\vep\daw0}\hat\tau_{(t,t);\vep}^{\sss(3)}(0,0)=\hat\tau_{(t,t)}
     ^{\sss(3)}(0,0),&&
    \lim_{\vep\daw0}A^{\sss(\vep)}=A
    \end{align}
both exist, so that
    \begin{align}
    \frac{A^{-3}}{(1+\gamma(t))\,t}\,\hat{\tau}_{(t,t)}^{\sss(3)}(0,0)
     \leq\liminf_{\vep\daw0}V^{\sss(\vep)}\leq\limsup_{\vep\daw0}V^{\sss
     (\vep)}\leq\frac{A^{-3}}{(1-\gamma(t))\,t}\,\hat{\tau}_{(t,t)}^{\sss
     (3)}(0,0).
    \end{align}
The above bounds are true for {\it any $t$}. Moreover, by the tree-graph inequality
and the fact that $\hat{\tau}_{t}^\smallsup{2}(0)$ is uniformly bounded in $t$,
we know that $\frac1t \hat{\tau}_{(t,t)}^\smallsup{3}(0,0)$ is uniformly bounded in $t$.
Indeed, the tree-graph inequality yields that
    \eq
    \hat{\tau}_{(t,t)}^\smallsup{3}(0,0)\leq 2\int_{0}^t
    \hat{\tau}_{s}^\smallsup{2}(0)(\lambda D\sstar \hat{\tau}_{t-s}^\smallsup{2})(0)
    \hat{\tau}_{t-s}^\smallsup{2}(0)ds.
    \en
Since $\hat{\tau}_{s}^\smallsup{2}(0)$ is uniformly bounded in $s$ by $K$, say,
we obtain that, uniformly in $t$,
    \eq
    \hat{\tau}_{(t,t)}^\smallsup{3}(0,0)\leq 2\lambda K^3 t.
    \en
Therefore, there exists a subsequence $\{t_l\}_{l=1}^{\infty}$
with $\lim_{l\rightarrow \infty} t_l=\infty$
such that the limit
    \eq
    \lim_{l\rightarrow \infty} \frac1{t_l} \hat{\tau}_{(t_l,t_l)}^\smallsup{3}(0,0)
    \equiv A^3 V
    \en
exists. Then, using that $\gamma(t)=o(1)$ as $t\rightarrow \infty$,
we come to the conclusion that
    \eq
    V=A^{-3}(A^3V)\leq\liminf_{\vep\daw0}V^{\smallsup{\vep}}
    \leq\limsup_{\vep\daw0}V^{\smallsup{\vep}}\leq A^{-3}(A^3V)=V,
    \en
that is, $\lim_{\vep\daw0}V^{\sss(\vep)}=V$.  This completes
the proof of Proposition~\ref{prop-disc}.
\qed

\newcommand{\piv}{{\tt piv}}
\newcommand{\Lb}[2]{\cL_{#1}^{\sss(#2)}}
\newcommand{\ta}{{\bar a}}
\newcommand{\bba}{{\underline a}}
\newcommand{\sN}{\sss N}

\newcommand{\bL}{{\bf L}}
\newcommand{\bS}{{\bf S}}
\newcommand{\eb}{\underline{e}}
\newcommand{\mB}{{\mathbb B}}
\newcommand{\Fone}{F'}
\newcommand{\Ftwo}{H}
\newcommand{\Pno}{P_\vno}
\newcommand{\tPno}{\tilde P_\vno}
\newcommand{\te}{\overline{e}}
\newcommand{\vno}{\varnothing}

\newcommand{\sB}{{\sss B}}

%\newcommand{\Edef}[3]
% {\sideset{\raisebox{7pt}{$#1$}\hspace{18pt}\raisebox{-15pt}{$#2$}}{}
% {\mathop{\raisebox{-16pt}{\hspace{-27pt}\includegraphics[scale=0.25]
% {Edef}}}}^{\hspace{-3pt}\raisebox{3pt}{$#3$}}}
%
%\newcommand{\EEdef}[3]
% {\overbrace{\sideset{\raisebox{-2pt}{$#1$}}{}{\mathop{\raisebox{-5pt}
% {\hspace{-20pt}\includegraphics[scale=0.22]{EEdef}}}}_{\hspace{5pt}
% \raisebox{-5pt}{$#2$}}}^{\dpst #3}}

\section{Linear expansion for the $r$-point function}\label{s:lace}
In this section, we derive the expansion \refeq{taunxvecr2} which
extracts an explicit $r$-point function $\tau(\vec{\xvec}_J-\vvec)$,
and an unexpanded contribution $A(\vec{\xvec}_J)$.  In
Section~\ref{s:exp-applexp}, we investigate $A(\vec{\xvec}_J)$ using
two expansions.  The first of these expansions extracts a factor
$\tau(\vec\xvec_{J\setminus I}-\yvec_1)$ from $A(\vec{\xvec}_J)$
in \refeq{BCE}, and the
second expansion extracts a factor $\tau(\vec\xvec_I-\yvec_2)$
from $A(\vec{\xvec}_J)$. This will lead to \refeq{Adec}--\refeq{BCE}.

From now on, we suppress the dependence on $\lambda$ and $\vep$ when no
confusion can arise.  The $r$-point function is defined by
\begin{align}\lbeq{rpt-remind}
\tau(\vec\xvec_J)=\mP(\ovec\conn\vec\xvec_J),
\end{align}
where we recall the notation \refeq{JJ1} and \refeq{conn-all}.
Rather than expanding \refeq{rpt-remind}, we expand a generalized version
of the $r$-point function defined below.

\begin{defn}[\ch{\textbf{Connections through $\bC$}}]\label{def:through}
{\rm Given a configuration and a set of sites $\bC$, we say that $\yvec$ is
connected to $\xvec$ {\it through} $\bC$, if every occupied path from $\yvec$
to $\xvec$ has at least one bond with an endpoint in $\bC$. This event is
written as $\yvec\ct{\bC}\xvec$.  Similarly, we write}
\begin{align}\lbeq{through}
\big\{\yvec\ct{\bC}\vec\xvec_J\big\}=\{\yvec\conn\vec\xvec_J\}\cap
\big\{\exists\,j\in J\text{ such that }\yvec\ct{\bC}\xvec_j\big\}.
\end{align}
\end{defn}

Below, we derive an expansion for $\mP\big(\vvec\ct{\bC}\vec\xvec_J\big)$.
This is more general than an expansion for the $r$-point function
$\tau(\vec\xvec_J)$, since
\begin{align}\lbeq{taur-rewr}
\tau(\vec\xvec_J)=\mP\big(\ovec\ctx{\{\ovec\}}\vec\xvec_J\big).
\end{align}
Thus, to obtain the linear expansion for the $r$-point function,
we need to specialize to $\yvec=\ovec$ and $\bC=\{\ovec\}$.
Before starting with the expansion, we introduce some further notation.

\begin{defn}[{\bf Clusters and pivotal bonds}]
{\rm Let $\bC(\xvec)=\{\yvec\in\Lambda:\xvec\conn\yvec\}$ denote the forward cluster
of $\xvec\in\Lambda$.  Given a bond $b$, we define $\tilde\bC^b(\xvec)\subseteq
\bC(\xvec)$ to be the set of sites to which $\xvec$ is connected in the
(possibly modified) configuration in which $b$ is made vacant. We say that $b$
is \emph{pivotal for} $\xvec\conn\yvec$ if
$\yvec\in\bC(\xvec)\setminus\tilde\bC^b(\xvec)$, i.e., if $\xvec$ is connected
to $\yvec$ in the possibly modified configuration in which the bond is made
occupied, whereas $\xvec$ is not connected to $\yvec$ in the possibly modified
configuration in which the bond is made vacant.}
%\item We say that $\xvec$ is {\em doubly connected to} $\yvec$,
%and we write $\xvec\dbc\yvec$, if $\xvec\conn\yvec$ but there are
%no pivotal bonds for $\xvec\conn\yvec$, i.e., if there are at least
%two bond-disjoint occupied paths from $\xvec$ to $\yvec$.
%By convention, we say that $\xvec\dbc\xvec$ for all $\xvec$.

\end{defn}

\ch{
\begin{rem}[\textbf{Clusters as collections of bonds}]
\label{rem-clusbonds}
We shall also often view $\bC(\xvec)$ and $\tilde\bC^b(\xvec)$
as collections of bonds, and abuse notation
to write, for a bond $a$, that $a\in \bC(\xvec)$
(resp.\ $a\in \tilde\bC^b(\xvec)$) when $\underline{a}\in \bC(\xvec)$ and $a$ is occupied
(resp.\ $\underline{a}\in \tilde\bC^b(\xvec)$ and $a$ is occupied).
\end{rem}
}
We now start the first step of the expansion.  For a bond
$b=(\xvec,\yvec)$, we write $\bb=\xvec$ and $\tb=\yvec$.  The event
$\{\vvec\ct{\bC}\vec\xvec_J\}$ can be decomposed into two disjoint
events depending on whether there is or is not a common pivotal bond $b$
for $\vvec\conn\xvec_j$ for all $j\in J$ such that $\vvec\ct{\bC}\bb$.  Let
\begin{align}
E'(\vvec,\vec\xvec_J;\bC)&=\big\{\vvec\ct{\bC}\vec\xvec_J\big\}\cap
 \big\{\nexists\text{ pivotal bond $b$ for }\vvec\conn\xvec_j~\forall
 j\in J\text{ such that }\vvec\ct{\bC}\bb\big\}.\lbeq{E'def}
\end{align}
See Figure~\ref{fig-E} for schematic representations of
$E'(\vvec,\xvec;\bC)$ and $E'(\vvec,\vec\xvec_J;\bC)$.
%%%%%FIGFIGFIGFIGFIGFIGFIGFIGFIGFIGFIGFIGFIGFIGFIGFIGFIGFIGFIGFIG
\begin{figure}[t]
\begin{center}
%\setlength{\unitlength}{0.00067in}
%{
%\begin{picture}(1000,3000)(0,500)
%{
%\put(500, 2950){\ellipse{200}{400}}
%\path(500,2600)(500,2750)
%\put(500, 2400){\ellipse{200}{400}}
%\path(500,2050)(500,2200)
%\put(500, 1850){\ellipse{200}{400}}
%\path(500,1500)(500,1650)
%\put(500, 1300){\ellipse{200}{400}}
%\path(500,950)(500,1100)
%\put(500, 750){\ellipse{200}{400}}
%
%\put(250,450){\makebox(0,0)[lb]{$b$}}
%\put(430,3250){\makebox(0,0)[lb]{$\xvec$}}
%\put(-300,2000){\makebox(0,0)[lb]{$\bC$}}
%
%\thicklines
%\qbezier(-200,1000)(0,2500)(650,3100)
%\path(430,550)(570,550)
%\path(430,450)(570,450)
%\thinlines
%}
%\end{picture}
%}\hskip 7pc
%\includegraphics[scale=0.3]{Edef}
\begin{align*}
E'(\vvec,\xvec;\bC)~=\quad
 \raisebox{-3.5pc}{\includegraphics[scale=0.20]{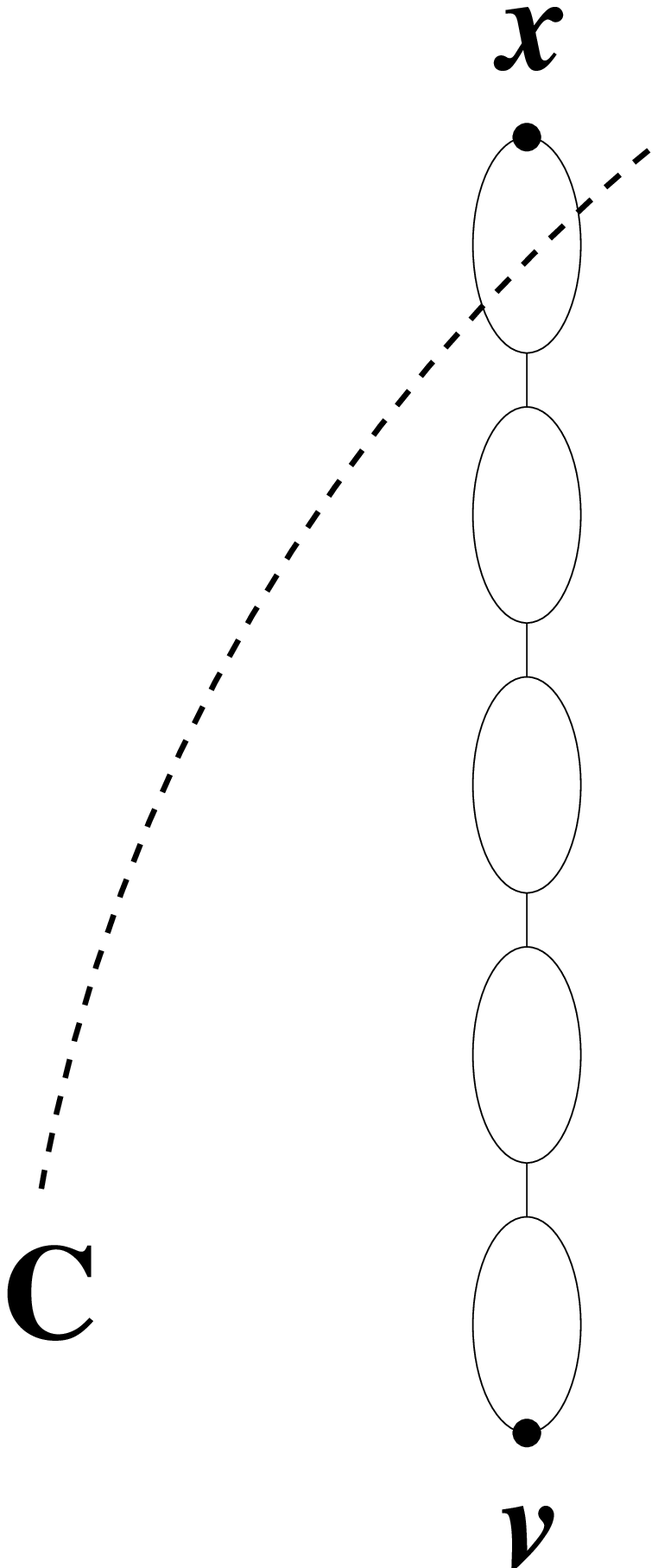}}&&
E'(\vvec,\vec\xvec_J;\bC)~=
  \raisebox{-4pc}{\includegraphics[scale=0.20]{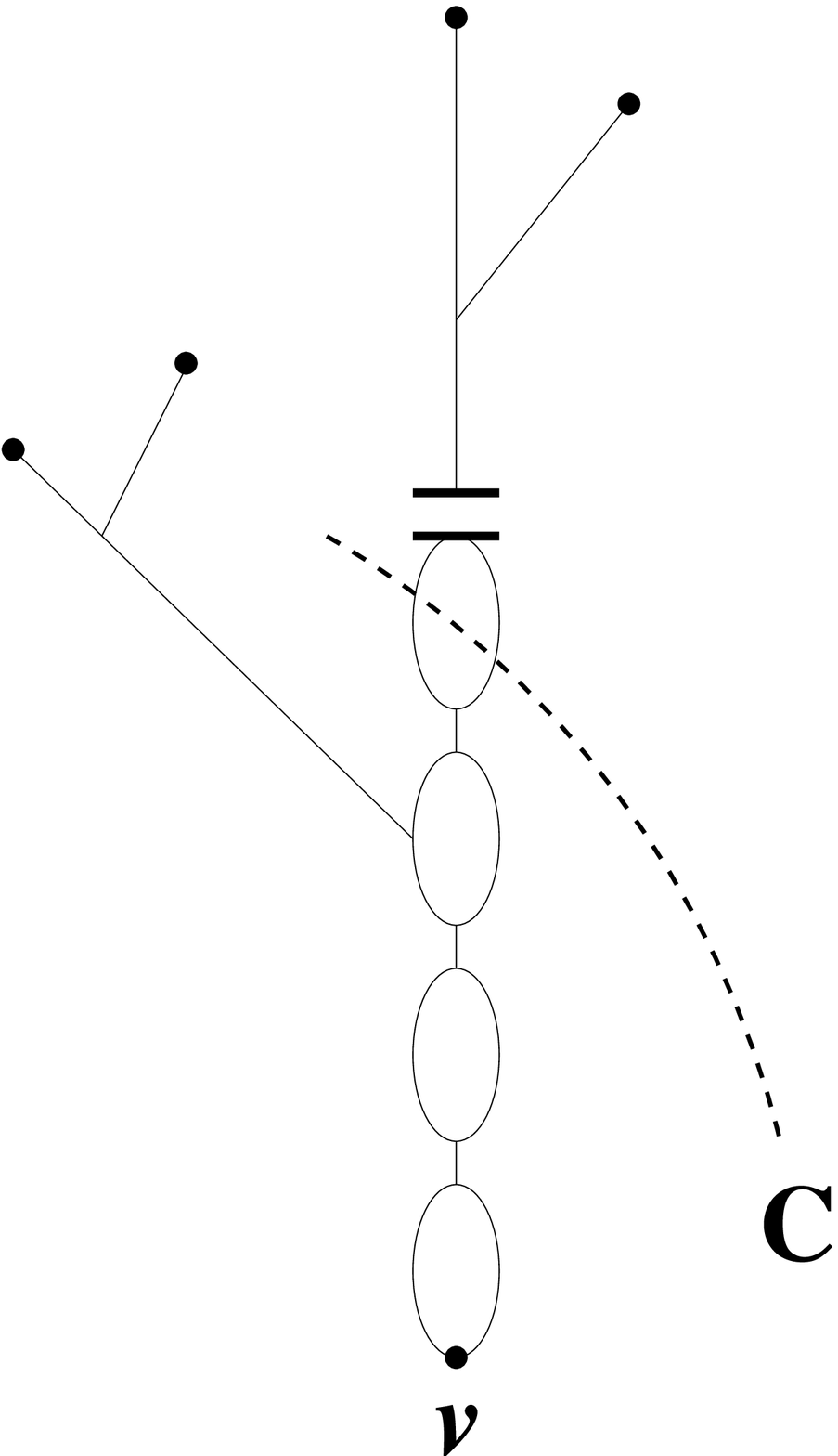}}
%\EEdef{\bC}{\vvec}{\vec\xvec_J}&&\Edef{\bC}{b}{\xvec}
\end{align*}
\caption{\label{fig-E}Schematic representations of
$E'(\vvec,\xvec;\bC)$ and $E'(\vvec,\vec\xvec_J;\bC)$.
The vertices at the top of the right figure are the components
of $\vec\xvec_J$.}
\end{center}
\end{figure}
%%%%%FIGFIGFIGFIGFIGFIGFIGFIGFIGFIGFIGFIGFIGFIGFIGFIGFIGFIGFIG

If there are such pivotal bonds, then we take the {\it first} bond among them.
This leads to the following \ch{partition}:

\begin{lem}[\textbf{Partition}]
\ch{For every $\vvec\in \Lambda$, $\vec\xvec_J\in \Lambda^{r-1}$ and $\bC\subseteq \Lambda$,}
\begin{align}\lbeq{decompE2}
\big\{\vvec\ct{\bC}\vec\xvec_J\big\}=E'(\vvec,\vec\xvec_J;\bC)\DDcup
\BDcup{b}\Big\{E'(\vvec,\bb;\bC)\cap\{b\text{ is occupied \& pivotal for }
\vvec\conn\xvec_j~\forall j\in J\}\Big\}.
\end{align}
\end{lem}

\begin{proof}
See \cite[Lemma~3.3]{HHS05b}.
\end{proof}

Defining
\begin{align}\lbeq{A0predef}
A^{\sss(0)}(\vvec,\vec\xvec_J;\bC)=\mP(E'(\vvec,\vec\xvec_J;\bC)),
\end{align}
we obtain
\begin{align}\lbeq{0thexp}
\mP\big(\vvec\ct{\bC}\vec\xvec_J\big)=A^{\sss(0)}(\vvec,\vec\xvec_J;\bC)+
 \sum_b\mP\Big(E'(\vvec,\bb;\bC)\cap\{b\text{ is occupied \& pivotal for }
 \vvec\conn\xvec_j~\forall j\in J\}\Big).
\end{align}

For the second term, we will use a {\it Factorization Lemma}
(see \cite{HS00a}, and, in particular, \cite[Lemma 2.2]{HHS05b}).
To state that lemma below, we first introduce some
notation.

\ch{
\begin{defn}[\textbf{Occurring in and on}]\label{def:inon}
{\rm For a bond configuration $\omega$ and a certain set of bonds $\mB$,
we denote by $\omega|_\mB$ the bond configuration which agrees with
$\omega$ for all bonds in $\mB$, and which has all other bonds vacant.
Given a (deterministic or random) set of vertices $\bC$, we let
$\mB_\bC=\big\{b:\{\bb,\tb\}\subset\bC\big\}$ and say that,
for events $E$,}
\begin{align}\lbeq{in-def}
\{E\text{ occurs in }\bC\}=\{\omega:\,\omega|_{\mB_\bC}\!\in E\}.
\end{align}
{\rm We adopt the convenient convention that $\{\xvec \conn \xvec$ in $\bC\}$
occurs if and only if $\xvec\in \bC$.}
\end{defn}
}

We will often omit ``occurs'' and simply write $\{E$ in $\bC\}$.  For example,
we define the {\it restricted $r$-point function} $\tau^{\sss\bC}(\vvec,\vec\xvec_J)$
by
\begin{align}\lbeq{restr-rpt}
\tau^{\sss\bC}(\vvec,\vec\xvec_J)=\prob{\vvec\conn\vec\xvec_J\text{ in }
 \Lambda\setminus\bC},
\end{align}
where we emphasize that, by the convention below \refeq{in-def},
$\tau^{\sss\bC}(\vvec,\vec\xvec_J)=0$ when
$\vvec \in \bC$. Note that, by Definition~\ref{def:through},
\begin{align}\lbeq{incl-excl}
\tau^{\sss\bC}(\vvec,\vec\xvec_J)=\tau(\vec\xvec_J-\vvec)-\mP\big(\vvec
 \ct{\bC}\vec\xvec_J\big).
\end{align}
A nice property of the notion \ch{of occurring} ``in'' is
its compatibility with operations in set theory (see
\cite[Lemma~2.3]{HS00a}):
\begin{align}\lbeq{on/inprop}
    \{E^{\rm c}\text{ in }\bC\}=\{E\text{ in }\bC\}^{\rm c},&&
    \{E\cap F\text{ in }\bC\}=\{E\text{ in }\bC\}\cap\{F\text{ in }\bC\}.
\end{align}

The statement of the Factorization Lemma is in terms of
{\it two} independent percolation configurations. The laws
of these independent configurations are indicated by subscripts,
i.e., $\Ebold_{\sss 0}$ denotes the expectation with respect to
the first percolation configuration, and $\Ebold_{\sss 1}$
denotes the expectation with respect to the second percolation configuration.
We also use the same subscripts for random variables, to indicate which law
describes their distribution. Thus, the law of
$\bC^{b}_{\sss 0}(\wvec)$
is described by $\Ebold_{\sss 0}$.

%%%%%%%%%%%%%%%%% LEMMA %%%%%%%%%%%%%%%%%%%%%%%%%%%%%%%%
\begin{lem}[\textbf{Factorization Lemma \cite[Lemma 2.2]{HHS05b}}]
\label{lem-cut1}
Given a site $\wvec\in\Lambda$, fix $\lamb\geq0$ such that $\bC(\wvec)$ is
almost surely finite.  For a bond $b$ and events $E,F$ determined by the
occupation status of bonds with time variables less than or equal to $t$
for some $t<\infty$,
\begin{align}\lbeq{lemcut1}
    \mE\big[\ind{E\text{ in }\tilde\bC^b(\wvec)}\;\ind{F\text{ in }\Lambda
    \setminus\tilde\bC^b(\wvec)}\big]=\mE_{\sss 0}\Big[\ind{E\text{ in }\tilde\bC^b_{\sss 0}
    (\wvec)}\;\mE_{\sss 1}\big[\ind{F\text{ in }\Lambda\setminus\tilde\bC^b_{\sss 0}(\wvec)}
    \big]\Big].
\end{align}
Moreover, when $E\subset\{\bb\in\tilde\bC^b(\wvec)\}\cap\{\tb\notin
\tilde\bC^b(\wvec)\}$, the event in the left-hand side is independent
of the occupation status of $b$.
\end{lem}
%%%%%%%%%%%%%%%%%%%%%%%%%%%%%%%%%%%%%%%%%%%%%%%%%%%%%%%

\medskip

We now apply this lemma to the second term in \refeq{0thexp}.
First, we note that
\begin{align}\lbeq{partition}
    &E'(\vvec,\bb;\bC)\cap\{b\text{ is occupied \& pivotal for }\vvec
    \conn\xvec_j~\forall j\in J\}\nn\\
    &~=\big\{E'(\vvec,\bb;\bC)
    \text{ in }\tilde\bC^b(\vvec)\big\}\cap\{b\text{ is occupied}\}
    \cap\big\{\tb\conn\vec\xvec_J\text{ in }\Lambda\setminus\tilde
    \bC^b(\vvec)\big\}.
\end{align}
Since $E'(\vvec,\bb;\bC)\subset\{\bb \in\tilde\bC^b(\vvec)\}$ and since the
event $\{\tb\conn\vec\xvec_J\text{ in }\Lambda\setminus\tilde \bC^b(\vvec)\}$
ensures that $\tb\notin\tilde\bC^b(\vvec)$, as required in
Lemma~\ref{lem-cut1}, the occupation status of $b$ is independent of the other
two events in \refeq{partition}.  Therefore, when we abbreviate
$p_b=p_\vep(\tb-\bb)$ (\ch{recall} \refeq{pabuse}) and make use of
\refeq{restr-rpt}--\refeq{incl-excl} as well as \refeq{lemcut1}, we obtain
\begin{align}\lbeq{fact-appl1}
    &\mP\Big(E'(\vvec,\bb;\bC)\cap\{b\text{ is occupied \& pivotal for }
    \vvec\conn\xvec_j~\forall j\in J\}\Big)\nn\\
    &~=\mE\Big[\ind{E'(\vvec,\bb;\bC)
    \text{ in }\tilde\bC^b(\vvec)}\;\ind{b\text{ is occupied}}\;\ind{\tb
    \conn\vec\xvec_J\text{ in }\Lambda\setminus\tilde\bC^b(\vvec)}\Big]\nn\\
    &~=p_b\,\mE\Big[\ind{E'(\vvec,\bb;\bC)
    \text{ in }\tilde\bC^b(\vvec)}~\mE\big[\ind{\tb\conn\vec
    \xvec_J\text{ in }\Lambda\setminus\tilde\bC^b(\vvec)}\big]\Big]\nn\\
    &~=p_b\,\mE\Big[\indic_{E'(\vvec,\bb;\bC)}
    \;\tau^{\sss\tilde\bC^b(\vvec)}(\tb,\vec\xvec_J)\Big]
    =p_b\,\mE\Big[\indic_{E'(\vvec,\bb;\bC)}\,\Big(\tau(\vec\xvec_J-\tb)
    -\mP\big(\tb\ctx{\sss\tilde\bC^b(\vvec)}\vec\xvec_J\big)\Big)\Big],
\end{align}
where we omit ``in $\tilde\bC^b(\vvec)$'' in the third equality, since
$E'(\vvec,\bb;\bC)$ depends only on bonds before time $t_{\bb}$
(where, for $\xvec=(x,t)\in \Lambda$, $t_{\xvec}=t$ denotes the
temporal component of $\xvec$).

Substituting \refeq{fact-appl1} in \refeq{0thexp}, we have
\begin{align}\lbeq{0thexp-rewr}
\mP\big(\vvec\ct{\bC}\vec\xvec_J\big)=A^{\sss(0)}(\vvec,\vec\xvec_J;\bC)
 +\sum_bp_b\,\mE\Big[\indic_{E'(\vvec,\bb;\bC)}\,\Big(\tau(\vec\xvec_J-
 \tb)-\mP\big(\tb\ctx{\sss\tilde\bC^b(\vvec)}\vec\xvec_J\big)\Big)\Big].
\end{align}
\ch{On the right-hand side of \refeq{0thexp-rewr}, again a generalised
$r$-point function appears, which allows us to iterate \refeq{0thexp-rewr},
by substituting the expansion for $\mP\big(\tb\ctx{\sss\tilde\bC^b(\vvec)}\vec\xvec_J\big)$
into the right-hand side of \refeq{0thexp-rewr}.}

In order to simplify the expressions arising in the expansion, we first
introduce some useful notation.  For a (random or deterministic) variable $X$,
we let
    \begin{align}\lbeq{B0def}
    M^{\sss(1)}_{\vvec,\vec\xvec_J;\bC}(X)=\mE\big[\indic_{E'(\vvec,\vec
    \xvec_J;\bC)}\;X\big],&&
    B^{\sss(0)}(\vvec,\yvec;\bC)=\sum_{b:\tb=\yvec}M^{\sss(1)}_{\vvec,\bb;
    \bC}(1)\;p_b.
    \end{align}
Note that, by this notation,
    \begin{align}\lbeq{A0def}
    A^{\sss(0)}(\vvec,\vec\xvec_J;\bC)=M^{\sss(1)}_{\vvec,\vec\xvec_J;\bC}(1).
    \end{align}
Then, \refeq{0thexp-rewr} equals
    \begin{align}\lbeq{Rn1rewr2}
    \mP\big(\vvec\ct{\bC}\vec\xvec_J\big)=A^{\sss(0)}(\vvec,\vec\xvec_J;\bC)+
    \sum_{\yvec}B^{\sss(0)}(\vvec,\yvec;\bC)\,\tau(\vec\xvec_J-\yvec)-\sum_b
    p_b\,M^{\sss(1)}_{\vvec,\bb;\bC}\Big(\mP\big(\tb\ctx{\sss\tilde\bC^b
    (\vvec)}\vec\xvec_J\big)\Big).
    \end{align}
This completes the first step of the expansion. We first take stock of what
we have achieved so far. In \refeq{Rn1rewr2}, we see that the generalized
$r$-point function $\mP\big(\vvec\ct{\bC}\vec\xvec_J\big)$ is written as
the sum of $A^{\sss(0)}(\vvec,\vec\xvec_J;\bC)$, a term which
is a convolution of some expansion term $B^{\sss(0)}(\vvec,\yvec;\bC)$
with an {\it ordinary} $r$-point function $\tau(\vec\xvec_J-\yvec)$
and a remainder term. The remainder term again involves a
generalized $r$-point function $\mP\big(\tb\ctx{\sss\tilde\bC^b
(\vvec)}\vec\xvec_J\big)$. Thus, we can iterate the above procedure,
until no more generalized $r$-point functions are present. This will prove
\refeq{taunxvecr2}.

%%%%%FIGFIGFIGFIGFIGFIGFIGFIGFIGFIGFIGFIGFIGFIGFIGFIGFIGFIGFIG
%%%%%FIGFIGFIGFIGFIGFIGFIGFIGFIGFIGFIGFIGFIGFIGFIGFIGFIGFIGFIG
\begin{figure}[t]
\begin{align*}
B^{\sss(0)}(\vvec,\xvec;\bC)~:~~
 \raisebox{-4pc}{\includegraphics[scale=0.20]{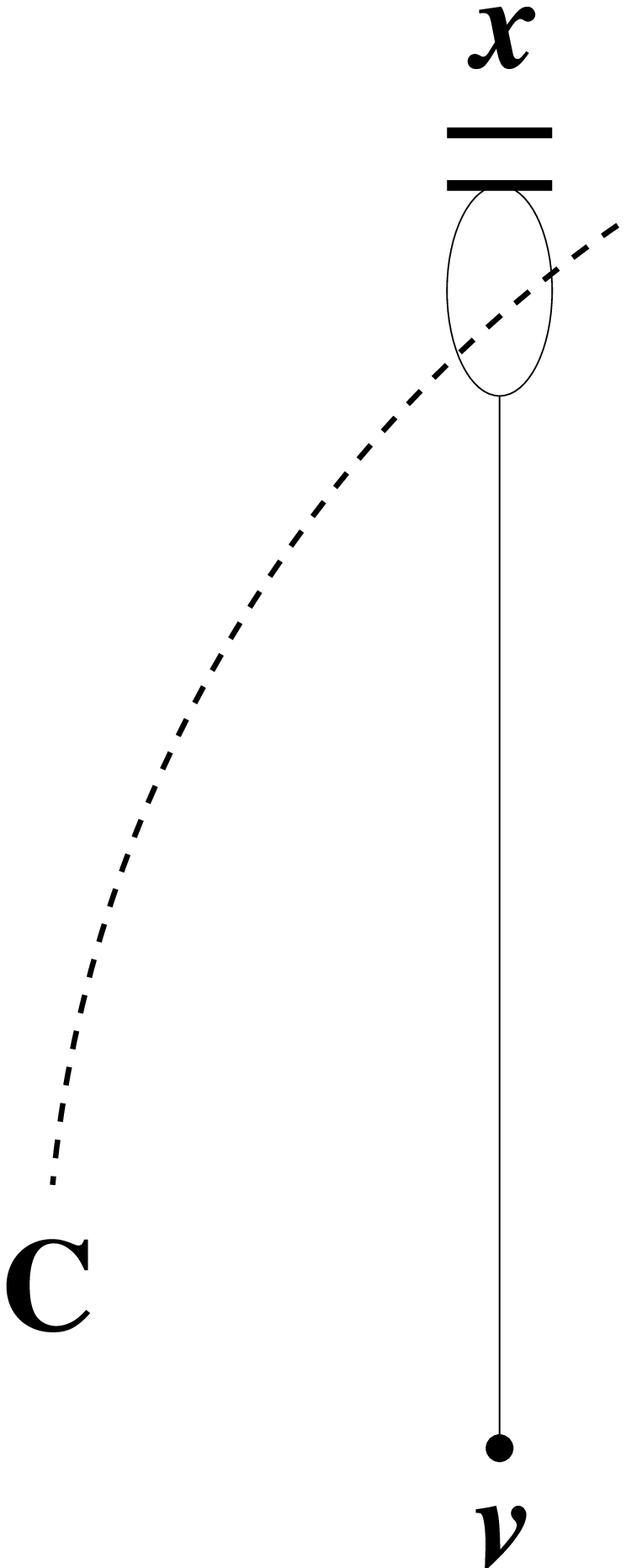}}&&
B^{\sss(1)}(\vvec,\xvec;\bC)~:~~
 \raisebox{-4pc}{\includegraphics[scale=0.20]{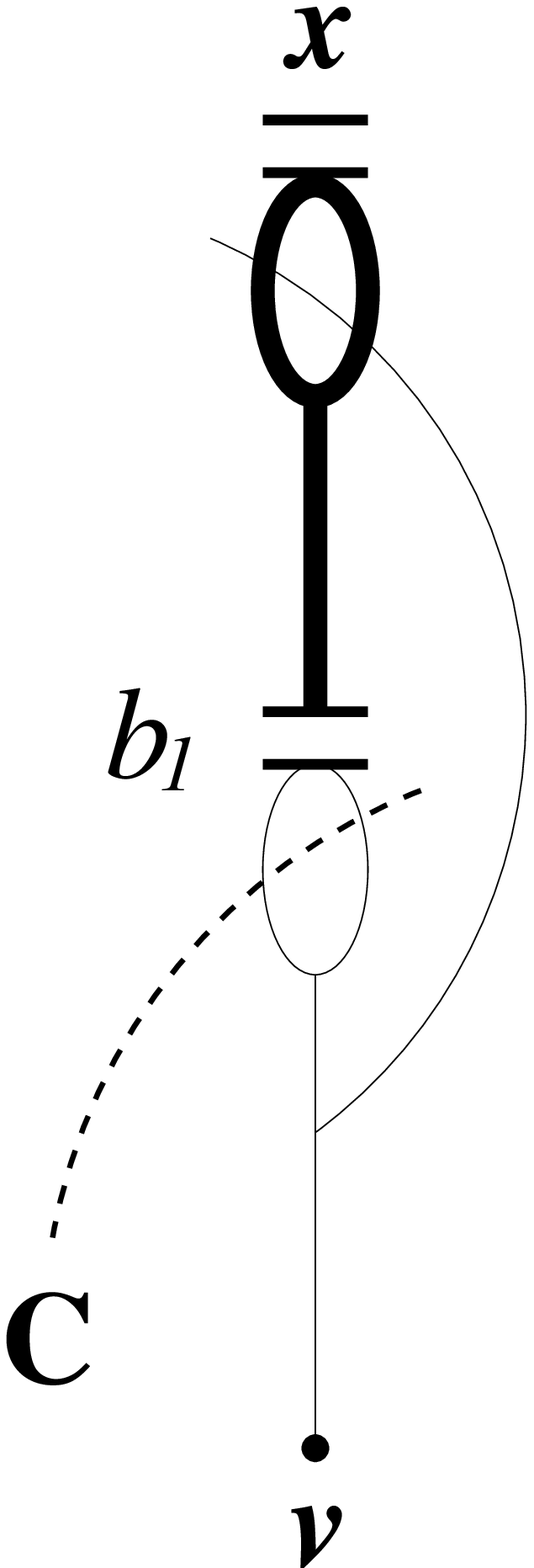}}
% \qquad\bigcup\quad\raisebox{-4pc}{\includegraphics[scale=0.20]{b11}}
\end{align*}
\caption{\label{fig-pi}Schematic representations of
$B^{\sss(0)}(\vvec,\xvec;\bC)$ and $B^{\sss(1)}(\vvec,\xvec;\bC)$.}
\end{figure}
%%%%%FIGFIGFIGFIGFIGFIGFIGFIGFIGFIGFIGFIGFIGFIGFIGFIGFIGFIGFIG
%%%%%FIGFIGFIGFIGFIGFIGFIGFIGFIGFIGFIGFIGFIGFIGFIGFIGFIGFIGFIG

In order to facilitate this iteration, and expand the right-hand side
in \refeq{Rn1rewr2} further, we first introduce
some more notation.  For $N\geq1$, we define
    \begin{align}\lbeq{M-def}
    M^{\sss(N+1)}_{\vvec,\vec\xvec_J;\bC}(X)&=\sum_{b_N}p_{b_N}
     M^{\sss(N)}_{\vvec,\bb_N;\bC}\Big(M^{\sss(1)}_{\tb_N,\vec\xvec_J;
     \tilde\bC_{N-1}}(X)\Big)\nn\\
    &=\sum_{\vec b_N=(b_1,\dots,b_N)}\prod_{i=1}^Np_{b_i}\,M^{\sss
     (1)}_{\vvec,\bb_1;\bC}\bigg(M^{\sss(1)}_{\tb_1,\bb_2;\tilde\bC_0}
     \Big(\cdots M^{\sss(1)}_{\tb_N,\vec\xvec_J;\tilde\bC_{N-1}}(X)\cdots
     \Big)\bigg),
    \end{align}
where the superscript $n$ of $M^{\sss(n)}$ denotes the number of involved
nested expectations, and, for $n\geq0$, we abbreviate
$\tilde\bC^{b_{n+1}}(\tb_n)=\tilde\bC_n$, where we use the convention that
$\tb_0=\vvec$, which is the initial vertex in
$M^{\sss(N+1)}_{\vvec,\vec\xvec_J;\bC}$.

Let
    \begin{align}\lbeq{ANBNdef}
    A^{\sss(N)}(\vvec,\vec\xvec_J;\bC)=M^{\sss(N+1)}_{\vvec,\vec\xvec_J;\bC}
    (1),&&
    B^{\sss(N)}(\vvec,\yvec;\bC)=\sum_{b:\tb=\yvec}M^{\sss(N+1)}_{\vvec,\bb;
    \bC}(1)\,p_b,
    \end{align}
which agree with \refeq{B0def}--\refeq{A0def} when $N=0$.  Note that
$A^{\sss(N)}(\vvec,\vec\xvec_J;\bC)=B^{\sss(N)}(\vvec,\yvec;\bC)=0$ for
$N\vep>\min_{j\in J}t_{\xvec_j}-t_{\vvec}$, since, by the recursive
definition \refeq{M-def}, the operation $M^{\sss(N+1)}$ eats up at least
$N$ time-units (where one time-unit is $\vep$).

We now resume the expansion of the right-hand side of \refeq{Rn1rewr2}.
As we notice, we have $\mP(\vvec\ct{\bC}\vec\xvec_J)$ again in the
right-hand side of \refeq{Rn1rewr2}, but now with $\vvec$ and $\bC$ being
replaced by $\tb$ and $\tilde\bC^b(\vvec)$, respectively.  Applying
\refeq{Rn1rewr2} to its own right-hand side, we obtain
    \begin{align}\lbeq{1stexp}
    \mP\big(\vvec\ct{\bC}\vec\xvec_J\big)=\Big(A^{\sss(0)}(\vvec,\vec\xvec_J;
    \bC)-A^{\sss(1)}(\vvec,\vec\xvec_J;\bC)\Big)&+\sum_{\yvec}\Big(B^{\sss(0)}
    (\vvec,\yvec;\bC)-B^{\sss(1)}(\vvec,\yvec;\bC)\Big)~\tau(\vec\xvec_J-
    \yvec)\nn\\
    &+\sum_{b_2}p_{b_2}\,M^{\sss(2)}_{\vvec,\bb_2;\bC}\Big(\mP\big(\tb_2
    \ctx{\sss\tilde\bC_1}\vec\xvec_J\big)\Big).
    \end{align}
\ch{Define
    \begin{align}\lbeq{AB-alt}
    A(\vvec,\vec\xvec_J;\bC)=\sum_{N=0}^\infty(-1)^NA^{\sss(N)}(\vvec,\vec
    \xvec_J;\bC),&&
    B(\vvec,\yvec;\bC)=\sum_{N=0}^\infty(-1)^NB^{\sss(N)}(\vvec,\yvec;\bC).
    \end{align}
By repeated application of \refeq{Rn1rewr2} to \refeq{1stexp} until the
remainder vanishes (which happens after a finite number of iterations, see
below \refeq{ANBNdef}), we arrive at the following conclusion, which
is the linear expansion for the generalised $r$-point function:}

\begin{prop}[\textbf{Linear expansion}]\label{prop:exp-first}
For any $J\ne\varnothing$, $\lamb\leq\lambc$ and $\vec\xvec_J\in\Lambda^{|J|}$,
    \begin{align}\lbeq{taur-exp}
    \mP\big(\vvec\ct{\bC}\vec\xvec_J\big)=A(\vvec,\vec\xvec_J;\bC)+\sum_{\yvec}
    B(\vvec,\yvec;\bC)\,\tau(\vec\xvec_J-\yvec).
    \end{align}
\end{prop}

\begin{figure}[t]
\begin{align*}
%\raisebox{-4pc}{\includegraphics[scale=0.17]{rptC1}}~
\mP\big(\vvec\ct{\bC}\vec\xvec_J\big)~=~
 \left(~\raisebox{-4pc}{\includegraphics[scale=0.17]{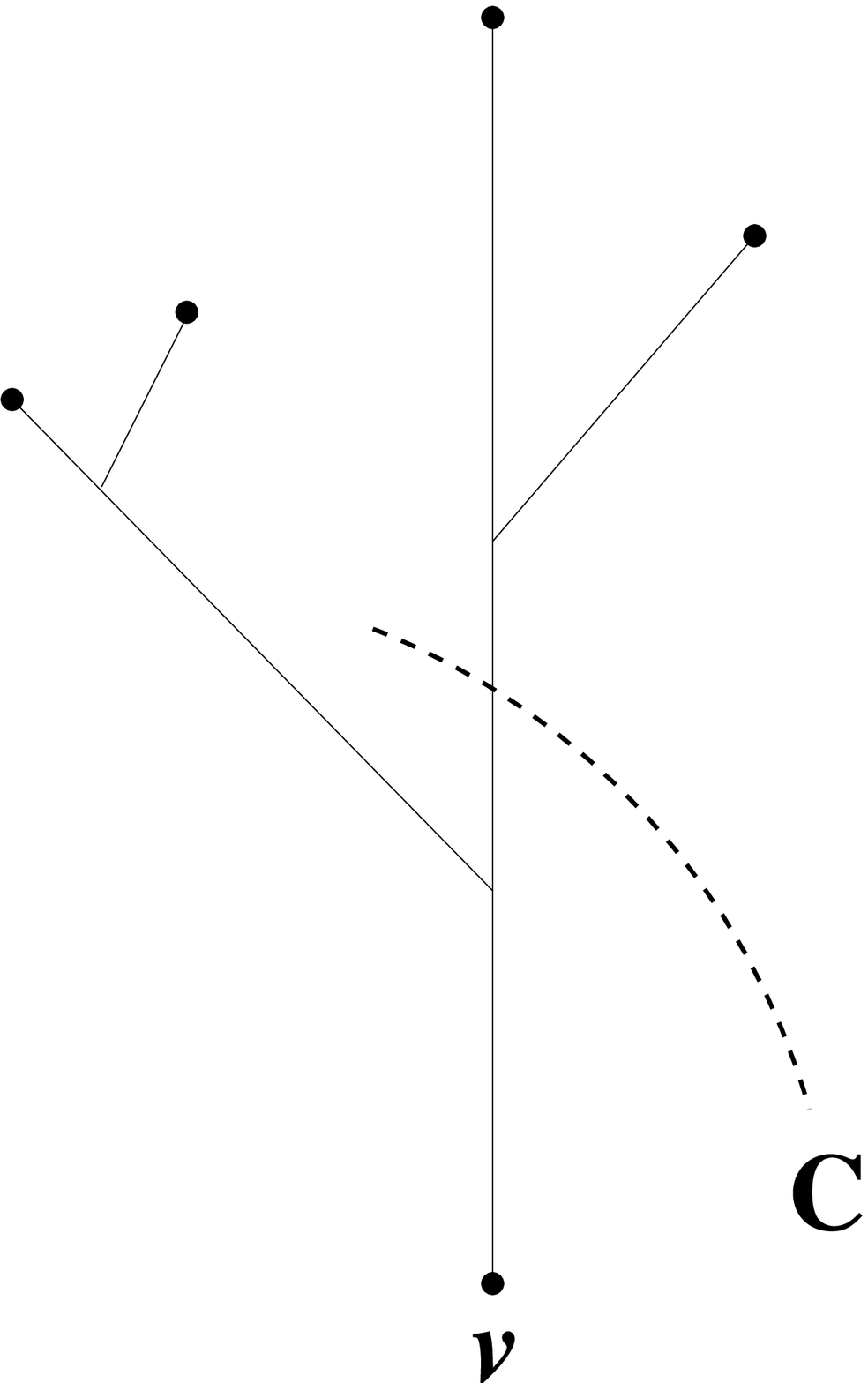}}~-
 \raisebox{-4pc}{\includegraphics[scale=0.17]{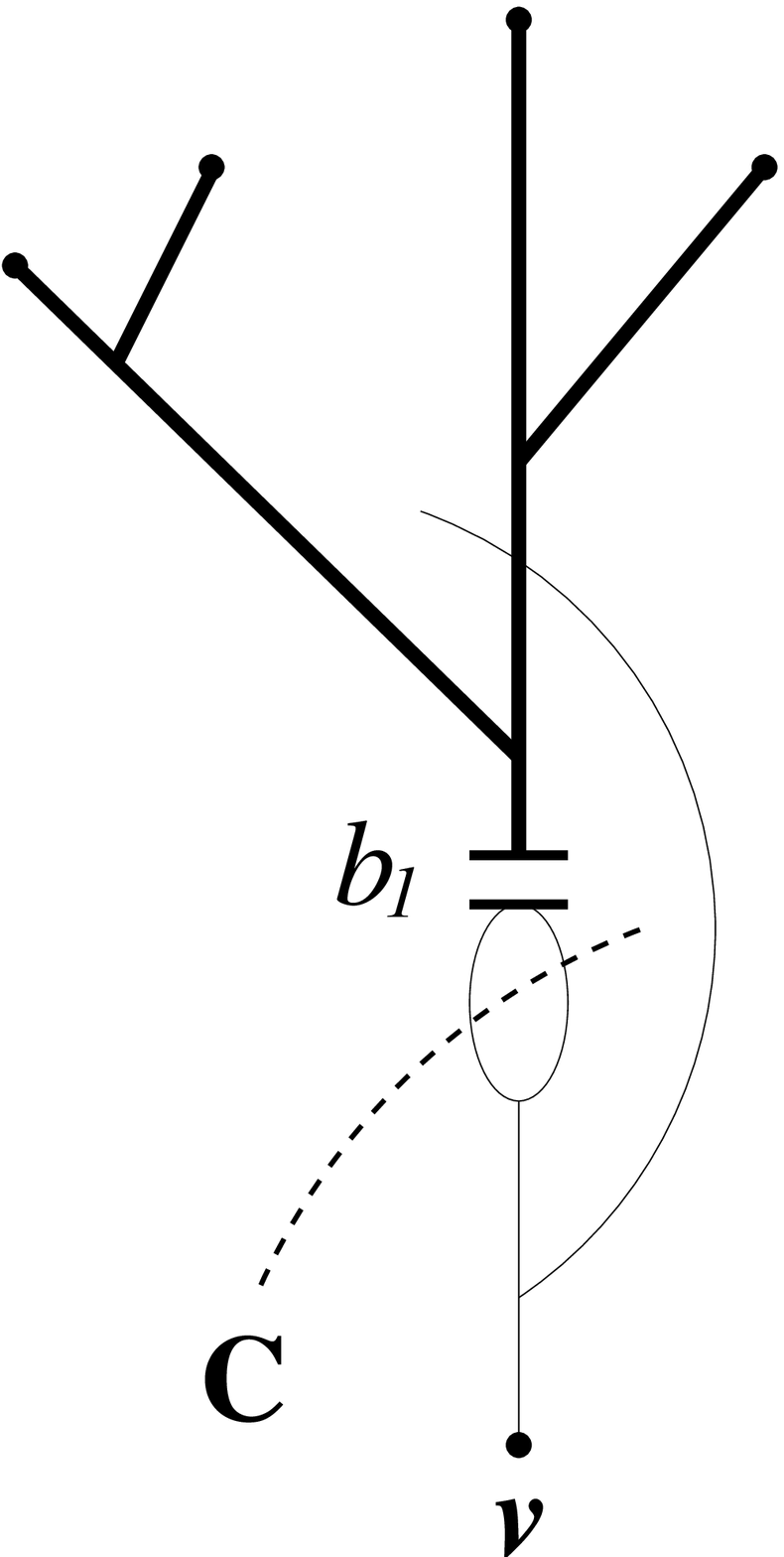}}~~+~\cdots~~\right)~+~
 \sum_{\yvec}\left(~\raisebox{-4pc}{\includegraphics[scale=0.17]{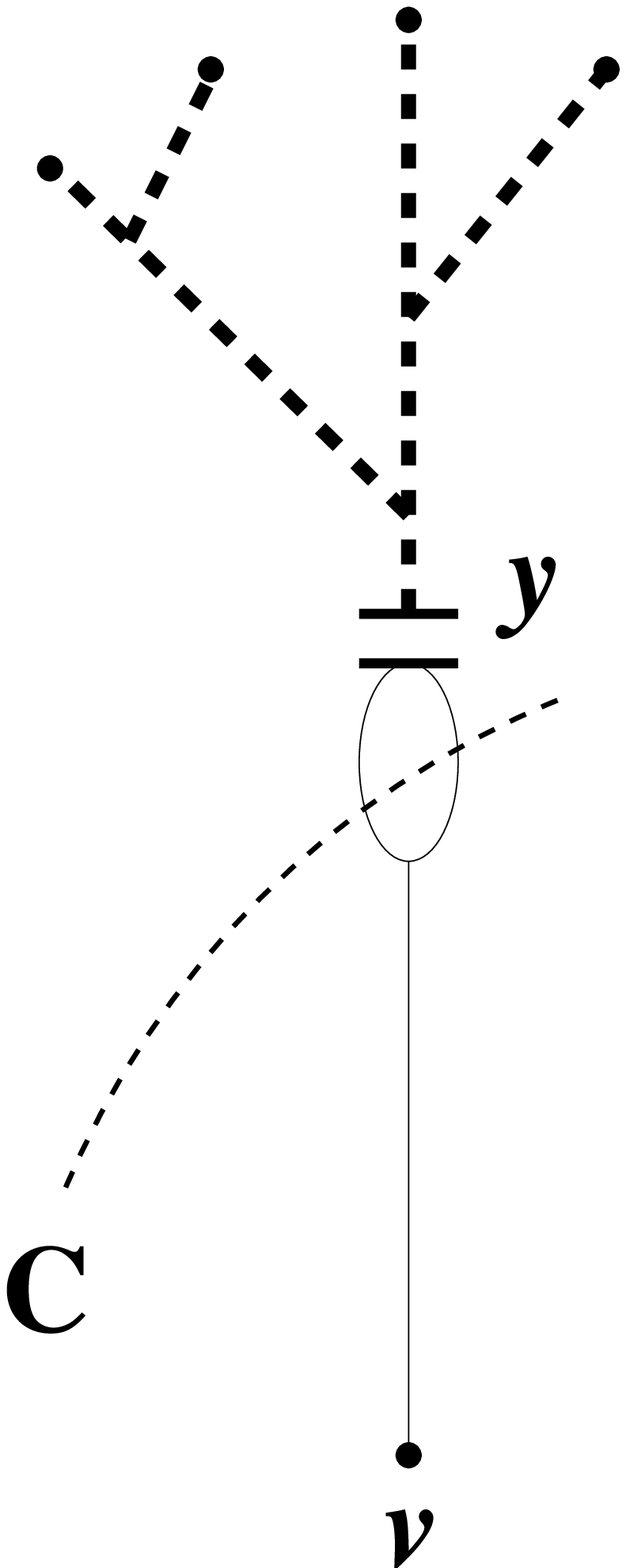}}~~
 -~\raisebox{-4pc}{\includegraphics[scale=0.17]{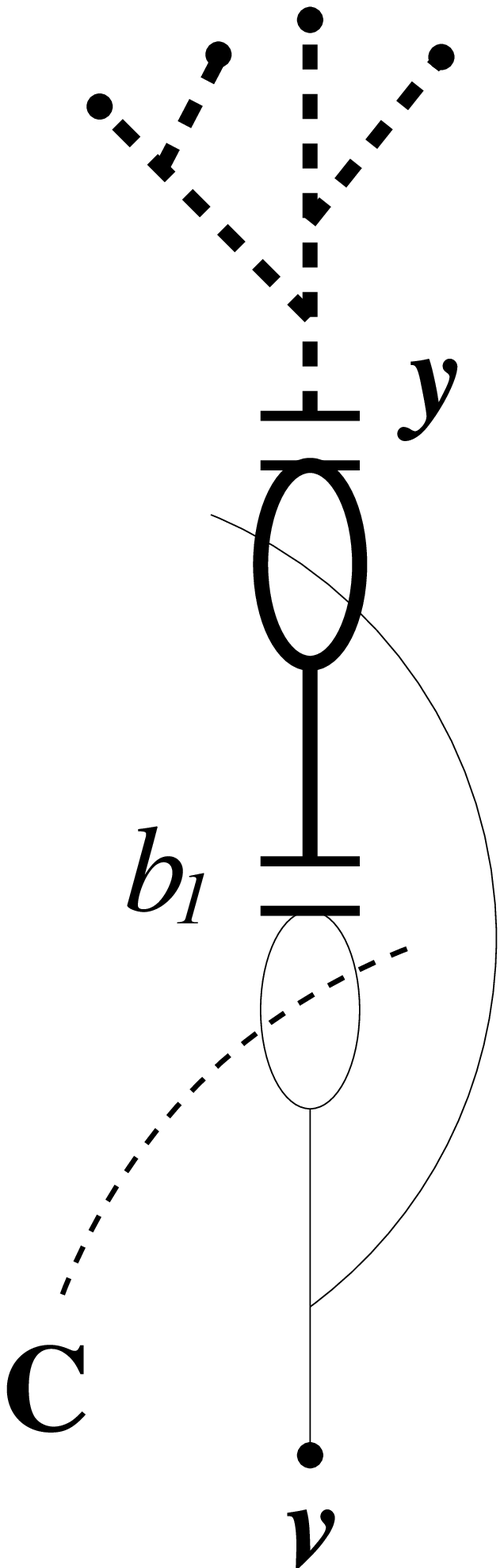}}~~+~\cdots~~\right)
\end{align*}
\caption{\label{fig-1stexp}A schematic representation of the expansion
\refeq{taur-exp}.  The vertices at the top of each diagram are the
components of $\vec\xvec_J$, as in Figure~\ref{fig-E}.  In the second
parentheses, the connection from $\yvec$ to $\vec\xvec_J$ in each diagram
(depicted in bold dashed lines) represents $\tau(\vec\xvec_J-\yvec)$.}
\end{figure}

Applying Proposition~\ref{prop:exp-first} to the $r$-point function
in \refeq{taur-rewr}, we arrive at
    \begin{align}\lbeq{taurexp}
    \tau(\vec\xvec_J)=A(\vec\xvec_J)+\sum_{\yvec}B(\yvec)\,\tau(\vec\xvec_J
    -\yvec),
    \end{align}
where we abbreviate
    \begin{align}
    \lbeq{ABzerodef}
    A(\vec\xvec_J)=A(\ovec,\vec\xvec_J;\{\ovec\}),&&
    B(\yvec)=B(\ovec,\yvec;\{\ovec\}),
    \end{align}
and similarly for
$A^{\sss(N)}(\vec\xvec_J)=A^{\sss(N)}(\ovec,\vec\xvec_J;\{\ovec\})$ and
$B^{\sss(N)}(\yvec)=B^{\sss(N)}(\ovec,\yvec;\{\ovec\})$. In the remainder of
this paper, we will specialise to the case where $\vvec=\ovec$ and
$\bC=\{\ovec\}$, and abbreviate
    \begin{align}
    M^{\sss(N)}_{\vec\xvec_J}(X)=M^{\sss(N)}_{\ovec,\vec\xvec_J;\{\ovec\}}
    (X)\qquad(N\geq1).
    \end{align}
This completes the proof of \refeq{taunxvecr2}.  In the next section,
we will use Proposition~\ref{prop:exp-first} for a general set $\bC$
in order to obtain the expansion for $A(\vec\xvec_J)$.

For future reference, we state a convenient recursion formula for
$M^{\sss(N+M)}_{\vvec,\vec\xvec_J;\bC}(X)$, valid for $N\geq 0, M\geq 1$:
    \begin{align}\lbeq{MN-rec}
    M^{\sss(N+M+1)}_{\vvec,\vec\xvec_J;\bC}(X)&
    =\sum_{b_{\sN}}p_{b_{\sN}}\,
    M^{\sss(N+1)}_{\vvec,
    \bb_{\sN};\bC}\Big(M^{\sss(M)}_{\tb_{\sN},\vec\xvec_J;\tilde\bC^{b_{\sN}}(\tb_{\sss N-1})}
    (X)\Big),
    \end{align}
which follows immediately from the second representation in \refeq{M-def}.

%%%%%%%%%%%%%%%%%%%%%%%%%%%%%%%%%%%%%%%%%%%%%%%%%%%%%%%%%%%%%%%%%%%%%%%%%%%%
%%%%%%%%%%%%%%%%%%%%%%%%%%%%%%%%%%%%%%%%%%%%%%%%%%%%%%%%%%%%%%%%%%%%%%%%%%%%
%%%%%%%%%%%%%%%%%%%%%%%%%%%%%%%%%%%%%%%%%%%%%%%%%%%%%%%%%%%%%%%%%%%%%%%%%%%%
%%%%%%%%%%%%%%%%%%%%%%%%%%%%%%%%%%%%%%%%%%%%%%%%%%%%%%%%%%%%%%%%%%%%%%%%%%%%

\section{Expansion for \protect $A(\vec{\xvec}_J)$}\label{s:exp-applexp}
We now consider $A(\vec\xvec_J)$ in \refeq{taurexp}.  Our goal is to extract
two factors $\tau(\vec\xvec_{J\setminus I}-\yvec_1)$ and
$\tau(\vec\xvec_I-\yvec_2)$ from $A(\vec\xvec_J)$, for some $I\subsetneq J$
with $I\ne\vno$ and some $\yvec_1,\yvec_2\in\Lambda$.  Let $r_1=|J\setminus
I|+1$ and $r_2=|I|+1$.  We devote Section~\ref{ss:cutting1} to the extraction
of the first $r_1$-point function $\tau(\vec\xvec_{J\setminus I}-\yvec_1)$, and
Section~\ref{ss:cutting2} to the extraction of the second $r_2$-point function
$\tau(\vec\xvec_I-\yvec_2)$.

\subsection{First cutting bond and decomposition of \protect $A^{\sss(N)}(\vec{\xvec}_J)$}\label{ss:cutting1}

First, we recall \refeq{A0def} and, by the recursive definition
\refeq{M-def} for $N\geq1$,
    \begin{align}\lbeq{AN-dec1}
    A^{\sss(N)}(\vec\xvec_J)&=M^{\sss(N+1)}_{\vec\xvec_J}(1)=\sum_{b_N}
    p_{b_N}\,M^{\sss(N)}_{\bb_N}\Big(\mP_{\sss N}\big(E'(\tb_{\sss N},
    \vec\xvec_J;\tilde\bC_{\sss N-1})\big)\Big),
    \end{align}
where the subscripts indicate which probability measure describes the
distribution of which cluster.  For example, $\tilde\bC_{\sss
N-1}\equiv\tilde\bC^{b_N}(\tb_{\sss N-1})$ is a random variable for $\mP_{\sss
N-1}$ that is hidden in the operation $M^{\sss(N)}_{\bb_N}$ (cf.,
\refeq{M-def}), but is deterministic for $\mP_{\sss N}$.  Therefore, to obtain
an expansion for $A^{\sss(N)}(\vec\xvec_J)$, it suffices to investigate
$\mP(E'(\vvec,\vec\xvec_J;\bC))$ for given $\vvec\in\Lambda$ and
$\bC\subset\Lambda$.  In this section, we shall extract an $r_1$-point function
$\tau(\vec\xvec_{J\setminus I}-\yvec_1)$ from $\mP(E'(\vvec,\vec\xvec_J;\bC))$.

Recall \refeq{through} and \refeq{E'def} to see that there must be a $j\in J$
such that $\vvec\ct{\bC}\xvec_j$.  We partition $E'(\vvec,\vec\xvec_J;\bC)$
according to the {\it first} component $\xvec_j$ which is connected from
$\vvec$ {\it through} $\bC$, i.e.,
    \begin{align}\lbeq{E'-dec}
    E'(\vvec,\vec\xvec_J;\bC)&=\BDcup{j\in J}\Big\{\{\vvec\conn\vec\xvec_J\}
    \cap\big\{\vvec\ct{\bC}(\xvec_1,\dots,\xvec_{j-1})\big\}^{\rm c}\cap
    \big\{\vvec\ct{\bC}\xvec_j\big\}\Big\}\nn\\
    &\quad~~\cap\big\{\nexists\text{ pivotal bond $b$ for }\vvec\conn\xvec_i~
    \forall i\in J\text{ such that }\vvec\ct{\bC}\bb\big\},
    \end{align}
where we use the convention that
    \eq
    \lbeq{conventiona}
    \{\vvec\ct{\sss\bC}(\xvec_1,\dots,\xvec_{j-1})\}=\vno\qquad\quad \text{if $j=1$.}
    \en
Because of this convention, for $j=1$, the event
$\{\vvec\ct{\sss\bC}(\xvec_1,\dots,\xvec_{j-1})\}^c$ is the whole probability
space. If $j\geq2$, then we can ignore the intersection in the second line of
\refeq{E'-dec}, because $\{\vvec\conn\vec\xvec_J\}\cap\{\vvec\ct{\bC}(\xvec_1,
\dots,\xvec_{j-1})\}^\text{c}$ implies \ch{that} $\vvec\conn\xvec_i$ in
$\Lambda\setminus\bC$ for $i=1,\dots,j-1$, so that the event in the second line
is automatically satisfied. \ch{We now define the first cutting bond:}

\begin{defn}[\textbf{First cutting bond}]{\rm
Given that $\vvec\ct{\bC}\xvec_j$, we say that a bond $b$ is the
\emph{$\xvec_j$-cutting bond} if it is the first occupied pivotal bond for
$\vvec\conn\xvec_j$ such that $\vvec\ct{\bC}\bb$.}
\end{defn}

Let
\begin{align}
\Fone(\vvec,\vec\xvec_J;\bC)&=\BDcup{j\in J}\Big\{\{\vvec\conn\vec\xvec_J
 \}\cap\big\{\vvec\ct{\bC}(\xvec_1,\dots,\xvec_{j-1})\big\}^{\rm c}\cap
 \big\{\vvec\ct{\bC}\xvec_j\big\}\cap\{\nexists
 \text{ $\xvec_j$-cutting bond}\}\Big\}\nn\\
&\quad~~\cap\big\{\nexists\text{ pivotal bond $b$ for }\vvec\conn\xvec_i~
 \forall i\in J\text{ such that }\vvec\ct{\bC}\bb\big\},
\end{align}
which, by definition and \refeq{E'def}, equals
\begin{align}\lbeq{F'1-def}
\Fone(\vvec,\vec\xvec_J;\bC)&=\BDcup{j\in J}\Big\{\{\vvec\conn\vec\xvec_J
 \}\cap\big\{\vvec\ct{\bC}(\xvec_1,\dots,\xvec_{j-1})\big\}^{\rm c}\cap
 E'(\vvec,\xvec_j;\bC)\Big\}\nn\\
&\quad~~\cap\big\{\nexists\text{ pivotal bond $b$ for }\vvec\conn\xvec_i~
 \forall i\in J\text{ such that }\vvec\ct{\bC}\bb\big\}.
\end{align}
Then, $E'(\vvec,\vec\xvec_J;\bC)$ equals
\begin{align}\lbeq{E'-dec1}
E'(\vvec,\vec\xvec_J;\bC)&=\Fone(\vvec,\vec\xvec_J;\bC)\nn\\
&\quad\DDcup\BDcup{b}\BDcup{j\in J}\Big\{\{\vvec\conn\vec\xvec_J\}\cap
 \big\{\vvec\ct{\bC}(\xvec_1,\dots,\xvec_{j-1})\big\}^{\rm c}\cap\big\{
 \vvec\ct{\bC}\xvec_j\big\}\cap\{b\text{ is $\xvec_j$-cutting}\}\Big\}
 \nn\\
&\hspace{6pc}\cap\big\{\nexists\text{ pivotal bond $b$ for }\vvec\conn
 \xvec_i~\forall i\in J\text{ such that }\vvec\ct{\bC}\bb\big\}.
\end{align}
The contribution due to $\Fone(\vvec,\vec\xvec_J;\bC)$ will turn out to be an
error term.

Next, we consider the union over $j\in J$ in \refeq{E'-dec1}.  When
$b$ is the $\xvec_j$-cutting bond, there is a unique nonempty set
$I\subset J_j\equiv J\setminus\{j\}$ such that $b$ is pivotal for
$\vvec\conn\xvec_i$ for all $i\in J\setminus I$, but not pivotal for
$\vvec\conn\xvec_i$ for any $i\in I$.  \ch{On this event,} the intersection in
the third line of \refeq{E'-dec1} can be ignored.  For a nonempty set
$I\subsetneq J$, we let $j_I$ be the minimal element in $J\setminus I$,
i.e.,
    \begin{align}\lbeq{minj}
    j_I=\min_{j\in J\setminus I}j.
    \end{align}
Then, the union over $j\in J$ in \refeq{E'-dec1} is rewritten as
    \begin{align}\lbeq{E'-F'1-dec1}
    &\BDcup{j\in J}\!\BDcup{\substack{\vno\ne I\subset J_j\\ j_I=j}}
    \Big\{\{\vvec\conn\vec\xvec_J\}\cap\big\{\vvec\ct{\bC}(\xvec_1,\dots,
    \xvec_{j-1})\big\}^{\rm c}\cap\big\{\vvec\ct{\bC}\xvec_j\big\}\cap\{b
    \text{ is $\xvec_j$-cutting}\}\nn\\
    &\hspace{5pc}\cap\big\{b\text{ is not pivotal for }\vvec\conn\xvec_i~
    \forall i\in I\big\}\cap\big\{b\text{ is pivotal for }\vvec\conn\xvec_i
    ~\forall i\in J\setminus I\big\}\Big\}\nn\\
    &=\BDcup{\vno\ne I\subsetneq J}\Big\{\{\vvec\conn\vec\xvec_J\}\cap\big\{
    \vvec\ct{\bC}(\xvec_1,\dots,\xvec_{j_I-1})\big\}^{\rm c}\cap\big\{\vvec
    \ct{\bC}\xvec_{j_I}\big\}\cap\{b\text{ is $\xvec_{j_I}$-cutting}\}\nn\\
    &\hspace{5pc}\cap\big\{b\text{ is not pivotal for }\vvec\conn\xvec_i~
    \forall i\in I\big\}\cap\big\{b\text{ is pivotal for }\vvec\conn\xvec_i
    ~\forall i\in J\setminus I\big\}\Big\}.
    \end{align}
To this event, we will apply Lemma~\ref{lem-cut1} and extract a factor
$\tau(\vec\xvec_{J\setminus I}-\tb)$.  To do so, we first rewrite
this event in a similar fashion to \refeq{partition} as follows:

\begin{prop}[\textbf{Setting the stage for the factorization I}]
\label{prop:settingI}
For all $\vec\xvec_J\in \Lambda^{r-1}$, any $I\subsetneq J$ with $I\ne\vno$
and any bond $b$,
    \begin{align}\lbeq{partition2}
    &\{\vvec\conn\vec\xvec_J\}\cap\big\{\vvec\ct{\bC}(\xvec_1,\dots,\xvec_{
    j_I-1})\big\}^{\rm c}\cap\big\{\vvec\ct{\bC}\xvec_{j_I}\big\}\cap\{b
    \text{ is $\xvec_{j_I}$-cutting}\}\nn\\
    &\qquad\qquad\cap\{b\text{ is not pivotal for }\vvec\conn\xvec_i~\forall
    i\in I\}\cap\{b\text{ is pivotal for }\vvec\conn\xvec_i~\forall i\in J
    \setminus I\}\nn\\
    &~=\Big\{\{\vvec\conn\vec\xvec_I\}\cap\big\{\vvec\ct{\bC}(\xvec_1,\dots,
    \xvec_{j_I-1})\big\}^{\rm c}\cap E'(\vvec,\bb;\bC)
    \text{ in }\tilde\bC^b(\vvec)\Big\}\nn\\
    &\qquad\qquad\cap\{b\text{ is occupied}\}\cap\big\{\tb\conn\vec\xvec_{J
    \setminus I}\text{ in }\Lambda\setminus\tilde\bC^b(\vvec)\big\},
    \end{align}
where the first and third events in the right-hand side are independent
of the occupation status of $b$.
\end{prop}

\begin{proof}
Since $\{\vvec\conn\vec\xvec_J\}=\{\vvec\conn\vec\xvec_I\}\cap\{\vvec\conn
\vec\xvec_{J\setminus I}\}$, the left-hand side of \refeq{partition2}
equals $\bigcap_{i=1}^3 H_i$, where
    \begin{align}
    H_1&=\{\vvec\conn\vec\xvec_I\}\cap\big\{\vvec\ct{\bC}(\xvec_1,\dots,
    \xvec_{j_I-1})\big\}^{\rm c}\cap\{b\text{ is not pivotal for }\vvec
    \conn\xvec_i~\forall i\in I\},\lbeq{H1-def}\\
    H_2&=\{\vvec\conn\vec\xvec_{J\setminus I}\}\cap\{b\text{ is pivotal
    for }\vvec\conn\xvec_i~\forall i\in J\setminus I\},\\
    H_3&=\big\{\vvec\ct{\bC}\xvec_{j_I}\big\}\cap\{b\text{ is
    $\xvec_{j_I}$-cutting}\}.
    \end{align}
Similarly to \refeq{partition}, $H_2$ and $H_3$ can be written as
    \begin{align}
    H_2&=\big\{\vvec\conn\bb\text{ in }
    \tilde\bC^b(\vvec)\big\}\cap\{b\text{ is occupied}\}\cap\big\{\tb\conn\vec
    \xvec_{J\setminus I}\text{ in }\Lambda\setminus\tilde\bC^b(\vvec)\big\},
    \lbeq{H2-rewr}\\
    H_3
    &=\big\{E'(\vvec,\bb;\bC)\text{ in }
    \tilde\bC^b(\vvec)\big\}\cap\{b\text{ is occupied}\}\cap\big\{\tb\conn
    \xvec_{j_I}\text{ in }\Lambda\setminus\tilde\bC^b(\vvec)\big\},
    \lbeq{H3-rewr}
    \end{align}
so that, also using that $E'(\vvec,\bb;\bC)\subseteq \{\vvec\conn\bb\}$ and $j_I\in J\setminus I$,
    \eq
    \lbeq{H23-rewr}
    H_2\cap H_3=\big\{E'(\vvec,\bb;\bC)\text{ in }
    \tilde\bC^b(\vvec)\big\}\cap\{b\text{ is occupied}\}\cap\big\{\tb\conn\vec
    \xvec_{J\setminus I}\text{ in }\Lambda\setminus\tilde\bC^b(\vvec)\big\}.
    \en

To prove \refeq{partition2}, it remains to show that
    \begin{align}\lbeq{H1-rewr}
    H_1=\Big\{\{\vvec\conn\vec\xvec_I\}\cap\big\{\vvec\ct{\bC}(\xvec_1,\dots,
    \xvec_{j_I-1})\big\}^{\rm c}\text{ in }\tilde\bC^b(\vvec)\Big\}.
    \end{align}
Due to \refeq{on/inprop} and $\{1,\dots,j_I-1\}\subset I$, \refeq{H1-def}
equals
    \begin{align}\lbeq{H1-rewr2}
    H_1&=\bigcap_{i=1}^{j_I-1}\big\{\{\vvec\conn\xvec_i\text{ in }\Lambda
    \setminus\bC\}\cap\{b\text{ is not pivotal for }\vvec\conn\xvec_i\}
    \big\}\nn\\
    &\qquad\cap\bigcap_{\substack{i'\in I\\ i'>j_I}}\big\{\{\vvec\conn
    \xvec_{i'}\}\cap\{b\text{ is not pivotal for }\vvec\conn\xvec_{i'}\}
    \big\}.
    \end{align}
When $j_I=1$, which is equivalent to $1\in I$, then the first
intersection is an empty intersection, so that, by convention,
it is equal to the whole probability space. We use that
    \begin{align}
    &\{\vvec\conn\xvec_i\text{ (in }\Lambda\setminus\bC)\}\cap\{b
    \text{ is not pivotal for }\vvec\conn\xvec_i\}\nn\\
    &=\{\vvec\conn\xvec_i\text{ (in }\Lambda\setminus\bC)\}\cap\big\{\vvec
    \conn\xvec_i\text{ in }\tilde\bC^b(\vvec)\big\}=\big\{\{\vvec\conn
    \xvec_i\text{ (in }\Lambda\setminus\bC)\}\text{ in }\tilde\bC^b(\vvec)
    \big\},
    \end{align}
where we write $(\text{in }\Lambda\setminus\bC)$ to indicate that the
equality is true with and without the restriction that the connections take place
in $\Lambda\setminus\bC$. Therefore, we can rewrite \refeq{H1-rewr2} as
    \begin{align}
    H_1=\bigcap_{i=1}^{j_I-1}\big\{\{\vvec\conn\xvec_i\text{ in }\Lambda
    \setminus\bC\}\text{ in }\tilde\bC^b(\vvec)\big\}\cap\bigcap_{\substack{
    i'\in I\\ i'>j_I}}\big\{\{\vvec\conn\xvec_{i'}\}\text{ in }\tilde\bC^b
    (\vvec)\big\},
    \end{align}
which equals \refeq{H1-rewr}.  This proves \refeq{partition2}.

As argued below \refeq{partition}, since $E'(\vvec,\bb;\bC)
\subset\{\bb\in\tilde\bC^b(\vvec)\}$ and since $\{\tb\conn\vec\xvec_{J
\setminus I}\text{ in }\Lambda\setminus\tilde\bC^b(\vvec)\}$
insures that $\tb \nin \tilde\bC^b(\vvec)$, by the independence statement in
lemma \ref{lem-cut1}, the occupation status of $b$ is independent of
the first and third events in the right-hand side of \refeq{partition2}.
This completes the proof of Proposition~\ref{prop:settingI}.
\end{proof}

We continue with the expansion of $\mP(E'(\vvec,\vec\xvec_J;\bC))$.
By \refeq{E'-dec1} and \refeq{E'-F'1-dec1}, as well as
Lemma~\ref{lem-cut1}, Proposition~\ref{prop:settingI} and
\refeq{incl-excl}, we obtain
    \begin{align}\lbeq{E'-F'1-dec1.5}
    &\mP(E'(\vvec,\vec\xvec_J;\bC))-\mP(\Fone(\vvec,\vec\xvec_J;\bC))\\
    &~=\sum_{\vno\ne I\subsetneq J}\sum_bp_b\,\mE\Big[\ind{\{\vvec\conn\vec
    \xvec_I\}\,\cap\,\{\vvec\ct{\bC}(\xvec_1,\dots,\xvec_{j_I-1})\}^{\rm c}\,
    \cap\,E'(\vvec,\bb;\bC)\,\text{ in }
    \tilde\bC^b(\vvec)}\;\ind{\tb\conn\vec\xvec_{J\setminus I}\text{ in }
    \Lambda\setminus\tilde\bC^b(\vvec)}\Big]\nn\\
    &~=\sum_{\vno\ne I\subsetneq J}\sum_bp_b\,\mE\Big[\indic_{E'(\vvec,\bb;
    \bC)}\;\ind{\{\vvec\conn\vec\xvec_I\}\,\cap\,\{\vvec\ct{\bC}(\xvec_1,
    \dots,\xvec_{j_I-1})\}^{\rm c}\text{ in }\tilde\bC^b(\vvec)}
    \;\tau^{\sss\tilde\bC^b(\vvec)}(\tb,\vec
    \xvec_{J\setminus I})\Big]\nn\\
    &~=\sum_{\vno\ne I\subsetneq J}\sum_bp_b\,M^{\sss(1)}_{\vvec,\bb;\bC}
    \bigg(\ind{\{\vvec\conn\vec\xvec_I\}\,\cap\,\{\vvec\ct{\bC}(\xvec_1,
    \dots,\xvec_{j_I-1})\}^{\rm c}\text{ in }\tilde\bC^b(\vvec)}\,\Big(\tau
    (\vec\xvec_{J\setminus I}-\tb)-\mP\big(\tb\ctx{\sss\tilde\bC^b(\vvec)}
    \vec\xvec_{J\setminus I}\big)\Big)\bigg),\nn
    \end{align}
where, in the second equality, we omit ``in $\tilde\bC^b(\vvec)$'' for
the event $E'(\vvec,\bb;\bC)$ due
to the fact that $E'(\vvec,\bb;\bC)$ depends only on bonds before time
$t_{\bb}$.  Applying
Proposition~\ref{prop:exp-first} to $\mP(\tb\ctx{\sss\tilde\bC^b(\vvec)}
\vec\xvec_{J\setminus I})$ and using the notation
    \begin{align}
    \lbeq{Bdelta-def}
    B_\delta(\tb,\yvec_1;\tilde\bC^b(\ovec))=\delta_{\tb,\yvec_1}
    -B(\tb,\yvec_1;\tilde\bC^b(\ovec)),
    \end{align}
we obtain
    \begin{align}\lbeq{E'-F'1-dec2}
    &\mP(E'(\vvec,\vec\xvec_J;\bC))-\mP(\Fone(\vvec,\vec\xvec_J;\bC))\nn\\
    &~=\sum_{\vno\ne I\subsetneq J}\sum_{\yvec_1}\sum_bp_b\,M^{\sss(1)}_{
    \vvec,\bb;\bC}\Big(\ind{\{\vvec\conn\vec\xvec_I\}\,\cap\,\{\vvec\ct{\bC}
    (\xvec_1,\dots,\xvec_{j_I-1})\}^{\rm c}\text{ in }\tilde\bC^b(\vvec)}
    \;B_\delta(\tb,\yvec_1;\tilde\bC^b(\vvec))\Big)\;\tau(\vec\xvec_{J
    \setminus I}-\yvec_1)\nn\\
    &\qquad-\sum_{\vno\ne I\subsetneq J}\sum_bp_b\,M^{\sss(1)}_{\vvec,\bb;
    \bC}\Big(\ind{\{\vvec\conn\vec\xvec_I\}\,\cap\,\{\vvec\ct{\bC}(\xvec_1,
    \dots,\xvec_{j_I-1})\}^{\rm c}\text{ in }\tilde\bC^b(\vvec)}\;A(\tb,
    \vec\xvec_{J\setminus I};\tilde\bC^b(\vvec))\Big).
    \end{align}

The first step of the expansion for $A^{\sss(N)}(\vec\xvec_J)$ is completed
by substituting \refeq{E'-F'1-dec2} into \refeq{AN-dec1} as follows.  Let
(see Figure~\ref{fig-E1capP})
    \eq
    a^{\sss(0)}(\vec\xvec_J;1)=\mP_{\sss 0}\big(\Fone(\ovec,\vec\xvec_J;\{\ovec\})\big),\lbeq{a01-def}
    \en
and, for $N\geq 1$,
    \eq
    a^{\sss(N)}(\vec\xvec_J;1)=\sum_{b_N}p_{b_N}M^{\sss(N)}_{\bb_N}
    \Big(\mP_{\sN}\big(\Fone(\tb_{\sN},\vec\xvec_J;\tilde\bC_{\sss N-1})\big)
    \Big).\lbeq{aN1-def}
    \en
Define, furthermore, for $N\geq 0$,
    \begin{align}
    \tilde B^{\sss(N)}(\yvec_1,\vec\xvec_I)
    &=\sum_{b_{N+1}}p_{b_{N+1}}\,
    M^{\sss(N+1)}_{\bb_{N+1}}\Big(\ind{\{\tb_N\conn\vec\xvec_I\}\,\cap\,\{
    \tb_N\ctx{\sss\tilde\bC_{N-1}}(\xvec_1,\dots,\xvec_{j_I-1})\}^{\rm c}
    \text{ in }\tilde\bC_N}\;B_\delta(\tb_{\sss N+1},\yvec_1;\tilde
    \bC_{\sN})\Big),\lbeq{tildeBN-def}\\
    a^{\sss(N)}(\vec\xvec_{J\setminus I},\vec\xvec_I;2)&=-\sum_{b_{N+1}}p_{
    b_{N+1}}\,M^{\sss(N+1)}_{\bb_{N+1}}\Big(\ind{\{\tb_N\conn\vec\xvec_I\}
    \,\cap\,\{\tb_N\ctx{\sss\tilde\bC_{N-1}}(\xvec_1,\dots,\xvec_{j_I-1})
    \}^{\rm c}\text{ in }\tilde\bC_N}\;A(\tb_{\sss N+1},\vec\xvec_{J\setminus
    I};\tilde\bC_{\sN})\Big),\lbeq{aN2-def}
    \end{align}
where we use the convention that, for $N=0$,
    \begin{align}\lbeq{convention}
    \tb_{\sss 0}=\ovec,&&
    \tilde\bC_{\sss -1}=\{\ovec\}.
    \end{align}
Here $a^{\sss(N)}(\vec\xvec_J;1)$ and
$a^{\sss(N)}(\vec\xvec_{J\setminus I},\vec\xvec_I;2)$
will \ch{turn out to} be error terms.  Then, using \refeq{AN-dec1}, \refeq{E'-F'1-dec2}, and the definitions in
\refeq{a01-def}--\refeq{aN2-def}, we arrive at the statement that for all $N\geq 0$,
    \begin{align}\lbeq{AN-dec2}
    A^{\sss(N)}(\vec\xvec_J)=a^{\sss(N)}(\vec\xvec_J;1)+\sum_{\vno\ne I
    \subsetneq J}\bigg(\sum_{\yvec_1}\tilde B^{\sss(N)}(\yvec_1,\vec\xvec_I)
    \;\tau(\vec\xvec_{J\setminus I}-\yvec_1)+a^{\sss(N)}(\vec\xvec_{J\setminus
    I},\vec\xvec_I;2)\bigg),
    \end{align}
where we further make use of the recursion relation in \refeq{M-def}.

\begin{figure}[t]
\begin{align*}
a^{\sss(1)}(\vec\xvec_J;1)\;:\raisebox{-3pc}{\includegraphics
 [scale=0.18]{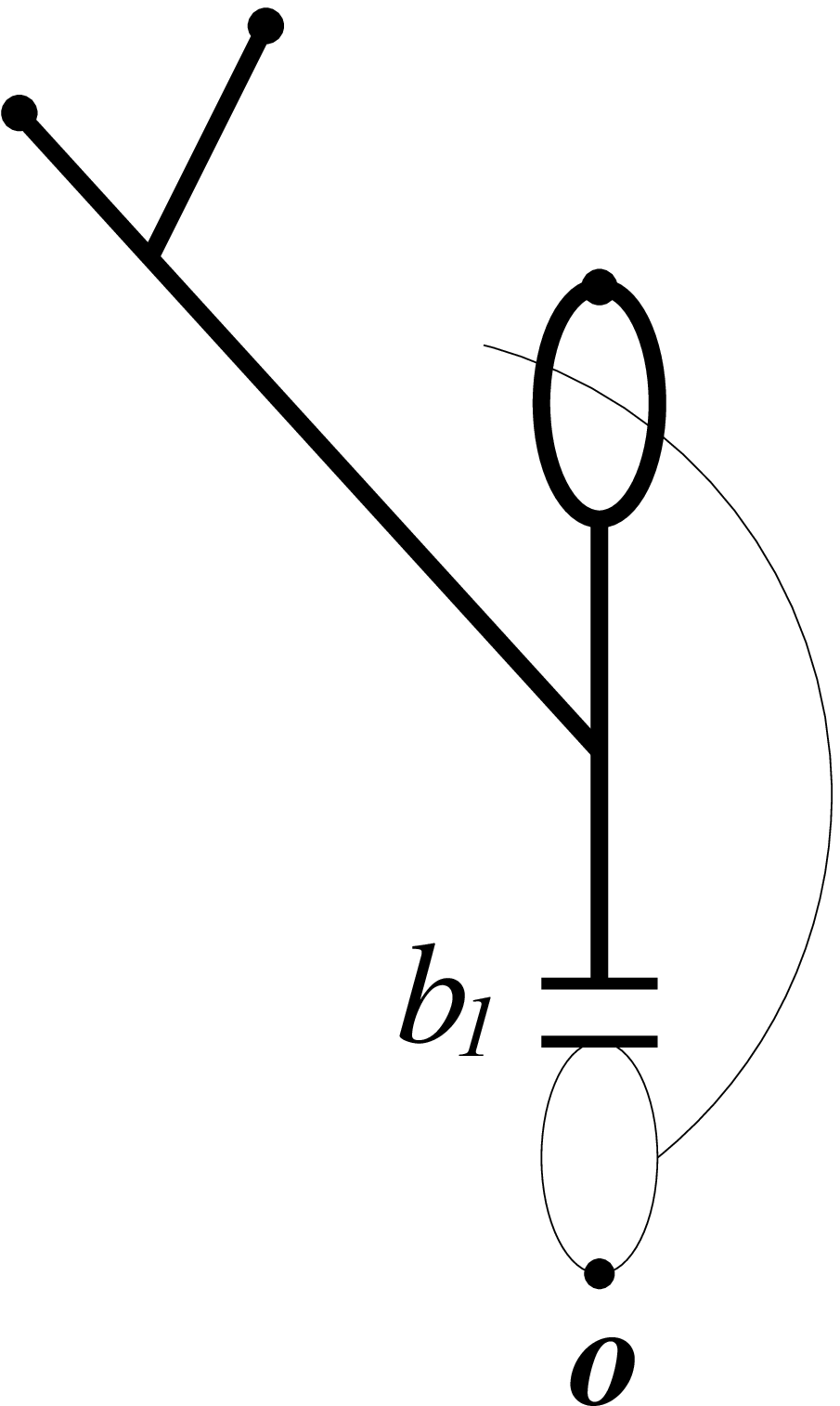}}&&
\tilde B^{\sss(1)}(\yvec_1,\vec\xvec_I)\;:\raisebox{-4pc}
 {\includegraphics[scale=0.18]{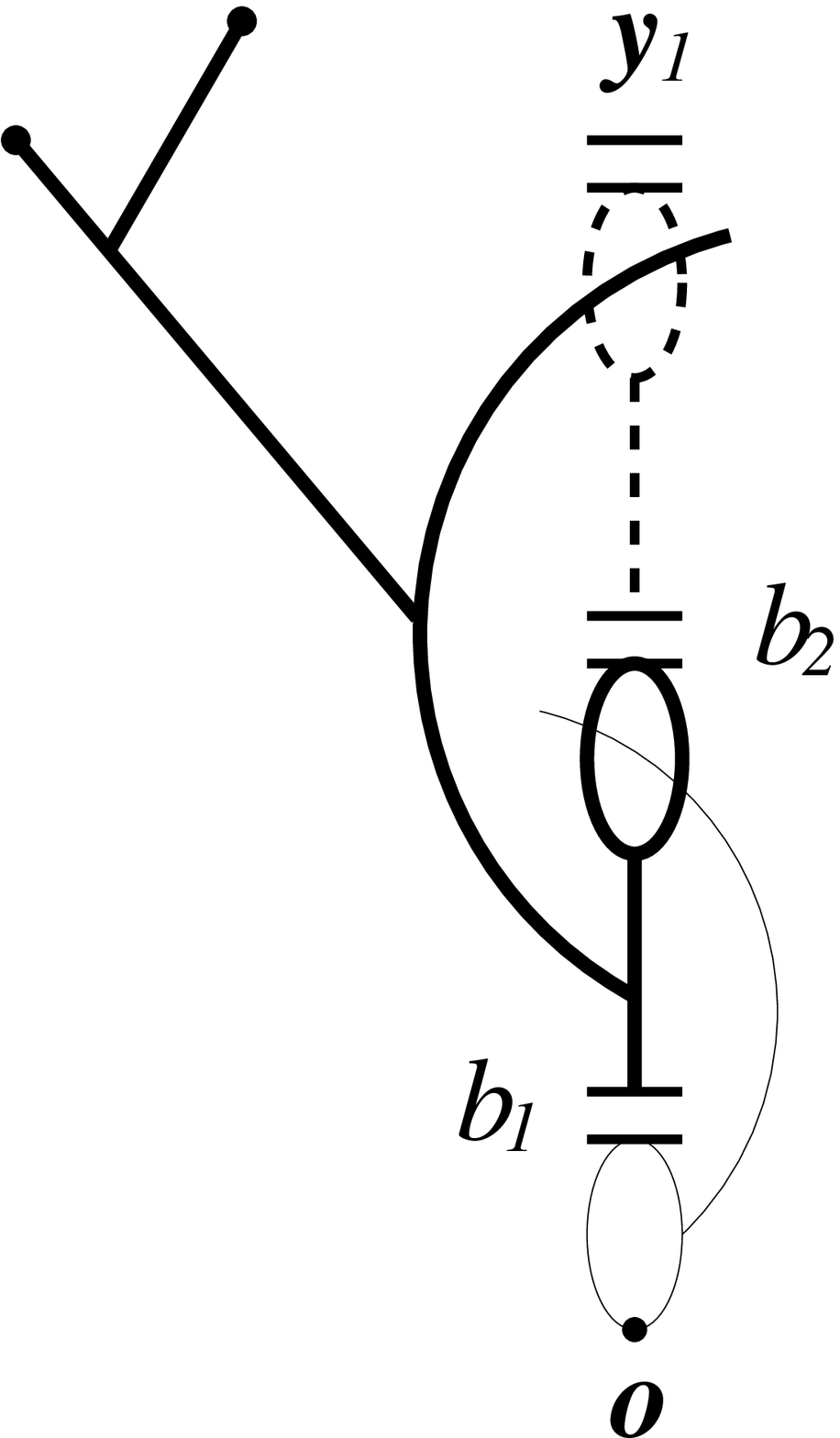}}&&
a^{\sss(1)}(\vec\xvec_{J\setminus I},\vec\xvec_I;2)\;:\raisebox{-4.2pc}
 {\includegraphics[scale=0.18]{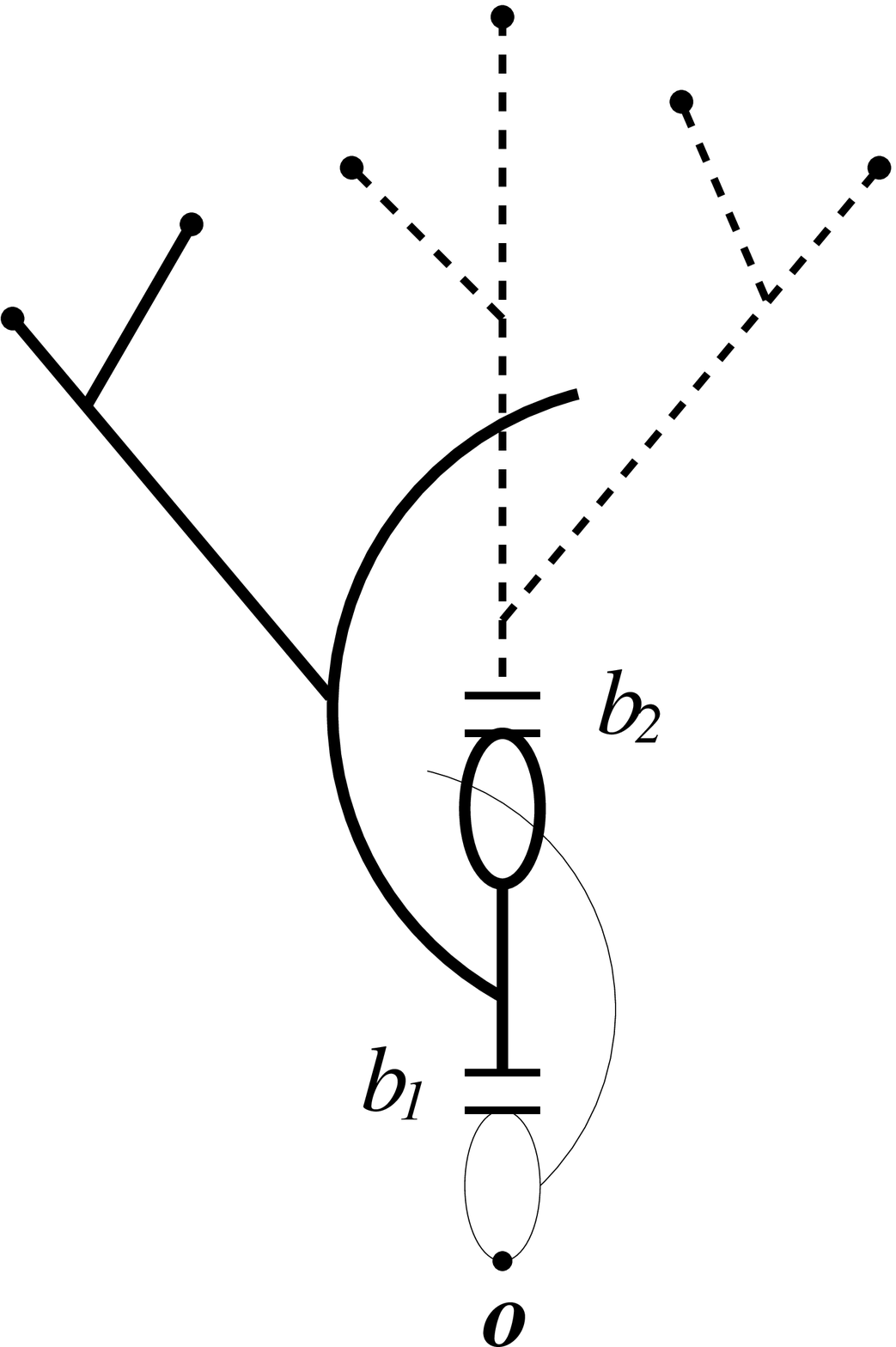}}
\end{align*}
\caption{\label{fig-E1capP}Schematic representations of
$a^{\sss(1)}(\vec\xvec_J;1)$, $\tilde B^{\sss(1)}(\yvec_1,\vec\xvec_I)$
and $a^{\sss(1)}(\vec\xvec_{J\setminus I},\vec\xvec_I;2)$, where
$B_\delta(\tb_2,\yvec_1;\tilde\bC_1)$ in
$\tilde B^{\sss(1)}(\yvec_1,\vec\xvec_I)$ and
$A(\tb_2,\vec\xvec_{J\setminus I};\tilde\bC_1)$ in
$a^{\sss(1)}(\vec\xvec_{J\setminus I},\vec\xvec_I;2)$ are reduced to
$B^{\sss(0)}(\tb_2,\yvec_1;\tilde\bC_1)$ and
$A^{\sss(0)}(\tb_2,\vec\xvec_{J\setminus I};\tilde\bC_1)$, respectively
(depicted in dashed lines).}
\end{figure}

In Section~\ref{ss:cutting2}, we extract a factor
$\tau(\vec\xvec_I-\yvec_2)$ out of
$\tilde B^{\sss(N)}(\yvec_1,\vec\xvec_I)$ and
complete the expansion for $A^{\sss(N)}(\vec\xvec_J)$.

\subsection{Second cutting bond and decomposition of
\protect $\tilde B^{\sss(N)}(\yvec_1,\vec{\xvec}_I)$}\label{ss:cutting2}

First, we recall that, for $N=0$,
    \begin{align}
    \lbeq{tildeB0-rewr}
    \tilde B^{\sss(0)}(\yvec_1,\vec\xvec_I)
    &=\!\sum_{b_{1}} p_{b_{1}}
    M^{\sss(1)}_{\bb_{1}}\Big(\ind{\{\ovec
    \conn\vec\xvec_I\}\,\cap\,\{\ovec\conn(\xvec_1,\dots,
    \xvec_{j_I-1})\}^{\rm c}\text{ in }\tilde\bC_{\sss 0}}\;B_\delta(\tb_{\sss 1},
    \yvec_1;\tilde\bC_{\sss 0})\Big)\bigg),
    \end{align}
where, by \refeq{conventiona}, for $j_I=1$, $\{\ovec\conn(\xvec_1,\dots,
    \xvec_{j_I-1})\}^{\rm c}$ is the whole probability space, while,
for $j_I>1$ and since $j_I-1\in I$ by \refeq{minj},
$\tilde B^{\sss(0)}(\yvec_1,\vec\xvec_I)\equiv 0$.
For $N\geq 1$, we recall that
    \begin{align}\lbeq{tildeBN-rewr}
    &\tilde B^{\sss(N)}(\yvec_1,\vec\xvec_I)\\
    &=\!\sum_{b_N,b_{N+1}}\!\!p_{b_N}p_{b_{N+1}}M^{\sss(N)}_{\bb_N}\bigg(
    M^{\sss(1)}_{\tb_N,\bb_{N+1};\tilde\bC_{\sss N-1}}\Big(\ind{\{\tb_N
    \conn\vec\xvec_I\}\,\cap\,\{\tb_N\ctx{\sss\tilde\bC_{N-1}}(\xvec_1,\dots,
    \xvec_{j_I-1})\}^{\rm c}\text{ in }\tilde\bC_N}\;B_\delta(\tb_{\sss N+1},
    \yvec_1;\tilde\bC_{\sN})\Big)\bigg).\nn
    \end{align}
Therefore, to decompose $\tilde B^{\sss(N)}(\yvec_1,\vec\xvec_I)$ and
extract $\tau(\vec\xvec_I-\yvec_2)$, it suffices to consider
    \begin{align}\lbeq{M1B-dec1}
    &M^{\sss(1)}_{\vvec,\bb;\bC}\Big(\ind{\{\vvec\conn\vec\xvec_I\}\,\cap\,
    \{\vvec\ct{\sss\bC}(\xvec_1,\dots,\xvec_{j_I-1})\}^{\rm c}\text{ in }
    \tilde\bC^b(\vvec)}\;B_\delta(\tb,\yvec_1;\tilde\bC^b(\vvec))\Big)\nn\\
    &~=M^{\sss(1)}_{\vvec,\bb;\bC}\Big(\ind{\{\vvec\conn\vec\xvec_I\}
    \text{ in }\tilde\bC^b(\vvec)}\;B_\delta(\tb,\yvec_1;\tilde\bC^b(\vvec))
    \Big)\nn\\
    &~~\quad-M^{\sss(1)}_{\vvec,\bb;\bC}\Big(\ind{\{\vvec\conn\vec\xvec_I\}\,
    \cap\,\{\vvec\ct{\sss\bC}(\xvec_1,\dots,\xvec_{j_I-1})\}\text{ in }
    \tilde\bC^b(\vvec)}\;B_\delta(\tb,\yvec_1;\tilde\bC^b(\vvec))\Big),
    \end{align}
for any fixed $I\subsetneq J$ with $I\ne\vno$, $\vvec\in\Lambda$,
$\bC\subset\Lambda$ and a bond $b$, where the second term is zero if
$j_I=1$ (see \refeq{conventiona}).  If $j_I>1$, then both terms in
the right-hand side are of the form
    \begin{align}\lbeq{preM1B-gen1}
    &M^{\sss(1)}_{\vvec,\bb;\bC}\Big(\ind{\{\vvec\conn\vec\xvec_I\}\,\cap\,\{
    \vvec\ct{\sss\bA}(\xvec_1,\dots,\xvec_{j_I-1})\}\text{ in }\tilde\bC^b
    (\vvec)}\;B_\delta(\tb,\yvec_1;\tilde\bC^b(\vvec))\Big)\nn\\
    &~=\mE\Big[\indic_{E'(\vvec,\bb;\bC)}\;\ind{\{\vvec\conn\vec\xvec_I\}
    \,\cap\,\{\vvec\ct{\sss\bA}(\xvec_1,\dots,\xvec_{j_I-1})\}\text{ in }
    \tilde\bC^b(\vvec)}\;B_\delta(\tb,\yvec_1;\tilde\bC^b(\vvec))\Big],
    \end{align}
with $\bA=\{\vvec\}$ and $\bA=\bC$, respectively.  To treat the case of
$j_I=1$ simultaneously, we temporarily adopt the convention that
    \eq
    \lbeq{conventionb}
    \{\vvec\ct{\{\vvec\}}
    (\xvec_1,\dots,\xvec_{j_I-1})\}=\Omega\qquad \qquad \text{for $j_I=1$},
    \en
where $\Omega$ is the whole probability space.
(Do not be confused with the convention in \refeq{conventiona}.)

We note that the random variables in the above expectation depend only on
bonds, other than $b$, whose both end-vertices are in $\tilde\bC^b(\vvec)$,
and are independent of the occupation status of $b$.  For an event $E$ and
a random variable $X$, we let
    \begin{align}\lbeq{condexp}
    \tilde\mP^b(E)=\mP\big(E\,\big|\,b\text{ is vacant}\big),&&
    \tilde\mE^b[X]=\mE\big[X\big|\,b\text{ is vacant}\big].
    \end{align}
Since $\tilde\bC^b(\vvec)=\bC(\vvec)$ almost surely with respect to
$\tilde\mP^b$, we can simplify \refeq{preM1B-gen1} as
    \begin{align}\lbeq{M1B-gen1}
    \tilde\mE^b\Big[\indic_{E'(\vvec,\bb;\bC)}\;\ind{\vvec\conn\vec\xvec_I\}
    \,\cap\,\{\vvec\ct{\sss\bA}(\xvec_1,\dots,\xvec_{j_I-1})}\;B_\delta(\tb,
    \yvec_1;\bC(\vvec))\Big].
    \end{align}

To investigate \refeq{M1B-gen1}, we now introduce a second cutting bond:

\begin{defn}[\textbf{Second cutting bond}]
\label{def-seccb}
{\rm
For $t\geq t_{\vvec}$, we say that a bond $e$ is the $t$-\emph{cutting bond for
} $\vvec\ctx{\bA} \vec\xvec_I$ if it is the first occupied pivotal bond for
$\vvec\conn\xvec_i$ \emph{for all} $i\in I$ such that $\vvec\ct{\bA}\eb$ and
$t_{\te}\geq t$.}
\end{defn}

Let
    \begin{align}\lbeq{F'2-def}
    \Ftwo_t(\vvec,\vec\xvec_I;\bA)=\{\vvec\conn\vec\xvec_I\}\cap\big\{\vvec
    \ct{\bA}(\xvec_1,\dots,\xvec_{j_I-1})\big\}\cap\{\nexists\text{ $t$-cutting bond for }
    \vvec\ctx{\bA} \vec\xvec_I\},
    \end{align}
which, for $\vec\xvec_I=\xvec$, equals
    \begin{align}\lbeq{F'2-1comp}
    \Ftwo_t(\vvec,\xvec;\bA)=\big\{\vvec\ct{\bA}\xvec\big\}\cap\{\nexists
    \text{ $t$-cutting bond for }\vvec\ctx{\bA} \xvec\}.
    \end{align}
\ch{Note in \refeq{M1B-gen1}, due to \refeq{condexp}, $b$ is
$\tilde\mP^b$-a.s.\ vacant. Also, by Definition \ref{def-seccb},
when $e$ is a cutting bond, then $e$ is occupied. Thus, we must
have that $e\neq b$.}
Using \refeq{M1B-gen1}--\refeq{F'2-def}, we have, for $j_I>1$,
    \begin{align}\lbeq{M1B-gen2}
    &\tilde\mE^b\Big[\indic_{E'(\vvec,\bb;\bC)}\,\ind{\vvec\conn\vec\xvec_I
    \}\,\cap\,\{\vvec\ct{\sss\bA}(\xvec_1,\dots,\xvec_{j_I-1})}\,B_\delta
    (\tb,\yvec_1;\bC(\vvec))\Big]-\tilde\mE^b\Big[\indic_{E'(\vvec,\bb;
    \bC)}\,\indic_{\Ftwo_{t_{\yvec_1}}(\vvec,\vec\xvec_I;\bA)}\,B_\delta
    (\tb,\yvec_1;\bC(\vvec))\Big]\nn\\
    &=\sum_{e(\ne b)}\tilde\mE^b\Big[\indic_{E'(\vvec,\bb;\bC)}\,\ind{
    \vvec\conn\vec\xvec_I\}\,\cap\,\{\vvec\ct{\sss\bA}(\xvec_1,\dots,
    \xvec_{j_I-1})\}\,\cap\,\{e\text{ is $t_{\yvec_1}$-cutting for }\vvec
    \ctx{\bA}\vec\xvec_I}\,B_\delta(\tb,\yvec_1;\bC(\vvec))\Big]\nn\\
    &=\sum_{e(\ne b)}\tilde\mE^b\Big[\indic_{E'(\vvec,\bb;\bC)}\,
    \ind{\vvec\ct{\bA}\xvec_i\;\forall i\in I\}\,\cap\,\{e
    \text{ is $t_{\yvec_1}$-cutting for }\vvec\ctx{\bA}\vec\xvec_I}\,
    B_\delta(\tb,\yvec_1;\bC(\vvec))\Big].
    \end{align}
By the convention \refeq{conventionb}, this equality also holds when $j_I=1$
and $\bA=\{\vvec\}$, so that in both cases we are left to analyse
\refeq{M1B-gen2}. To the right-hand side, we will apply Lemma~\ref{lem-cut1}
and extract \ch{a factor} $\tau(\vec\xvec_I-\yvec_2)$.  To do so, we first rewrite the event
in the second indicator on the right-hand side as follows:

\begin{prop}[\textbf{Setting the stage for the factorization II}]
\label{prop:settingII}
For $\bA\subset\Lambda$, $t\geq t_{\vvec}$ and a bond $e$,
    \begin{align}\lbeq{exp-onin2}
    &\{\vvec\ct{\bA}\xvec_i~\forall i\in I\}\cap\{e\text{ is
    $t$-cutting for }\vvec\ctx{\bA}\vec\xvec_I\}\nn\\
    &\quad=\big\{\Ftwo_t(\vvec,\eb;\bA)
    \text{ in }\tilde\bC^e(\vvec)\big\}\cap\{e\text{ is occupied}\}\cap
    \big\{\te\conn\vec\xvec_I\text{ in }\Lambda\setminus\tilde\bC^e(\vvec)
    \big\},
    \end{align}
where the first and third events in the right-hand side are independent
of the occupation status of $b$.
\end{prop}

\begin{proof}
By definition, we immediately obtain (cf., \refeq{partition} and
\refeq{H3-rewr})
    \begin{align}
    &\{\vvec\ct{\bA}\xvec_i~\forall i\in I\}\cap\{e\text{ is
    $t$-cutting for }\vvec\ctx{\bA}\vec\xvec_I\}\nn\\
    &=\Big\{\big\{\vvec\ct{\bA}\eb\big\}\cap\{\nexists
    \text{ $t$-cutting bond for }\vvec\ctx{\bA} \eb\}
    \text{ in }\tilde\bC^e(\vvec)\Big\}\cap\{e\text{ is occupied}\}
    \cap\big\{\te\conn\vec\xvec_I\text{ in }\Lambda\setminus\tilde\bC^e
    (\vvec)\big\}\nn\\
    &=\big\{\Ftwo_t(\vvec,\eb;\bA)
    \text{ in }\tilde\bC^e(\vvec)\big\}\cap\{e\text{ is occupied}\}
    \cap\big\{\te\conn\vec\xvec_I\text{ in }\Lambda\setminus\tilde\bC^e
    (\vvec)\big\},
    \end{align}
which proves \refeq{exp-onin2}.

The statement below \refeq{exp-onin2} also holds, since $\Ftwo_t(\vvec,\eb;
\bA)\subset\{\eb\in\tilde\bC^e(\vvec)
\}$, while $\te\conn\vec\xvec_I\text{ in }\Lambda\setminus\tilde\bC^e
(\vvec)$ ensures that $\te\notin\tilde\bC^e(\vvec)$ occurs
(see the similar arguments below \refeq{partition} and
\refeq{H3-rewr}).  This completes the proof of
Proposition~\ref{prop:settingII}.
\end{proof}

We continue with the expansion of the right-hand side of \refeq{M1B-gen2}.
First, we note that $B_\delta(\tb,\yvec_1;\bC(\vvec))$ is random only
when $t_{\yvec_1}$ is strictly larger than $t_{\tb}$, and depends only
on bonds whose both endvertices are in $\bC(\vvec;t_{\yvec_1}-\vep)$,
where we define, for $T\geq t_{\vvec}$,
    \begin{align}
    \bC(\vvec;T)=\bC(\vvec)\cap\big(\Zd\times[t_{\vvec},T]\big),
    \end{align}
which is almost surely finite as long as the interval $[t_{\vvec},T]$ is
finite.  As a result, we claim that, \ch{a.s.},
    \begin{align}\lbeq{Bident}
    B_\delta(\tb,\yvec_1;\bC(\vvec))=B_\delta(\tb,\yvec_1;\bC(\vvec;
     t_{\yvec_1}-\vep)).
    \end{align}
Indeed, this follows since the first term of $B_\delta(\tb,\yvec_1;\bC(\vvec))$
in \refeq{Bdelta-def} does not depend on $\bC(\vvec)$ at all, while the other
term, due to the definition of $B(\tb,\yvec_1;\bC(\vvec))$ in \refeq{ANBNdef}
only depends on $\bC(\vvec)$ up to time $t_{\yvec_1}-\vep$.

As a result, by conditioning on $\bC(\vvec;t_{\yvec_1}-\vep)$ and using
Proposition~\ref{prop:settingII}, the summand in \refeq{M1B-gen2} for $e\ne b$
can be written as
    \begin{align}\lbeq{propappl}
    &\sum_{\bB\subset\Lambda}\tilde\mE^b\Big[\indic_{E'(\vvec,\bb;\bC)}\;\ind{
    \Ftwo_{t_{\yvec_1}}(\vvec,\eb;\bA)
    \text{ in }\tilde\bC^e(\vvec)}\;\ind{\bC(\vvec;t_{\yvec_1}-\vep)=\bB}
    \;B_\delta(\tb,\yvec_1;\bB)
    \ind{e\text{ is occupied}}\;\ind{\te\conn\vec\xvec_I\text{ in }
    \Lambda\setminus\tilde\bC^e(\vvec)}\Big]\;\nn\\
    &=p_e\sum_{\bB\subset\Lambda}B_\delta(\tb,\yvec_1;\bB)\;\tilde\mE^b\Big[
    \ind{E'(\vvec,\bb;\bC)\,\cap\,\Ftwo_{t_{\yvec_1}}(\vvec,\eb;\bA)\,\cap\,
    \{\bC(\vvec;t_{\yvec_1}-\vep)
    =\bB\}\text{ in }\tilde\bC^e(\vvec)}
    \ind{\te\conn\vec\xvec_I\text{ in }\Lambda\setminus\tilde\bC^e(\vvec)}\Big],
    \end{align}
where the second expression is obtained by using
$t_{\tb}\leq t_{\yvec_1}\leq t_{\te}$ and the fact that the event $\{e$
is occupied\} is independent of the other events.  To the expectation on
the right-hand side of \refeq{propappl}, we apply Lemma~\ref{lem-cut1}
with $\mE$ in \refeq{lemcut1} being replaced by $\tilde\mE^b\,$, which, we recall, is
the expectation for oriented percolation defined over the bonds other than
$b$.  Then, \refeq{propappl} equals
    \begin{align}\lbeq{pre-rev-cond}
    &p_e\sum_{\bB\subset\Lambda}B_\delta(\tb,\yvec_1;\bB)\,\tilde\mE^b\Big[
    \ind{E'(\vvec,\bb;\bC)\,\cap\,\Ftwo_{t_{\yvec_1}}(\vvec,\eb;\bA)\,\cap
    \,\{\bC(\vvec;t_{\yvec_1}-\vep)=\bB
    \}\text{ in }\tilde\bC^e(\vvec)}\,\tilde\mE^b\big[\ind{\te\conn\vec
    \xvec_I\text{ in }\Lambda\setminus\tilde\bC^e(\vvec)}\big]\Big]\\
    &=p_e\sum_{\bB\subset\Lambda}\tilde\mE^b\Big[\indic_{E'(\vvec,\bb;\bC)}
    \,\ind{\Ftwo_{t_{\yvec_1}}(\vvec,\eb;\bA)\text{ in }\tilde\bC^e(\vvec)}
    \,\ind{\bC(\vvec;t_{\yvec_1}-\vep)=\bB}\,B_\delta(\tb,\yvec_1;\bB)
    \mE\big[\ind{\te\conn\vec\xvec_I\text{ in }
    \Lambda\setminus\tilde\bC^e(\vvec)}\big]\Big]\nn\\
    &=p_e\sum_{\bB\subset\Lambda}\tilde\mE^b\Big[\indic_{E'(\vvec,\bb;\bC)}\,
    \ind{\Ftwo_{t_{\yvec_1}}(\vvec,\eb;\bA)\text{ in }\tilde\bC^e(\vvec)}\,
    \ind{\bC(\vvec;t_{\yvec_1}-\vep)=\bB}\,B_\delta(\tb,\yvec_1;\bB)\,\Big(\tau
    (\vec\xvec_I-\te)-\mP\big(\te\ctx{\sss\tilde\bC^e(\vvec)}\vec\xvec_I\big)
    \Big)\Big],\nn
    \end{align}
where the first equality is due to the fact that the event
$\{\te\conn\vec\xvec_I$ in $\Lambda\setminus\tilde\bC^e(\vvec)\}$ depends only
on bonds after $t_{\te}\;(\geq t_{\tb})$, so that $\tilde\mE^b$ can be replaced
by $\mE$, and the second equality is obtained by using
\refeq{restr-rpt}--\refeq{incl-excl}.  By performing the sum over
$\bB\subset\Lambda$ and using \refeq{Bident}, \refeq{pre-rev-cond} equals
\begin{align}\lbeq{rev-cond}
p_e\,\tilde\mE^b\Big[\indic_{E'(\vvec,\bb;\bC)}\;\ind{\Ftwo_{t_{\yvec_1}}
 (\vvec,\eb;\bA)\text{ in }\tilde\bC^e(\vvec)}\;B_\delta(\tb,\yvec_1;\bC
 (\vvec))\;\Big(\tau(\vec\xvec_I-\te)-\mP\big(\te\ctx{\sss\tilde\bC^e
 (\vvec)}\vec\xvec_I\big)\Big)\Big].
\end{align}
For notational convenience, we define
\begin{align}\lbeq{tildeM-def}
\tilde M^b_{\vvec,\bb;\bC}(X)=\tilde\mE^b\big[\indic_{E'(\vvec,\bb;\bC)}
 \;X\big].
\end{align}
Note that $\tilde M^b_{\vvec,\bb;\bC}(X)=M^{\sss(1)}_{\vvec,\bb;\bC}(X)$
if $X$ depends only on bonds before $t_{\bb}$.  As in the derivation of
\refeq{E'-F'1-dec2} from \refeq{E'-F'1-dec1.5}, we use
Proposition~\ref{prop:exp-first} to conclude that, by \refeq{M1B-gen2}
and \refeq{rev-cond}--\refeq{tildeM-def},
\begin{align}\lbeq{M1B-gen3}
&\tilde M^b_{\vvec,\bb;\bC}\Big(\ind{\vvec\conn\vec\xvec_I\}\,\cap\,\{
 \vvec\ct{\sss\bA}(\xvec_1,\dots,\xvec_{j_I-1})}\;B_\delta(\tb,\yvec_1;
 \bC(\vvec))\Big)-\tilde M^b_{\vvec,\bb;\bC}\Big(\indic_{\Ftwo_{t_{
 \yvec_1}}(\vvec,\vec\xvec_I;\bA)}\;B_\delta(\tb,\yvec_1;\bC(\vvec))
 \Big)\nn\\
&=\sum_{\yvec_2}\sum_{e(\ne b)}p_e\,\tilde M^b_{\vvec,\bb;\bC}\Big(
 \ind{\Ftwo_{t_{\yvec_1}}(\vvec,\eb;\bA)\text{ in }\tilde\bC^e(\vvec)}
 \;B_\delta(\tb,\yvec_1;\bC(\vvec))\;B_\delta(\te,\yvec_2;\tilde\bC^e
 (\vvec))\Big)\;\tau(\vec\xvec_I-\yvec_2)\nn\\
&\quad-\sum_{e(\ne b)}p_e\,\tilde M^b_{\vvec,\bb;\bC}\Big(\ind{\Ftwo_{
 t_{\yvec_1}}(\vvec,\eb;\bA)\text{ in }\tilde\bC^e(\vvec)}\;B_\delta
 (\tb,\yvec_1;\bC(\vvec))\;A(\te,\vec\xvec_I;\tilde\bC^e(\vvec))\Big).
\end{align}

\begin{figure}[t]
\begin{align*}
a^{\sss(1)}(\yvec_1,\vec\xvec_I;3)_{\sss+}:
 \raisebox{-4pc}{\includegraphics[scale=0.17]{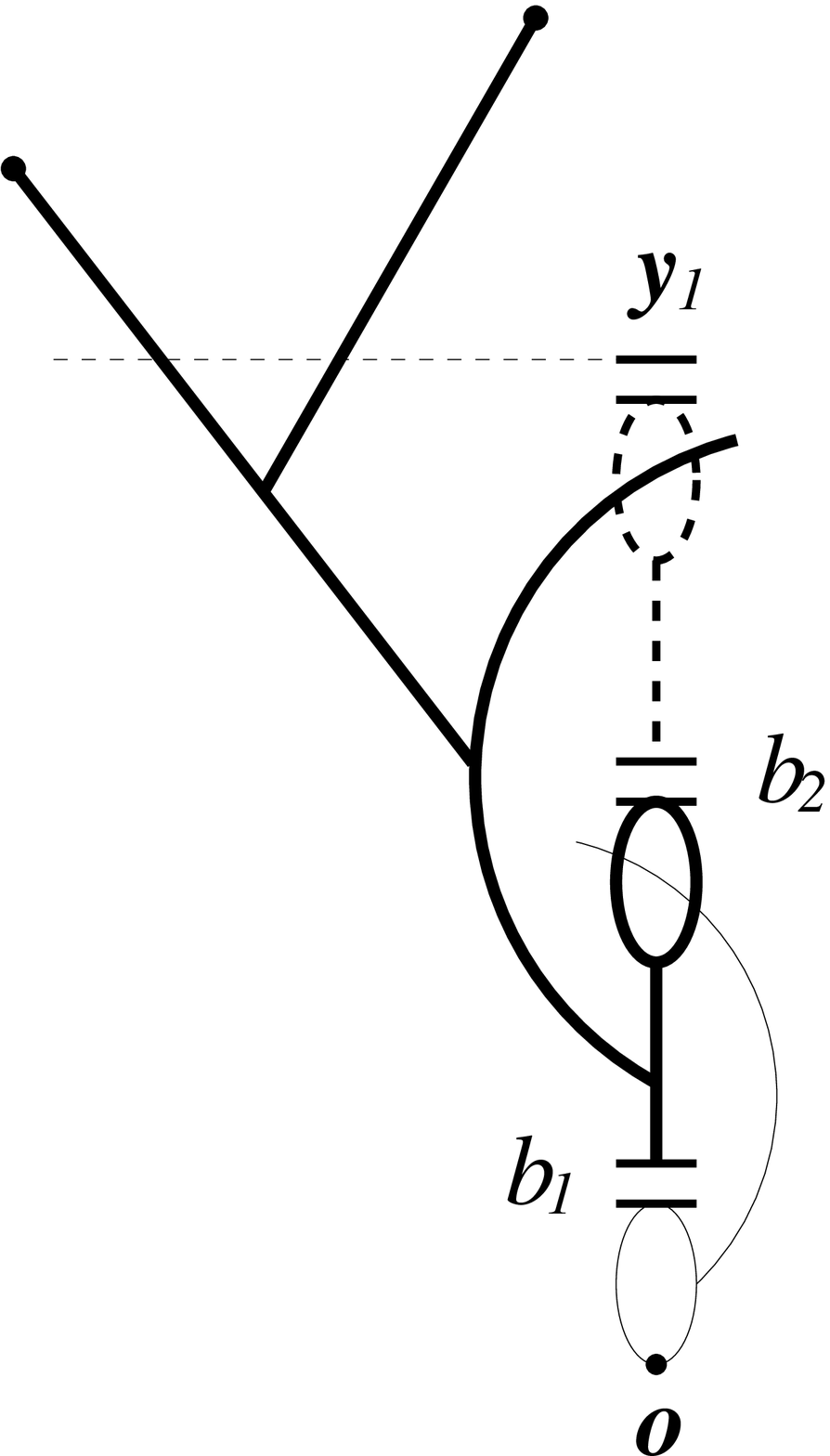}}\hspace{5pc}
a^{\sss(1)}(\yvec_1,\vec\xvec_I;3)_{\sss-}:
 \raisebox{-4pc}{\includegraphics[scale=0.17]{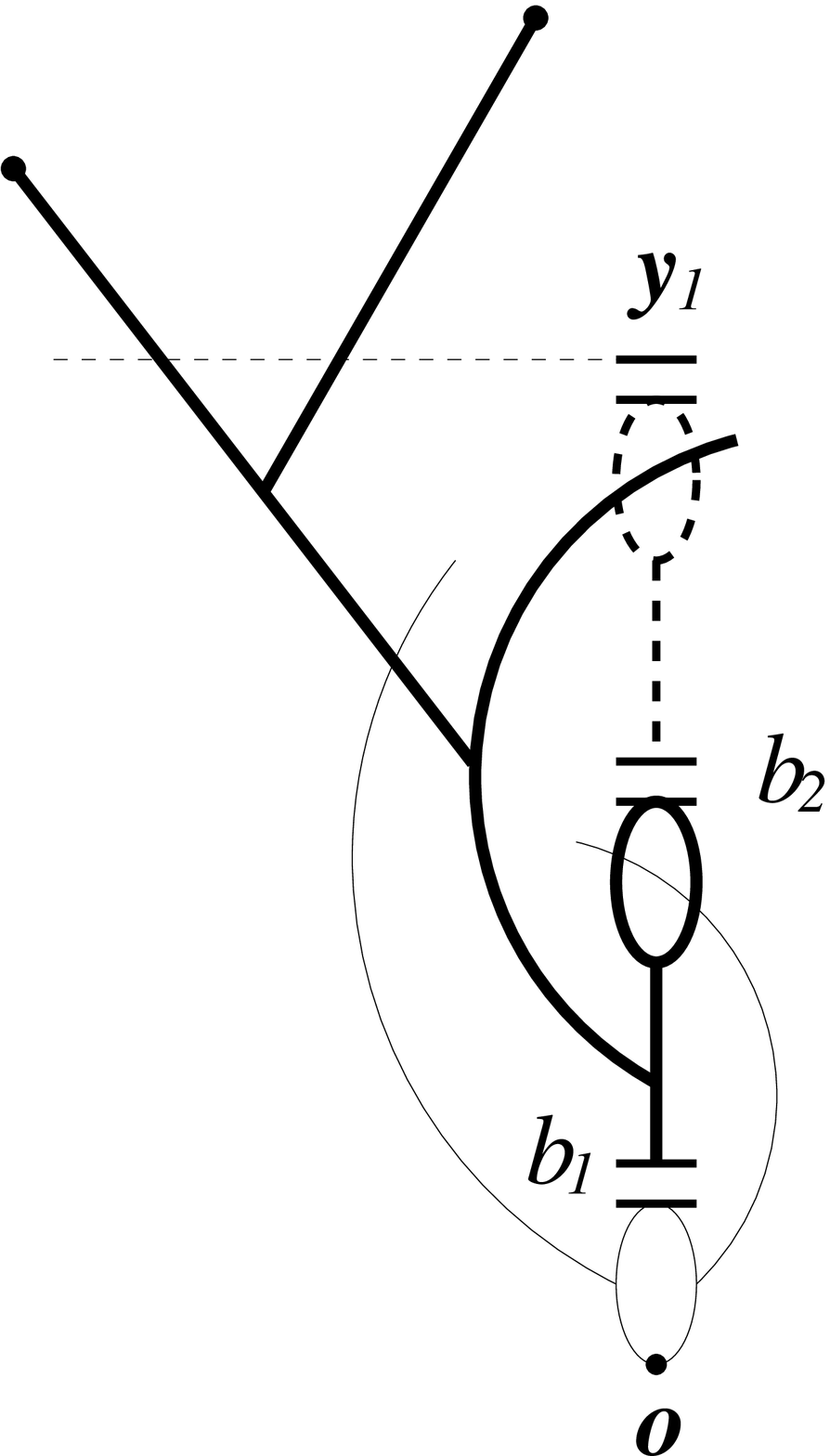}}\qquad+\quad
 \raisebox{-4pc}{\includegraphics[scale=0.17]{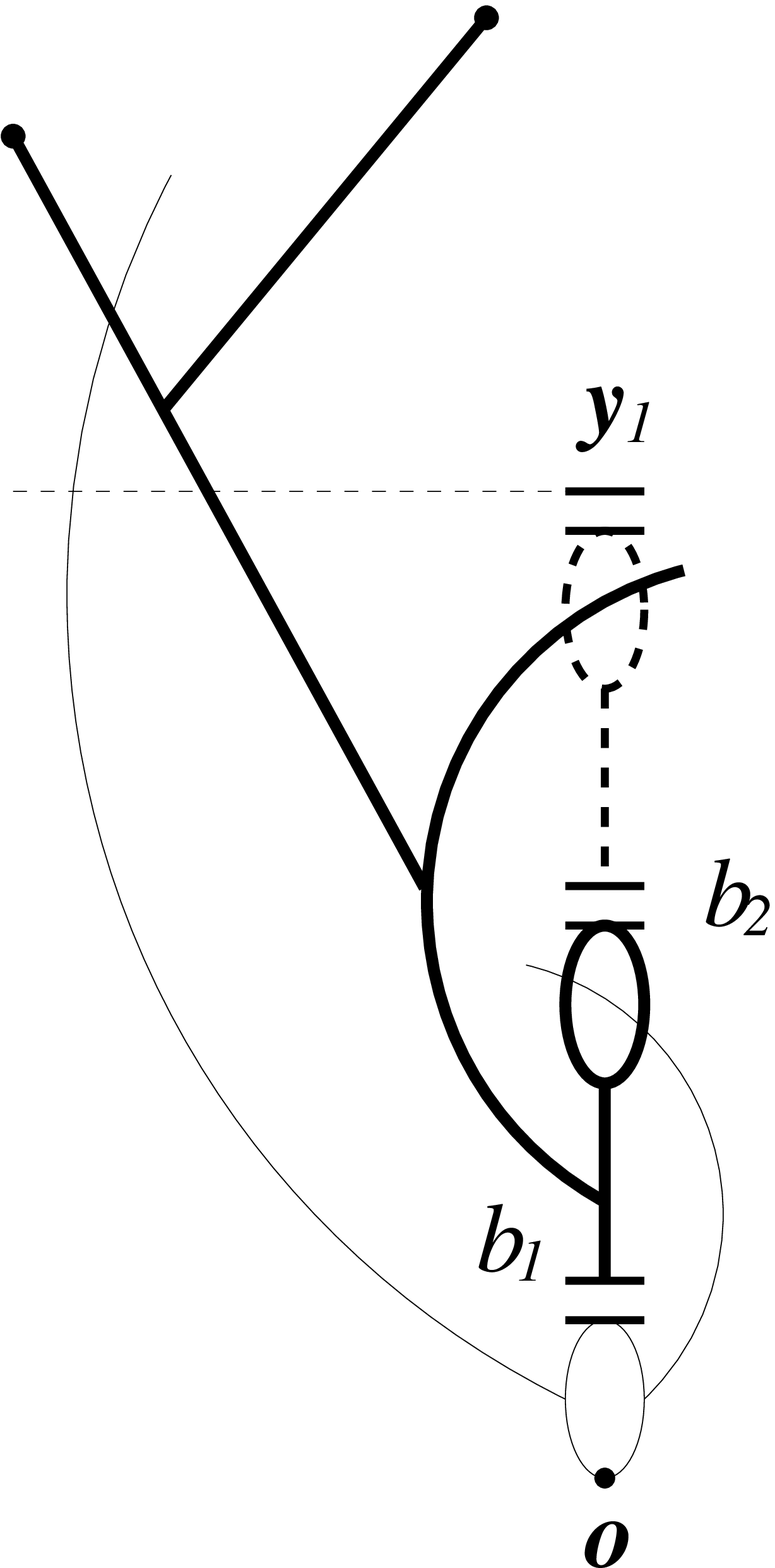}}
\end{align*}
\caption{\label{fig-2ndexp-error}Schematic representations of
$a^{\sss(1)}(\yvec_1,\vec\xvec_I;3)_{\sss\pm}$.  The random variable
$B_\delta(\tb_{\sss N+1},\yvec_1;\bC(\tb_{\sN}))$ in \refeq{a3IN-def} for
$N=1$ is reduced to $B^{\sss(0)}(\tb_2,\yvec_1;\bC(\tb_1))$ (in bold
dashed lines).}
\end{figure}

\begin{figure}[t]
\begin{gather*}
\phi^{\sss(1)}(\yvec_1,\yvec_2)_{\sss+}:
 \raisebox{-5pc}{\includegraphics[scale=0.17]{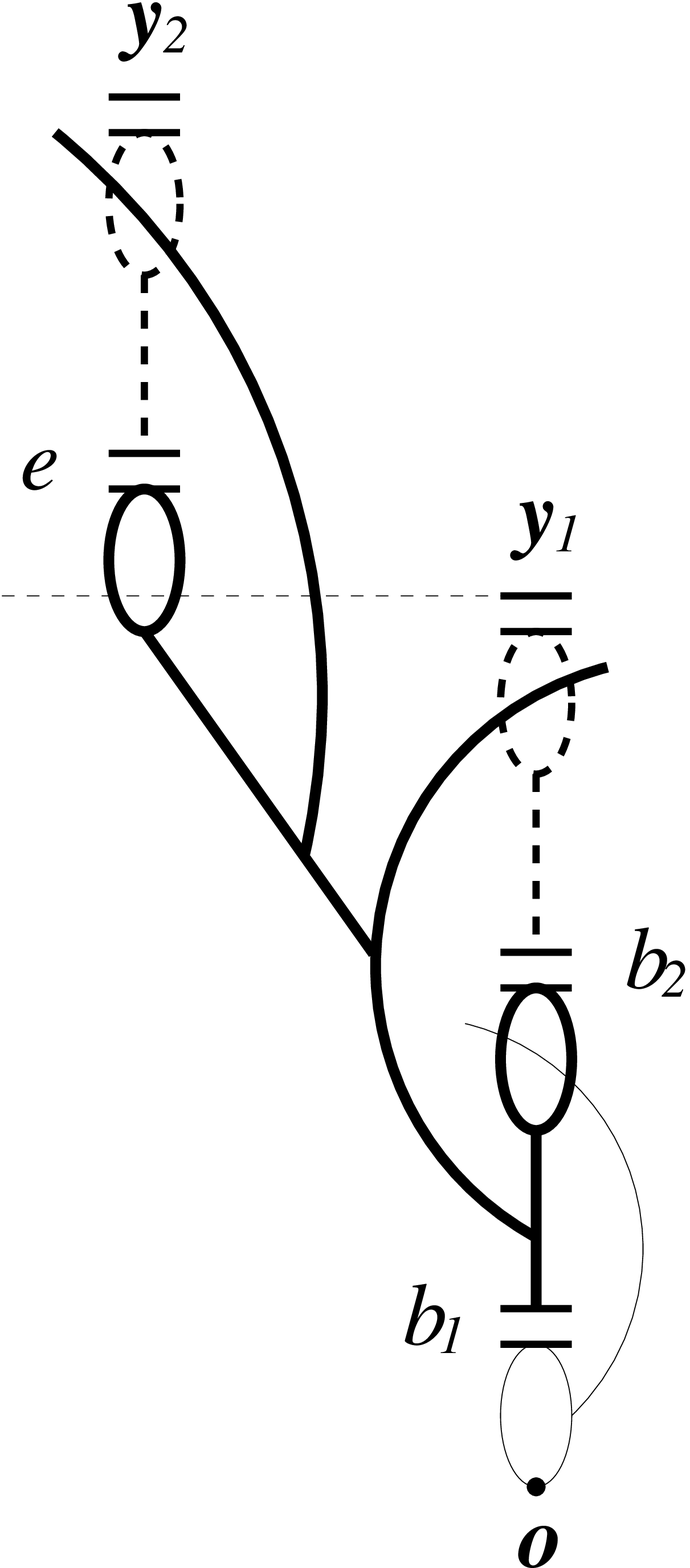}}\hspace{5pc}
\phi^{\sss(1)}(\yvec_1,\yvec_2)_{\sss-}:
 \raisebox{-5pc}{\includegraphics[scale=0.17]{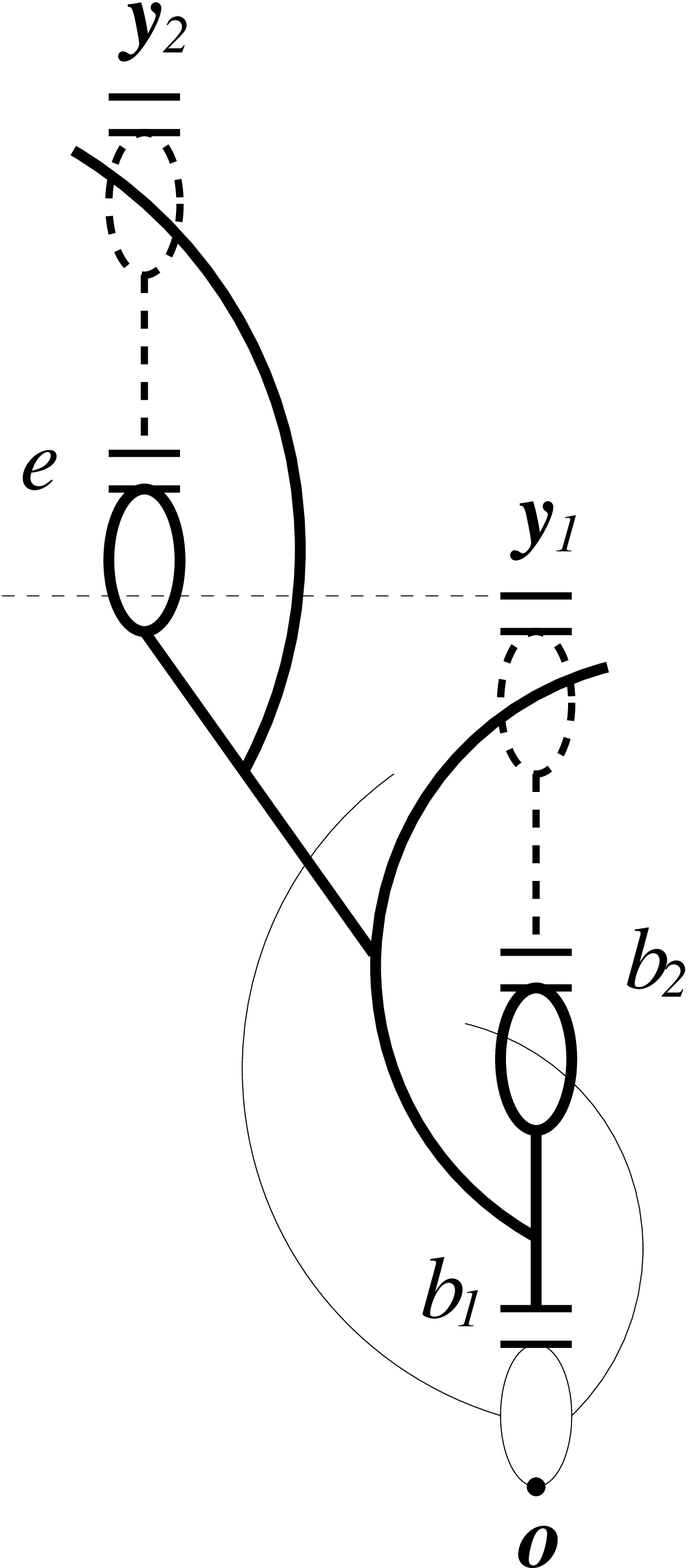}}\qquad+~\quad
 \raisebox{-5pc}{\includegraphics[scale=0.17]{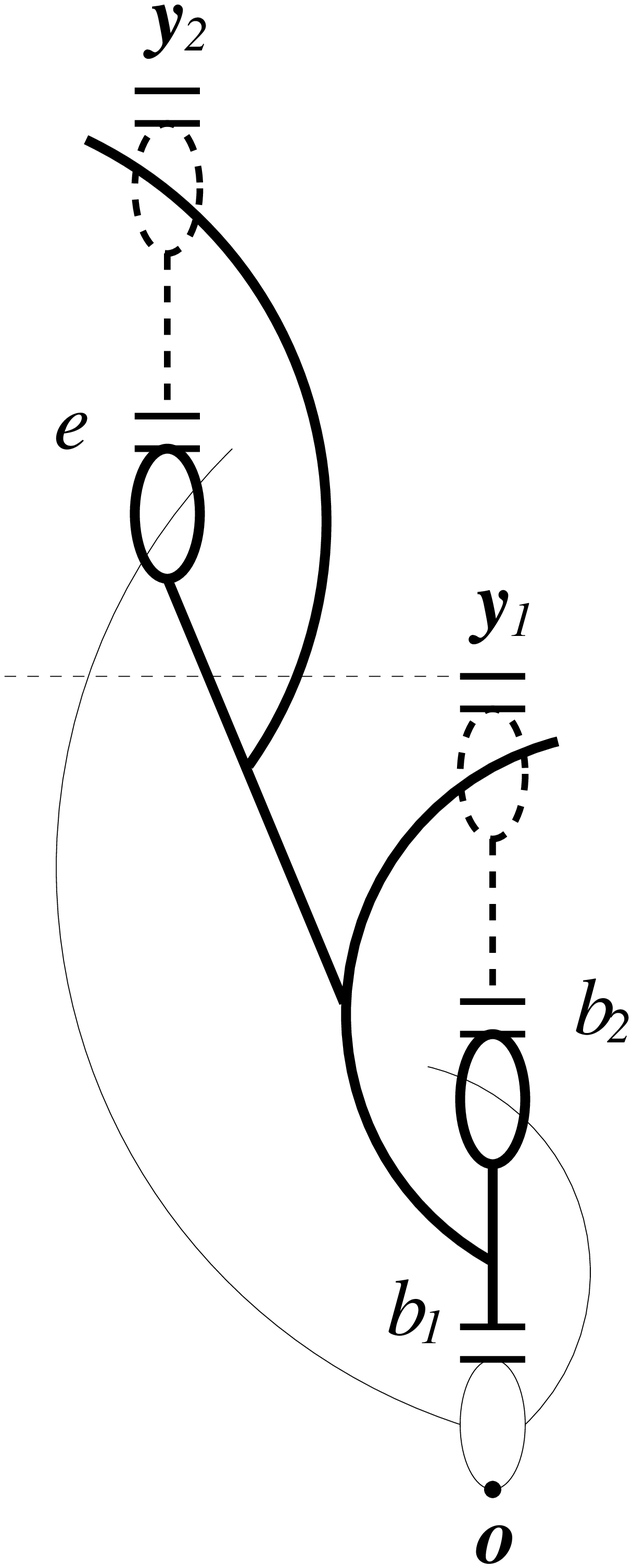}}\\[7pt]
a^{\sss(1)}(\yvec_1,\vec\xvec_I;4)_{\sss+}:
 \raisebox{-5pc}{\includegraphics[scale=0.17]{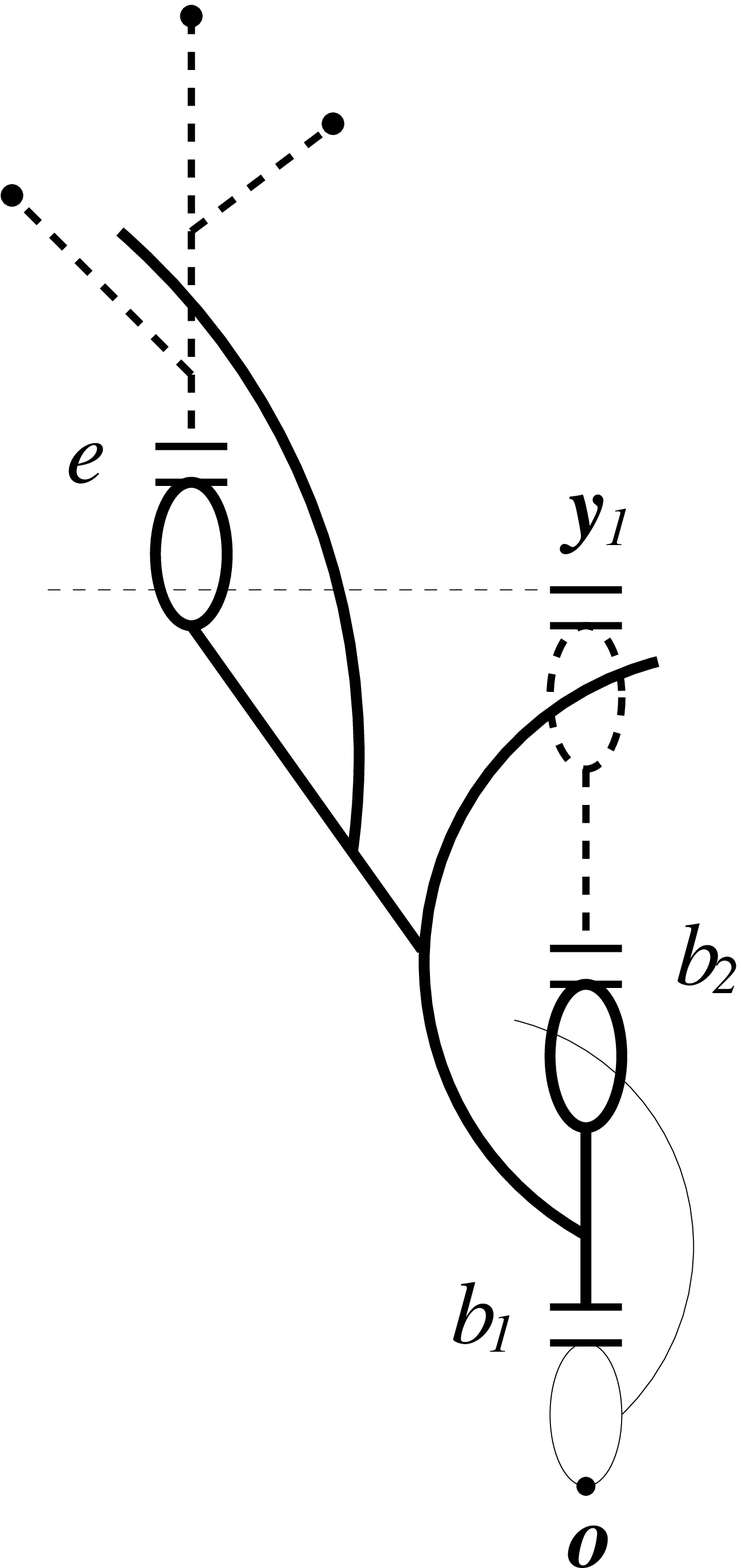}}\hspace{5pc}
a^{\sss(1)}(\yvec_1,\vec\xvec_I;4)_{\sss-}:
 \raisebox{-5pc}{\includegraphics[scale=0.17]{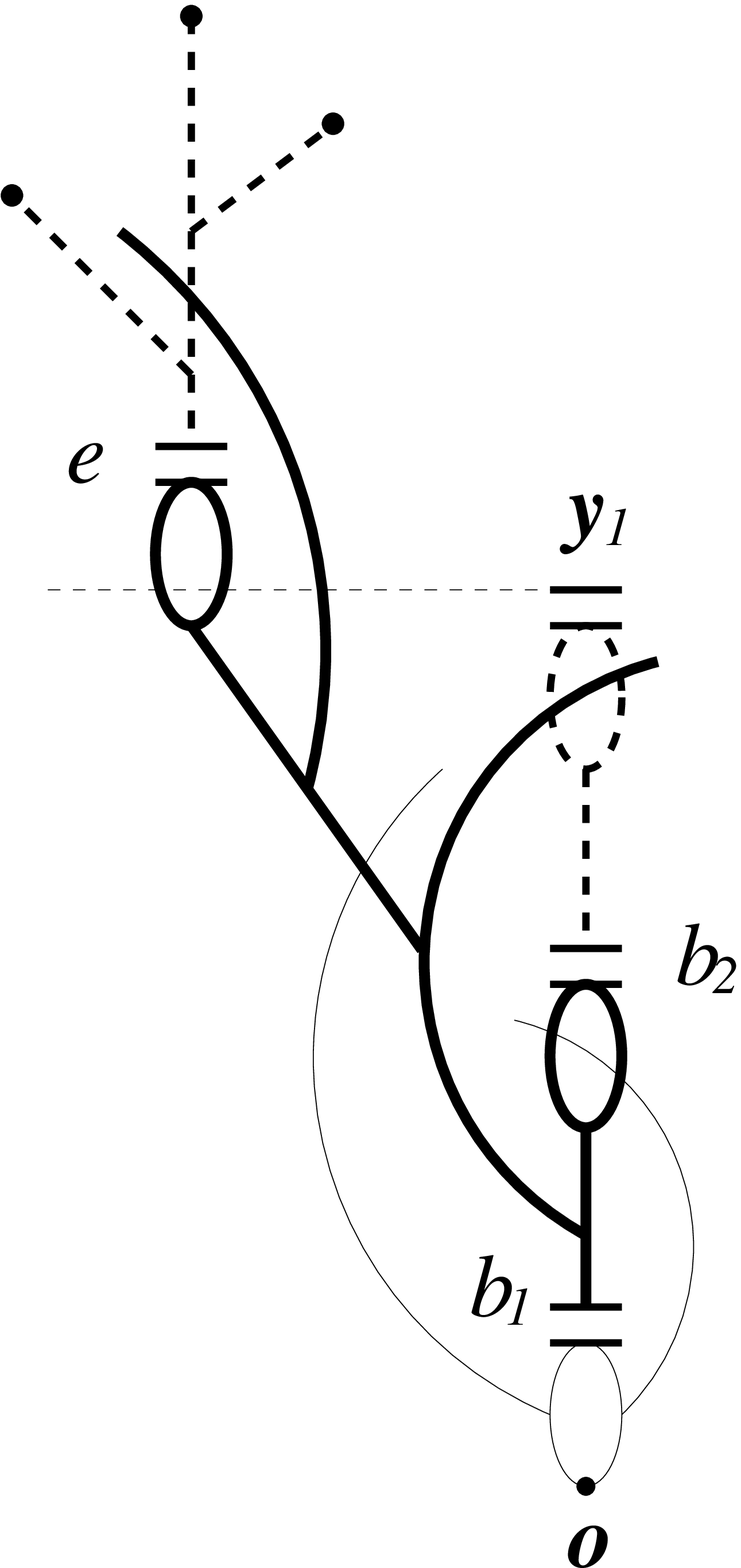}}\qquad+~\quad
 \raisebox{-5pc}{\includegraphics[scale=0.17]{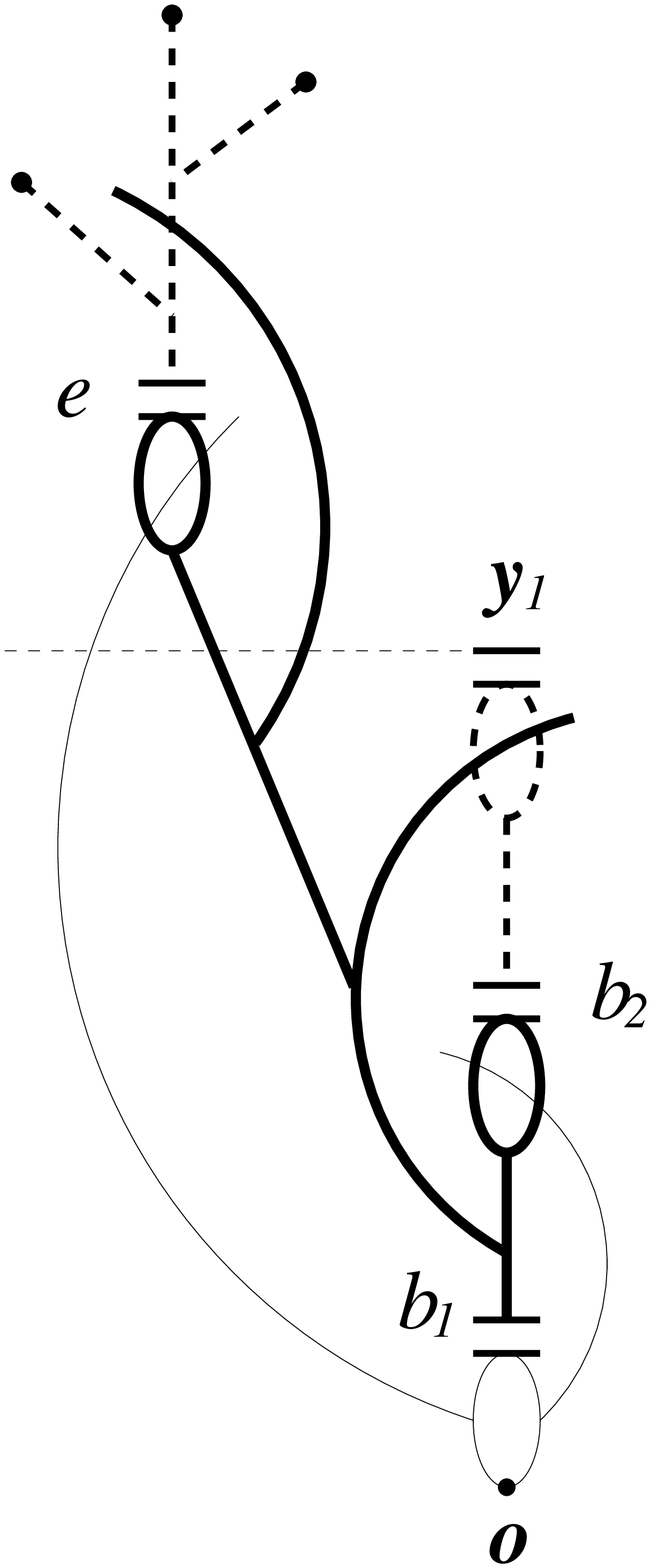}}
\end{gather*}
\caption{\label{fig-2ndexp-main}Schematic representations of
$\phi^{\sss(1)}(\yvec_1,\yvec_2)_{\sss\pm}$ and
$a^{\sss(1)}(\yvec_1,\vec\xvec_I;4)_{\sss\pm}$.  The random variables
$B_\delta(\tb_{\sss N+1},\yvec_1;\bC(\tb_{\sN}))$,
$B_\delta(\te,\yvec_2;\tilde\bC^e(\tb_{\sN}))$ and
$A(\te,\vec\xvec_I;\tilde\bC^e(\tb_{\sN}))$ in
\refeq{phipm-def}--\refeq{a4pm-def} for $N=1$ are reduced,
respectively, to $B^{\sss(0)}(\tb_2,\yvec_1;\bC(\tb_1))$,
$B^{\sss(0)}(\te,\yvec_2;\tilde\bC^e(\tb_1))$ and
$A^{\sss(0)}(\te,\vec\xvec_I;\tilde\bC^e(\tb_1))$ (depicted
in bold dashed lines).}
\end{figure}

The expansion for $\tilde B^{\sss(N)}(\yvec_1,\vec\xvec_I)$ is completed
by using \refeq{tildeBN-rewr}--\refeq{M1B-dec1} and \refeq{M1B-gen3} as
follows.  For convenience, we let
    \begin{align}\lbeq{tildeMN-def}
    \tilde M_{b_1}^{\sss(1)}(X)=\tilde M^{b_1}_{\ovec,\bb_1;\{\ovec\}}(X),&&
    \tilde M_{b_{N+1}}^{\sss(N+1)}(X)=\sum_{b_N}p_{b_N}M_{\bb_N}^{\sss(N)}
    \Big(\tilde M^{b_{N+1}}_{\tb_N,\bb_{N+1};\tilde\bC_{N-1}}(X)\Big)\qquad
    (N\geq1).
    \end{align}
Using this notation, as well as the abbreviations $\bC_{\sN}=\bC(\tb_{\sN})$, $\tilde\bC^e_{\sN}=\tilde\bC^e(\tb_{\sN})$, $\bC_{\sss+}=\{\tb_{\sN}\}$ and
$\bC_{\sss-}=\tilde\bC_{\sss N-1}$, we define, for $N\geq 0$,
\begin{align}
\dpst\phi^{\sss(N)}(\yvec_1,\yvec_2)_{\sss\pm}&=\sum_{\substack{b_{N+1},e\\
 b_{N+1}\ne e}}p_{b_{N+1}}p_e\,\tilde M_{b_{N+1}}^{\sss(N+1)}\Big(\ind{
 \Ftwo_{t_{\yvec_1}}(\tb_N,\eb;\bC_\pm)\text{ in }\tilde\bC^e_N}\,B_\delta
 (\tb_{\sss N+1},\yvec_1;\bC_{\sN})\,B_\delta(\te,\yvec_2;\tilde\bC^e_{\sN})
 \Big),\lbeq{phipm-def}
\end{align}
and, for $\ell=3,4$,
    \eq
    a^{\sss(N)}(\yvec_1,\vec\xvec_I;\ell)
    =a^{\sss(N)}(\yvec_1,\vec\xvec_I;\ell)_{
    \sss+}-
    \ind{j_I>1}\;a^{\sss(N)}(\yvec_1,\vec\xvec_I;\ell)_{\sss-},
    \lbeq{a3IN-def}
    \en
where
\begin{align}
a^{\sss(N)}(\yvec_1,\vec\xvec_I;3)_{\sss\pm}&=\sum_{b_{N+1}}p_{b_{N+1}}\tilde
 M_{b_{N+1}}^{\sss(N+1)}\Big(\indic_{\Ftwo_{t_{\yvec_1}}(\tb_N,\vec\xvec_I;
 \bC_\pm)}\,B_\delta(\tb_{\sss N+1},\yvec_1;\bC_{\sN})\Big),\lbeq{a3pm-def}\\
a^{\sss(N)}(\yvec_1,\vec\xvec_I;4)_{\sss\pm}&=-\sum_{\substack{b_{N+1},e\\
 b_{N+1}\ne e}}p_{b_{N+1}}p_e\,\tilde M_{b_{N+1}}^{\sss(N+1)}\Big(\ind{
 \Ftwo_{t_{\yvec_1}}(\tb_N,\eb;\bC_\pm)\text{ in }\tilde\bC^e_N}\,B_\delta
 (\tb_{\sss N+1},\yvec_1;\bC_{\sN})\,A(\te,\vec\xvec_I;\tilde\bC^e_{\sN})
 \Big).\lbeq{a4pm-def}
\end{align}
These functions correspond to the second term in the left-hand side of
\refeq{M1B-gen3} and the first and second terms in the right-hand side of
\refeq{M1B-gen3}, respectively, when \refeq{M1B-gen3} is substituted into
\refeq{tildeBN-rewr}.  We note that the functions
\refeq{a3IN-def} depend on $I$ via the indicator
$\ind{j_I>1}$, which is due to the fact that both terms in the right-hand
side of \refeq{M1B-dec1} contribute to the case of $j_I>1$, while for the
case of $j_I=1$, the contribution is only from the first term that has been
treated as the case of $\bA=\{\tb_{\sN}\}$.  Now we arrive at
    \begin{align}\lbeq{tildeBN-dec}
    \tilde B^{\sss(N)}(\yvec_1,\vec\xvec_I)-a^{\sss(N)}(\yvec_1,\vec\xvec_I;3)
    =\sum_{\yvec_2}
    \big(\phi^{\sss(N)}(\yvec_1,\yvec_2)_+-\indic_{\{j_I>1\}}\phi^{\sss(N)}(\yvec_1,\yvec_2)_-\big)
    ~\tau(\vec\xvec_I-\yvec_2)
    +a^{\sss(N)}(\yvec_1,\vec\xvec_I;4),
    \end{align}
where $a^{\sss(N)}(\yvec_1,\vec\xvec_I;\ell)$ for $\ell=3,4$ \ch{turn out to} be error
terms.  This extracts the factor $\tau(\vec\xvec_I-\yvec_2)$
from $\tilde B^{\sss(N)}(\yvec,\vec\xvec_I)$.

\subsection{Summary of the expansion for \protect $A(\vec{\xvec}_J)$}\label{ss:grocery}
Recall \refeq{AN-dec2} and \refeq{tildeBN-dec}, and define,
for $N\geq 0$,
    \begin{align}\lbeq{aNxI-def}
    a^{\sss(N)}(\vec\xvec_{J\setminus
    I},\vec\xvec_I)&=a^{\sss(N)}(\vec\xvec_{J\setminus
    I},\vec\xvec_I;2)+\sum_{\yvec_1}\Big(a^{\sss(N)}(\yvec_1,\vec\xvec_I;3)
    +a^{\sss(N)}(\yvec_1,\vec\xvec_I;4)\Big)\,\tau(\vec\xvec_{J\setminus I}-
    \yvec_1),
    \end{align}
let $a^{\sss(N)}(\vec\xvec_J)$ be given by \refeq{edef} and
define
    \begin{align}\lbeq{aphi-def}
    a(\vec\xvec_J)=\sum_{N=0}^\infty(-1)^Na^{\sss(N)}(\vec\xvec_J),&&
    \phi(\yvec_1,\yvec_2)_\pm=\sum_{N=0}^\infty(-1)^N\phi^{\sss(N)}(\yvec_1,
    \yvec_2)_\pm.
    \end{align}
Now, we can summarize the expansion in the previous two subsections
as follows:

\begin{prop}[\textbf{Expansion for $A(\vec\xvec_J)$}]\label{prop:grocery}
For any $\lamb\geq0$, $J\ne\varnothing$ and $\vec\xvec_J\in\Lambda^{|J|}$,
    \begin{align}\lbeq{expA}
    A(\vec\xvec_J)=a(\vec\xvec_J)+\sum_{\vno\ne I\subsetneq J_1}\,\sum_{\yvec_1,
    \yvec_2}C(\yvec_1,\yvec_2)\;\tau(\vec\xvec_{J\setminus I}-\yvec_1)\;
    \tau(\vec\xvec_I-\yvec_2),
    \end{align}
where
    \begin{align}\lbeq{Cdef}
    C(\yvec_1,\yvec_2)=\phi(\yvec_1,\yvec_2)_{\sss+}+\phi(\yvec_2,
    \yvec_1)_{\sss+}-\ch{\phi(\yvec_2,\yvec_1)_{\sss-}}.
    \end{align}
\end{prop}

\begin{proof}
We substitute \refeq{tildeBN-dec} into \refeq{AN-dec2}.
\ch{Note that, by \refeq{minj}, $j_I>1$ precisely when $1\in I$.
Thus, also taking} notice of the difference in $J\setminus I$, which contains
1 in \refeq{BCE}, but may not in \refeq{AN-dec2}, we split the sum
over $I$ arising from in \refeq{AN-dec2} as
    \begin{align}
    &\sum_{\yvec_1,\yvec_2}\bigg(\sum_{\vno\ne I\subset J_1}\ch{\phi(\yvec_1,
    \yvec_2)_{\sss +}}\;\tau(\vec\xvec_{J\setminus I}-\yvec_1)\;\tau(\vec\xvec_I-
    \yvec_2)\nn\\
    &\qquad+\sum_{1\in I\subsetneq J}
    \ch{\Big(\phi(\yvec_1,\yvec_2)_{\sss+}-\phi(\yvec_1,\yvec_2)_{\sss-}\Big)}\;\tau
    (\vec\xvec_{J\setminus I}-\yvec_1)\;\tau(\vec\xvec_I-\yvec_2)\bigg)\nn\\
    &~=\sum_{\yvec_1,\yvec_2}\,\sum_{\vno\ne I\subset J_1}\ch{\phi(\yvec_1,
    \yvec_2)_{\sss+}}\;\tau(\vec\xvec_{J
    \setminus I}-\yvec_1)\;\tau(\vec\xvec_I-\yvec_2)\nn\\
    &\qquad+\sum_{\yvec'_1,\yvec'_2}\,\sum_{\vno\ne I'\subset J_1}
    \ch{\Big(\phi(\yvec'_2,\yvec'_1)_{\sss+}-\phi(\yvec'_2,\yvec'_1)_{\sss+}\Big)}
    \;\tau(\vec\xvec_{J\setminus I'}-\yvec'_1)\;
    \tau(\vec\xvec_{I'}-\yvec'_2)\nn\\
    &~=\sum_{\yvec_1,\yvec_2}\,\sum_{\vno\ne I\subset J_1}\Big(\phi(\yvec_1,
    \yvec_2)_{\sss+}+\phi(\yvec_2,\yvec_1)_{\sss+}-\ch{\phi(\yvec_2,\yvec_1)_{
    \sss-}}\Big)\;\tau(\vec\xvec_{J\setminus I}-\yvec_1)\;\tau(\vec\xvec_I-
    \yvec_2),
    \lbeq{Cphipf}
    \end{align}
where $\yvec'_1, \yvec'_2$ and $I'$ in the middle expression correspond to
$\yvec'_1=\yvec_2,\yvec'_2=\yvec_1$ and $I'=J\setminus I$ on the left hand side of
\refeq{Cphipf}. Therefore, we arrive at \refeq{expA}--\refeq{Cdef}.
This completes the derivation of the lace expansion for the $r$-point
function.
\end{proof}

\subsection{Proof of \refeq{psimainrep} and a comparison
to the survival probability expansion coefficients}
\label{sec-Cvepvep}
In this section, we prove \refeq{psimainrep} and compare the lace-expansion
coefficients for the $r$-point functions to the ones of the survival
probability derived in \cite{HHS05b}.

First we prove \refeq{psimainrep}.  Note that, by \refeq{psidef},
\refeq{psimainrep} is equivalent to
    \begin{align}\lbeq{Cmain}
    C_{\vep,\vep}(y_1,y_2)=p_\vep(y_1)\,p_\vep(y_2)\,(1-\delta_{y_1,y_2}).
    \end{align}
By \refeq{Cdef}, \refeq{Cmain} follows when we show that
    \begin{align}\lbeq{psimainreprep}
    \phi_{\vep,\vep}(y_1,y_2)_{\sss\pm}=p_\vep(y_1)\,p_\vep(y_2)\,
    (1-\delta_{y_1,y_2}).
    \end{align}
According to \refeq{phipm-def},
$\phi_{\vep,\vep}^{\sss(N)}(y_1,y_2)_{\sss\pm}=0$ unless $N=0$.  Also, by
\refeq{convention}, we see that
$\phi_{\vep,\vep}^{\sss(0)}(y_1,y_2)_{\sss+}=\phi_{\vep,\vep}^{\sss(0)}(y_1,
y_2)_{\sss-}$.  Therefore, \ch{since $p_\vep(y_1)\,p_\vep(y_2)\,
(1-\delta_{y_1,y_2})$ is symmetric in $y_1, y_2$,} it suffices to show that
    \begin{align}\lbeq{phipm0-suff}
    \phi_{\vep,\vep}^{\sss(0)}(y_1,y_2)_{\sss+}&\equiv\sum_{b,e:b\ne e}p_b
    p_e\,\tilde\mE^b\Big[\indic_{E'(\ovec,\bb;\{\ovec\})}\ind{\Ftwo_\vep
    (\ovec,\eb;\{\ovec\})\text{ in }\tilde\bC^e(\ovec)}\,B_\delta(\tb,(y_1,
    \vep);\bC(\ovec))\,B_\delta(\te,(y_2,\vep);\tilde\bC^e(\ovec))\Big]\nn\\
    &=p_\vep(y_1)\,p_\vep(y_2)\,(1-\delta_{y_1,y_2}).
    \end{align}
However, this immediately follows from the fact that the product of the
two indicators in $\tilde\mE^b$ is $\ind{\bb=\eb=\ovec}$ (cf.,
\refeq{E'def} and \refeq{F'2-def}) and that, by \refeq{Bdelta-def},
$B_\delta(\tb,(y_1,\vep);\bC(\ovec))=\delta_{\tb,(y_1,\vep)}$ and
$B_\delta(\te,(y_2,\vep);\tilde\bC^e(\ovec))=\delta_{\te,(y_2,\vep)}$.
This completes the proof of \refeq{psimainrep}.

Next we compare the lace-expansion coefficients for the $r$-point functions
to the ones of the survival probability derived in \cite{HHS05b}.
We recall that, by \refeq{psidef},
    \eq
    \lbeq{psiCrel}
    \hat\psi_{s_1,s_2}(0,0)=p_\vep\hat{C}_{s_1-\vep,s_2-\vep}(0,0),
    \en
where $p_\vep=\hat p_\vep(0)\equiv1+(\lamb-1)\vep$.
In \cite{HHS05b}, it is shown that survival probability
    \eq
    \theta_t=\mathbb{P}(\exists x\in \Z^d: (0,0)\conn (x,t)).
    \en
satisfies a lace expansion of the form
    \eq
    \lbeq{laceexpeqmain}
    \theta_{t} = \ddsum_{0\leq s\leq t-\vep} \pi_t p_\vep \theta_{t-\vep-s}
    -\ddsum_{\vep\leq s_1\leq\lfloor t/2 \rfloor}\ddsum_{\vep\leq s_2 \leq t}
    \phi_{s_1,s_2} \theta_{t-s_1}\theta_{t-s_2} +e_t,
    \en
where $\pi_t=\hat{\pi}_t(0)$ arises in the lace expansion for the 2-point
function and $e_t$ are error terms.

We claim that the lace-expansion coefficients
for the $r$-point functions and the ones for the survival probability in
\cite{HHS05b} satisfy
    \eq
    \lbeq{phipsirel}
    \hat\psi_{s_1,s_2}(0,0)=2p_\vep\phi_{s_1-\vep,s_2-\vep}.
    \en
This relation shall play an essential role in \cite{hsa06}. Indeed, we bound
$\hat\psi_{s_1,s_2}(0,0)$ in the present paper, and therefore, shall allow to
make use of the bounds derived here in the sequel to this paper \cite{hsa06}.
We now prove \refeq{phipsirel}.

By \ch{\refeq{psiCrel},} \refeq{phipsirel} is
equivalent to
    \begin{align}\lbeq{phipsirel-equiv}
    \hat C_{s_1,s_2}(0,0)=2\phi_{s_1,s_2}.
    \end{align}
We refer to \cite[(4.3) and (5.17)]{HHS05b} for the lace-expansion coefficients
of the oriented-percolation survival probability:
    \begin{align}
    \phi_{s_1,s_2}=\sum_{N=0}^\infty(-1)^N\phi_{s_1,s_2}^{\sss(N)},&&
    2\phi_{s_1,s_2}^{\sss(N)}=2\phi_{s_1,s_2}^{\sss(N)}(\{\vvec_{\sss N-1}
    \})-\phi_{s_1,s_2}^{\sss(N)}(\tilde C_{\sss N-1}),
    \end{align}
where $\phi_{s_1,s_2}^{\sss(N)}(\{\vvec_{\sss N-1}\})$ and
$\phi_{s_1,s_2}^{\sss(N)}(\tilde C_{\sss N-1})$ correspond
respectively to $\hat\phi_{s_1,s_2}^{\sss(N)}(0,0)_{\sss+}$ and
$\hat\phi_{s_1,s_2}^{\sss(N)}(0,0)_{\sss-}$ in our paper.
Therefore, \refeq{phipsirel-equiv} follows from \refeq{aphi-def} and
\refeq{Cdef}.  This completes the proof of \refeq{phipsirel}.
We refer to \cite{hsa06} for a more detailed discussion of the implications
of \refeq{phipsirel}.

%%%%%%%%%%%%%%%%%%%%%%%%%%%%%%%%%%%%%%%%%%%%%%%%%%%%%%%%%%%%%%%%%%%%%%%%
%%%%%%%%%%%%%%%%%%%%%%%%%%%%%%%%%%%%%%%%%%%%%%%%%%%%%%%%%%%%%%%%%%%%%%%%
%%%%%%%%%%%%%%%%%%%%%%%%%%%%%%%%%%%%%%%%%%%%%%%%%%%%%%%%%%%%%%%%%%%%%%%%
%%%%%%%%%%%%%%%%%%%%%%%%%%%%%%%%%%%%%%%%%%%%%%%%%%%%%%%%%%%%%%%%%%%%%%%%
%\input{rpt4}

% June 1, 2007, RvdH

\newcommand{\hatbetaT}{\beta_1T^{-\alphamin}}
\newcommand{\Nsp}{\#_{\tt spat}}

\section{Bounds on \protect $B(\xvec)$ and \protect $A(\vec{\xvec}_J)$}\label{s:bounds}
In this section, we prove the following proposition, in which we denote the
second-largest element of $\{t_j\}_{j\in J}$ by $\bar t=\bar t_J$:

\begin{prop}[\textbf{Bounds on the coefficients of the linear expansion}]
 \label{prop:BAbds}
\begin{enumerate}[(i)]
\item
Let $d>4$ and $L\gg1$.  For $\lamb\le\lambc^{\sss(\vep)}$, $N\ge0$,
$t\in\vep\mN$, $\vec t_J\in(\vep\Zp)^{|J|}$ and $q=0,2$,
    \begin{align}
    \sum_x|x|^qB_t^{\sss(N)}(x)&\le\big((1-\vep)\delta_{q,0}+\lamb\vep
     \sigma^q\big)\delta_{t,\vep}\delta_{N,0}+\vep^2\frac{O(\beta)^{1\vee
     N}\sigma^q}{(1+t)^{(d-q)/2}},\lbeq{Bbd}\\
    \sum_{\vec x_J}A^{\sss(N)}_{\vec t_J}(\vec x_J)&\le\vep O(\beta)^N
     O\big((1+\bar t)^{r-3}\big),\lbeq{Abd}
    \end{align}
where the constant in the $O(\beta)$ term is independent of $\vep,L,N$ and $t$
(or $\bar t$ in \refeq{Abd}).
\item
Let $d\le4$ with $\alpha\equiv bd-\frac{4-d}2>0$,
$\hat\beta_\sT=\hatbetaT$ with $\alphamin\in(0,\alpha)$, and $L_1\gg1$.
For $\lamb\le\lambc^{\sss(\vep)}$, $N\ge0$, $t\in\vep\mN\cap[0,T\log T]$,
$\vec t_J\in(\vep\Zp)^{|J|}$ with $\max_{j\in J}t_j\le T\log T$ and $q=0,2$,
    \begin{align}
    \sum_x|x|^qB_t^{\sss(N)}(x)&\le\big((1-\vep)\delta_{q,0}+\lamb\vep
     \sigma_{\sT}^q\big)\delta_{t,\vep}\delta_{N,0}+\vep^2\frac{O(\beta
     _{\sT})\,O(\hat\beta_{\sT})^{0\vee(N-1)}\sigma_{\sT}^q}
     {(1+t)^{(d-q)/2}},\lbeq{Bbd<4}\\
    \sum_{\vec x_J}A^{\sss(N)}_{\vec t_J}(\vec x_J)&\le\vep O(\hat
     \beta_{\sT})^NO\big((1+\bar t)^{r-3}\big),\lbeq{Abd<4}
    \end{align}
where the constants in the $O(\beta_{\sT})$ and $O(\hat\beta_{\sT})$ terms are
independent of $\vep,L_1,T,N$ and $t$ (or $\bar t$ in \refeq{Abd<4}).
\end{enumerate}
\end{prop}

In Section~\ref{sec-PN}, we define several constructions that will be used
later to define bounding diagrams for $B(\xvec)$, $A(\vec\xvec)$,
$C(\yvec_1,\yvec_2)$ and $a(\vec\xvec)$.  There, we also summarize effects
of these constructions.  Then, we prove the above bounds on $B(\xvec)$ in
Section~\ref{ss:Bbd}, and the bounds on $A(\vec{\xvec}_J)$ in
Section~\ref{ss:Abd}. Throughout Sections \ref{s:bounds}--\ref{ss:ebd}, we
shall frequently assume that $\lambda\le 2$, which follows from
\refeq{estimates-discr} for $d>4$ and $L\gg1$, and from the restriction
on $\lamb_{\sT}$ in Theorem \ref{thm:2pt} for $d\le4$ and $L_1\gg1$.

\subsection{Constructions: I}\label{sec-PN}
First, in Section~\ref{sss:constr-def}, we introduce several constructions that
will be used in the following sections to define bounding diagrams on relevant
quantities.  Then, in Section~\ref{sss:constr-eff}, we show that these
constructions can be used iteratively by studying the effect of applying
constructions to diagram functions.  Such iterative bounds will be
crucial in Sections~\ref{ss:Bbd}--\ref{ss:Abd} to prove
Proposition~\ref{prop:BAbds}.

\subsubsection{Definitions of constructions}\label{sss:constr-def}
For $b=(\uvec,\vvec)$ with $\uvec=(u,s)$ and $\vvec=(v,s+\vep)$, we will abuse
notation to write $p(b)$ or $p(\vvec-\uvec)$ for $p_\vep(v-u)$, and $D(b)$ or
$D(\vvec-\uvec)$ for $D(v-u)$. Let
    \begin{align}\lbeq{vhi-def}
    \vhi(\xvec-\uvec)=(p\sstar\tau)(\xvec-\uvec),
    \end{align}
and (see Figure~\ref{fig:L-def})
    \begin{align}\lbeq{Lne-def}
    L(\uvec,\vvec;\xvec)&=
    \begin{cases}
    \vhi(\xvec-\uvec)~(\tau\sstar\lamb\vep D)(\xvec-\vvec)+(\vhi\sstar
     \lamb\vep D)(\xvec-\uvec)~\tau(\xvec-\vvec)&(\uvec\ne\vvec),\\
    (\lamb\vep D\sstar\tau)(\xvec-\uvec)~(\tau\sstar\lamb\vep D)(\xvec
     -\uvec)+(\lambda\vep D\sstar\tau\sstar\lamb\vep D)(\xvec-\uvec)~
     \tau(\xvec-\uvec)&(\uvec=\vvec),
    \end{cases}
    \end{align}
where $\vhi$ for $\uvec\ne\vvec$ corresponds to $\lamb\vep D\sstar\tau$ for
$\uvec=\vvec$.  We
call the lines from $\uvec$ to $\xvec$ in $L(\uvec,\vvec;\xvec)$ the
\emph{$L$-admissible lines}.  Here, with lines, we mean $\vhi(\xvec-\uvec)$
and $(\vhi\sstar\lamb\vep D)(\xvec-\uvec)$ when $\uvec\ne\vvec$.  If
$\uvec=\vvec$, then we define both lines from $\uvec$ to $\xvec$ in each term
in $L(\uvec,\uvec;\xvec)$ to be $L$-admissible.  We note that these lines can
be represented by 2-point functions as, e.g.,
    \begin{align}
    \lbeq{lines-two}
    (\vhi\sstar\lamb\vep D)(\xvec-\uvec)=\sum_{b=(\uvec,\,\cdot\,)}\;\sum_{
     \substack{b'=(\,\cdot\,,\xvec)\\ \text{spatial}}}\tau(\tb-\bb)\;\tau
     (\bb'-\tb)\;\tau(\tbp-\bb').
    \end{align}
Thus, below, we will frequently interpret lines to denote 2-point functions.

\begin{figure}[t]
\begin{align*}
\text{(a)}\qquad\raisebox{-3pc}{\includegraphics[scale=0.18]{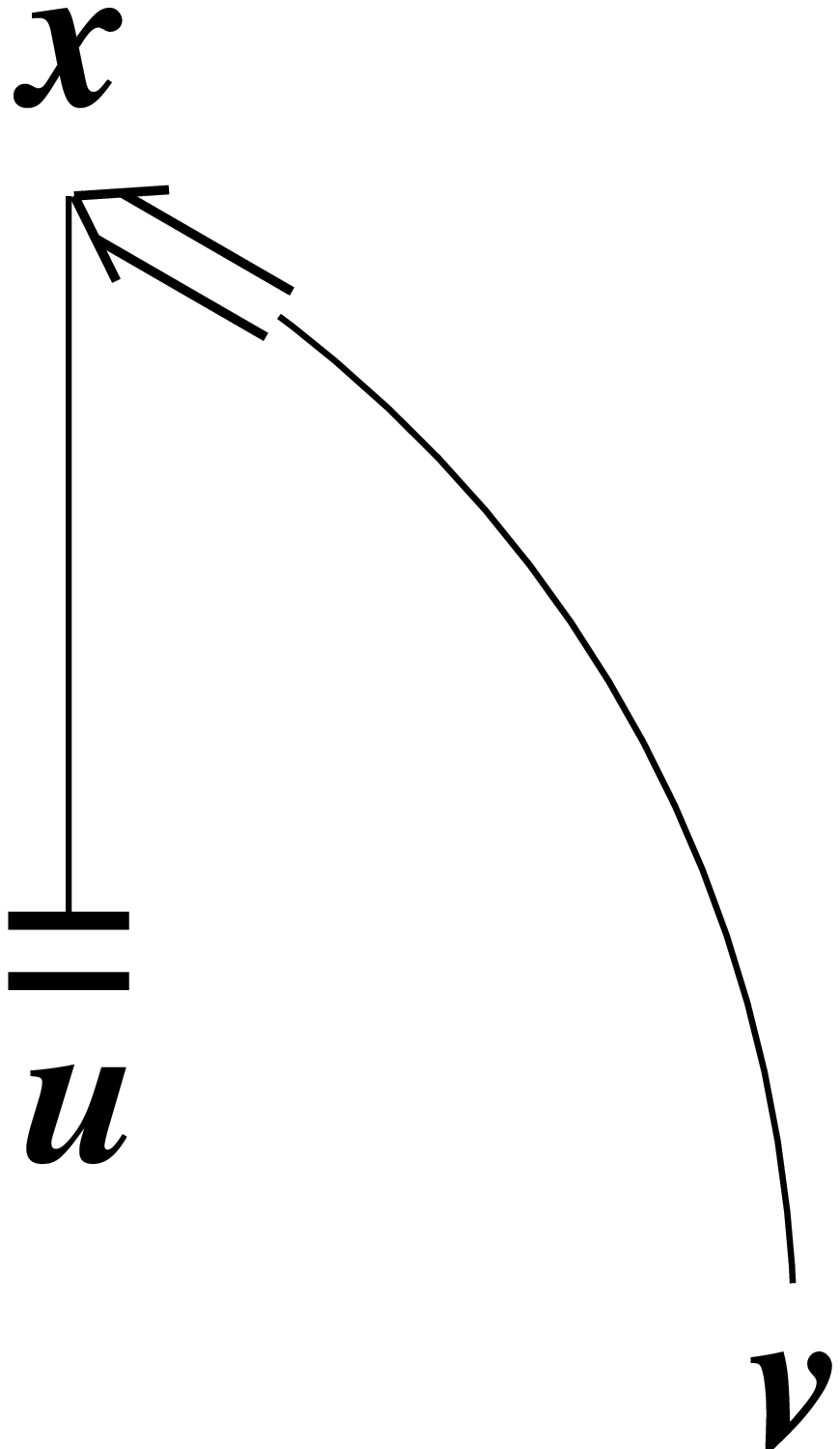}}\quad
 +\qquad\raisebox{-3pc}{\includegraphics[scale=0.18]{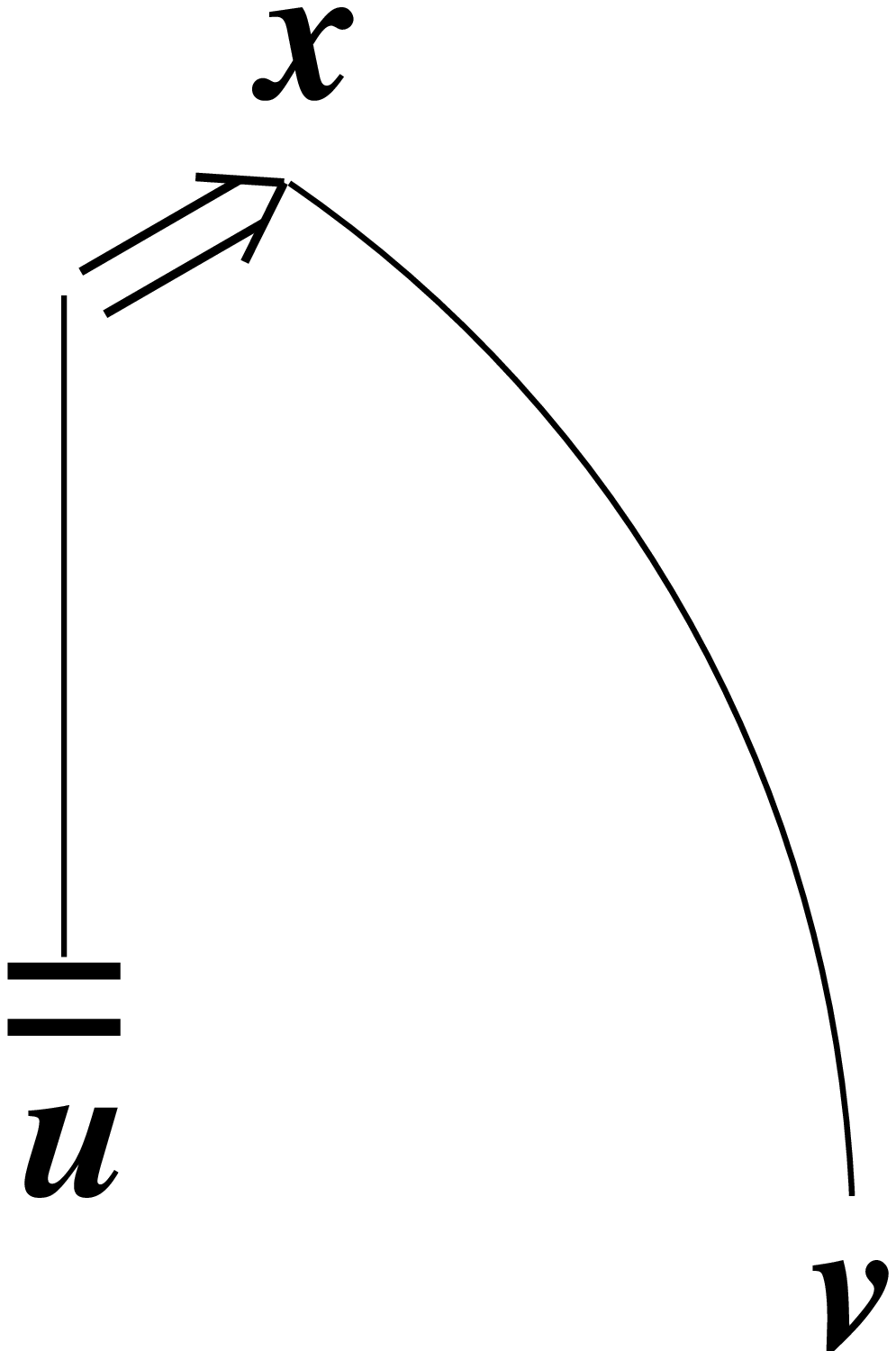}}&&&&
\text{(b)}\qquad\raisebox{-3pc}{\includegraphics[scale=0.18]{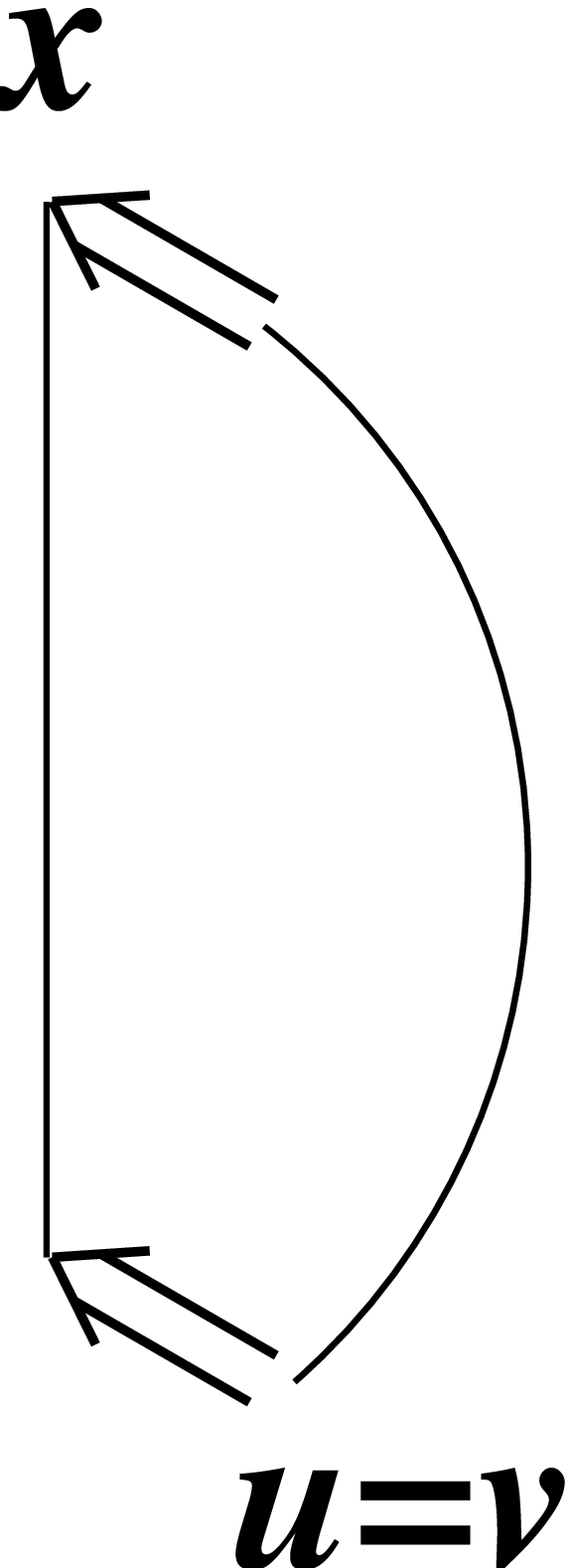}}\qquad
 +~\qquad\raisebox{-3pc}{\includegraphics[scale=0.18]{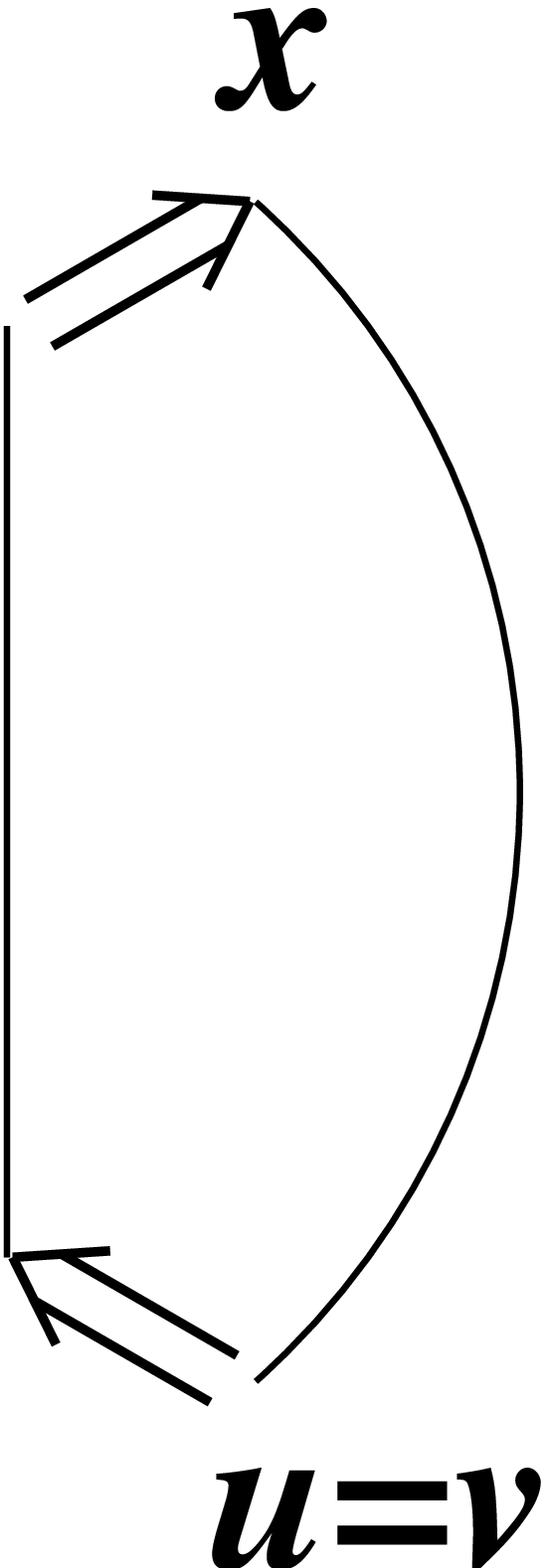}}
\end{align*}
\caption{\label{fig:L-def}Schematic representation of $L(\uvec,\vvec;\xvec)$
for (a) $\uvec\ne\vvec$ and (b) $\uvec=\vvec$.  Here, the tilted arrows
denote spatial bonds, while the short double line segments at $\uvec$ in Case
(a) denote unspecified bonds that could be spatial or temporal.}
\end{figure}

We will use the following constructions to prove Proposition~\ref{prop:BAbds}:

\begin{defn}[\textbf{Constructions~$B$, $\ell$, $2^{\sss(i)}$ and $E$}]
 \label{def-constr}
\begin{enumerate}[(i)]
\item
\emph{Construction~$B$.}  Given any diagram line $\twop$, say
$\tau(\xvec-\vvec)$, and given $\yvec\ne\xvec$, we define
Construction~$\bspat^{\twop}(\yvec)$ to be the operation in which
$\tau(\xvec-\vvec)$ is replaced by
    \begin{align}\lbeq{bspat}
    \tau(\yvec-\vvec)~(\lamb\vep D\sstar\tau)(\xvec-\yvec)\quad =
     \quad\raisebox{-3pc}{\includegraphics[scale=0.14]{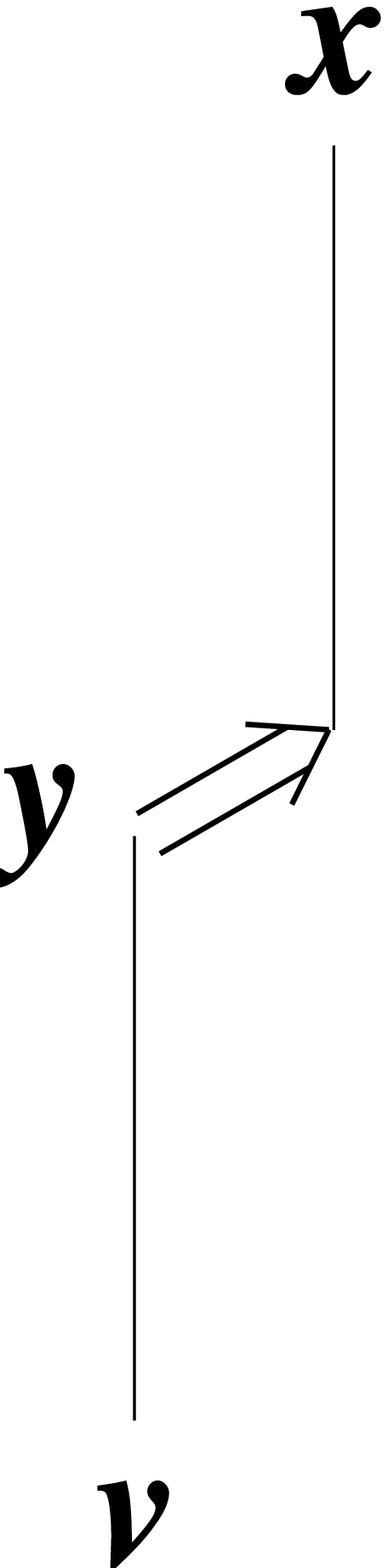}}~~,
    \end{align}
and define Construction~$\btemp^{\twop}(\yvec)$ to be the operation in which
$\tau(\xvec-\vvec)$ is replaced by
\begin{align}\lbeq{btemp}
\sum_{b:\tb=\yvec}\tau(\bb-\vvec)~\lamb\vep D(b)~\mP((\bb,\bb_+)\conn
 \xvec)\quad = \quad\raisebox{-3pc}{\includegraphics[scale=0.14]{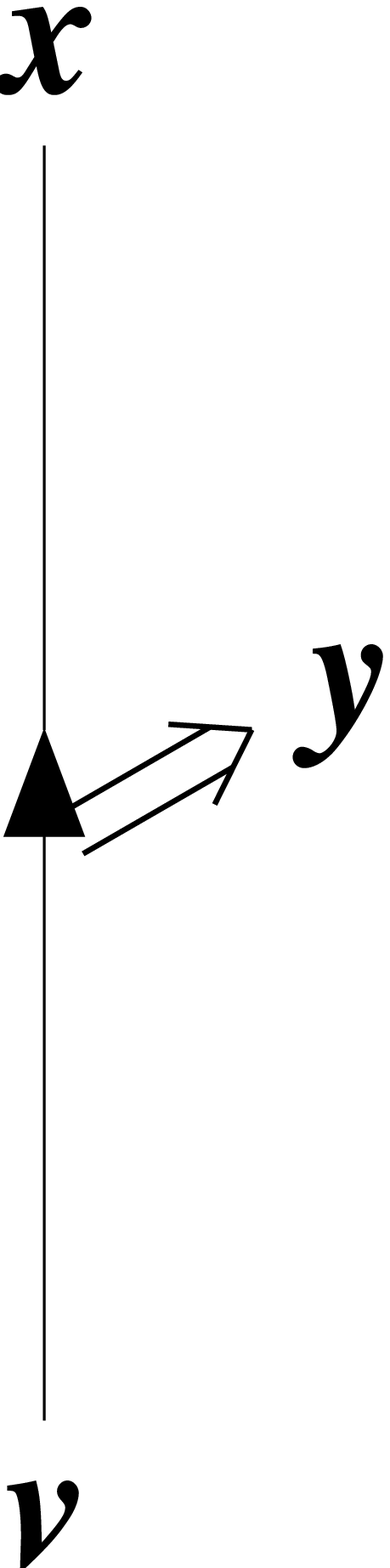}}~~,
\end{align}
where $\{b\conn\xvec\}=\{b$ is occupied$\}\cap\{\tb\conn\xvec\}$ and
$\vvec_+=(v,t_{\vvec}+\vep)$ for $\vvec=(v,t_{\vvec})$.
Construction~$B^{\twop}(\yvec)$ applied to $\tau(\xvec-\vvec)$ is the sum of
$\tau(\xvec-\vvec)\delta_{\xvec,\yvec}$ and the results of
Construction~$\bspat^{\twop}(\yvec)$ and Construction~$\btemp^{\twop}(\yvec)$
applied to $\tau(\xvec-\vvec)$.  Construction~$B^{\twop}(s)$ is the operation
in which Construction~$B^{\twop}(y,s)$ is performed and then followed by
summation over $y\in\Zd$.  Constructions~$\bspat^{\twop}(s)$ and
$\btemp^{\twop}(s)$ are defined similarly.  We omit the superscript $\twop$
and write, e.g., Construction~$B(\yvec)$ when we perform
Construction~$B^\twop(\yvec)$ followed by a sum over \emph{all} possible lines
$\twop$.  We denote the result of applying Construction~$B(\yvec)$ to a
diagram function $f(\xvec)$ by $f(\xvec;B(\yvec))$, and define
$f(\xvec;\bspat(\yvec))$ and $f(\xvec;\btemp(\yvec))$ similarly.  For example,
we denote the result of applying Construction~$\bspat(\yvec)$ to the line
$\varphi(\xvec)$ by
\begin{align}
\varphi(\xvec;\bspat(\yvec))\equiv(p\sstar\tau)(\xvec;\bspat(\yvec))
 =\delta_{\ovec,\yvec}(\lambda\vep D\sstar\tau)(\xvec)+\varphi(\yvec)\,
 (\lambda\vep D\sstar\tau)(\xvec-\yvec),
\end{align}
where $\delta_{\ovec,\yvec}(\lambda\vep D\sstar\tau)(\xvec)$ is the
contribution in which $p$ of $\varphi$ is replaced by $\lambda\vep D$.
\item
\emph{Construction~$\ell$.}  Given any diagram line $\twop$,
Construction~$\ell^{\twop}(\yvec)$ is the operation in which a line to $\yvec$
is inserted into the line $\twop$.  This means, for example, that the 2-point function $\tau(\uvec-\vvec)$ corresponding to the line $\twop$ is replaced by
\eq\lbeq{Constr-ell}
\sum_{\zvec}\tau(\uvec-\vvec;B^\twop(\zvec))\,\tau(\yvec-\zvec).
%\vep \lambda  \sum_{b} D(b)[\tau(\vvec-\tb)\tau(\bb-\uvec)\tau(\yvec-\bb)
%+\tau(\vvec-\bb)\tau(\bb-\uvec)\tau(\yvec-\tb)\big].
\en
 We omit the superscript $\twop$ and write Construction~$\ell(\yvec)$ when we
perform Construction~$\ell^\twop(\yvec)$ followed by a sum over \emph{all}
possible lines $\twop$.  We write $F(\vvec,\yvec;\ell(\zvec))$ for the diagram
where Construction~$\ell(\zvec)$ is performed on the diagram $F(\vvec,\yvec)$.
Similarly, for $\vec\yvec=(\yvec_1,\dots,\yvec_j)$,
Construction~$\ell(\vec\yvec)$ is the repeated application of
Construction~$\ell(\yvec_i)$ for $i=1,\dots,j$.  We note that the order of
application of the different Construction~$\ell(\yvec_i)$ is irrelevant.
\item
\emph{Constructions~$2^{\sss(i)}$\,and $E$.}  For a diagram $F(\vvec,\uvec)$
with two vertices carrying labels $\vvec$ and $\uvec$ and with a certain set of
admissible lines, Constructions~$2_{\uvec}^{\sss(1)}(\wvec)$ and
$2_{\uvec}^{\sss(0)}(\wvec)$ produce the diagrams
    \begin{align}
    F(\vvec,\uvec;2_{\uvec}^{\smallsup{1}}(\wvec))&=\sum_\twop\sum_{
     \uvec,\zvec}F(\vvec,\uvec;B^{\twop}(\zvec))\,L(\uvec, \zvec;\wvec),
     \lbeq{tilF1}\\
    F(\vvec,\uvec;2_{\uvec}^{\smallsup{0}}(\wvec))&=F(\vvec,\wvec)+
     F(\vvec,\uvec;2_{\uvec}^{\smallsup{1}}(\wvec)),\lbeq{tilF0}
    \end{align}
where $\sum_\twop$ is the sum over the set of admissible lines for
$F(\vvec,\uvec)$.  Here and elsewhere, we use Einstein's summation convention:
each diagram function $F(\vvec,\uvec;2_{\uvec}^{\sss(i)}(\wvec))$ depends only
on $\vvec$ and $\wvec$, but not on $\uvec$.  We call the $L$-admissible lines
of the added factor $L(\uvec,\zvec;\wvec)$ in \refeq{tilF1} the
\emph{$2^{\sss(1)}$-admissible lines} for
$F(\vvec,\uvec;2_{\uvec}^{\sss(1)}(\wvec))$.
%For a diagram $F(\vvec, \yvec)$, we write $F(\yvec;2_{\yvec}^{\smallsup{0}}(\wvec))$
%for the result of an application of Construction~$2_{\yvec}^{\smallsup{0}}(\wvec)$
%to $F(\vvec,\yvec)$, and $F(\yvec;2_{\yvec}^{\smallsup{1}}(\wvec))$ for the result
%of an application of Construction~$2_{\yvec}^{\smallsup{1}}(\wvec)$ to $F(\vvec,\yvec)$.
%We emphasize here that the diagram $F(\vvec,\yvec;2_{\yvec}^{\smallsup{i}}(\wvec))$
%only depends on $\vvec, \wvec$, and not on $\yvec$.
Construction~$E_{\yvec}(\wvec)$ is the successive applications of
Constructions~$2_{\yvec}^{\sss(1)}(\zvec)$ and $2_{\zvec}^{\sss(0)}(\wvec)$
(cf., Figure~\ref{fig-tilFz}):
    \begin{align}\lbeq{tilFz}
    F(\vvec,\yvec;E_{\yvec}(\wvec))&=F\big(\vvec,\yvec;2_{\yvec}^{\sss(1)}(\uvec),
     2_{\uvec}^{\sss(0)}(\wvec)\big)\nn\\[5pt]
    &\equiv F\big(\vvec,\yvec;2_{\yvec}^{\sss(1)}(\wvec)\big)+\sum_\twop\sum_{
     \uvec,\zvec}F\big(\vvec,\yvec;2_{\yvec}^{\sss(1)}(\uvec),B^{\twop}(\zvec)
     \big)\,L(\uvec,\zvec;\wvec),
    \end{align}
where $\sum_\twop$ is the sum over the $2^{\sss(1)}$-admissible lines for
$F(\vvec,\yvec;2_{\yvec}^{\sss(1)}(\uvec))$.  Note that
$F(\vvec,\yvec;E_{\yvec}(\wvec))$ also depends only on $\vvec$ and $\wvec$, but
not on $\yvec$.
%In words, Construction~$E_{\yvec}(\wvec)$ is the same as
%Construction~$2_{\yvec}^{\smallsup{1}}(\wvec)$ followed by
%Construction~$2_{\wvec}^{\smallsup{0}}(\yvec)$, where the unique admissible
%line prior to the application of Construction~$2_{\wvec}^{\smallsup{0}}(\yvec)$
%is the $L$-admissible line in $L(\uvec, \zvec;\wvec)$ added to the diagram
%in the application of Construction~$2_{\uvec}^{\smallsup{1}}(\wvec)$
%in \refeq{tilF1}.\\
We further define the \emph{$E$-admissible lines} to be all the lines added
in the Constructions~$2_{\yvec}^{\sss(1)}(\zvec)$ and
$2_{\zvec}^{\sss(0)}(\wvec)$.
%
%arising in the Construction~$2_{\wvec}^{\smallsup{0}}(\yvec)$ for the
%second term in \refeq{tilF0}, and the $L$-admissible lines
%arising in the Construction~$2_{\wvec}^{\smallsup{1}}(\yvec)$ for the
%first term in \refeq{tilF0}.
\end{enumerate}
\end{defn}

\begin{figure}[t]
\begin{align*}
\raisebox{-3.5pc}{\includegraphics[scale=0.17]{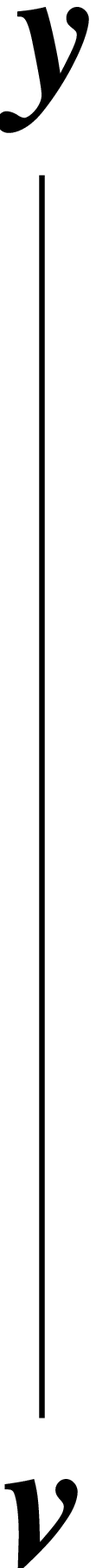}}\quad
 \underset{2_{\yvec}^{\sss(1)}(\zvec)}{\longrightarrow}\quad
\raisebox{-3.5pc}{\includegraphics[scale=0.17]{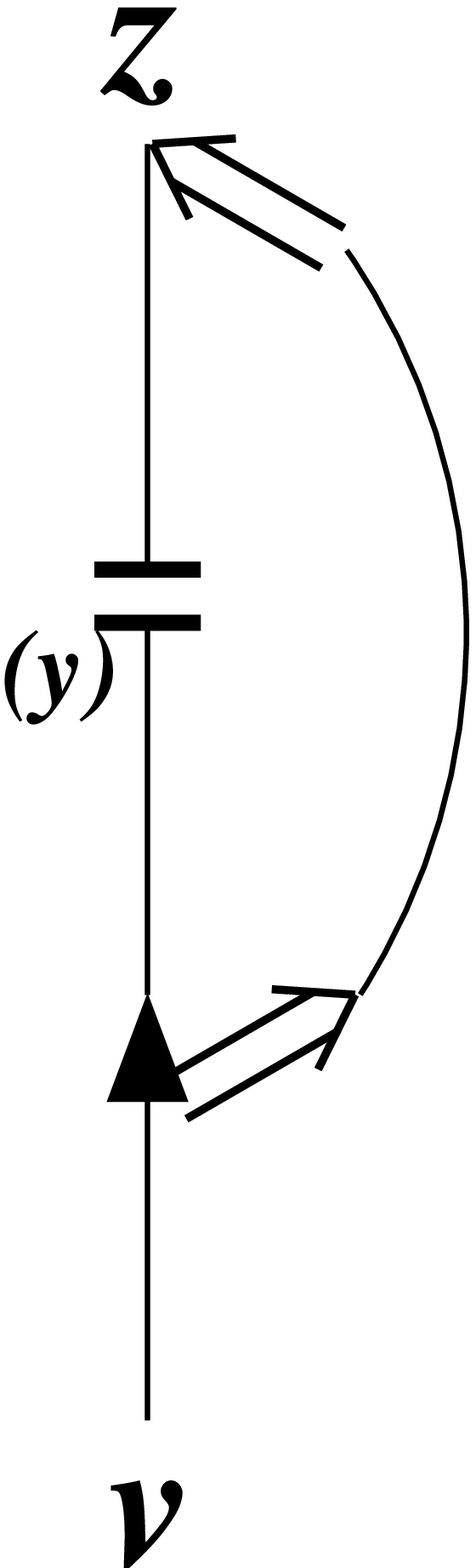}}\quad(+
 \text{ 5 other possibilities})\quad\underset{2_{\zvec}^{\sss(0)}
 (\wvec)}{\longrightarrow}\quad
\raisebox{-4.5pc}{\includegraphics[scale=0.17]{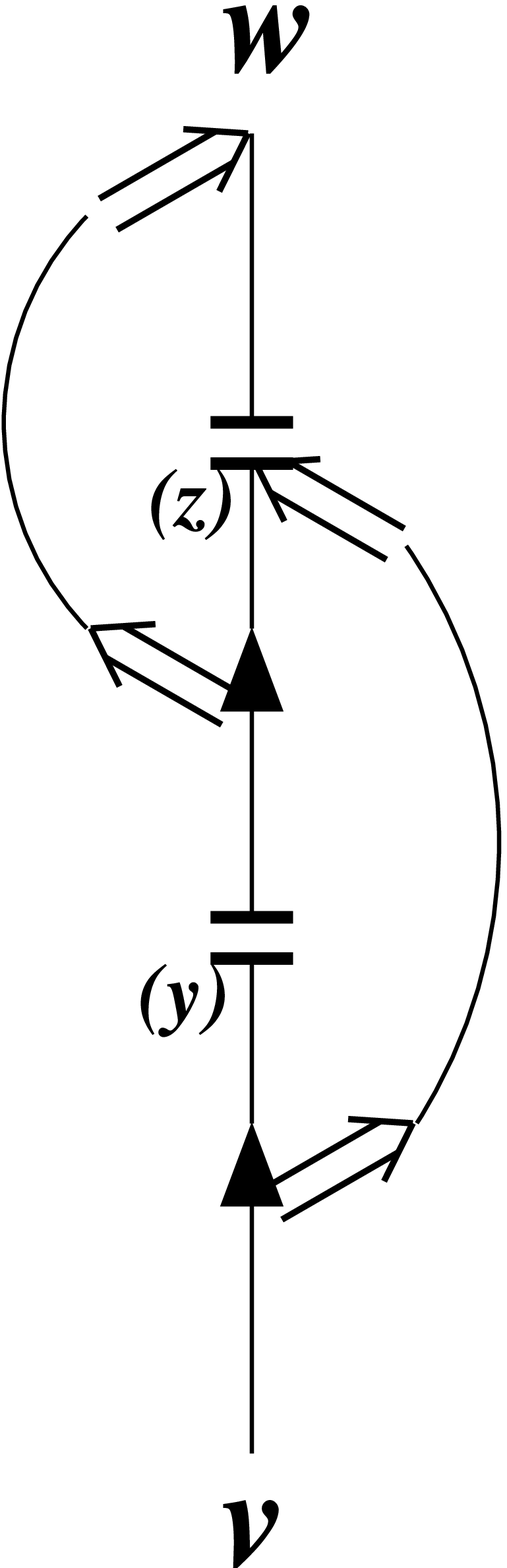}}\quad(+
 \text{ 53 other possibilities})
\end{align*}
\caption{\label{fig-tilFz}Construction~$E_{\yvec}(\wvec)$ in \refeq{tilFz}
applied to $F(\vvec,\yvec)=\tau(\yvec-\vvec)-\delta_{\vvec,\yvec}$.  The
$6~(=4+2)$ possibilities of the result of applying
Construction~$2_{\yvec}^{\sss(1)}(\zvec)$ are due to the fact that
$L(\yvec,\uvec;\zvec)$ for some $\uvec$ consists of 2 terms, and that the
result of Construction~$B^\twop(\uvec)$ consists of $3~(=2+1)$ terms, one of
which is the trivial contribution: $F(\vvec,\yvec)\,\delta_{\yvec,\uvec}$.  The
number of admissible lines in the resulting diagram is 2 for this trivial
contribution, otherwise 1. Therefore, the number of resulting terms at the end
is 54, which is the sum of 6 (due to the identity in \refeq{tilF0}),
$24~(=4\times6$, due to the non-trivial contribution in the first stage
followed by Construction~$2_{\zvec}^{\sss(0)}(\wvec)$) and
$24~(=2\times2\times6$, due to the trivial contribution having 2 admissible
lines followed by Construction~$2_{\zvec}^{\sss(0)}(\wvec)$).}
\end{figure}

\subsubsection{Effects of constructions}\label{sss:constr-eff}
In this section, we summarize the effects of applying the above constructions
to diagrams, i.e., we prove bounds on diagrams obtained by applying
constructions on simpler diagrams in terms of the bounds on those simpler
diagrams.  We also use the following bounds on $\hat\tau_t$ that
were proved in \cite{hsa04}: there is a $K=K(d)$ such that, for $d>4$ with any
$t\ge0$,
\begin{align}\lbeq{fbd}
\hat\tau_t(0)\le K,&&
|\nabla^2\hat\tau_t(0)|\le Kt\sigma^2,&&
\|\wD^2\,\hat{\tau}_t\|_1\le\frac{K\beta}{(1+t)^{d/2}}.
\end{align}
For $d\le4$ with $0\le t\le T\log T$, we replace $\beta$ by
$\beta_{\sT}=L_{\sT}^{-d}$,  and $\sigma$ by $\sigma_{\sT}=O(L_{\sT}^2)$.
Furthermore, by \cite[Lemma 4.5]{hsa04}, we have that, for $q=0,2$ and $d>4$,
\begin{align}\lbeq{tau-ass.1}
\sum_x(\tau_t*D)(x)\le K,&&
\sup_x|x|^q(\tau_t*D)(x)\le\frac{cK\sigma^q\beta}{(1+t)^{(d-q)/2}},
\end{align}
for some $c<\infty$.  Again, for $d\le4$, we replace $\sigma^q\beta$ by
$\sigma_\sT^q\beta_{\sT}$,

%for $d>4$ and
%$t\ge0$,
%\begin{align}\lbeq{fbd}
%\hat\tau_t(0)\le K,&&
%|\nabla^2\hat\tau_t(0)|\le Kt\sigma^2,&&
%\|\wD^2\,\hat{\tau}_t\|_1\le\frac{K\beta}{(1+t)^{d/2}}.
%\end{align}
%For $d\le4$, we restrict the time variable to $t\le T\log T$ and replace
%$\beta$ and $\sigma$ by $\beta_{\sT}=L_{\sT}^{-d}$ ($\equiv\beta_1T^{-bd}$)
%and $\sigma_{\sT}=O(L_{\sT}^2)$, respectively.  Furthermore, by
%\cite[Lemma 4.5]{hsa04}, we have, for $d>4$ and $q=0,2$,
%\begin{align}\lbeq{tau-ass.1}
%\sum_x(\tau_t*D)(x)\le K,&&
%\sup_x|x|^q(\tau_t*D)(x)\le\frac{cK\sigma^q\beta}{(1+t)^{(d-q)/2}},
%\end{align}
%for some $c<\infty$.  The same replacement needed in \refeq{fbd} for $d\le4$
%is also necessary in \refeq{tau-ass.1}.

\begin{lem}[\textbf{Effects of Constructions~$B$ and $\ell$}]\label{lem:constr1}
Let $s\wedge\min_{i\in I}t_i\ge0$, and let $f_{\vec t_I}(\vec x_I)$ be a
diagram function that satisfies $\sum_{\vec x_I}f_{\vec t_I}(\vec x_I)\le
F(\vec t_I)$ by assigning $l_1$ or $l_\infty$ norm to each diagram line and
using \refeq{fbd}--\refeq{tau-ass.1} in order to estimate those norms.  Let
$d>4$.  Then, there exist $C_1,C_2<\infty$ which are independent of $\vep,s$ and
$\vec t_I$ such that, for any line $\twop$ and $q=0,2$,
\begin{align}
\sum_{\vec x_I,y}|y|^qf_{\vec t_I}(\vec x_I;B^\twop(y,s))&\le(N_\eta\sigma^2
 s)^{q/2}(\delta_{s,t_\twop}+\vep C_1)\,F(\vec t_I),\lbeq{constr1-bd-a}\\
\sum_{\vec x_I,y}|y|^qf_{\vec t_I}(\vec x_I;\ell^\twop(y,s))&\le C_2(N_\eta
 \sigma^2s)^{q/2}(1+s\wedge t_\twop)\,F(\vec t_I),\lbeq{constr1-bd-b}
\end{align}
where $N_\eta$ is the number of lines (including $\eta$) contained in the
shortest path of the diagram from $\ovec$ to $\eta$, and $t_\twop$ is the
temporal component of the terminal point of the line $\twop$.  When $d\le4$,
$\sigma$ in \refeq{constr1-bd-a}--\refeq{constr1-bd-b} is replaced by
$\sigma_\sT$.
\end{lem}

\begin{proof}%[Proof of Lemma~\ref{lem:constr1}]
The first inequality \refeq{constr1-bd-a}, where $\delta_{s,t_\twop}$ is due to
the trivial contribution in $B^\twop(y,s)$, is a generalisation of
\cite[Lemma~4.6]{hsa04}, where $\twop$ was an admissible line.  For $q=2$,
in particular, we first bound $|y|^2$ by
$N_\eta\sum_{i=1}^{N_\eta}|y_i-y_{i-1}|^2$, where $(y_0,s_0)\equiv
\ovec,(y_1,s_1),(y_2,s_2),\dots,(y_{N_\eta},s_{N_\eta})\equiv(y,s)$ are the
endpoints of the diagram lines along the (shortest) path from $\ovec$ to
$(y,s)$. Then, we estimate each contribution from $|\Delta
y_i|^2\equiv|y_i-y_{i-1}|^2$ using the bound on
$|\nabla^2\hat\tau_{s_i-s_{i-1}}(0)|$ in \refeq{fbd} or the bound on
$\sup_{\Delta y_i}|\Delta y_i|^2(\tau_{s_i-s_{i-1}}*D)(\Delta y_i)$ in
\refeq{tau-ass.1}.  As a result, we gain an extra factor
$O(s_i-s_{i-1})\sigma^2$ or $O(s_i-s_{i-1})\sigma_\sT^2$ depending on the value
of $d$.  Summing all contributions yields the factor $O(s)\sigma^2$ or
$O(s)\sigma_\sT^2$.  The rest of the proof is similar to that of
\cite[Lemma~4.6]{hsa04}.

To prove the second inequality \refeq{constr1-bd-b}, we note that
\begin{align}\lbeq{constr1-Btau}
\sum_{\vec x_I,y}|y|^qf_{\vec t_I}(\vec x_I;\ell^\twop(y,s))
 \le2^q\ddsum_{r\le s\wedge t_\twop}\sum_{\vec x_I,y,z}(|z|^q+|y-z|^q)\,
 f_{\vec t_I}(\vec x_I;B^\twop(z,r))\,\tau_{s-r}(y-z).
\end{align}
We first perform the sum over $y$ using \refeq{fbd}--\refeq{tau-ass.1} and then
perform the sum over $z$ using \refeq{constr1-bd-a}.  This yields, for $d>4$,
\begin{align}
\sum_{\vec x_I,y}|y|^qf_{\vec t_I}(\vec x_I;\ell^\twop(y,s))&\le K\ddsum_{r
 \le s\wedge t_\twop}\sum_{\vec x_I,z}\big(|z|^q+\sigma^q(s-r)^{q/2}\big)\,
 f_{\vec t_I}(\vec x_I;B^\twop(z,r))\nn\\
&\le KF(\vec t_I)\,\sigma^q\ddsum_{r\le s\wedge t_\twop}\underbrace{\big(
 (N_\eta r)^{q/2}+(s-r)^{q/2}\big)}_{\le\;2(N_\eta s)^{q/2}}(\delta_{r,
 t_\twop}+\vep C_1)\nn\\
&\le2K(N_\eta\sigma^2s)^{q/2}\big(1+C_1(s\wedge t_\twop)\big)\,F(\vec t_I).
\end{align}
For $d\le4$, we only need to replace $\sigma$ in the above computation by
$\sigma_\sT$.  This completes the proof of Lemma~\ref{lem:constr1}.
\end{proof}

\begin{lem}[\textbf{Effects of Constructions~$2^{\sss(1)}$ and $E$}]
 \label{lem:constr1'}
Suppose that $t>0$ for all $d\ge1$ and that $t\le T\log T$ for $d\le4$.  Let
$f(\xvec)$ ($\equiv f_t(x)$ for $\xvec=(x,t)$) be a diagram function such that
$\sum_x|x|^qf_t(x)\le C_f(1+t)^{-(d-q)/2}$ for $q=0,2$, and that $f_t(x)$ has
at most $\cL_f$ lines at any fixed time between 0 and $t$.  There is a
constant $c<\infty$ which does not depend on $f,\cL_f,C_f$ and $t$ such that,
for $d>4$,
\begin{align}\lbeq{lem-constr1'a}
\sum_x|x|^qf(\uvec;2_{\uvec}^{\sss(1)}(x,t))\le\frac{c\cL_fC_f\beta}
 {(1+t)^{(d-q)/2}},
\end{align}
hence
\begin{align}\lbeq{lem-constr1'b}
\sum_x|x|^qf(\uvec;E_{\uvec}(x,t))\le\frac{c\cL_fC_f(1+c\cL_f\beta)\beta}
 {(1+t)^{(d-q)/2}}.
\end{align}
When $d\le4$, $\beta$ in \refeq{lem-constr1'a}--\refeq{lem-constr1'b} is
replaced by $\hat\beta_\sT$.
\end{lem}

\begin{proof}%[Proof of Lemma~\ref{lem:constr1'}]
The idea of the proof is the same as that of \cite[Lemma~4.7]{hsa04}.  Here we
only explain the case of $q=0$; the extension to $q=2$ is proved identically as
the extension to $q=2$ in \cite[Lemma~4.8]{hsa04}.

First we recall the definition \refeq{tilF1}.  Then, we have
    \begin{align}\lbeq{effBl1}
    \sum_xf(\uvec;2_{\uvec}^{\sss(1)}(x,t))\le\ddsum_{\substack{s<t\\ s'\le
     s}}\bigg(\sum_\twop\sum_{u,v}f((u,s);B^\twop(v,s'))\bigg)\bigg(\sup_{u,v}
     \sum_xL((u,s),(v,s');(x,t))\bigg).
    \end{align}
Since $f_s(u)$ has at most $\cL_f$ lines at any fixed time between 0 and $s$,
by Lemma~\ref{lem:constr1}, we obtain
    \begin{align}\lbeq{effBl1-1stblock}
    \sum_\twop\sum_uf((u,s);B^\twop(s'))\le\cL_f\frac{\delta_{s,s'}+\vep
     C_1}{(1+s)^{d/2}}.
    \end{align}
By \refeq{Lne-def} and \refeq{tau-ass.1}, we have that, for $d>4$ and any
$u,v\in\Zd$ and $s,s'\le t$,
\begin{align}\lbeq{supL-bd}
\sum_xL((u,s),(v,s');(x,t))\le\frac{c'\vep^{1+\delta_{(u,s),(v,s')}}\beta}
 {(1+t-s\wedge s')^{d/2}}.
\end{align}
For $d\le4$, $\beta$ is replaced by $\beta_\sT$.
The factor $\vep^{\delta_{(u,s),(v,s')}}$ will be crucial when we introduce the
$0^\text{th}$ order bounding diagram (see, e.g., \refeq{P0vyv} and
\refeq{Pcal0-bd} below).  To bound the convolution \refeq{effBl1}, however, we
simply ignore this factor. Then, the contribution to \refeq{effBl1} from
$\delta_{s,s'}$ in \refeq{effBl1-1stblock} is bounded by $c'\beta$ or
$c'\beta_\sT$ (depending on $d$) multiplied by
    \begin{align}
    \ddsum_{s<t}\frac1{(1+s)^{d/2}}\,\frac{\vep}{(1+t-s)^{d/2}}\le c''\times
     \begin{cases}
     (1+t)^{-d/2}&(d>2),\\
     (1+t)^{-1}\log(2+t)&(d=2),\\
     (1+t)^{1-d}&(d<2).
     \end{cases}
    \end{align}
Similarly, the contribution to \refeq{effBl1} from $\vep C_1$ in
\refeq{effBl1-1stblock} is bounded by $c'\beta$ or $c'\beta_\sT$ multiplied by
(cf., \cite[Lemma~4.7]{hsa04})
    \begin{align}\lbeq{convlem}
    \ddsum_{\substack{s<t\\ s'\le s}}\frac{\vep C_1}{(1+s)^{d/2}}\,
     \frac{\vep}{(1+t-s')^{d/2}}\le c'''\times
     \begin{cases}
     (1+t)^{-d/2}&(d>4),\\
     (1+t)^{-2}\log(2+t)&(d=4),\\
     (1+t)^{2-d}&(d<4).
     \end{cases}
    \end{align}
The above constants $c'',c'''$ are independent of $\vep$ and $t$.  To obtain
the required factor $(1+t)^{-d/2}$ for $d\le4$, we use $t\le T\log T$,
$\beta_\sT\equiv\beta_1T^{-bd}$ and
$\hat\beta_\sT\equiv\hatbetaT$ with $\alphamin<bd-\frac{4-d}2$ as follows:
\begin{align}\lbeq{convappl<4}
\beta_\sT(1+t)^{2-d}\big(\log(2+t)\big)^{\delta_{d,4}}
 =\frac{\beta_\sT(1+t)^{(4-d)/2}(\log(2+t))^{\delta_{d,4}}}{(1+t)^{d/2}}
 \le\frac{O(\hat\beta_\sT)}{(1+t)^{d/2}}.
\end{align}
This completes the proof.
\end{proof}

\subsection{Bound on \protect $B(\xvec)$}\label{ss:Bbd}
In this section, we estimate $B(\xvec)$.  First, in Section~\ref{sss:Bdiagbd},
we prove a $d$-independent diagrammatic bound on
$B^{\sss(N)}(\vvec,\yvec;\bC)$, where we recall
$B^{\sss(N)}(\xvec)=B^{\sss(N)}(\ovec,\xvec;\{\ovec\})$ (cf.,
\refeq{ABzerodef}).  Then, in Section~\ref{sec-BNbd}, we prove the bounds on
$B^{\sss(N)}(\xvec)$: \refeq{Bbd} for $d>4$ and \refeq{Bbd<4} for $d\le4$.

\subsubsection{Diagrammatic bound on \protect $B^{\sss(N)}(\vvec,\yvec;\bC)$}
\label{sss:Bdiagbd} First we define bounding diagrams for
$B^{\sss(N)}(\vvec,\yvec;\bC)$.  For $\vvec,\wvec,\cvec\in\Lambda$, we
let
    \begin{align}\lbeq{S00-def}
    S^{\sss(0,0)}(\vvec,\wvec;\cvec)&=\delta_{\wvec,\cvec}\times
    \begin{cases}
    \delta_{\vvec,\wvec}&(t_{\vvec}=t_{\wvec}),\\[2pt]
    (\tau\sstar\lamb\vep D)(\wvec-\vvec)&(t_{\vvec}<t_{\wvec}),
    \end{cases}\\
    \lbeq{S01-def}
    S^{\sss(0,1)}(\vvec,\wvec;\cvec)&=(1-\delta_{\wvec,\cvec})\times
    \begin{cases}
    0&(t_{\vvec}=t_{\wvec}),\\[2pt]
    \tau(\wvec-\vvec)&(t_{\vvec}<t_{\wvec}),
    \end{cases}
    \end{align}
and
    \begin{align}\lbeq{S0-def}
    S^{\sss(0)}(\vvec,\wvec;\cvec)=S^{\sss(0,0)}(\vvec,\wvec;\cvec)+
    S^{\sss(0,1)}(\vvec,\wvec;\cvec)\,\lamb\vep D(\wvec-\cvec).
    \end{align}
For $\vvec,\wvec\in\Lambda$ and $\bC\subseteq \Lambda$, we define
$\wvec_-=(w,t_{\wvec}-\vep)$ and
    \begin{align}\lbeq{S0C-def}
    S^{\sss(0)}(\vvec,\wvec;\bC)=\sum_{\cvec\in\bC}\Big(
    S^{\sss(0,0)}(\vvec,\wvec;\cvec)+
    S^{\sss(0,1)}(\vvec,\wvec;\cvec)
    \ind{(\cvec,\wvec)\in\bC}(1-\delta_{\cvec,\wvec_-})
    \Big),
    \end{align}
where $(\cvec,\wvec)\in\bC$ precisely when the bond $(\cvec,\wvec)$
is a part of $\bC$.  We now comment on this issue in more detail.

Note that $\bC\subseteq \Lambda$ appearing in $B^{\sss(N)}(\vvec,\yvec;\bC)$
is a set of sites.  However, we will only need bounds on
$B^{\sss(N)}(\vvec,\yvec;\bC)$ for $\bC=\tilde\bC_{\sN}$ for some $N$.  As a
result, the set $\bC$ of sites here have a special structure, which we will
conveniently make use of.  That is, in the sequel, we will consider $\bC$ to
consist of sites and bonds simultaneously, as in Remark~\ref{rem-clusbonds}
in the beginning of Section~\ref{s:lace}, and call $\bC$ a
\emph{cluster-realization} when
\begin{align}\lbeq{clusterreal-def}
\Pbold\big(\bC(\cvec)=\bC\big)>0
\end{align}
for some $\cvec\in\Lambda$.

The diagram $S^{\sss(0)}(\vvec,\wvec;\bC)$ is closely related to the diagram
$\sum_{\cvec\in \bC}S^{\sss(0)}(\vvec,\wvec;\cvec)$, apart from the fact that
$S^{\sss(0,1)}(\vvec,\wvec;\cvec)$ is multiplied by $\lamb\vep
D(\wvec-\cvec)$ in \refeq{S0-def} and by $\ind{(\cvec,\wvec)\subseteq
\bC}(1-\delta_{\cvec,\wvec_-})$ in \refeq{S0C-def}. In all our applications,
the role of $\bC$ is played by a $\tilde \bC_{\sss N}$-cluster, and, in such
cases, since $(\cvec, \wvec)$ is a spatial bond, $\lamb\vep D(\wvec-\cvec)$ is
the probability that the bond $(\cvec,\wvec)$ is occupied.  This factor $\vep$ is
crucial in our bounds.

Furthermore, we define
\begin{align}\lbeq{P0cdef}
P^{\sss(0)}(\vvec,\yvec;\cvec)=S^{\sss(0)}(\vvec,\wvec;\cvec,2^{\sss(0)}
 _{\wvec}(\yvec))=\quad\raisebox{-3pc}{\includegraphics[scale=0.14]{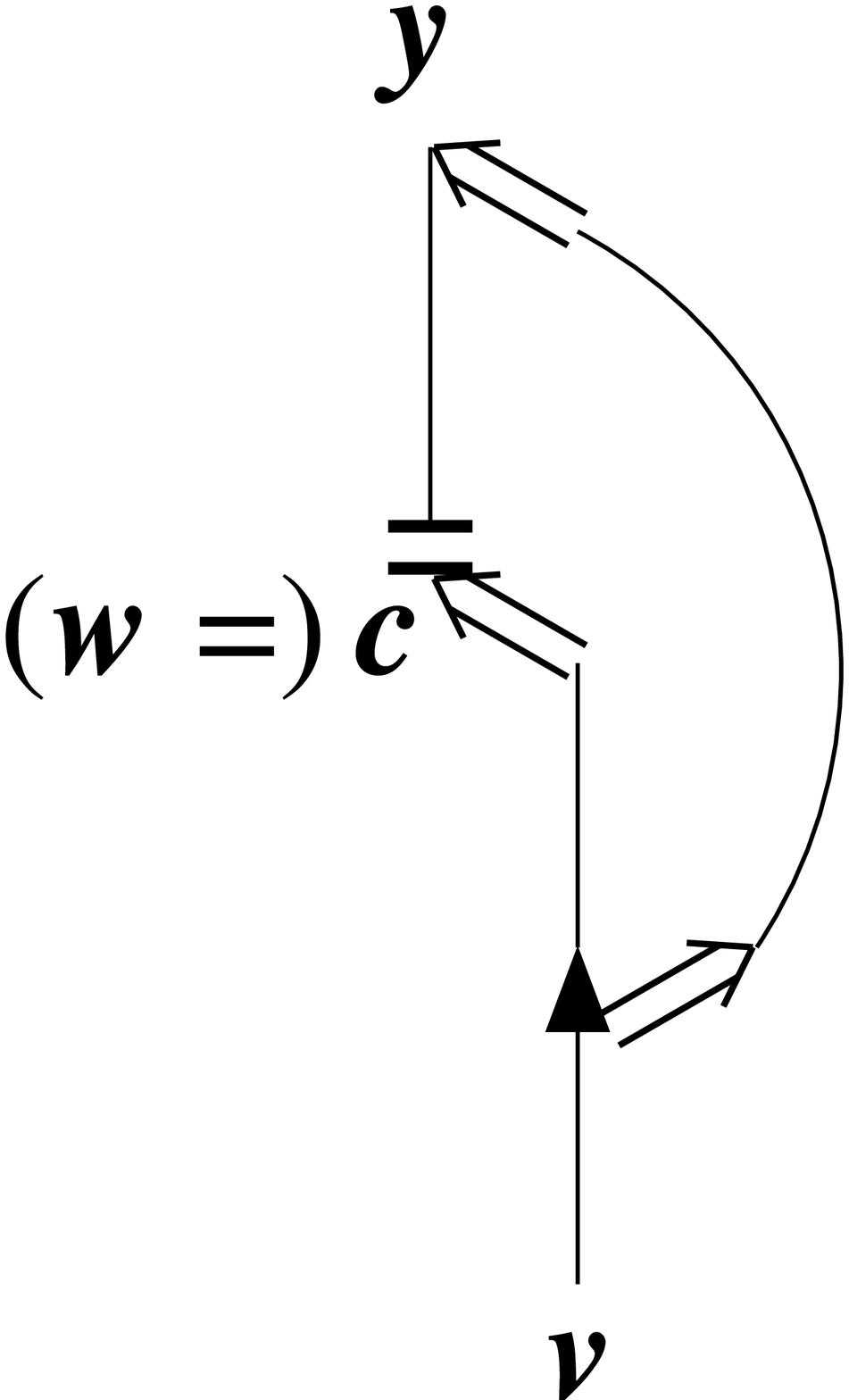}}
 \quad+\quad\raisebox{-3pc}{\includegraphics[scale=0.14]{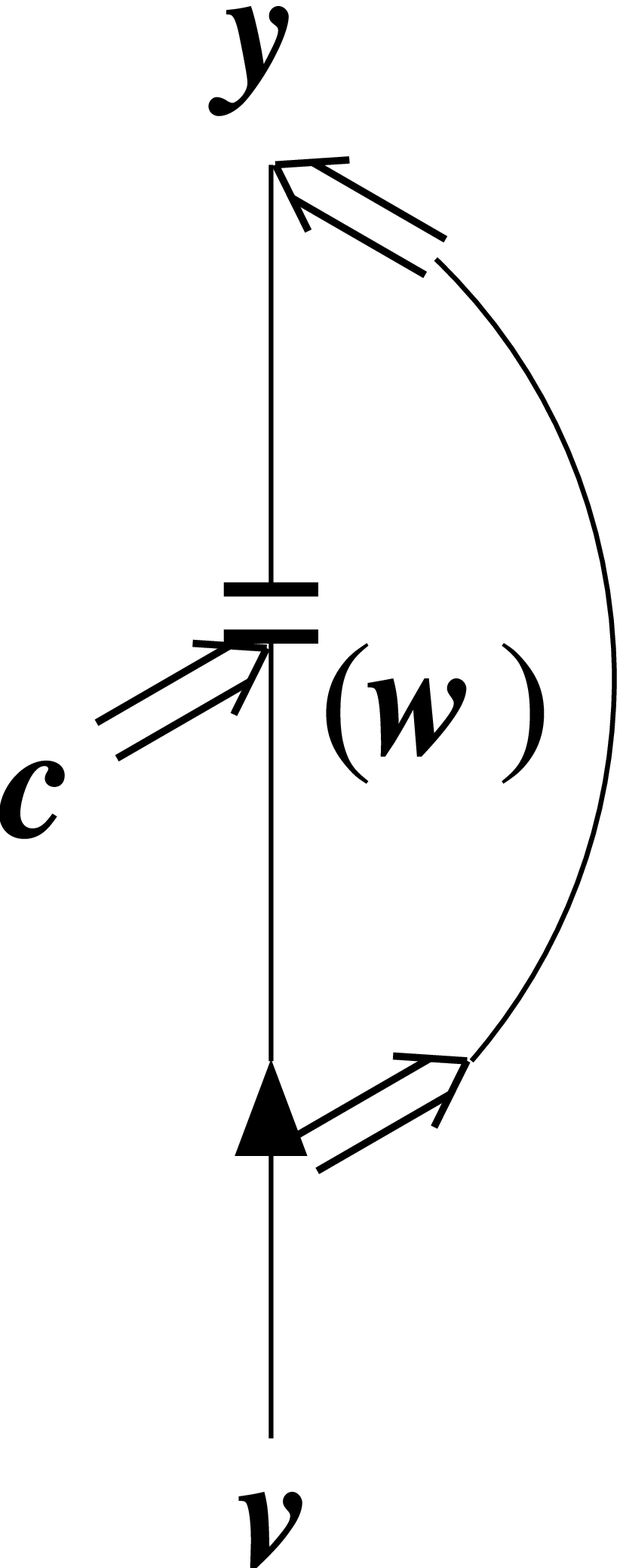}}\quad(+
 \text{ 12 other possibilities})
\end{align}
and
    \begin{align}
    \lbeq{P0Cdef}
    P^{\sss(0)}(\vvec,\yvec;\bC)=S^{\sss(0)}(\vvec,\wvec;\bC,2^{\sss
     (0)}_{\wvec}(\yvec)),
    \end{align}
where the admissible lines for the application of
Construction~$2^{\sss(0)}_{\wvec}(\yvec)$ in \refeq{P0cdef}--\refeq{P0Cdef}
are $(\tau\sstar\lamb\vep D)(\wvec-\vvec)$ and $\tau(\wvec-\vvec)$ in the
second lines of \refeq{S00-def}--\refeq{S01-def}.  If $\cvec=\vvec$, then,
by the first line of \refeq{S00-def} and recalling \refeq{tilF0} and the
definition of Construction~$B$ applied to a ``line'' of length zero (see
below \refeq{btemp}), we have
\begin{align}\lbeq{P0vyv}
P^{\sss(0)}(\vvec,\yvec;\vvec)=\delta_{\vvec,\yvec}+L(\vvec,\vvec;\yvec).
\end{align}
We further define the diagram $P^{\sss(N)}(\vvec,\yvec;\cvec)$ (resp.,
$P^{\sss(N)}(\vvec,\yvec;\bC)$) by $N$ applications of Construction~$E$ to
$P^{\sss(0)}(\vvec,\yvec;\cvec)$ in \refeq{P0cdef} (resp.,
$P^{\sss(0)}(\vvec,\yvec;\bC)$ in \refeq{P0Cdef}).
We call the $E$-admissible lines, arising in the final Construction~$E$, the
\emph{$N^\text{th}$ admissible lines}.

We note that, by \refeq{Lne-def} and this notation, it is not hard to see that
    \begin{align}
    L(\yvec,\uvec;\zvec,2^{\sss(0)}_{\zvec}(\wvec))=\sum_{\zvec}\sum_{b:
     \bb=\yvec}\tau(\zvec-\uvec)\,p_bP^{\sss(0)}(\tb,\wvec;\zvec).
    \end{align}
Therefore, an equivalent way of writing \refeq{tilFz} is
    \begin{align}\lbeq{tilFzrep}
    F(\vvec,\yvec;2_{\yvec}^{\sss(1)}(\zvec),2_{\zvec}^{\sss(0)}(\wvec))
    &=\sum_\twop\sum_{\uvec,\yvec}F(\vvec,\yvec;B^\twop(\uvec))\,
     L(\yvec,\uvec;\zvec,2_{\zvec}^{\sss(0)}(\wvec))\nn\\
    &=\sum_\twop\sum_{\zvec}\sum_bF(\vvec,\bb;\ell^\twop(\zvec))\,p_b
     P^{\sss(0)}(\tb,\wvec;\zvec),
    \end{align}
where $\sum_\twop$ is the sum over the admissible lines for $F(\vvec,\yvec)$.
In particular, we obtain the recursion
\begin{align}\lbeq{PArecrep}
P^{\sss(N)}(\vvec,\wvec;\bC)&\equiv P^{\sss(N-1)}(\vvec,\yvec;\bC,
 2_{\yvec}^{\sss(1)}(\zvec),2_{\zvec}^{\sss(0)}(\wvec))\nn\\
&=\sum_\twop\sum_{\zvec}\sum_bP^{\sss(N-1)}(\vvec,\bb;\bC,\ell^\twop(\zvec))\,
 p_bP^{\sss(0)}(\tb,\wvec;\zvec),
\end{align}
where $\sum_\twop$ is the sum over $(N-1)^\text{th}$~admissible lines.

The following lemma states that the diagrams constructed above indeed
bound $B^{\sss(N)}(\vvec,\yvec;\bC)$:

\begin{lem}\label{lem-BNbd}
For $N\ge0$, $\vvec,\yvec\in\Lambda$, and a cluster-realization
$\bC\subset\Lambda$ with $\min_{\cvec\in \bC} t_{\cvec}<t_{\vvec}$,
    \begin{align}\lbeq{piPbd}
    B^{\sss(N)}(\vvec,\yvec;\bC)\le\sum_{b:\tb=\yvec}P^{\sss(N)}
     (\vvec,\bb;\bC)\,p_b.
    \end{align}
\end{lem}

\begin{proof}
A similar bound was proved in \cite[Proposition~6.3]{HHS05b}, and we follow its
proof as closely as possible, paying attention to the powers of $\vep$.

We prove by induction on $N$ the following two statements:
\begin{align}
M_{\vvec,\yvec;\bC}^{\sss(N+1)}(1)&\le P^{\sss(N)}(\vvec,\yvec;
 \bC),\lbeq{MP1}\\[5pt]
M_{\vvec,\yvec;\bC}^{\sss(N+1)}\big(\ind{\wvec\in\bC_N}\big)&\le
 \sum_\twop P^{\sss(N)}(\vvec,\yvec;\bC,\ell^\twop(\wvec)),\lbeq{MPI}
\end{align}
where $\sum_\twop$ is the sum over the $N^\text{th}$~admissible lines.  The
first inequality together with \refeq{ANBNdef} immediately imply \refeq{piPbd}.

To verify \refeq{MP1} for $N=0$, we first prove
\begin{align}\lbeq{E'bd}
E'(\vvec,\yvec;\bC)&\subseteq\Ecal(\vvec,\yvec;\bC)\nn\\
&\equiv\bigcup_{\substack{\cvec,\wvec\in\bC\\ \uvec\in\Lambda}}\bigg\{\Big\{
 \{\vvec\conn\uvec\}\circ\{\uvec\conn\wvec\}\circ\{\wvec\conn\yvec\}\circ\{
 \uvec\conn\yvec\}\Big\}\nn\\
&\hskip4pc\cap\Big\{\{\cvec=\wvec,~\uvec\centernot\conn\wvec_-\}\cup\{\cvec\ne
 \wvec_-,~(\cvec,\wvec)\in\bC\}\Big\}\bigg\},
\end{align}
where $E\circ F$ denotes disjoint
occurrence of the events $E$ and $F$.  It is immediate that (see, e.g.,
\cite[(6.12)]{HHS05b})
\begin{align}\lbeq{E'bdprep}
E'(\vvec,\yvec;\bC)\subseteq\bigcup_{\substack{\cvec\in\bC\\ \uvec\in\Lambda}}
 \Big\{\{\vvec\conn\uvec\}\circ\{\uvec\conn\cvec\}\circ\{\cvec\conn\yvec\}\circ
 \{\uvec\conn\yvec\}\Big\}.
\end{align}
However, when $\vep\ll 1$, the above bound is not good enough, since it does
not produce sufficiently many factors of $\vep$. Therefore, we now improve the
inclusion.  We denote by $\wvec$ the element in $\bC$ with the smallest time
index such that $\vvec\conn \wvec$.  Such an element must exist, since
$E'(\vvec,\yvec;\bC)\subset\{\vvec\ct{\bC}\yvec\}$.  Then, there are two
possibilities, namely, that $\vvec$ is not connected to
$\wvec_-\equiv(w,t_{\wvec}-\vep)$, or that $\wvec_-\nin\bC$.  In the latter
case, since $\bC$ is a cluster-realization with
$\min_{\cvec\in\bC}t_{\cvec}<t_{\vvec}$, there must be a vertex $\cvec\in\bC$
such that the spatial bond $(\cvec,\wvec)$ is a part of $\bC$.  Together with
\refeq{E'bdprep}, it is not hard to see that \refeq{E'bd} holds.

Recall that a spatial bond $b$ has probability $\lamb\vep D(b)$ of being
occupied.  We note that, since $\{\uvec\conn\wvec\}\circ\{\uvec\conn\yvec\}$
occurs, and when $\wvec\ne\uvec$ and $\yvec\ne\uvec$, there must be at least
one spatial bond $b$ with $\bb=\uvec$, such that either $b\conn\wvec$ or
$b\conn\yvec$. Therefore, this produces a factor $\vep$.  Also, when
$\wvec\ne\yvec$ and $\uvec\ne\yvec$, then the disjoint connections in
$\{\wvec\conn\yvec\}\circ\{\uvec\conn\yvec\}$ produce a spatial bond pointing
at $\yvec$.  Taking all of the different possibilities into account, and using
the BK inequality (see, e.g., \cite{g99}), we see that
    \begin{align}\lbeq{pi0bd}
    M_{\vvec,\yvec;\bC}^{\sss(1)}(1)=\mE\big[\indic_{E'(\vvec,\yvec;\bC)}
     \big]\le\Pbold(\Ecal(\vvec,\yvec;\bC))\le P^{\sss(0)}(\vvec,\yvec;
     \bC),
    \end{align}
which is \refeq{MP1} for $N=0$.

To verify \refeq{MPI} for $N=0$, we use the fact that
    \begin{gather}
    M_{\vvec,\yvec;\bC}^{\sss(1)}(\ind{\wvec\in\bC_{\sss0}})=\mE\big[
     \indic_{E'(\vvec,\yvec;\bC)}\ind{\vvec\conn\wvec}\big]\le\mP\big(
     \Ecal(\vvec,\yvec;\bC)\cap\{\vvec\conn\wvec\}\big)\nn\\[7pt]
    \le\sum_\twop P^{\sss(0)}(\vvec,\yvec;\bC,\ell^{\twop}(\wvec))
     \equiv\sum_\twop\sum_{\zvec}P^{\sss(0)}(\vvec,\yvec;\bC,B^{\twop}
     (\zvec))\,\tau(\wvec-\zvec),\lbeq{2pt.100}
    \end{gather}
where $\sum_\twop$ is the sum over the $0^\text{th}$~admissible lines. Indeed,
to relate \refeq{2pt.100} to \refeq{pi0bd}, fix a backward occupied path from
$\wvec$ to $\vvec$.  We note that this must share some part with the occupied
paths from $\vvec$ to $\yvec$. Let $\uvec$ be the vertex with highest time
index of this common part. Then, there must be a spatial bond $b$ with
$\bb=\uvec$.  Recall that the result of Construction~$B^\twop(\zvec)$ is the
sum of $\sum_\twop P^{\sss(0)}(\vvec,\yvec;\bC)\,\tau(\wvec-\yvec)$ and the
results of Construction~$\bspat^\twop(\zvec)$ and
Construction~$\btemp^\twop(\zvec)$. We also recall
\refeq{bspat}--\refeq{btemp}.  Therefore, $\zvec$ in \refeq{2pt.100} is
$\bb=\uvec$ in the contribution due to Construction~$\bspat^\twop(\zvec)$, and
$\zvec=\tb$ in the contribution from Construction~$\btemp^\twop(\zvec)$.  This
completes the proof of \refeq{MPI} for $N=0$.

To advance the induction, we fix $N \ge 1$ and assume that
\refeq{MP1}--\refeq{MPI} hold for $N-1$.  By \refeq{M-def}, \refeq{pi0bd},
\refeq{P0Cdef} and \refeq{S0C-def}, we have
\begin{align}\lbeq{MP1pfafirst}
M_{\vvec,\yvec;\bC}^{\sss(N+1)}(1)&=\sum_bp_bM_{\vvec,\bb;\bC}^{\sss(N)}
 \Big(M_{\tb,\yvec;\tilde\bC_{\sN-1}}^{\sss(1)}(1)\Big)\nn\\
&\le\sum_bp_bM_{\vvec,\bb;\bC}^{\sss(N)}\Big(P^{\sss(0)}(\tb,\yvec;\tilde
 \bC_{\sN-1})\Big)\nn\\
&=\sum_bp_b\sum_{\cvec}\bigg(M_{\vvec,\bb;\bC}^{\sss(N)}\big(\ind{\cvec\in
 \tilde\bC_{N-1}}\big)\,S^{\sss(0,0)}(\tb,\wvec;\cvec,2_{\wvec}^{\sss(0)})\nn\\
&\hskip5pc+M_{\vvec,\bb;\bC}^{\sss(N)}\big(\ind{(\cvec,\wvec)\in\tilde\bC_{\sN
 -1}}\big)\,(1-\delta_{\cvec,\wvec_-})\,S^{\sss(0,1)}(\tb,\wvec;\cvec,
 2_{\wvec}^{\sss(0)})\bigg).
\end{align}
%The term $S^{\sss(0,0)}(\vvec,\wvec;\cvec)$
%is multiplied by an indicator $\ind{\cvec\in\tilde\bC_{\sss N-1}}$, which allows
%to use the induction hypothesis.  The term $S^{\sss(0,1)}(\vvec,\wvec;\cvec)$ is
%multiplied by an indicator
%$\ind{(\cvec,\wvec)\in\tilde\bC_{\sN-1}}(1-\delta_{\cvec,\wvec_-})$ for some
%$\wvec$ with $t_{\wvec}\ge t_{\tb}$, and we are lead to investigate
%$M^{\sss(N)}_{\vvec,\bb;\bC}\big(\ind{(\cvec,\wvec)\in\tilde\bC_{\sN-1}}\big)
%(1-\delta_{\cvec,\wvec_-})$, which is closely
%related, but not exactly equal, to the term bounded in the induction
%hypothesis.  However, since
Note that $t_{\cvec}\ge t_{\bb}$.  By the Markov property, we obtain
\begin{align}\lbeq{MN-indbond}
M^{\sss(N)}_{\vvec,\bb;\bC}\big(\ind{(\cvec,\wvec)\in\tilde\bC_{\sN-1}}\big)
 \,(1-\delta_{\cvec,\wvec_-})=M^{\sss(N)}_{\vvec,\bb;\bC}\big(\ind{\cvec\in
 \tilde\bC_{\sN-1}}\big)\,\lambda\vep D(\wvec-\cvec).
\end{align}
Substitution of \refeq{MN-indbond} into \refeq{MP1pfafirst} and using
\refeq{S0-def} and \refeq{P0cdef}, we arrive at
\begin{align}\lbeq{MP1pfa}
M_{\vvec,\yvec;\bC}^{\sss(N+1)}(1)\le\sum_b\sum_{\cvec}M_{\vvec,\bb;\bC}^{\sss
 (N)}\big(\ind{\cvec\in\tilde\bC_{\sN-1}}\big)p_bP^{\sss(0)}(\tb,\yvec;\cvec).
\end{align}
We apply the induction hypothesis to bound
$M_{\vvec,\bb;\bC}^{\sss(N)}(\ind{\cvec\in\bC_{\sN-1}})$
(\,$\ge M_{\vvec,\bb;\bC}^{\sss(N)}(\ind{\cvec\in\tilde\bC_{\sN-1}})$\,)
and then use \refeq{PArecrep} to conclude \refeq{MP1}.

Similarly, for \refeq{MPI}, we have
    \eqalign\lbeq{MPIa}
    M^{\sss(N+1)}_{\vvec,\yvec;\bC}\big(\ind{\wvec\in\bC_N}\big)
    =\sum_bp_bM^{\sss(N)}_{\vvec,\bb;\bC}\Big(M^{\sss(1)}_{\tb,
    \yvec;\tilde\bC_{N-1}}\big(\ind{\wvec\in\bC_N}\big)\Big),
    \enalign
and substitution of the bound \refeq{MPI} for $N=0$ yields
    \eqalign\lbeq{MPIa1}
    M^{\sss(N+1)}_{\vvec,\yvec;\bC}\big(\ind{\wvec\in\bC_N}\big)
    &\le\sum_bp_b M^{\sss(N)}_{\vvec,\bb;\bC}\bigg(
    \sum_\twop P^{\sss(0)}\big(\tb,
    \yvec;\tilde\bC_{\sN-1},\ell^{\twop}(\wvec)\big)\bigg),
    \enalign
where $\sum_\twop$ is the sum over the admissible lines for
$P^{\sss(0)}(\tb,\yvec;\tilde\bC_{\sN-1})$.  The argument in
\refeq{MP1pfafirst}--\refeq{MP1pfa} then proves that
\begin{align}\lbeq{MPIa2}
M^{\sss(N+1)}_{\vvec,\yvec;\bC}\big(\ind{\wvec\in\bC_N}\big)
 &\le\sum_bp_b\sum_{\cvec}M^{\sss(N)}_{\vvec,\bb;\bC}\big(
 \ind{\cvec\in\tilde\bC_{N-1}}\big)\sum_\twop P^{\sss(0)}(\tb,
 \yvec;\cvec,\ell^{\twop}(\wvec)).
\end{align}
%where $\sum_\twop$ is the sum over the admissible lines for
%$P^{\sss(0)}(\tb,\yvec;\cvec)$.
We use the induction hypothesis \refeq{MPI} to bound
$M^{\sss(N)}_{\vvec,\bb;\bC}(\ind{\cvec\in\bC_{N-1}})$ $(\,\ge
M^{\sss(N)}_{\vvec,\bb;\bC}(\ind{\cvec\in\tilde\bC_{N-1}})\,)$,
as well as the fact that (cf., \refeq{PArecrep})
\begin{align}\lbeq{PArecrep-gen}
P^{\sss(N)}\big(\vvec,\yvec;\bC,\ell^{\twop}(\wvec)\big)=\sum_{\twop'}
 \sum_{\zvec}\sum_bP^{\sss(N-1)}\big(\vvec,\bb;\bC,\ell^{\twop'}(\zvec)\big)
 \,p_bP^{\sss(0)}\big(\tb,\yvec;\zvec,\ell^{\twop}(\wvec)\big),
\end{align}
where $\sum_{\twop'}$ is the sum over the $(N-1)^\text{th}$~admissible lines.
This leads to
\begin{align}
M^{\sss(N+1)}_{\vvec,\yvec;\bC}\big(\ind{\wvec\in\bC_N}\big)
 &\le\sum_\twop P^{\sss(N)}(\vvec,\yvec;\bC,\ell^{\twop}(\wvec)).
\end{align}
%where $\sum_\twop$ is the sum over the $N^\text{th}$~admissible lines.
This completes the advancement of \refeq{MPI}.
\end{proof}

We close this section by listing a few related results that will be used later
on.  First, it is not hard to see that \refeq{MPI} can be generalised to
    \eq
    \lbeq{MPIgen}
    M_{\vvec,\yvec;\bC}^\smallsup{N+1}(\ind{\vec\xvec \in \bC_{\sss N}})
    \le P^{\smallsup{N}}(\vvec,\yvec;\bC, \ell(\vec\xvec)).
    \en
Next, we let
\begin{align}\lbeq{P0zerodef}
P^\smallsup{N}(\xvec)=P^{\sss(N)}(\ovec,\xvec;\ovec),
\end{align}
By \refeq{ABzerodef} and Lemma~\ref{lem-BNbd}, we have
    \eq
    \lbeq{BNPNbd}
    B^{\sss(N)}(\xvec)\le\sum_{b:\tb=\xvec}P^{\sss(N)}(\bb)\,p_b.
    \en
We will use the recursion formula (cf., \refeq{PArecrep})
    \eq\lbeq{recPNM}
    P^{\sss(N+M)}(\xvec)=\sum_\twop\sum_{\avec}\sum_bP^{\sss(N)}(\bb;
    \ell^\twop(\avec))\,p_bP^{\sss(M-1)}(\tb,\xvec;\avec),
    \en
where $\sum_\twop$ is the sum over the $N^\text{th}$~admissible lines.  This
can easily be checked by induction on $M$ (see also
\cite[(6.21)--(6.24)]{HHS05b}).

We will also make use of the following lemma, which generalises
\refeq{recPNM} to cases where more constructions are applied:

\begin{lemma}\label{lem-recPNMthetas}
For every $N,M\ge 0$,
    \eq\lbeq{recPNMthetas}
    \sum_\twop\sum_{\avec}\sum_bP^{\sss(N)}(\bb;\ell^{\twop}(\avec),
    \ell(\vec\xvec))\,p_bP^{\sss(M)}(\tb,\yvec;\avec,\ell(\vec\zvec))
    \le P^{\sss(N+M+1)}(\yvec;\ell(\vec\xvec),\ell(\vec\zvec)),
    \en
where $\sum_\twop$ is the sum over the $N^\text{th}$~admissible lines for
$P^{\sss(N)}(\bb)$.  Recall that Construction~$\ell(\vec\xvec)$ for
$\vec\xvec=(\xvec_1,\dots,\xvec_j)$ is the repeated application of
Construction~$\ell^{\twop_i}(\xvec_i)$ for $i=1,\dots,j$, followed by sums over
\emph{all} possible lines $\twop_i$ for $i=1,\dots,j$.
\end{lemma}

\begin{proof}
The above inequality is similar to \refeq{recPNM}, but now with two extra
construction performed to the arising diagrams.  The equality in
\refeq{recPNM} is replaced by an upper bound in \refeq{recPNMthetas},
since on the right-hand side there are more possibilities for the lines on
which the Constructions~$\ell(\vec\xvec)$ and $\ell(\vec \zvec)$
can be performed.
\end{proof}

\subsubsection{Proof of the bound on \protect $B^{\sss(N)}(\xvec)$}\label{sec-BNbd}
We now specialise to $\vvec=\ovec$ and $\bC=\{\ovec\}$, for which we recall
\refeq{ABzerodef} and \refeq{P0zerodef}--\refeq{BNPNbd}. The main result in
this section is the following bound on $P^{\sss(N)}_t(x)\equiv
P^{\sss(N)}((x,t))$, from which, together with Lemma~\ref{lem-BNbd}, the
inequalities \refeq{Bbd} and \refeq{Bbd<4} easily follow.

\begin{lem}[\textbf{Bounds on $P_t^{\sss(N)}$}]\label{lem-Ptbds}
\begin{enumerate}[(i)]
\item
Let $d>4$ and $L\gg1$.  For $\lamb\le\lambc^{\sss(\vep)}$, $N\ge0$,
$t\in\vep\Zp$ and $q=0,2$,
\begin{align}\lbeq{PN-bd}
\sum_x|x|^q\,P_t^{\sss(N)}(x)&\le\delta_{q,0}\delta_{t,0}\delta_{N,0}+
 \vep^2\frac{O(\beta)^{1\vee N}\sigma^q}{(1+t)^{(d-q)/2}},
\end{align}
where the constant in the $O(\beta)$ term is independent of $\vep,L,N$ and $t$.
\item
Let $d\le4$ with $bd-\frac{4-d}2>0$,
$\hat\beta_\sT=\hatbetaT$ and $L_1\gg1$.  For
$\lamb\le\lambc^{\sss(\vep)}$, $N\ge0$, $t\in\vep\Zp\cap[0,T\log T]$ and $q=0,2$,
\begin{align}\lbeq{PN-bd-lower}
\sum_x|x|^q\,P_t^{\sss(N)}(x)&\le\delta_{q,0}\delta_{t,0}\delta_{N,0}+
 \vep^2\frac{O(\beta_{\sT})\,O(\hat\beta_{\sT})^{0\vee(N-1)}\sigma_{\sT
 }^q}{(1+t)^{(d-q)/2}},
\end{align}
where the constants in the $O(\beta_{\sT})$ and $O(\hat\beta_{\sT})$ terms are
independent of $\vep,T,N$ and $t$.
\end{enumerate}
\end{lem}

\begin{proof}
Let
\begin{align}\lbeq{PcalNdef}
\Pcal^{\sss(0)}(\xvec)=P^{\sss(0)}(\xvec),&&
\Pcal^{\sss(N)}(\xvec)=\Pcal^{\sss(N-1)}(\uvec;2_{\uvec}^{\sss(1)}
 (\xvec))\qquad(N\ge1).
\end{align}
We note from \cite[Lemma~4.4]{hsa04} that the inequalities
\refeq{PN-bd}--\refeq{PN-bd-lower} were shown for a similar quantity
to $\Pcal^{\sss(N)}(\xvec)$, where $L(\uvec,\vvec;\xvec)$ in
\cite[(4.18)]{hsa04} was not our $L(\uvec,\vvec;\xvec)$ in
\refeq{Lne-def} (compare \refeq{supL-bd} with
\cite[(4.42)]{hsa04}). The main differences between $L(\uvec,\vvec;\xvec)$ in
\cite[(4.18)]{hsa04} and $L(\uvec,\vvec;\xvec)$ in
\refeq{Lne-def} for $\uvec\neq\vvec$ is that
$\varphi$ in \refeq{vhi-def} has a term $\delta_{\uvec,\xvec}$ less than
the one in \cite[(4.17)]{hsa04}, and, for $\uvec=\vvec$,
our $L(\uvec,\vvec;\xvec)$ has a factor $\lambda\vep$ more than
the one in \cite[(4.18)]{hsa04}.

The proof of \cite[Lemma~4.4]{hsa04} was based on the recursion relation
\cite[(4.24)]{hsa04} that is equivalent to \refeq{PcalNdef}.  Since
$\lamb\le\lambc^{\sss(\vep)}\le2$ when $L$ is sufficiently large, our
$L(\uvec,\vvec;\xvec)$ in \refeq{Lne-def} is smaller than twice
$L(\uvec,\vvec;\xvec)$ in \cite[(4.18)]{hsa04}, so that \cite[Lemma 4.4]{hsa04}
also applies to $\Pcal^{\sss(N)}(\xvec)$.  For $N=0$ with $d>4$, we have (cf.,
\refeq{P0vyv})
\begin{align}\lbeq{Pcal0-bd}
\sum_x\Pcal^{\sss(0)}(x,t)\equiv\sum_xP^{\sss(0)}(x,t)=\sum_x\Big(
 \delta_{x,o}\delta_{t,0}+L((o,0),(o,0);(x,t))\Big)\le\delta_{t,0}
 +\vep^2\frac{O(\beta)}{(1+t)^{d/2}}.
\end{align}
The factor $O(\beta)$ is replaced by $O(\beta_\sT)$ if $d\le4$. For $N\ge1$, we
apply Lemma~\ref{lem:constr1'} to \refeq{Pcal0-bd} $N$ times.

We now relate $\Pcal^{\sss(N)}(\xvec)$ with $P^{\sss(N)}(\xvec)$.
Note that, by \refeq{tilF0}--\refeq{tilFz}, we have
\begin{align}\lbeq{PN2rec}
P^{\sss(N)}(\xvec)=P^{\sss(N-1)}(\uvec;2_{\uvec}^{\sss(1)}(\wvec),
 2_{\wvec}^{\sss(0)}(\xvec))=P^{\sss(N-1)}(\uvec;2_{\uvec}^{\sss
 (1)}(\xvec))+P^{\sss(N-1)}(\uvec;2_{\uvec}^{\sss(1)}(\wvec),
 2_{\wvec}^{\sss(1)}(\xvec)).
\end{align}
It follows by \refeq{PcalNdef} and \refeq{PN2rec} that
\begin{align}\lbeq{PPcalrel}
P^{\sss(N)}(\xvec)=\sum_{M=0}^N\binom{N}{M}\Pcal^{\sss(N+M)}(\xvec)
 \le2^N \sum_{M=0}^N\Pcal^{\sss(N+M)}(\xvec).
\end{align}
where the inequality is due to $\binom{N}{M}\le 2^N$.  By
Lemma~\ref{lem:constr1'}, we have, for $d>4$,
\begin{align}\lbeq{PcalN-bd}
\sum_x|x|^q\Pcal_t^{\sss(N)}(x)\le\delta_{q,0}\delta_{t,0}\delta_{N,0}
 +\vep^2\frac{(c\beta)^{1\vee N}\sigma^q}{(1+t)^{(d-q)/2}}
 \qquad(N\ge0),
\end{align}
for some $c<\infty$.  For $d\le4$, we can simply replace $\beta^{1\vee N}$ by
$\beta_\sT\hat\beta_\sT^{0\vee(N-1)}$ and $\sigma^2$ by $\sigma_\sT^2$.
Therefore,
\begin{align}\lbeq{PNbdconcl}
\sum_x|x|^qP_t^{\sss(N)}(x)\le2^N\sum_{M=0}^N\sum_x|x|^q\Pcal_t^{\sss(N
 +M)}(x)
&\le2^N\sum_{M=0}^N\bigg(\delta_{q,0}\delta_{t,0}\delta_{N+M,0}+\vep^2
 \frac{(c\beta)^{N+M}\sigma^q}{(1+t)^{(d-q)/2}}\bigg)\nn\\
&\le\delta_{q,0}\delta_{t,0}\delta_{N,0}+\vep^2\frac{(2c\beta)^N}{1-c\beta}
 \,\frac{\sigma^q}{(1+t)^{(d-q)/2}}.
\end{align}
This completes the proof of Lemma~\ref{lem-Ptbds}.
\end{proof}

\subsection{Bound on \protect$A(\vec{\xvec}_J)$}\label{ss:Abd}
In this section, we investigate $A(\vec\xvec_J)$.  First, in
Section~\ref{sss:Adiagbd}, we prove a $d$-independent diagrammatic bound on
$A^{\sss(N)}(\vvec,\vec\xvec_J;\bC)$, where we recall
$A^{\sss(N)}(\vec\xvec_J)=A^{\sss(N)}(\ovec,\vec\xvec_J;\{\ovec\})$ in
\refeq{ABzerodef}.  Then, in Section~\ref{sss:Abd}, we prove the bound
\refeq{Abd} for $d>4$ and the bound \refeq{Abd<4} for $d\le4$ simultaneously.

\subsubsection{Diagrammatic bound on \protect $A^{\sss(N)}(\vvec,\vec\xvec_J;\bC)$}\label{sss:Adiagbd}
The main result proved in this section is the following proposition:

\begin{lem}[\textbf{Diagrammatic bound on $A^{\sss(N)}(\vvec,\vec\xvec_J;\bC)$}]
 \label{lem-ANdiagbd}
For $r\ge3$, $\vec\xvec_J\in\Lambda^{r-1}$, $\vvec\in\Lambda$ and
$\bC\subset\Lambda$,
\begin{align}\lbeq{APbd}
&A^{\sss(N)}(\vvec,\vec\xvec_J;\bC)\\
&\le\begin{cases}
 \dpst\sum_{I\ne\vno,J}\bigg(\ind{\vvec\in\bC}\,\mP\big(\{\vvec\conn
  \vec\xvec_I\}\circ\{\vvec\conn\vec\xvec_{J\setminus I}\}\big)+\sum_{
  \zvec\ne\vvec}P^{\sss(0)}(\vvec,\zvec;\bC,\ell(\vec\xvec_I))\,\tau
  (\vec\xvec_{J\setminus I}-\zvec)\bigg)&(N=0),\\[15pt]
 \dpst\sum_{I\ne\vno,J}\sum_{\zvec}\Big(P^{\sss(N)}(\vvec,\zvec;\bC)
  \,\tau(\vec\xvec_I-\zvec)+P^{\sss(N)}(\vvec,\zvec;\bC,\ell(\vec
  \xvec_I))\,\Big)\,\tau(\vec\xvec_{J\setminus I}-\zvec)&(N\ge1).
 \end{cases}\nn
\end{align}
%\begin{align}\lbeq{APbd0}
%A^{\sss(0)}(\vvec,\vec\xvec_J;\bC)\le\sum_{I\ne\vno,J}\bigg(\ind{\vvec
% \in\bC}\,\mP\big(\{\vvec\conn\vec\xvec_I\}\circ\{\vvec\conn\vec\xvec_{
% J\setminus I}\}\big)+\sum_{\zvec\ne\vvec}P^{\sss(0)}(\vvec,\zvec;\bC,
% \ell(\vec\xvec_I))\,\tau(\vec\xvec_{J\setminus I}-\zvec)\bigg),
%\end{align}
%and for $N\ge1$,
%\begin{align}\lbeq{APbd1}
%A^{\sss(N)}(\vvec,\vec\xvec_J;\bC)\le\sum_{I\ne\vno,J}\sum_{\zvec}
% \bigg(P^{\sss(N)}(\vvec,\zvec;\bC)\,\tau(\vec\xvec_I-\zvec)+P^{\sss
% (N)}(\vvec,\zvec;\bC;\ell(\vec\xvec_I))\,\bigg)\,\tau(\vec\xvec_{J
% \setminus I}-\zvec).
%\end{align}
%%\begin{align}
%%&A^{\sss(N)}(\vvec,\vec\xvec_J;\bC)\nn\\
%%&~~\le\sum_{I\ne\vno,J}\sum_{\zvec}\bigg(P^{\sss(N)}(\vvec,\zvec;
%% \bC)\,\mP\big(\{\zvec\conn\vec\xvec_I\}\circ\{\zvec\conn\vec\xvec_{
%% J\setminus I}\}\big)+P^{\sss(N)}(\vvec,\zvec;\bC;\ell(\vec\xvec_I))
%% \,\tau(\vec\xvec_{J\setminus I}-\zvec)\bigg).
%%\end{align}
\end{lem}

\bigskip

To prove Lemma \ref{lem-ANdiagbd}, we first note that, by
\refeq{B0def}--\refeq{A0def} and \refeq{M-def}--\refeq{ANBNdef},
\begin{align}\lbeq{ANBNrep}
A^{\sss(N)}(\vvec,\vec\xvec_J;\bC)
%=M^{\sss(N+1)}_{\vvec,\vec\xvec_J;\bC}(1)&=\sum_{b_N}p_{b_N}
% M^{\sss(N)}_{\vvec,\bb_N;\bC}\Big(M^{\sss(1)}_{\tb_N,\vec
% \xvec_J;\tilde\bC_{N-1}}(1)\Big)\nn\\
 =\begin{cases}
 \mP\big(E'(\vvec,\vec\xvec_J;\bC)\big)&(N=0),\\[5pt]
 \sum_{b_N}p_{b_N}M^{\sss(N)}_{\vvec,\bb_N;\bC}\Big(\mP\big(E'(\tb_{\sN},
  \vec\xvec_J;\tilde\bC_{\sss N-1})\big)\Big)&(N\ge1).
\end{cases}
\end{align}
Thus, we are lead to study $\mP\big(E'(\vvec,\vec\xvec_J;\bC)\big)$.
As a result, Lemma~\ref{lem-ANdiagbd} is a consequence of the following lemma:

\begin{lemma}\label{lem-PEbd}
For $r\ge3$, $\vec\xvec_J\in\Lambda^{r-1}$, $\vvec\in\Lambda$ and
$\bC\subset\Lambda$,
\begin{align}\lbeq{PE'vecxbd}
\mP\big(E'(\vvec,\vec\xvec_J;\bC)\big)\le\sum_{I\ne\vno,J}\bigg(
 \ind{\vvec\in\bC}\,\mP\big(\{\vvec\conn\vec\xvec_I\}\circ\{\vvec
 \conn\vec\xvec_{J\setminus I}\}\big)+\sum_{\zvec\ne\vvec}P^{\sss(0)}
 (\vvec,\zvec;\bC,\ell(\vec\xvec_I))\,\tau(\vec\xvec_{J\setminus I}-
 \zvec)\bigg).
\end{align}
\end{lemma}

\begin{proof}[Proof of Lemma~\ref{lem-ANdiagbd} assuming Lemma~\ref{lem-PEbd}]
Since Lemma~\ref{lem-PEbd} and \refeq{ANBNrep} immediately imply \refeq{APbd}
for $N=0$, it thus suffices to prove \refeq{APbd} for $N\ge1$.

Substituting \refeq{PE'vecxbd} with $\vvec=\tb_{\sN}$,
$\bC=\tilde\bC_{\sss N-1}$ into \refeq{ANBNrep} and then using
\refeq{MPIa1}--\refeq{MPIa2}, we obtain
\begin{align}
&A^{\sss(N)}(\vvec,\vec\xvec_J;\bC)\nn\\
&\le\sum_{I\ne\vno,J}\sum_{b_N}p_{b_N}\bigg(M^{\sss(N)}_{\vvec,\bb_N;
 \bC}\big(\ind{\tb_{\sN}\in\tilde\bC_{N-1}}\big)\,\mP\big(\{\tb_{\sN}
 \conn\vec\xvec_I\}\circ\{\tb_{\sN}\conn\vec\xvec_{J\setminus I}\}\big)
 \nn\\
&\qquad\qquad+\sum_{\zvec\ne\tb_N}M^{\sss(N)}_{\vvec,\bb_N;\bC}\Big(
 P^{\sss(0)}(\tb_{\sN},\zvec;\tilde\bC_{\sN-1},\ell(\vec\xvec_I))\Big)\,
 \tau(\vec\xvec_{J\setminus I}-\zvec)\bigg)\nn\\
&\le\sum_{I\ne\vno,J}\sum_{\zvec}\Bigg(\bigg(\underbrace{\sum_\twop
 \sum_{b_N}P^{\sss(N-1)}(\vvec,\bb_{\sN};\bC;\ell^\twop(\tb_{\sN}))\,p_{
 b_N}\delta_{\tb_N,\zvec}}_{X}\bigg)\,\mP\big(\{\zvec\conn\vec\xvec_I
 \}\circ\{\zvec\conn\vec\xvec_{J\setminus I}\}\big)\nn\\
&\qquad\qquad+\bigg(\underbrace{\sum_\twop\sum_{\cvec}\sum_{b_N:\tb_N
 \ne\zvec}P^{\sss(N-1)}(\vvec,\bb_{\sN};\bC;\ell^\twop(\cvec))\,p_{
 b_N}P^{\sss(0)}(\tb_{\sN},\zvec;\cvec,\ell(\vec\xvec_I))}_{Y}\bigg)\,
 \tau(\vec\xvec_{J\setminus I}-\zvec)\Bigg),
\end{align}
where $\sum_\twop$ is the sum over the $(N-1)^\text{th}$~admissible lines for
$P^{\sss(N-1)}(\vvec,\bb_{\sN};\bC)$.  Ignoring the restriction
$\tb_{\sN}\ne\zvec$ and using an extension of \refeq{PArecrep-gen}, we obtain
\begin{align}
Y\le P^{\sss(N)}(\vvec,\zvec;\bC,\ell(\vec\xvec_I)).
\end{align}
For $X$, we use \refeq{P0vyv} and \refeq{PArecrep} to obtain
\begin{align}
X&\le\sum_\twop\sum_{b_N}P^{\sss(N-1)}(\vvec,\bb_{\sN};\bC;\ell^\twop
 (\tb_{\sN}))\,p_{b_N}P^{\sss(0)}(\tb_{\sN},\zvec;\tb_{\sN})\nn\\
&\le\sum_\twop\sum_{\yvec}\sum_{b_N}P^{\sss(N-1)}(\vvec,\bb_{\sN};\bC;
 \ell^\twop(\yvec))\,p_{b_N}P^{\sss(0)}(\tb_{\sN},\zvec;\yvec)=P^{\sss
 (N)}(\vvec,\zvec;\bC).
\end{align}
Finally, we use the BK inequality to bound $\mP(\{\zvec\conn\vec\xvec_I\}
\circ\{\zvec\conn\vec\xvec_{J\setminus I}\})$ by
$\tau(\vec\xvec_I-\zvec)\,\tau(\vec\xvec_{J\setminus I}-\zvec)$.  This
completes the proof.
\end{proof}

\begin{proof}[Proof of Lemma~\ref{lem-PEbd}]
Recall \refeq{E'bd}.  We show below that
\begin{align}\lbeq{BKvecx}
E'(\vvec,\vec\xvec_J;\bC)\subset\bigcup_{I\ne\vno,J}\bigcup_{\zvec}
 \Big\{\big\{\Ecal(\vvec,\zvec;\bC)\cap\{\vvec\conn\vec\xvec_I\}\big\}
 \circ\{\zvec\conn\vec\xvec_{J\setminus I}\}\Big\}.
\end{align}

First, we prove \refeq{PE'vecxbd} assuming \refeq{BKvecx}.
Substituting \refeq{BKvecx} into
$\mP(E'(\vvec,\vec\xvec_J;\bC))$, we have
\begin{align}
&\mP\big(E'(\vvec,\vec\xvec_J;\bC)\big)\nn\\
&\le\sum_{I\ne\vno,J}\sum_{\zvec}\mP\Big(\big\{\Ecal(\vvec,\zvec;\bC)
 \cap\{\vvec\conn\vec\xvec_I\}\big\}\circ\{\zvec\conn\vec\xvec_{J
 \setminus I}\}\Big)\\
&=\sum_{I\ne\vno,J}\bigg(\ind{\vvec\in\bC}~\mP\big(\{\vvec\conn\vec
 \xvec_I\}\circ\{\vvec\conn\vec\xvec_{J\setminus I}\}\big)+\sum_{\zvec
 \ne\vvec}\mP\Big(\big\{\Ecal(\vvec,\zvec;\bC)\cap\{\vvec\conn\vec
 \xvec_I\}\big\}\circ\{\zvec\conn\vec\xvec_{J\setminus I}\}\Big)\bigg).
 \nn
\end{align}
For the sum over $\zvec\ne\vvec$, we use the BK inequality to extract
$\mP(\zvec\conn\vec\xvec_{J\setminus I})\equiv\tau(\vec\xvec_{J\setminus
I}-\zvec)$ and apply the following inequality that is a result of an
extension of the argument around \refeq{2pt.100}:
\begin{align}
\mP\big(\Ecal(\vvec,\zvec;\bC)\cap\{\vvec\conn\vec\xvec_I\}\big)\le
 P^{\sss(0)}(\vvec,\zvec;\bC,\ell(\vec\xvec_I)).
\end{align}
This completes the proof of \refeq{PE'vecxbd}.

It remains to prove \refeq{BKvecx}.
Summarising \refeq{F'1-def}--\refeq{partition2}, we can rewrite
$E'(\vvec,\vec\xvec_J;\bC)$ as
\begin{align}\lbeq{E'supset-pre}
E'(\vvec,\vec\xvec_J;\bC)&=\bigg\{\BDcup{j\in J}\Big\{\{\vvec\conn
 \vec\xvec_J\}\cap\big\{\vvec\ct{\bC}(\xvec_1,\dots,\xvec_{j-1})\big\}
 ^{\rm c}\cap E'(\vvec,\xvec_j;\bC)\Big\}\nn\\
&\hspace{3pc}\cap\big\{\nexists\text{ pivotal bond $b$ for }\vvec\conn
 \xvec_i~\forall i\text{ such that }\vvec\ct{\bC}\bb\big\}\bigg\}\nn\\
&\quad\DDcup~\Bigg\{\BDcup{\vno\ne I\subsetneq J}\BDcup{b}\bigg\{\Big\{\{\vvec
 \conn\vec\xvec_I\}\cap\big\{\vvec\ct{\bC}(\xvec_1,\dots,\xvec_{j_I-1})
 \big\}^{\rm c}\cap E'(\vvec,\bb;\bC)\text{ in }\tilde\bC^b(\vvec)\Big\}\nn\\
&\hspace{8pc}\cap\Big\{b\text{ is occupied},~~\tb\conn\vec\xvec_{J
 \setminus I}\text{ in }\Lambda\setminus\tilde\bC^b(\vvec)\Big\}
 \bigg\}\Bigg\}.
\end{align}
Ignoring $\{\vvec\ct{\bC}(\xvec_1,\dots,\xvec_{j-1})\}^{\rm c}$
and using $E'(\vvec,\zvec;\bC)\subset\Ecal(\vvec,\zvec;\bC)$, we have
\begin{align}
E'(\vvec,\vec\xvec_J;\bC)&\subset\bigg\{\bigcup_{j\in J}\big\{\Ecal(
 \vvec,\xvec_j;\bC)\cap\{\vvec\conn\vec\xvec_{J_j}\}\big\}\bigg\}\nn\\
&\quad~\cup\;\bigg\{\bigcup_{\vno\ne I\subsetneq J}\,\bigcup_{\zvec}\Big\{
 \big\{\Ecal(\vvec,\zvec;\bC)\cap\{\vvec\conn\vec\xvec_I\}\big\}\circ\{\zvec
 \conn\vec\xvec_{J\setminus I}\}\Big\}\bigg\}.
\end{align}
Note that the first event on the right-hand side is a subset of the second
event, when $I=J_j$ and $\zvec=\xvec_j$, for which $J\setminus I=\{j\}$
and $\{\zvec\conn\vec\xvec_{J\setminus I}\}=\{\xvec_j\conn\xvec_j\}$ is the
trivial event.  This completes the proof of \refeq{BKvecx} and hence of
Lemma~\ref{lem-PEbd}.
\end{proof}

\subsubsection{Proof of the bound on \protect$A^{\sss(N)}(\vec{\xvec}_J)$}
\label{sss:Abd}
%We first use Lemma~\ref{lem-ANdiagbd} to see that we need
%to bound
%    \eq
%    \sum_{\vec x} A^{\smallsup{N}}_{\vec t_J} (\vec x_J)
%    \le \sum_{I\neq \varnothing, J} \sum_{\vec x} \sum_{\zvec}
%    P^{\sss (N)}(\zvec;\ell(\vec\xvec_I))\tau(\vec\xvec_{J\backslash I}-\zvec).
%    \en
%When $\zvec=0$, we can improve the bound above, since one bond
%emanating from $\ovec$ must be spatial. Therefore,
%we obtain
%    \eq
%    A^{\smallsup{N}}(\vec \xvec_J)
%    \le \sum_{I\neq \varnothing, J}\Big(\lamb \vep (D\sstar \tau)(\vec\xvec_I)
%    \tau(\vec\xvec_{J\backslash I})+\sum_{\zvec\neq \ovec}
%    P^{\sss (N)}(\zvec;\ell(\vec\xvec_I))\tau(\vec\xvec_{J\backslash I}-\zvec)\Big).
%    \en
%
%\[\]
%
We prove \refeq{Abd} for $d>4$ and \refeq{Abd<4} for $d\le4$ simultaneously,
using Lemmas~\ref{lem:constr1} and \ref{lem-Ptbds}--\ref{lem-ANdiagbd}.

Below, we will frequently use
\begin{align}\lbeq{treegraph}
\sum_{\vec x_I}\tau_{\vec t_I}(\vec x_I)\le
 O\big((1+\bar t_I)^{|I|-1}\big),
\end{align}
where we recall $\max_{i\in I}t_i\le T\log T$ for $d\le4$.  For simplicity, let
$I=\{1,\dots,i\}$.  Then, \refeq{treegraph} is an easy consequence of
Lemma~\ref{lem:constr1} and the tree-graph inequality \cite{an84}:
\begin{align}\lbeq{an-tree}
\sum_{\vec x_I}\tau_{\vec t_I}(\vec x_I)\le\sum_{\vec x_{I_i},x_i}\tau_{
 \vec t_{I_i}}(\vec x_{I_i};\ell(x_i,t_i))\le\cdots\le\sum_{x_1,\dots,
 x_i}\tau_{t_1}\big(x_1;\ell(x_2,t_2),\cdots,\ell(x_i,t_i)\big).
\end{align}

First we prove \refeq{Abd}, for which $d>4$, for $N\ge1$. By
Lemma~\ref{lem-ANdiagbd}, we have
\begin{align}\lbeq{ANgeq1-prebd}
A^{\sss(N)}(\vec\xvec_J)\equiv A^{\sss(N)}(\ovec,\vec\xvec_J;\{\ovec\})
 \le\sum_{I\ne\vno,J}\sum_{\zvec}\Big(P^{\sss(N)}(\zvec)\,\tau(\vec
 \xvec_I-\zvec)+P^{\sss(N)}(\zvec;\ell(\vec\xvec_I))\,\Big)\,\tau(\vec
 \xvec_{J\setminus I}-\zvec).
\end{align}
Note that the number of lines contained in each diagram for
$P^{\sss(N)}(\zvec)$ at any fixed time between 0 and $t_{\zvec}$ is bounded,
say, by $\cL$, due to its construction.
%It is not hard to see (see also \cite[below (7.24)]{HHS05b}) that
%$P^{\sss(N)}(\zvec)$ has at most 4 lines at any fixed time.
Therefore, by Lemmas~\ref{lem:constr1}
and \ref{lem-Ptbds}, we obtain
\begin{align}\lbeq{PN-constr-bd:l1}
\sum_{z,x_1}P^{\sss(N)}((z,s);\ell(x_1,t_1))\le\cL\frac{\vep^2O(\beta)^N}{(1
 +s)^{d/2}}\,(1+s\wedge t_1)\le\cL\frac{\vep^2O(\beta)^N}{(1+s)^{(d-2)/2}},
\end{align}
and further that
\begin{align}\lbeq{PN-constr-bd:l2}
\sum_{z,x_1,x_2}P^{\sss(N)}\big((z,s);\ell(x_1,t_1),\ell(x_2,t_2)\big)
%&\le\vep^2\frac{O(\beta)^N((1+s\wedge t_2)+(1+t_1\wedge t_2))}{(1+s)
% ^{(d-2)/2}}\nn\\
&\le\cL(\cL+1)\frac{\vep^2O(\beta)^N}{(1+s)^{(d-2)/2}}\,\big(1+(s\vee
t_1)\wedge
 t_2\big).
\end{align}
More generally, by denoting the second-largest element of $\{s,\vec t_I\}$ by
$\bar s_{\vec t_I}$, we have
\begin{align}\lbeq{PN-constr-bd}
\sum_{z,\vec x_I}P^{\sss(N)}\big((z,s);\ell(\vec x_I,\vec t_I)\big)\le
 \frac{(\cL+|I|-1)!}{(\cL-1)!}\,\frac{\vep^2O(\beta)^N}{(1+s)^{(d-2)/
 2}}\,(1+\bar s_{\vec t_I})^{|I|-1},
\end{align}
where the combinatorial factor $\frac{(\cL+|I|-1)!}{(\cL-1)!}$ is independent
of $\beta$ and $N$.  Substituting this and \refeq{PN-bd} into \refeq{ANgeq1-prebd} and
using \refeq{treegraph}, we obtain that, since $(d-2)/2>1$,
\begin{align}\lbeq{AN-conv-bd}
\sum_{\vec x_J}A_{\vec t_J}^{\sss(N)}(\vec x_J)&\le\vep O(\beta)^N\sum_{I\ne
 \vno,J}\Bigg(\vep\ddsum_{s\le\underline t_J}\frac1{(1+s)^{d/2}}\,O\big((\bar
 t_I-s)^{|I|-1}\big)\,O\big((\bar t_{J\setminus I}-s)^{|J\setminus I|-1}\big)
 \nn\\
&\hspace{7pc}+\vep\ddsum_{s\le\underline{t}_{J\setminus I}}\frac{O((1+
 \bar s_{\vec t_I})^{|I|-1})}{(1+s)^{(d-2)/2}}\,O\big((\bar t_{J\setminus
 I}-s)^{|J\setminus I|-1}\big)\Bigg)\nn\\
&\le\vep O(\beta)^NO\big((1+\bar t)^{|J|-2}\big),
% \equiv\vep O(\beta)^NO(N+1)^{r-2}O\big((1+\bar t)^{r-3}\big).
\end{align}
where $\bar t=\bar t_J$.  This proves \refeq{Abd} for $N\ge1$.

To prove \refeq{Abd<4}, for which $d\le4$, for $N\ge1$, we simply replace
$O(\beta)^N$ in \refeq{PN-constr-bd} by
$O(\beta_{\sT})\,O(\hat\beta_\sT)^{N-1}$ using Lemma~\ref{lem-Ptbds}(ii)
instead of Lemma~\ref{lem-Ptbds}(i). Then, we use the factor $\beta_{\sT}$ to
control the sums over $s\in\vep\Zp$ in \refeq{AN-conv-bd}, as in
\refeq{convappl<4}.  Since $\underline t_{J\setminus I}\le T\log T$,
$\beta_\sT\equiv\beta_1T^{-bd}$ and $\hat\beta_\sT\equiv\hatbetaT$ with
$\alphamin<bd-\frac{4-d}2$, we have
\begin{align}\lbeq{PN-constr-bd:d<4}
\beta_\sT\,\vep\ddsum_{s\le\underline t_{J\setminus I}}(1+s)^{-(d-2)/2}\le
 O(\beta_\sT)(1+\underline t_{J\setminus I})^{(4-d)/2}\big(\log(1+\underline
 t_{J\setminus I})\big)^{\delta_{d,4}}\le O(\hat\beta_\sT).
\end{align}
This completes the proof of \refeq{Abd<4} for $N\ge1$.

Next we consider the case of $N=0$.  Similarly to the above computation, the
contribution from the latter sum in \refeq{APbd} over $\zvec\ne\vvec~(=\ovec$
in the current setting) equals $\vep O(\beta(1+\bar t)^{r-3})$ for $d>4$ and
$\vep O(\hat\beta_{\sT}(1+\bar t)^{r-3})$ for $d\le4$.  It remains to
estimate the contribution from
$\mP(\{\ovec\conn\vec\xvec_I\}\circ\{\ovec\conn\vec\xvec_{J\setminus I}\})$ in
\refeq{APbd}.

If $\vep$ is large (e.g., $\vep=1$), then we simply use the BK inequality
to obtain
\begin{align}\lbeq{PE'vecx-prebd3}
\mP\big(\{\ovec\conn\vec\xvec_I\}\circ\{\ovec\conn\vec\xvec_{J\setminus
 I}\}\big)\le\tau(\vec\xvec_I)\,\tau(\vec\xvec_{J\setminus I}).
\end{align}
Therefore, by \refeq{treegraph}, we have
\begin{align}
\sum_{\vec x_J}A_{\vec t_J}^{\sss(0)}(\vec x_J)\le O\big((1+\bar
 t)^{r-3}\big).
\end{align}

If $\vep\ll1$, then we should be more careful.  Since
$\{\ovec\conn\vec\xvec_I\}$ and $\{\ovec\conn\vec\xvec_{J\setminus
I}\}$ occur bond-disjointly, and since there is only one temporal bond growing
out of $\ovec$, there must be a nonempty subset $I'$ of $I$ or $J\setminus I$
and a spatial bond $b$ with $\bb=\ovec$ such that
$\{b\conn\vec\xvec_{I'}\}\circ\{\ovec\conn\vec\xvec_{J\setminus I'}\}$ occurs.
Then, by the BK inequality and
\refeq{treegraph}, we obtain
\begin{align}\lbeq{PE'vecx-prebd4}
\sum_{\vec x_J}\mP\big(\{\ovec\conn\vec\xvec_I\}\circ\{\ovec\conn\vec\xvec_{J
 \setminus I}\}\big)&\le\sum_{\vec x_J}\sum_{\vno\ne I'\subsetneq J}\,
 \sum_{\substack{b\text{ spatial}\\ \bb=\ovec}}\mP\big(\{b\conn\vec
 \xvec_{I'}\}\circ\{\ovec\conn\vec\xvec_{J\setminus I'}\}\big)\nn\\
&\le\sum_{\vec x_J}\sum_{\vno\ne I'\subsetneq J}(\lamb\vep D\sstar\tau)(\vec
 \xvec_{I'})\,\tau(\vec\xvec_{J\setminus I'})\nn\\
&\le\vep\,O\Big((1+\bar t_{I'})^{|I'|-1}(1+\bar t_{J\setminus I'})^{|J\setminus
 I'|-1}\Big)\le\vep\,O\big((1+\bar t)^{|J|-2}\big).
\end{align}
This completes the proof of \refeq{Abd} for $d>4$ and \refeq{Abd<4} for
$d\le4$. \qed

%%%%%%%%%%%%%%%%%%%%%%%%%%%%%%%%%%%%%%%%%%%%%%%%%%%%%%%%%%%%%%%%%%%%%%%%
%%%%%%%%%%%%%%%%%%%%%%%%%%%%%%%%%%%%%%%%%%%%%%%%%%%%%%%%%%%%%%%%%%%%%%%%
%%%%%%%%%%%%%%%%%%%%%%%%%%%%%%%%%%%%%%%%%%%%%%%%%%%%%%%%%%%%%%%%%%%%%%%%
%%%%%%%%%%%%%%%%%%%%%%%%%%%%%%%%%%%%%%%%%%%%%%%%%%%%%%%%%%%%%%%%%%%%%%%%
%\input{rpt5}

% June 2, 2007, RvdH

\section{Bound on \protect$\phi(\yvec_1,\yvec_2)_{\sss\pm}$}\label{sec-Cbd}
To prove the bound on $\hat\psi_{s_1,s_2}(k_1,k_2)$ in
Proposition~\ref{thm-psivphibd}, we first recall \refeq{psidef} and
\refeq{Cdef}:
    \begin{align}
    \psi(\yvec_1,\yvec_2)=\sum_{\vvec}p_\vep(\vvec)\,C(\yvec_1-\vvec,
     \yvec_2-\vvec),\qquad
    C(\yvec_1,\yvec_2)=\phi(\yvec_1,\yvec_2)_{\sss+}+\phi(\yvec_2,
     \yvec_1)_{\sss+}-\phi(\yvec_2,\yvec_1)_{\sss-},
    \end{align}
hence
    \begin{align}\lbeq{psi-phi-id}
    \hat\psi_{s_1,s_2}(k_1,k_2)=\hat p_\vep(k_1+k_2)\Big(\hat\phi_{s_1-\vep,
     s_2-\vep}(k_1,k_2)_{\sss+}+\hat\phi_{s_2-\vep,s_1-\vep}(k_2,k_1)_{\sss+}
     -\hat\phi_{s_2-\vep,s_1-\vep}(k_2,k_1)_{\sss-}\Big).
    \end{align}
Therefore, to show the bound on $\hat\psi_{s_1,s_2}(k_1,k_2)$, it suffices to
investigate $\phi(\yvec_1,\yvec_2)_{\sss\pm}$.

Recall the definition of $\phi^{\sss(N)}(\yvec_1,\yvec_2)_{\sss\pm}$ in
\refeq{phipm-def}, where $B_\delta(\tb_{\sss N+1},\yvec_1;\bC_{\sN})$ and
$B_\delta(\te,\yvec_2;\tilde\bC^e_{\sN})$ appear.  We also recall
$B^{\sss(N)}(\vvec,\yvec;\bC)$ in \refeq{AB-alt} and $B_\delta(\vvec,\yvec;\bC)$ in
\refeq{Bdelta-def}.  Let
    \begin{align}\lbeq{BdeltaNdef}
    B_\delta^{\sss(N)}(\vvec,\yvec;\bC)=\begin{cases}
     \delta_{\vvec,\yvec}&(N=0),\\
     B^{\sss(N-1)}(\vvec,\yvec;\bC)&(N\ge1),
     \end{cases}
    \end{align}
so that $B_\delta^{\sss(N)}(\vvec,\yvec;\bC)\ge0$ and
    \begin{align}\lbeq{Bdeltaalt}
    B_\delta(\vvec,\yvec;\bC)=\sum_{N=0}^\infty (-1)^NB_\delta^{\sss(N)}
     (\vvec,\yvec;\bC).
    \end{align}
Let $\phi^{\sss(N,N_1,N_2)}(\yvec_1,\yvec_2)_{\sss\pm}$ be the contribution to
$\phi^{\sss(N)}(\yvec_1,\yvec_2)_{\sss\pm}$ from
$B_\delta^{\sss(N_1)}(\tb_{\sss N+1},\yvec_1;\bC_{\sN})$ and
$B_\delta^{\sss(N_2)}(\te,\yvec_2;\tilde\bC^e_{\sN})$. Then,
$\phi^{\sss(N,N_1,N_2)}(\yvec_1,\yvec_2)_{\sss\pm}\ge0$ and
    \begin{align}\lbeq{phialtbd}
    \phi(\yvec_1,\yvec_2)_{\sss\pm}=\sum_{N,N_1,N_2=0}^\infty(-1)^{N+N_1+N_2}
     \phi^{\sss(N,N_1,N_2)}(\yvec_1,\yvec_2)_{\sss\pm}.
    \end{align}

Now we state the bound on $\phi_{s_1,s_2}^{\sss(N,N_1,N_2)}(y_1,y_2)_{\sss\pm}$ in the
following proposition. Since we have already shown in Section~\ref{sec-Cvepvep} that
    \begin{align}\lbeq{phi-vep-vep-sum}
    \phi_{\vep,\vep}^{\sss(N,N_1,N_2)}(y_1,y_2)_{\sss\pm}=
    \begin{cases}
    p_\vep(y_1)\,p_\vep(y_2)\,(1-\delta_{y_1,y_2})&\text{if }N=N_1=N_2=0,\\
    0&\text{otherwise},
    \end{cases}
    \end{align}
we only need to bound $\phi_{s_1,s_2}^{\sss(N,N_1,N_2)}(y_1,y_2)_{\sss\pm}$ for
$s_2\ge s_1\ge\vep$ with $(s_1,s_2)\neq(\vep,\vep)$.  For $j=1,2$, we let (cf.,
\refeq{bs-def}--\refeq{ns-def})
\begin{align}
\tilde n_{s_1,s_2}^{\sss(j)}&=n_{s_1+j\vep,s_2+j\vep}\equiv3-\delta_{s_1,s_2}
 -\delta_{s_1,(2-j)\vep}\delta_{s_2,(2-j)\vep},\\[5pt]
\tilde b_{s_1,s_2}^{\sss(j)}&=\frac{\vep^{\tilde n_{s_1,s_2}^{(j)}}\ind{s_1\le
 s_2}}{(1+s_1)^{(d-2)/2}}\times
 \begin{cases}
 (1+s_2-s_1)^{-(d-2)/2}\quad&(d>2),\\
 \log(1+s_2)&(d=2),\\
 (1+s_2)^{(2-d)/2}&(d<2),
 \end{cases}\lbeq{tildebj-def}
\end{align}
where $\tilde n_{s_1,s_2}^{\sss(0)}=n_{s_1,s_2}$ and $\tilde
b_{s_1,s_2}^{\sss(0)}=b_{s_1,s_2}^{\sss(\vep)}$.
%Also, for $j=0,1$, since
%$\tilde n_{s_1-\vep,s_2-\vep}^{\sss(j+1)}=\tilde n_{s_1,s_2}^{\sss(j)}$,
%we have
%\begin{align}
%\tilde b_{s_1-\vep,s_2-\vep}^{\sss(j+1)}=(1+\vep)^{0\vee\frac{d-2}2}\,
% \tilde b_{s_1,s_2}^{\sss(j)}.
%\end{align}
Then, the bound on $\phi^{\sss(N,N_1,N_2)}_{s_1,s_2}$ proved in this section
reads as follows:

\begin{prop}\label{prop-phibds}
Let $\lamb=\lambc$ for $d>4$, and $\lamb=\lamb_\sT$ for $d\le4$.
Let $s_2\ge s_1\ge\vep$ with $(s_1,s_2)\ne(\vep,\vep)$
and $s_2\le T\log T$ if $d\le4$.  For $q=0,2$ and $N,N_1,N_2\ge0$ ($N\ge1$ for
$\phi_{s_1,s_2}^{\sss(N,N_1,N_2)}(y_1,y_2)_{\sss-}$),
\begin{align}\lbeq{phi1bdN}
&\sum_{y_1,y_2}|y_i|^q\phi^{\sss(N,N_1,N_2)}_{s_1,s_2}(y_1,y_2)_{\sss\pm}\nn\\
&\le(1+s_i)^{q/2}\,\tildeb_{s_1,s_2}^{\sss(1)}\times
 \begin{cases}
 (\delta_{s_1,s_2}\delta_{N_2,0}+\beta)\,O(\beta)^{\sss1\vee(N+N_1)+0\vee(N_2
  -1)}\sigma^q&(d>4),\\
 (\delta_{s_1,s_2}\delta_{N_2,0}+\beta_\sT)\,O(\beta_{\sT})O(\hat\beta_{\sT})
  ^{\sss0\vee(N+N_1-1)+0\vee(N_2-1)}\sigma_\sT^q&(d\le4).
 \end{cases}
\end{align}
\end{prop}

The bound on $\hat\psi_{s_1,s_2}(k_1,k_2)$ in Proposition~\ref{thm-psivphibd}
now follows from Proposition~\ref{prop-phibds} as well as \refeq{psi-phi-id},
\refeq{phialtbd}--\refeq{phi-vep-vep-sum} and
\begin{align}\lbeq{Cmainbd}
\big|\nabla_{k_i}^q\hat\phi_{s_1,s_2}^{\sss(N,N_1,N_2)}(k_1,k_2)_{\sss
 \pm}\big|\le\sum_{y_1,y_2}|y_i|^q\phi_{s_1,s_2}^{\sss(N,N_1,N_2)}(y_1,
 y_2)_{\sss\pm}.
\end{align}

The remainder of this section is organised as follows.  In
Section~\ref{sec-diagramsphi}, we define bounding diagrams for
$\phi_{s_1,s_2}^{\sss(N,N_1,N_2)}(y_1,y_2)_{\sss\pm}$.  In
Section~\ref{sec-diagramsphi-bd}, we prove that those diagrams are so bounded
as to imply Proposition~\ref{prop-phibds}.  In Section~\ref{sec-redphi}, we
prove that $\phi_{s_1,s_2}^{\sss(N,N_1,N_2)}(y_1,y_2)_{\sss\pm}$ are indeed
bounded by those diagrams.

\subsection{Constructions: II}\label{sec-diagramsphi}
To define bounding diagrams for
$\phi_{s_1,s_2}^{\sss(N,N_1,N_2)}(y_1,y_2)_{\sss\pm}$, we first introduce two
more constructions:

\begin{defn}[\textbf{Constructions~$V_t$ and $\Ecal_t$}]\label{def-VEcal}
Given a diagram $F(\yvec_1)$ with
two vertices carrying labels $\ovec$ and $\yvec_1$, Construction~$V_t(\yvec_2)$
and Construction~$\Ecal_t(\yvec_2)$ produce the diagrams
\begin{align}
F\big(\yvec_1;V_t(\yvec_2)\big)&=\sum_{\vvec:t_{\vvec}=t}F\big(\yvec_1;\ell
 (\vvec),2_{\vvec}^{\sss(0)}(\yvec_2)\big),\lbeq{FV}\\
F\big(\yvec_1;\Ecal_t(\yvec_2)\big)&=\sum_{\zvec}\sum_{\avec:t_{\avec}\ge t}
 F\big(\yvec_1;B(\zvec),\ell(\avec)\big)\,P^{\sss(0)}(\zvec,\yvec_2;\avec).
 \lbeq{FEcal}
\end{align}
\end{defn}

\begin{rem}
Recall that Construction~$\ell(\vvec)$ (resp., Construction~$B(\vvec)$) is the
result of applying Construction~$\ell^\twop(\vvec)$ (resp.,
Construction~$B^\twop(\vvec)$) followed by a sum over \emph{all} possible lines
$\twop$.  Construction~$2^{\sss(0)}_{\vvec}(\yvec_2)$ in \refeq{FV} is applied
to a certain set of admissible lines for $F(\yvec_1)$ (e.g., the $N^\text{th}$
admissible lines for $P^{\sss(N)}(\yvec_1)$) and the line added due to
Construction~$\ell(\vvec)$.
%Construction~$\Ecal_t(\yvec_2)$ is equivalent to a
%Construction~$2^{\sss (1)}(\avec)$ with $t_{\avec}\ge t$,
%followed by a Construction~$2^{\smallsup{0}}_{\avec}(\yvec_2)$.
\end{rem}

Now we use the above constructions to define bounding diagrams for
$\phi^{\sss(N)}(\yvec_1,\yvec_2)_{\sss\pm}$.  Define
    \begin{align}
    R^{\sss(N)}(\yvec_1,\yvec_2)&=P^{\sss(N)}\big(\yvec_1;V_{t_{\yvec_1}}
     (\yvec_2)\big)\equiv\sum_{\vvec:t_{\vvec}=t_{\yvec_1}}P^{\sss(N)}
     \big(\yvec_1;\ell(\vvec),2^{\sss(0)}_{\vvec}(\yvec_2)\big),\lbeq{PN0def}\\
    Q^{\sss(N)}(\yvec_1,\yvec_2)&=P^{\sss(N)}\big(\yvec_1;\Ecal_{t_{\yvec_1}}
     (\yvec_2)\big)\equiv\sum_{\zvec}\sum_{\avec:t_{\avec}\ge t_{\yvec_1}}
     P^{\sss(N)}\big(\yvec_1;B(\zvec),\ell(\avec)\big)\,P^{\sss(0)}(\zvec,\yvec_2;
     \avec).\lbeq{QN-def}
    \end{align}
In Figure~\ref{fig-2ndexp-main}, we see a close resemblance between the
shown example of $\phi^{\sss(1)}(\yvec_1,\yvec_2)_{\sss+}$ (as well as the
first figure for $\phi^{\sss(1)}(\yvec_1,\yvec_2)_{\sss-}$) and
the bounding diagram
    \begin{align}
    \sum_{\substack{b:\tb=\yvec_1\\ b':\tbsp=\yvec_2}}p_bp_{b'}\sum_e\sum_{\cvec}
     R^{\sss(2)}(\bb,\eb;\ell(\cvec))\,p_eP^{\sss(0)}(\te,\bb';\cvec),
    \end{align}
and between the shown example of $\phi^{\sss(1)}(\yvec_1,\yvec_2)_{\sss-}$ and
the bounding diagram
    \begin{align}
    \sum_{\substack{b:\tb=\yvec_1\\ b':\tbsp=\yvec_2}}p_bp_{b'}\sum_e\sum_{\cvec}
     Q^{\sss(2)}(\bb,\eb;\ell(\cvec))\,p_eP^{\sss(0)}(\te,\bb';\cvec).
    \end{align}
Let $R^{\sss(N,N')}(\yvec_1,\yvec_2)$ (resp.,
$Q^{\sss(N,N')}(\yvec_1,\yvec_2)$) be the result of $N'$ applications of
Construction~$E$ applied to the second argument $\vvec$ of
$R^{\sss(N)}(\yvec_1,\vvec)$ (resp., $Q^{\sss(N)}(\yvec_1,\vvec)$).  By
convention, we write
$R^{\sss(N,0)}(\yvec_1,\yvec_2)=R^{\sss(N)}(\yvec_1,\yvec_2)$ and
$Q^{\sss(N,0)}(\yvec_1,\yvec_2)=Q^{\sss(N)}(\yvec_1,\yvec_2)$.
%These diagrams depend only on $N,N_1$ through $N+N_1$.
%By convention, $R^{\sss(N,0)}(\yvec_1,\yvec_2)=R^{\sss(N)}(\yvec_1,\yvec_2)$
%and $Q^{\sss(N,0)}(\yvec_1,\yvec_2)=Q^{\sss(N)}(\yvec_1,\yvec_2)$.

In Section~\ref{sec-redphi}, we will prove the following diagrammatic
bounds on $\phi^{\sss(N,N_1,N_2)}(\yvec_1,\yvec_2)_{\sss\pm}$:

\begin{lemma}[\textbf{Bounding diagrams for $\phi^{\sss(N,N_1,N_2)}(\yvec_1,\yvec_2)_{\sss\pm}$}]\label{lem-phibdslow}
Let $\yvec_1,\yvec_2\in\Lambda$ with $t_{\yvec_2}\ge t_{\yvec_1}>0$, and let
$N_1,N_2\ge0$.  For $N\ge 0$,
\begin{align}\lbeq{phi1bd}
\phi^{\sss(N,N_1,N_2)}(\yvec_1,\yvec_2)_{\sss+}\le\sum_{\uvec_1,\uvec_2}
 R^{\sss(N+N_1,N_2)}(\uvec_1,\uvec_2)~p_{\vep}(\yvec_1-\uvec_1)\,p_{\vep}
 (\yvec_2-\uvec_2),
\end{align}
and, for $N\ge1$,
\begin{align}\lbeq{phi3bd}
\phi^{\sss(N,N_1,N_2)}(\yvec_1,\yvec_2)_{\sss-}\le\sum_{\uvec_1,\uvec_2}
 \Big(R^{\sss(N+N_1,N_2)}(\uvec_1,\uvec_2)+Q^{\sss(N+N_1,N_2)}(\uvec_1,
 \uvec_2)\Big)\,p_{\vep}(\yvec_1-\uvec_1)\,p_{\vep}(\yvec_2-\uvec_2).
\end{align}
\end{lemma}

\subsection{Bounds on \protect$\phi_{s_1,s_2}^{\sss(N,N_1,N_2)}
 (y_1,y_2)_{\sss\pm}$ assuming their diagrammatic bounds}
 \label{sec-diagramsphi-bd}
In this section, we prove the following bounds on $R^{\sss(N,N')}$ and
$Q^{\sss(N,N')}$:

\begin{lemma}\label{lem-PQNbds}
Let $\lamb=\lambc$ for $d>4$, and $\lamb=\lamb_\sT$ for $d\le4$.
Let $s_2\ge s_1\ge0$ with $(s_1,s_2)\ne(0,0)$ and $s_2\le
T\log T$ if $d\le4$, and let $q=0,2$ and $N'\ge0$.  For $N\ge0$,
\begin{align}\lbeq{goal2}
&\sum_{y_1,y_2}|y_i|^qR^{\sss(N,N')}_{s_1,s_2}(y_1,y_2)\nn\\
&\le(1+s_i)^{q/2}\,\tildeb_{s_1,s_2}^{\sss(2)}\times
 \begin{cases}
 (\delta_{s_1,s_2}\delta_{N',0}+\beta)\,O(\beta)^{\sss1\vee N+0\vee(N'-1)}
  \sigma^q&(d>4),\\
 (\delta_{s_1,s_2}\delta_{N',0}+\beta_\sT)\,O(\beta_{\sT})O(\hat
  \beta_{\sT})^{\sss0\vee(N-1)+0\vee(N'-1)}\sigma_\sT^q&(d\le4),
 \end{cases}
\end{align}
and, for $N\ge1$,
\begin{align}\lbeq{Qgoal}
\sum_{y_1,y_2}|y_i|^qQ^{\sss(N,N')}_{s_1,s_2}(y_1,y_2)&\le
 (1+s_i)^{q/2}\,\tildeb_{s_1,s_2}^{\sss(2)}\times
 \begin{cases}
 O(\beta)^{N+N'+1}\sigma^q&(d>4),\\
 O(\beta_{\sT})^2O(\hat\beta_{\sT})^{N+N'-1}\sigma_\sT^q&(d\le4).
 \end{cases}
\end{align}
\end{lemma}

Proposition~\ref{prop-phibds} is an immediate consequence of
Lemmas~\ref{lem-phibdslow}--\ref{lem-PQNbds}.

\begin{proof}[Proof of Lemma~\ref{lem-PQNbds}]
Let
\begin{align}
\tilde R^{\sss(N)}(\yvec_1,\yvec_2)&=P^{\sss(N)}(\yvec_1;\ell(\yvec_2))\,
 \delta_{t_{\yvec_1},t_{\yvec_2}},\lbeq{Rtildedef}\\
\tilde Q^{\sss(N)}(\yvec_1,\yvec_2)&=\sum_{\zvec,\wvec}P^{\sss(N)}\big(
 \yvec_1;B(\zvec),B(\wvec)\big)\,L(\zvec,\wvec;\yvec_2).\lbeq{SN-def}
\end{align}
By \refeq{PN0def}--\refeq{QN-def} and \refeq{tilFz}, we have
\begin{align}
R^{\sss(N,N')}(\yvec_1,\yvec_2)&=\tilde R^{\sss(N)}\Big(\yvec_1,\vvec_0;
 2_{\vvec_0}^{\sss(0)}(\vvec_1),E_{\vvec_1}(\vvec_2),\cdots,E_{\vvec_{N'}}
 (\vvec_{\sss N'+1})\Big)\,\delta_{\vvec_{N'+1},\yvec_2},\lbeq{RtilR}\\
Q^{\sss(N,N')}(\yvec_1,\yvec_2)&=\tilde Q^{\sss(N)}\Big(\yvec_1,\vvec_0;
 2_{\vvec_0}^{\sss(0)}(\vvec_1),E_{\vvec_1}(\vvec_2),\cdots,E_{\vvec_{N'}}
 (\vvec_{\sss N'+1})\Big)\,\delta_{\vvec_{N'+1},\yvec_2},\lbeq{QtilQ}
\end{align}
where Construction~$2_{\vvec_0}^{\sss(0)}(\vvec_1)$ in \refeq{RtilR} is applied
to the $N^\text{th}$\,admissible lines for $P^{\sss(N)}(\yvec_1)$ and the added
line due to Construction~$\ell(\vvec_0)$ in the definition of $\tilde
R^{\sss(N)}(\yvec_1,\vvec_0)$, while
Construction~$2_{\vvec_0}^{\sss(0)}(\vvec_1)$ in \refeq{QtilQ} is applied to
the $L$-admissible lines of the factor $L(\zvec,\wvec;\yvec_2)$ in the
definition of $\tilde Q^{\sss(N)}(\yvec_1,\vvec_0)$ in \refeq{SN-def}.  We
will show below that, for $s_2\ge s_1\ge0$ with $(s_1,s_2)\ne(0,0)$ and $s_2\le
T\log T$ if $d\le4$, and for $N\ge0$,
\begin{align}
\sum_{y_1,y_2}|y_i|^q\tilde R^{\sss(N)}_{s_1,s_2}(y_1,y_2)&\le
 (1+s_i)^{q/2}\,\tildeb_{s_1,s_2}^{\sss(2)}\delta_{s_1,s_2}\times
 \begin{cases}
 O(\beta)^{1\vee N}\sigma^q&(d>4),\\
 O(\beta_{\sT})O(\hat\beta_{\sT})^{0\vee(N-1)}\sigma_\sT^q&(d\le4),
 \end{cases}\lbeq{tilR-suffbd}
\end{align}
and, for $N\ge1$,
\begin{align}
\sum_{y_1,y_2}|y_i|^q\tilde Q^{\sss(N)}_{s_1,s_2}(y_1,y_2)&\le
 (1+s_i)^{q/2}\,\tildeb_{s_1,s_2}^{\sss(2)}\times
 \begin{cases}
 O(\beta)^{N+1}\sigma^q&(d>4),\\
 O(\beta_{\sT})^2O(\hat\beta_{\sT})^{N-1}\sigma_\sT^q&(d\le4).
 \end{cases}\lbeq{tilQ-suffbd}
\end{align}
These bounds are sufficient for \refeq{goal2}--\refeq{Qgoal}, due to
Lemma~\ref{lem:constr1'}.  For example, consider \refeq{RtilR} for
$2<d\le4$ with $N'=1$ and $0<s_1\le s_2\le T\log T$.  By
\refeq{tilF0}--\refeq{tilFz},
\begin{align}\lbeq{tilR-suffbd-ex}
R^{\sss(N,1)}_{s_1,s_2}(y_1,y_2)&=\tilde R^{\sss(N)}\big((y_1,s_1),\vvec;2_{
 \vvec}^{\sss(1)}(y_2,s_2)\big)+2\tilde R^{\sss(N)}\Big((y_1,s_1),\vvec;2_{
 \vvec}^{\sss(1)}(\vvec'),2_{\vvec'}^{\sss(1)}(y_2,s_2)\Big)\nn\\
&\quad+\tilde R^{\sss(N)}\Big((y_1,s_1),\vvec;2_{\vvec}^{\sss(1)}(\vvec'),2_{
 \vvec'}^{\sss(1)}(\vvec''),2_{\vvec''}^{\sss(1)}(y_2,s_2)\Big).
\end{align}
By Lemma~\ref{lem:constr1} and \refeq{effBl1-1stblock}--\refeq{supL-bd},
we obtain
\begin{align}\lbeq{tilR-suffbd-exbd0}
&\sum_{y_1,y_2}\tilde R^{\sss(N)}\big((y_1,s_1),\vvec;2_{\vvec}^{\sss(1)}(y_2,
 s_2)\big)\nn\\
&\le\sum_{y_1,y_2}\sum_\twop\sum_{(v,t),(w,s)}\tilde R^{\sss(N)}\big((y_1,s_1),
 (v,t);B^\twop(w,s)\big)\,L\big((v,t),(w,s);(y_2,s_2)\big)\nn\\
&\le O(\beta_{\sT})\,O(\hat\beta_{\sT})^{0\vee(N-1)}\,\tildeb_{s_1,s_1}^{
 \sss(2)}\ddsum_{s<s_1}\sum_\twop(\delta_{s,t_\eta}+\vep C_1)\frac{c'\vep
 \beta_{\sT}}{(1+s_2-s)^{d/2}}\nn\\
&\le O(\beta_{\sT})\,O(\hat\beta_{\sT})^{0\vee(N-1)}\underbrace{\frac{c'\vep^3
 \beta_\sT}{(1+s_1)^{(d-2)/2}}\bigg(\frac{\cL_1}{(1+s_2-s_1)^{d/2}}+\ddsum_{s<
 s_1}\frac{\vep C_1\cL_2}{(1+s_2-s)^{d/2}}\bigg)}_{\le~O(\beta_\sT)\,
 \tildeb_{s_1,s_2}^{\sss(2)}},
\end{align}
where we have used the fact that the number $\cL_1$ of admissible lines $\twop$
is finite and that $\tilde R^{\sss(N)}$ has a finite number $\cL_2$ of lines at
any fixed time.  In fact, since $P^{\sss(N)}$ has at most 4 lines at any
fixed time, by \refeq{Rtildedef}, $\tilde R^{\sss(N)}$ has at most 5 lines at
any fixed time, so that $\cL_2=5$.  Furthermore, the sum over $y_1,y_2$ of
$\tilde R^{\sss(N)}(\yvec_1,\vvec;2_{\vvec}^{\sss(1)}(\vvec'),2_{\vvec'}^{\sss
(1)}(\yvec_2))$ in \refeq{tilR-suffbd-ex}, where $\yvec_1=(y_1,s_1)$ and
$\yvec_2=(y_2,s_2)$, is bounded similarly as
\begin{align}\lbeq{tilR-suffbd-exbd1}
&\sum_{y_1,y_2}\tilde R^{\sss(N)}\Big((y_1,s_1),\vvec;2_{\vvec}^{\sss(1)}(v',
 t'),2_{(v',t')}^{\sss(1)}(y_2,s_2)\Big)\nn\\
&\le O(\beta_{\sT})^2O(\hat\beta_{\sT})^{0\vee(N-1)}\ddsum_{t'<s_2}
 \tildeb_{s_1,t'}^{\sss(2)}\ddsum_{s\le t'}\sum_\twop(\delta_{s,t_\eta}+\vep
 C_1)\frac{c'\vep\beta_{\sT}}{(1+s_2-s)^{d/2}}\nn\\
&\le O(\beta_{\sT})^2O(\hat\beta_{\sT})^{0\vee(N-1)}\underbrace{\frac1{(1
 +s_1)^{(d-2)/2}}\ddsum_{s_1\le t'<s_2}\frac{\vep^{n_{s_1,t'}^{\sss(2)}}}{(1
 +t'-s_1)^{(d-2)/2}}\,\frac{\vep O(\beta_{\sT})}{(1+s_2-t')^{(d-2)/2}}}_{\le~
 O(\hat\beta_\sT)\,\tildeb_{s_1,s_2}^{\sss(2)}}.
\end{align}
The contribution from the third term in \refeq{tilR-suffbd-ex} can be estimated
similarly and is further smaller than the bound \refeq{tilR-suffbd-exbd1} by a
factor of $\hat\beta_{\sT}$.  We have shown \refeq{goal2} for $2<d\le4$
with $q=0$, $N'=1$ and $0<s_1\le s_2\le T\log T$.

Now it remains to show \refeq{tilR-suffbd}--\refeq{tilQ-suffbd}.  First we
prove \refeq{tilR-suffbd}, which is trivial when $s_2>s_1=0$ because $\tilde
R_{0,s_2}^{\sss(N)}(y_1,y_2)\equiv0$.  Let $s_2\ge s_1>0$.  By applying
\refeq{constr1-bd-b} for $q=0$ to the bounds in
\refeq{PN-bd}--\refeq{PN-bd-lower}, we obtain that, for $q=0,2$,
\begin{align}\lbeq{tildeR+bd1}
\sum_{y_1,y_2}|y_1|^q\tilde R_{s_1,s_2}^{\sss(N)}(y_1,y_2)\le\frac{C_2(1+s_2)
 \vep^2\delta_{s_1,s_2}}{(1+s_1)^{(d-q)/2}}\times
 \begin{cases}
 O(\beta)^{1\vee N}\sigma^q&(d>4),\\
 O(\beta_\sT)O(\hat\beta_\sT)^{0\vee(N-1)}\sigma_\sT^q&(d\le4).
 \end{cases}
\end{align}
To bound $\sum_{y_1,y_2}|y_2|^2\tilde R_{s_1,s_2}^{\sss(N)}(y_1,y_2)$, we apply
\refeq{constr1-bd-b} for $q=2$ to the bounds in
\refeq{PN-bd}--\refeq{PN-bd-lower} for $q=0$.  Then, we obtain
\begin{align}\lbeq{tildeR+bd2}
\sum_{y_1,y_2}|y_2|^2\tilde R_{s_1,s_2}^{\sss(N)}(y_1,y_2)\le\frac{C_2(N+1)s_2
 (1+s_2)\vep^2\delta_{s_1,s_2}}{(1+s_1)^{d/2}}\times
 \begin{cases}
 O(\beta)^{1\vee N}\sigma^q&(d>4),\\
 O(\beta_\sT)O(\hat\beta_\sT)^{0\vee(N-1)}\sigma_\sT^q&(d\le4),
 \end{cases}
\end{align}
where we have used the fact that the number of diagram lines to which
Construction~$\ell(y_2,s_2)$ is applied is at most $N+1$.  Absorbing the factor
$N+1$ into the geometric term, we can summarise
\refeq{tildeR+bd1}--\refeq{tildeR+bd2} as
\begin{align}\lbeq{tildeR+bd}
\sum_{y_1,y_2}|y_i|^q\tilde R_{s_1,s_2}^{\sss(N)}(y_1,y_2)\le(1+s_i)^{q/2}\,
 \tilde b_{s_1,s_2}^{\sss(2)}\delta_{s_1,s_2}\times
 \begin{cases}
 O(\beta)^{1\vee N}\sigma^q&(d>4),\\
 O(\beta_\sT)O(\hat\beta_\sT)^{0\vee(N-1)}\sigma_\sT^q&(d\le4).
 \end{cases}
\end{align}
This completes the proof of \refeq{tilR-suffbd}.

Next we prove \refeq{tilQ-suffbd} for $N\ge1$ (hence $s_1>0$).  For $i=1$ and
$q=0,2$, we have
\begin{align}\lbeq{tildeQy1qprebd}
\sum_{y_1,y_2}|y_1|^q\tilde Q^{\sss(N)}_{s_1,s_2}(y_1,y_2)\le\ddsum_{s',s''
 =0}^{s_1}\bigg(\sum_{y_1}|y_1|^qP^{\sss(N)}\big((y_1,s_1);B(s'),B(s'')\big)
 \bigg)\nn\\
\times\bigg(\sup_{z,w}\sum_{y_2}L\big((z,s'),(w,s'');(y_2,s_2)\big)\bigg).
\end{align}
We bound the sum over $y_1$ in the right-hand side by applying
\refeq{constr1-bd-a} for $q=0$ to \refeq{PN-bd}--\refeq{PN-bd-lower}, and bound
the sum over $y_2$ by using \refeq{supL-bd}.  Then, we obtain
\begin{align}\lbeq{tildeQy1qbd}
\refeq{tildeQy1qprebd}&\le\frac{\vep^3}{(1+s_1)^{(d-q)/2}}\,\ddsum_{s',s''=0}
 ^{s_1}\frac{(\delta_{s_1,s'}+\vep C_1)(\delta_{s_1,s''}+\vep C_1)}{(1+s_2
 -s'\wedge s'')^{d/2}}\times
 \begin{cases}
 O(\beta)^{N+1}\sigma^q&(d>4)\\
 O(\beta_\sT)^2O(\hat\beta_\sT)^{N-1}\sigma_\sT^q&(d\le4)
 \end{cases}\nn\\
&\le(1+s_1)^{q/2}\,\tilde b_{s_1,s_2}^{\sss(2)}\times
 \begin{cases}
 O(\beta)^{N+1}\sigma^q&(d>4),\\
 O(\beta_\sT)^2O(\hat\beta_\sT)^{N-1}\sigma_\sT^q&(d\le4).
 \end{cases}
\end{align}
For $i=2$ and $q=2$, we have
\begin{align}\lbeq{tildeQy22prebd}
\sum_{y_1,y_2}|y_2|^2\tilde Q^{\sss(N)}_{s_1,s_2}(y_1,y_2)\le\ddsum_{s',s''
 =0}^{s_1}\,\sum_{\substack{y_1,y_2\\ w,z}}&(|w|^2+|y_2-w|^2)\,P^{\sss(N)}
 \big((y_1,s_1);B(z,s'),B(w,s'')\big)\nn\\
&\times L\big((z,s'),(w,s'');(y_2,s_2)\big),
\end{align}
where, by applying \refeq{constr1-bd-a} to \refeq{PN-bd}--\refeq{PN-bd-lower}
for $q=0$ and using \refeq{supL-bd}, the contribution from $|w|^2$ is bounded
as
\begin{align}\lbeq{tildeQy22prebd-1}
&\ddsum_{s',s''=0}^{s_1}\bigg(\sum_{y_1,w}|w|^2P^{\sss(N)}\big((y_1,s_1);B(s'),
 B(w,s'')\big)\bigg)\bigg(\sup_{z,w}\sum_{y_2}L\big((z,s'),(w,s'');(y_2,s_2)
 \big)\bigg)\nn\\
&\le\frac{(N+1)\vep^3}{(1+s_1)^{d/2}}\,\ddsum_{s',s''=0}^{s_1}\frac{s''
 (\delta_{s_1,s'}+\vep C_1)(\delta_{s_1,s''}+\vep C_1)}{(1+s_2-s'\wedge s'')
 ^{d/2}}\times
 \begin{cases}
 O(\beta)^{N+1}\sigma^2&(d>4),\\
 O(\beta_\sT)^2O(\hat\beta_{\sT})^{N-1}\sigma_{\sT}^2&(d\le4).
 \end{cases}
\end{align}
On the other hand, by using \refeq{PN-bd}--\refeq{PN-bd-lower} for $q=0$ and
\refeq{fbd}--\refeq{tau-ass.1}, the contribution from $|y_2-w|^2$ in
\refeq{tildeQy22prebd} is bounded as
\begin{align}\lbeq{tildeQy22prebd-2}
&\ddsum_{s',s''=0}^{s_1}\bigg(\sum_{y_1}P^{\sss(N)}\big((y_1,s_1);B(s'),B(s'')
 \big)\bigg)\bigg(\sup_{z,w}\sum_{y_2}|y_2-w|^2L\big((z,s'),(w,s'');(y_2,s_2)
 \big)\bigg)\nn\\
&\le\frac{\vep^3}{(1+s_1)^{d/2}}\,\ddsum_{s',s''=0}^{s_1}\frac{(s_2-s'')
 (\delta_{s_1,s'}+\vep C_1)(\delta_{s_1,s''}+\vep C_1)}{(1+s_2-s'\wedge
 s'')^{d/2}}\times
 \begin{cases}
 O(\beta)^{N+1}\sigma^2&(d>4),\\
 O(\beta_\sT)^2O(\hat\beta_{\sT})^{N-1}\sigma_{\sT}^2&(d\le4).
 \end{cases}
\end{align}
Summing \refeq{tildeQy22prebd-1} and \refeq{tildeQy22prebd-2} and absorbing the
factor $N+1$ into the geometric term, we obtain
\begin{align}\lbeq{tildeQy22bd}
\refeq{tildeQy22prebd}&\le s_2\,\tilde b_{s_1,s_2}^{\sss(2)}\times
 \begin{cases}
 O(\beta)^{N+1}\sigma^2&(d>4),\\
 O(\beta_\sT)^2O(\hat\beta_{\sT})^{N-1}\sigma_{\sT}^2&(d\le4).
 \end{cases}
\end{align}
Summarising \refeq{tildeQy1qbd} and \refeq{tildeQy22bd} yields
\refeq{tilQ-suffbd} for $N\ge1$.  This completes the proof of
Lemma~\ref{lem-PQNbds}.
\end{proof}

\subsection{Diagrammatic bounds on \protect$\phi^{\sss(N,N_1,N_2)}(\yvec_1,
 \yvec_2)_{\sss\pm}$}\label{sec-redphi}
In this section, we prove Lemma~\ref{lem-phibdslow}.  First we recall the
convention \refeq{convention} and the definition
\refeq{phipm-def} and \refeq{BdeltaNdef}--\refeq{phialtbd}:
\begin{align} \lbeq{phiNNN}
&\phi^{\sss(N,N_1,N_2)}(\yvec_1,\yvec_2)_{\sss\pm}\nn\\
&\quad=\sum_{\substack{b_{N+1},e\\ b_{N+1}\ne e}}p_{b_{N+1}}p_e\,\tilde
 M_{b_{N+1}}^{\sss(N+1)}\Big(\ind{\Ftwo_{t_{\yvec_1}}(\tb_N,\eb;\bC_\pm)
 \text{ in }\tilde\bC^e_N}\,B_\delta^{\sss(N_1)}(\tb_{\sss N+1},\yvec_1;
 \bC_{\sN})\,B_\delta^{\sss(N_2)}(\te,\yvec_2;\tilde\bC^e_{\sN})\Big),
\end{align}
where we recall $\Ftwo_t(\vvec,\xvec;\bA)=\{\vvec\ct{\bA}\xvec\}\cap\{\nexists
t$-cutting bond for $\vvec\ctx{\bA}\xvec\}$, as defined in \refeq{F'2-1comp},
and $\bC_{\sss+}=\{\tb_{\sN}\}$ and $\bC_{\sss-}=\tilde\bC_{\sss N-1}$.  If the
factors $\ind{\Ftwo_{t_{\yvec_1}}(\tb_N,\eb;\bC_\pm)\text{ in }\tilde\bC^e_N}$
and $B_\delta^{\sss(N_2)}(\te,\yvec_2;\tilde\bC^e_{\sN})$ were absent, then
\refeq{phiNNN} would simplify to
$\pi^{\sss(N+N_1)}(\yvec_1)\le P^{\sss(N+N_1)}(\yvec_1)$.
Therefore, our task is to investigate the effect of these changes.

We will prove Lemma~\ref{lem-phibdslow} using the following three lemmas:

\begin{lemma}\label{lem-inclG}
For $\vvec,\xvec\in\Lambda$ and $t_{\vvec}<t\le t_{\xvec}$, (cf.,
Figure~\ref{fig-VEcal})
    \begin{align}\lbeq{H1incl}
    H_t(\vvec,\xvec;\{\vvec\})\subset V_{t-\vep}(\vvec,\xvec)\equiv
     \bigcup_{\zvec:t_{\zvec}\le t-\vep}\{\vvec\conn\zvec\dbc\xvec\}.
    \end{align}
Moreover, for $\bA\subset\Lambda$, let
    \begin{align}\lbeq{Gi1def}
    G_t^{\sss(1)}(\vvec,\xvec;\bA)=H_t(\vvec,\xvec;\bA)\cap V_{t-\vep}(\vvec,
     \xvec),&&
    G_t^{\sss(2)}(\vvec,\xvec;\bA)=H_t(\vvec,\xvec;\bA)\setminus V_{t-\vep}
     (\vvec,\xvec).
    \end{align}
Then,
\begin{align}\lbeq{Ecaldef}
G_t^{\sss(1)}(\vvec,\xvec;\bA)\subseteq V_{t-\vep}(\vvec,\xvec),&&
G_t^{\sss(2)}(\vvec,\xvec;\bA)\subseteq{\cal E}_t(\vvec,\xvec;\bA),
\end{align}
where
\begin{align}\lbeq{Ecaldefrep}
{\cal E}_t(\vvec,\xvec;\bA)&=\bigcup_{\avec,\wvec\in\bA}~\bigcup_{\substack{
 \zvec\in\Lambda\\ t_{\zvec}\ge t}}\bigg\{\Big\{\{\vvec\conn
 \zvec\}\circ\{\zvec\conn\wvec\}\circ\{\wvec\conn
 \xvec\}\circ\{\zvec\conn\xvec\}\Big\}\nn\\
&\hskip6pc\cap\Big\{\{\avec=\wvec,~\zvec\centernot\conn\wvec_-\}\cup\{\avec\ne
 \wvec_-,~(\avec,\wvec)\in\bA\}\Big\}\bigg\}.
\end{align}
\end{lemma}

%%%%%FIGFIGFIGFIGFIGFIGFIGFIGFIGFIGFIGFIGFIGFIGFIGFIGFIGFIGFIGFIG
%%%%%FIGFIGFIGFIGFIGFIGFIGFIGFIGFIGFIGFIGFIGFIGFIGFIGFIGFIGFIGFIG
\begin{figure}[t]
\begin{center}
\setlength{\unitlength}{0.0005in}
{
\begin{picture}(6000,3000)(3000,1200)
{ \put(4200,2900){\makebox(0,0)[lb]{$t$}}
\qbezier[50](2700,3000)(3400,3000)(4100,3000)

\thinlines
\path(3400,900)(3400, 2700)
\put(3400, 3000){\ellipse{300}{600}}
\put(3330,3400){\makebox(0,0)[lb]{$\xvec$}}
\put(3330,700){\makebox(0,0)[lb]{$\vvec$}}

%Second figure
\put(9130,2900){\makebox(0,0)[lb]{$t$}}

\put(8430,700){\makebox(0,0)[lb]{$\vvec$}}
\put(8620,3500){\makebox(0,0)[lb]{{$\sss \bullet$}}}
\put(8800,3450){\makebox(0,0)[lb]{${\scriptstyle\avec\in\bA}$}}

\thinlines
\put(8500, 3500){\ellipse{300}{600}}

\path(8500,900)(8500,3200) \put(8430,3900){\makebox(0,0)[lb]{$\xvec$}}
\qbezier[50](7600,3000)(8300,3000)(9000,3000) }
\end{picture}
}
\end{center}
\caption{\label{fig-VEcal}Schematic representations of the events
(a)~$V_{t-\vep}(\vvec,\xvec)$ and (b)~$\Ecal_t(\vvec,\xvec;\bA)$.}
\end{figure}
%%%%%FIGFIGFIGFIGFIGFIGFIGFIGFIGFIGFIGFIGFIGFIGFIGFIGFIGFIGFIG
%%%%%FIGFIGFIGFIGFIGFIGFIGFIGFIGFIGFIGFIGFIGFIGFIGFIGFIGFIGFIG

\begin{lemma}\label{lem-remtilde}
Let $X$ be a non-negative random variable which is independent of the
occupation status of the bond $b$, while $F$ is an increasing event.  Then,
    \begin{align}
    \tilde\mE^b[X\indic_{F}]\le\mE[X\indic_F].
    \end{align}
\end{lemma}

%Before stating Lemma \ref{lem-bdsphi}, we introduce a final construction.
%Indeed, for a bond $b$, Construction $\ell(b)$ is the result of
%Construction $\ell(\bb)$ followed by Construction $\ell_{\bb}(\tb)$.
%Construction $\ell(b)$ is equivalent to Construction $\ell(\bb)$ followed
%by a multiplication with $p(b)$.

\begin{lemma}\label{lem-bdsphi}
Let $\yvec_1,\yvec_2\in\Lambda$ and $\vec\xvec\in\Lambda^j$ for some $j\ge0$.
For $N,N_1\ge0$,
    \begin{align}\lbeq{Vdiagrbd}
    \sum_{b_{N+1}}p_{b_{N+1}}M^{\sss(N+1)}_{\bb_{N+1}}\Big(\indic_{V_{t_{\yvec_1}
     -\vep}(\tb_N,\yvec_2)\,\cap\,\{\vec\xvec\in\tilde\bC_N\}}B_\delta^{\sss(N_1)}
     (\tb_{\sss N+1},\yvec_1;\tilde\bC_{\sss N})\Big)\le\sum_{b:\tb=
     \yvec_1}R^{\sss(N+N_1)}(\bb,\yvec_2;\ell(\vec\xvec))\,p_b,
    \end{align}
and for $N\ge1$ and $N_1\ge0$,
    \begin{align}\lbeq{Ediagrbd}
    \sum_{b_{N+1}}p_{b_{N+1}}M^{\sss(N+1)}_{\bb_{N+1}}\Big(\indic_{\Ecal_{t_{
     \yvec_1}}(\tb_N,\yvec_2;\tilde\bC_{N-1})\,\cap\,\{\vec\xvec\in\tilde\bC_N
     \}}B_\delta^{\sss(N_1)}(\tb_{\sss N+1},\yvec_1;\tilde\bC_{\sss N})\Big)
     \le\sum_{b:\tb=\yvec_1}Q^{\sss(N+N_1)}(\bb,\yvec_2;\ell(\vec\xvec))\,p_b.
    \end{align}
%The bounds in \refeq{Vdiagrbd} and \refeq{Ediagrbd} are also valid
%when $\{\vec\xvec\in\tilde\bC_N\}$ on the left-hand sides are replaced
%by $\{b\in\tilde\bC_N\}\cap \{\vec\xvec\in\tilde\bC_N\}$, and
%Construction~$\ell(\vec\xvec)$ on the right-hand sides are
%replaced by Construction~$\ell(b)$ followed
%by Construction~$\ell(\vec\xvec)$.
\end{lemma}

The remainder of this subsection is organised as follows.  In
Section~\ref{sec-phibd1}, we prove Lemma~\ref{lem-phibdslow} assuming
Lemmas~\ref{lem-inclG}--\ref{lem-bdsphi}.  Lemma~\ref{lem-inclG} is an
adaptation of \cite[Lemmas~7.15 and 7.17]{HHS05b} for oriented percolation,
which applies here as the discretized contact process is an oriented
percolation model. The origin of the event
$\{\zvec\centernot\conn\wvec_-\}\cup\{\wvec_-\notin\bA\}$ in
\refeq{Ecaldefrep} is similar to the intersection with the second line
in \refeq{E'bd}, for which we refer to the proof of \refeq{E'bd}.
Lemma~\ref{lem-remtilde} is identical to \cite[Lemma~7.16]{HHS05b}.
We omit the proofs of these two lemmas.  In Section~\ref{sec-pfProp67},
we prove Lemma~\ref{lem-bdsphi}.

\subsubsection{Proof of Lemma~\ref{lem-phibdslow} assuming
Lemmas~\ref{lem-inclG}--\ref{lem-bdsphi}}\label{sec-phibd1}

\begin{proof}[Proof of Lemma~\ref{lem-phibdslow} for $N_2=0$]
First we prove the bound on $\phi^{\sss(N,N_1,0)}(\yvec_1,\yvec_2)_{\sss+}$,
where, by \refeq{tildeM-def} and \refeq{tildeMN-def},
\begin{align}\lbeq{tobeboundedphi1}
&\phi^{\sss(N,N_1,0)}(\yvec_1,\yvec_2)_{\sss+}=\sum_{\substack{b_{N+1},e\\
 b_{N+1}\ne e}}p_{b_{N+1}}p_e\,\tilde M_{b_{N+1}}^{\sss(N+1)}\Big(\indic_{
 \Ftwo_{t_{\yvec_1}}(\tb_N,\eb;\{\tb_N\})}B_\delta^{\sss(N_1)}(\tb_{\sss
 N+1},\yvec_1;\bC_{\sN})\Big)\,\delta_{\te,\yvec_2}\nn\\
&=\sum_{\substack{b_N,b_{N+1},e\\ b_{N+1}\ne e}}p_{b_N}p_{b_{N+1}}p_e\,
 M_{\bb_N}^{\sss(N)}\bigg(\tilde\mE^{b_{N+1}}\Big[\indic_{E'(\tb_N,\bb_{
 N+1};\tilde\bC_{N-1})}\indic_{\Ftwo_{t_{\yvec_1}}(\tb_N,\eb;\{\tb_N\})}
 B_\delta^{\sss(N_1)}(\tb_{\sss N+1},\yvec_1;\bC_{\sN})\Big]\bigg)\,
 \delta_{\te,\yvec_2}.
\end{align}
Note that, by Lemma~\ref{lem-inclG}, $H_{t_{\yvec_1}}(\tb_{\sss
N},\eb;\{\tb_{\sss N}\})$ is a subset of $V_{t_{\yvec_1}-\vep}(\tb_{\sss
N},\eb)$, which is an increasing event.  We also note that the event
$E'(\tb_{\sss N},\bb_{\sss N+1};\tilde\bC_{\sss N-1})$ and the random variable
$B_{\delta}^{\sss(N_1)}(\tb_{\sss N+1},\yvec_1;\tilde\bC_{\sss N})$, where
$\tilde\bC_{\sss N}=\tilde\bC^{b_{\sss N+1}}(\tb_{\sss N})$, are independent of
the occupation status of $b_{\sss N+1}$.  By Lemma~\ref{lem-remtilde} and using
\refeq{B0def} and \refeq{M-def}, we obtain
\begin{align}
\refeq{tobeboundedphi1}&\le\sum_{\substack{b_N,b_{N+1},e\\ b_{N+1}\ne e}}
 p_{b_N}p_{b_{N+1}}p_e\,M_{\bb_N}^{\sss(N)}\bigg(\mE\Big[\indic_{E'(\tb_N,
 \bb_{N+1};\tilde\bC_{N-1})}\indic_{V_{t_{\yvec_1}-\vep}(\tb_{\sss N},\eb)}
 B_\delta^{\sss(N_1)}(\tb_{\sss N+1},\yvec_1;\bC_{\sN})\Big]\bigg)\,
 \delta_{\te,\yvec_2}\nn\\
&=\sum_{\substack{b_{N+1},e\\ b_{N+1}\ne e}}p_{b_{N+1}}p_eM_{\bb_{N+1}}
 ^{\sss(N+1)}\Big(\indic_{V_{t_{\yvec_1}-\vep}(\tb_{\sss N},\eb)}B_\delta
 ^{\sss(N_1)}(\tb_{\sss N+1},\yvec_1;\bC_{\sN})\Big)\,\delta_{\te,\yvec_2}.
 \lbeq{remtildephi1}
\end{align}
The bound \refeq{phi1bd} for $N_2=0$ now follows from Lemma~\ref{lem-bdsphi}.

Next we prove the bound on $\phi^{\sss(N,N_1,0)}(\yvec_1,\yvec_2)_{\sss-}$,
where, similarly to \refeq{tobeboundedphi1},
\begin{align}\lbeq{remtildephi2}
&\phi^{\sss(N,N_1,0)}(\yvec_1,\yvec_2)_{\sss-}\\%=\sum_{\substack{b_{N+1},
% e\\ b_{N+1}\ne e}}p_{b_{N+1}}p_e\,\tilde M_{b_{N+1}}^{\sss(N+1)}\Big(
% \indic_{\Ftwo_{t_{\yvec_1}}(\tb_N,\eb;\tilde\bC_{N-1})}B_\delta^{\sss
% (N_1)}(\tb_{\sss N+1},\yvec_1;\bC_{\sN})\Big)\,\delta_{\te,\yvec_2}
&=\sum_{\substack{b_N,b_{N+1},e\\ b_{N+1}\ne e}}p_{b_N}p_{b_{N+1}}p_e\,
 M_{\bb_N}^{\sss(N)}\bigg(\tilde\mE^{b_{N+1}}\Big[\indic_{E'(\tb_N,\bb_{
 N+1};\tilde\bC_{N-1})}\indic_{\Ftwo_{t_{\yvec_1}}(\tb_N,\eb;\tilde
 \bC_{N-1})}B_\delta^{\sss(N_1)}(\tb_{\sss N+1},\yvec_1;\bC_{\sN})\Big]
 \bigg)\,\delta_{\te,\yvec_2}.\nn
\end{align}
By \refeq{Gi1def}, we have the partition
\begin{align}\lbeq{Hisplit}
H_{t_{\yvec_1}}(\tb_{\sss N},\eb;\tilde\bC_{\sss N-1})=G_{t_{\yvec_1}}
 ^{\sss(1)}(\tb_{\sss N},\eb;\tilde\bC_{\sss N-1})\DDcup~G_{t_{\yvec_1}}
 ^{\sss(2)}(\tb_{\sss N},\eb;\tilde\bC_{\sss N-1}).
\end{align}
See Figure~\ref{fig-5} for schematic representations of the events
$E'(\tb_{\sss N},\bb_{\sss N+1};\tilde\bC_{\sss N-1})\,\cap\,
G_{t_{\yvec_1}}^{\sss(i)}(\tb_{\sss N},\eb;\tilde\bC_{\sss N-1})$ for
$i=1,2$.  By Lemma~\ref{lem-inclG}, we have
\begin{align}\lbeq{inclG2V}
\indic_{E'(\tb_N,\bb_{N+1};\tilde\bC_{N-1})}\indic_{G_{t_{\yvec_1}}
 ^{\sss(1)}(\tb_{\sss N},\eb;\tilde\bC_{\sss N-1})}\le\indic_{
 E'(\tb_N,\bb_{N+1};\tilde\bC_{N-1})}\indic_{V_{t_{\yvec_1}-\vep}
 (\tb_{\sss N},\eb)},
\end{align}
so that, by \refeq{remtildephi1}, the contribution from
$G_{t_{\yvec_1}}^{\sss(1)}(\tb_{\sss N},\eb;\tilde\bC_{\sss N-1})$ obeys the
same bound as $\phi^{\sss(N,N_1,0)}(\yvec_1,\yvec_2)_{\sss+}$, which is the
term in \refeq{phi3bd} proportional to $R^{\sss(N+N_1,0)}$.

%%%%%FIGFIGFIGFIGFIGFIGFIGFIGFIGFIGFIGFIGFIGFIGFIGFIGFIGFIGFIGFIG
%%%%%FIGFIGFIGFIGFIGFIGFIGFIGFIGFIGFIGFIGFIGFIGFIGFIGFIGFIGFIGFIG
\begin{figure}[t]
\begin{center}
\setlength{\unitlength}{0.0005in} {
\begin{picture}(6000,3700)(2000,700)
{\put(3500, 1700){\ellipse{200}{400}} \path(3500,1350)(3500,1500) \put(3500,
1150){\ellipse{200}{400}} \put(4130,2900){\makebox(0,0)[lb]{$t_{\yvec_1}$}}

\put(3430,600){\makebox(0,0)[lb]{$\tb_{\sss N}$}}
\put(1930,800){\makebox(0,0)[lb]{$\tilde\bC_{\sss N-1}$}}

\put(3830,1850){\makebox(0,0)[lb]{$b_{\sss N+1}$}}

\qbezier[90](1000,3000)(2500,3000)(4000,3000)

\Thicklines \qbezier(4200,500)(4200,1350)(3350,1830)
\qbezier(2800,500)(2600,1350)(3200,2130)

\path(3400,1900)(3600,1900) \path(3400,2050)(3600,2050)

\thinlines \path(3400,1750)(2590, 2760) \put(2300, 3000){
\begin{rotate}{130}
\ellipse{600}{300}
\end{rotate}
}
%\path(2100,3175)(2250, 3305)
%\path(2050,3225)(2200, 3355)
\put(1870,3200){\makebox(0,0)[lb]{$\eb$}}
%\qbezier(2125,3290)(1600,4000)(1500,5000)

%\path(1000,5000)(4000,5000)
%\put(4050,4950){\makebox(0,0)[lb]{$n$}}

%Second figure
\put(8500, 1700){\ellipse{200}{400}} \path(8500,1350)(8500,1500) \put(8500,
1150){\ellipse{200}{400}} \put(9130,2900){\makebox(0,0)[lb]{$t_{\yvec_1}$}}

\put(8430,600){\makebox(0,0)[lb]{$\tb_{\sss N}$}}
\put(6530,800){\makebox(0,0)[lb]{$\tilde\bC_{\sss N-1}$}}

\put(8830,1850){\makebox(0,0)[lb]{$b_{\sss N+1}$}}

\qbezier[90](6000,3000)(7500,3000)(9000,3000)

\Thicklines \qbezier(9200,600)(9200,1350)(8400,1830)
\path(8400,1900)(8600,1900) \path(8400,2050)(8600,2050)
\qbezier(7800,600)(6800,1350)(6900,3830)
%\path(6625,3710)(6775, 3840)
%\path(6675,3665)(6825, 3790)
\thinlines \put(6880, 3500){
\begin{rotate}{130}
\ellipse{600}{300}
\end{rotate}
} \path(8400,1750)(7133,3265) \put(6370,3700){\makebox(0,0)[lb]{$\eb$}}
%\qbezier(6715,3790)(6400,4300)(6500,5000)
%\path(6000,5000)(9000,5000)
%\put(9050,4950){\makebox(0,0)[lb]{$n$}}

}
\end{picture}
}
\end{center}
\caption{\label{fig-5} Schematic representations of the events
(a)~$E'(\tb_{\sss N}, \bb_{\sss N+1};\tilde\bC_{\sss N-1})\cap
G_{t_{\yvec_1}}^{\sss(1)}(\tb_{\sss N},\eb;\tilde\bC_{\sss N-1})$ and
(b)~$E'(\tb_{\sss N},\bb_{\sss N+1}; \tilde\bC_{\sss N-1})\cap
G_{t_{\yvec_1}}^{\sss(2)}(\tb_{\sss N},\eb; \tilde\bC_{\sss N-1})$.}
\end{figure}
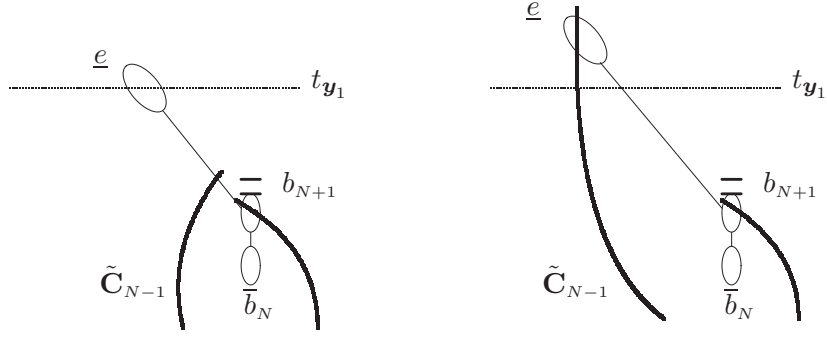
%%%%%FIGFIGFIGFIGFIGFIGFIGFIGFIGFIGFIGFIGFIGFIGFIGFIGFIGFIGFIG
%%%%%FIGFIGFIGFIGFIGFIGFIGFIGFIGFIGFIGFIGFIGFIGFIGFIGFIGFIGFIG

For the contribution to $\phi^{\sss(N,N_1,0)}(\yvec_1,\yvec_2)_{\sss-}$ from
$G_{t_{\yvec_1}}^{\sss(2)}(\tb_{\sss N},\eb;\tilde\bC_{\sss N-1})$, we can
assume that $N\ge1$ because $G^{\sss(2)}_{t_{\yvec_1}}(\tb_{\sss0},\eb;
\bC_{\sss-1})=\vno$ when $N=0$ (cf., \refeq{convention}).  Note that, by
Lemma~\ref{lem-inclG}, $G_{t_{\yvec_1}}^{\sss(2)}(\tb_{\sss N},\eb;\tilde
\bC_{\sss N-1})$ is a subset of ${\cal E}_{t_{\yvec_1}}(\tb_{\sss N},
\eb;\tilde\bC_{\sss N-1})$, which is an increasing event. Therefore, similarly
to the analysis in \refeq{remtildephi1}, we use Lemma~\ref{lem-remtilde} to
obtain
\begin{align}\lbeq{remtildephi3}
&\sum_{\substack{b_N,b_{N+1},e\\ b_{N+1}\ne e}}p_{b_N}p_{b_{N+1}}p_e
 \,M_{\bb_N}^{\sss(N)}\bigg(\tilde\mE^{b_{N+1}}\Big[\indic_{E'(\tb_N,
 \bb_{N+1};\tilde\bC_{N-1})}\indic_{G_{t_{\yvec_1}}^{\sss(2)}(\tb_{\sN},
 \eb;\tilde\bC_{\sN-1})}B_\delta^{\sss(N_1)}(\tb_{\sss N+1},\yvec_1;
 \bC_{\sN})\Big]\bigg)\,\delta_{\te,\yvec_2}\nn\\
&\le\sum_{\substack{b_N,b_{N+1},e\\ b_{N+1}\ne e}}
 p_{b_N}p_{b_{N+1}}p_e\,M_{\bb_N}^{\sss(N)}\bigg(\mE\Big[\indic_{E'(\tb_N,
 \bb_{N+1};\tilde\bC_{N-1})}\indic_{{\cal E}_{t_{\yvec_1}}(\tb_{\sss N},
 \eb;\tilde\bC_{\sss N-1})}B_\delta^{\sss(N_1)}(\tb_{\sss N+1},\yvec_1;
 \bC_{\sN})\Big]\bigg)\,\delta_{\te,\yvec_2}\nn\\
&=\sum_{\substack{b_{N+1},e\\ b_{N+1}\ne e}}p_{b_{N+1}}p_eM_{\bb_{N+1}}
 ^{\sss(N+1)}\Big(\indic_{{\cal E}_{t_{\yvec_1}}(\tb_{\sss N},\eb;\tilde
 \bC_{\sss N-1})}B_\delta^{\sss(N_1)}(\tb_{\sss N+1},\yvec_1;\bC_{\sN})
 \Big)\,\delta_{\te,\yvec_2}.
\end{align}
The bound \refeq{phi3bd} for $N_2=0$ now follows from Lemma~\ref{lem-bdsphi}.
This completes the proof of Lemma~\ref{lem-phibdslow} for $N_2=0$.
\end{proof}

\begin{proof}[Proof of Lemma~\ref{lem-phibdslow} for $N_2\ge1$]
First we prove the bound on $\phi^{\sss(N,N_1,1)}(\yvec_1,\yvec_2)_{\sss+}$,
where, by \refeq{phiNNN}--\refeq{H1incl}, \refeq{BdeltaNdef} and \refeq{piPbd},
\begin{align}\lbeq{varphiA1repp}
\phi^{\sss(N,N_1,1)}(\yvec_1,\yvec_2)_{\sss+}\le\sum_{\substack{b_{N+1},e\\
 b_{N+1}\ne e}}p_{b_{N+1}}p_e\,\tilde M_{b_{N+1}}^{\sss(N+1)}\Big(\indic_{
 V_{t_{\yvec_1}-\vep}(\tb_N,\eb)}\,B_\delta^{\sss(N_1)}(\tb_{\sss N+1},
 \yvec_1;\bC_{\sN})\,B_\delta^{\sss(0)}(\te,\yvec_2;\tilde\bC^e_{\sN})\Big).
\end{align}
Following the argument around \refeq{MP1pfafirst}--\refeq{MP1pfa} , we have
\begin{align}\lbeq{varphiA1rep}
\refeq{varphiA1repp}\le\sum_{\substack{b_{N+1},e,e'\\ \te'=\yvec_2}}\sum_{
 \cvec}p_{b_{N+1}}M_{\bb_{N+1}}^{\sss(N+1)}\Big(\indic_{V_{t_{\yvec_1}-
 \vep}(\tb_N,\eb)\,\cap\,\{\cvec\in\tilde\bC^e_{\sN}\}}\,B_\delta^{\sss(N_1)}
 (\tb_{\sss N+1},\yvec_1;\tilde\bC_{\sN})\Big)\,p_eP^{\sss(0)}(\te,\eb';\cvec)\,
 p_{e'},
\end{align}
where $\tilde\bC_{\sss N}=\tilde\bC^{b_{\sss N+1}}(\tb_{\sss N})$.
%By \refeq{P0Cdef}, \refeq{S0C-def} and \refeq{S00-def}--\refeq{S00-def},
%there are two contributions to $P^{\sss(0)}(\te,\eb';\tilde\bC^e_{\sN})$,
%namely, one containing an indicator of $\cvec\in \tilde\bC^e_{\sN}$ and
%one containing an indicator $(\cvec, \wvec)\in \tilde\bC^e_{\sN}$.
By \refeq{Vdiagrbd} with $\vec\xvec=\cvec$
%, resp.\ $\vec\xvec=\varnothing$ and $b=(\cvec, \wvec)$,
and \refeq{tilFzrep}, we obtain
\begin{align}\lbeq{varphiA1rep2}
\refeq{varphiA1rep}&\le\sum_{\substack{b:\tb=\yvec_1\\ e':\te'=\yvec_2}}
 p_b\,p_{e'}\underbrace{\sum_\twop\sum_{\cvec}\sum_eR^{\sss(N+N_1)}(\bb,
 \eb;\ell^\twop(\cvec))\,p_eP^{\sss(0)}(\te,\eb';\cvec)}_{R^{\sss(N+N_1)}
 \big(\bb,\yvec;2_{\yvec}^{\sss(1)}(\cvec),2_{\cvec}^{\sss(0)}(\eb')\big)}
 =\sum_{\substack{b:\tb=\yvec_1\\ e':\te'=\yvec_2}}p_b\,p_{e'}
 R^{\sss(N+N_1)}(\bb,\yvec;E_{\yvec}(\eb')).
\end{align}
%Here we note that the Construction~$\ell((\cvec, \wvec))$ in
%\refeq{Vdiagrbd} for $\vec\xvec=\varnothing$ and $b=(\cvec, \wvec)$ produces
%the factor $\lambda \vep D(\wvec-\cvec)$ that is present in \refeq{S0-def}.
This shows that
\begin{align}\lbeq{varphiA1rep3}
\phi^{\sss(N,N_1,1)}(\yvec_1,\yvec_2)_{\sss+}\le\sum_{\uvec_1,\uvec_2}
 p_\vep(\yvec_1-\uvec_1)\,p_\vep(\yvec_2-\uvec_2)\,R^{\sss(N+N_1,1)}
 (\uvec_1,\uvec_2),
\end{align}
as required.

To extend the proof of \refeq{phi1bd} to all $N_2$, we estimate
$B_{\delta}^{\sss (N_2)}(\te, \yvec_2;\tilde{\bC}_{\sss N}^e)$ using
\refeq{piPbd}. Since the bound on $B_{\delta}^{\sss(N_2)}(\te,\yvec_2;
\tilde\bC_{\sN}^e)$ is the same as $N_2-1$ applications of Construction~$E$
to $P^{\sss(0)}(\te,\uvec_2;\tilde\bC^e_{\sN})$, the bound follows by the
definition of $R^{\sss(N+N_1,N_2)}(\yvec_1,\yvec_2)$.

The proof of \refeq{phi3bd} for
$\phi^{\sss(N,N_1,N_2)}(\yvec_1,\yvec_2)_{\sss-}$ proceeds similarly,
when we use \refeq{Ediagrbd} rather than \refeq{Vdiagrbd}.  This
completes the proof of Lemma~\ref{lem-phibdslow}.
\end{proof}

\subsubsection{Proof of Lemma~\ref{lem-bdsphi}}\label{sec-pfProp67}
\begin{proof}[Proof of Lemma~\ref{lem-bdsphi} for $N_1=0$]
Since $B_\delta^{\sss(0)}(\tb_{\sss N+1},\yvec_1;\tilde\bC_{\sss
N})=\delta_{\tb_{N+1},\yvec_1}$, the sums over $b_{\sss N+1}$ on the left-hand
sides of \refeq{Vdiagrbd}--\refeq{Ediagrbd} are identical to the sums on the
right-hand sides over $b$ with $\tb=\yvec_1$.  We also note that
$t_{\yvec_1}=t_{\bb_{N+1}}+\vep$ in this case.  By the definitions of
$R^{\sss(N)}$ and $Q^{\sss(N)}$ in \refeq{PN0def}--\refeq{QN-def}, to prove
Lemma~\ref{lem-bdsphi} for $N_1=0$ it suffices to show
\begin{align}
M_{\bb_{N+1}}^{\sss(N+1)}\Big(\indic_{V_{t_{\bb_{N+1}}}(\tb_N,\yvec_2)\,\cap\,
 \{\vec\xvec\in\tilde\bC_N\}}\Big)&\le P^{\sss(N)}\big(\bb_{\sN+1};V_{t_{\bb_{N
 +1}}}(\yvec_2),\ell(\vec\xvec)\big)\qquad(N\ge0),\lbeq{Vdiagrbd:N1=0}\\
M_{\bb_{N+1}}^{\sss(N+1)}\Big(\indic_{\Ecal_{t_{\bb_{N+1}}+\vep}(\tb_N,\yvec_2;
 \tilde\bC_{N-1})\,\cap\,\{\vec\xvec\in\tilde\bC_N\}}\Big)&\le P^{\sss(N)}\big(
 \bb_{\sN+1};\Ecal_{t_{\bb_{N+1}}}(\yvec_2),\ell(\vec\xvec)\big)\qquad(N\ge1).
 \lbeq{Ediagrbd:N1=0}
\end{align}
%and a similar bound with $\{\vec\xvec\in\tilde\bC_{\sN}\}$ on the left-hand
%side replaced by $\{b\in\tilde\bC_{\sN}\}\cap\{\vec\xvec\in\tilde\bC_{\sN}\}$
%and with Construction~$\ell(\vec\xvec)$ on the right-hand side replaced by
%Construction~$\ell(b)$ followed with Construction~$\ell(\vec\xvec)$.  Below,
%we shall only deal with the cases in
%\refeq{Vdiagrbd:N1=0}--\refeq{Ediagrbd:N1=0}, the extensions to deal with
%$\{b\in\tilde\bC_{\sN}\}$ are minor and will be omitted.
%
By the nested structure of $M^{\sss(N+1)}_{\bb_{N+1}}$ (cf.,
\refeq{MN-rec}),
    \begin{align}
    \text{LHS of \refeq{Vdiagrbd:N1=0}}&=\sum_{b_N}p_{b_N}M^{\sss(N)}_{\bb_N}\bigg(
     M_{\tb_N,\bb_{N+1};\tilde\bC_{N-1}}^{\sss(1)}\Big(\indic_{V_{t_{\bb_{N+1}}}
     (\tb_N,\yvec_2)\,\cap\,\{\vec\xvec\in\tilde\bC_N\}}\Big)\bigg),\\
    \text{LHS of \refeq{Ediagrbd:N1=0}}&=\sum_{b_{N-1}}p_{b_{N-1}}M^{\sss(N-1)}_{
     \bb_{N-1}}\bigg(M_{\tb_{N-1},\bb_{N+1};\tilde\bC_{N-2}}^{\sss(2)}\Big(\indic_{
     \Ecal_{t_{\bb_{N+1}}+\vep}(\tb_N,\yvec_2;\tilde\bC_{N-1})\,\cap\,\{\vec\xvec
     \in\tilde\bC_N\}}\Big)\bigg).
    \end{align}
On the other hand, by the recursive definition of $P^{\sss(N)}$ (cf.,
\refeq{recPNM}),
    \begin{align}
    \text{RHS of \refeq{Vdiagrbd:N1=0}}&=\sum_{b_N}p_{b_N}\sum_{\cvec}P^{\sss(N
     -1)}(\bb_{\sss N};\ell(\cvec))\,P^{\sss(0)}\big(\tb_{\sss N},\bb_{\sss N+1}
     ;\cvec,V_{t_{\bb_{N+1}}}(\yvec_2),\ell(\vec\xvec)\big),
     \lbeq{Vdiagrbd:N1=0rhs}\\
    \text{RHS of \refeq{Ediagrbd:N1=0}}&=\sum_{b_{N-1}}p_{b_{N-1}}\sum_{\cvec}
     P^{\sss(N-2)}(\bb_{\sss N-1};\ell(\cvec))\,P^{\sss(1)}\big(\tb_{\sss N-1},
     \bb_{\sss N+1};\cvec,\Ecal_{t_{\bb_{\sss N+1}}}(\yvec_2),\ell(\vec\xvec)
     \big),\lbeq{Ediagrbd:N1=0rhs}
    \end{align}
where Construction~$\ell(\cvec)$ in \refeq{Vdiagrbd:N1=0rhs} is applied to
the $(N-1)^\text{th}$ admissible lines of $P^{\sss(N)}$ and that in
\refeq{Ediagrbd:N1=0rhs} is applied to the $(N-2)^\text{th}$ admissible
lines.  By comparing the above expressions and following the argument around
\refeq{MP1pfafirst}--\refeq{MP1pfa}, it thus suffices
to prove
    \begin{align}
    M_{\tb_N,\bb_{N+1};\tilde\bC_{N-1}}^{\sss(1)}\Big(\indic_{V_{t_{\bb_{N+1}}}
     (\tb_N,\yvec_2)\,\cap\,\{\vec\xvec\in\tilde\bC_N\}}\Big)&\le P^{\sss(0)}\big(
     \tb_{\sss N},\bb_{\sss N+1};\tilde\bC_{\sss N-1},V_{t_{\bb_{N+1}}}(\yvec_2),
     \ell(\vec\xvec)\big),\lbeq{Vdiagrbd:N1=0rewr}\\
    M_{\tb_{N-1},\bb_{N+1};\tilde\bC_{N-2}}^{\sss(2)}\Big(\indic_{\Ecal_{t_{\bb_{N
     +1}}+\vep}(\tb_N,\yvec_2;\tilde\bC_{N-1})\,\cap\,\{\vec\xvec\in\tilde\bC_N\}}
     \Big)&\le P^{\sss(1)}\big(\tb_{\sss N-1},\bb_{\sss N+1};\tilde\bC_{\sss N-2},
     \Ecal_{t_{\bb_{\sss N+1}}}(\yvec_2),\ell(\vec\xvec)\big).
     \lbeq{Ediagrbd:N1=0rewr}
    \end{align}

First we prove \refeq{Vdiagrbd:N1=0rewr}.  Note that, by \refeq{B0def},
    \begin{align}\lbeq{Vincl}
    \text{LHS of \refeq{Vdiagrbd:N1=0rewr}}=\mP\Big(E'(\tb_{\sN},\bb_{\sss N+1};
     \tilde\bC_{\sss N-1})\cap V_{t_{\bb_{N+1}}}(\tb_{\sN},\yvec_2)\cap\{\vec
     \xvec\in\tilde\bC_{\sN}\}\Big).
    \end{align}
Using \refeq{H1incl}, we obtain
    \begin{align}
    V_{t_{\bb_{N+1}}}(\tb_{\sss N},\yvec_2)\subseteq\bigcup_{\vvec:t_{\vvec}=t_{
     \bb_{N+1}}}\bigcup_{\zvec}\Big\{\{\tb_{\sN}\conn\zvec\}\circ\{\zvec\conn
     \vvec\}\circ\{\vvec\conn\yvec_2\}\circ\{\zvec\conn\yvec_2\}\Big\},
    \end{align}
hence
\begin{align}\lbeq{Vincl-bd}
\refeq{Vincl}\le\sum_{\vvec:t_{\vvec}=t_{\bb_{N+1}}}\mP\bigg(&E'(\tb_{\sN},
 \bb_{\sss N+1};\tilde\bC_{\sss N-1})\cap\{\vec\xvec\in\tilde\bC_{\sN}\}\nn\\
&\quad\cap\bigcup_{\zvec}\Big\{\{\tb_{\sN}\conn\zvec\}\circ\{\zvec\conn\vvec\}
 \circ\{\vvec\conn\yvec_2\}\circ\{\zvec\conn\yvec_2\}\Big\}\bigg).
\end{align}
The event $E'(\tb_{\sN},\bb_{\sss N+1};\tilde\bC_{\sss N-1})$ implies that
there are disjoint connections necessary to obtain the bounding diagram
$P^{\sss(0)}(\tb_{\sss N},\bb_{\sss N+1};\tilde\bC_{\sss N-1})$.  The event
$\{\tb_{\sN}\conn\vvec\}$ ($=\bigcup_{\zvec}
\{\{\tb_{\sN}\conn\zvec\}\circ\{\zvec\conn\vvec\}\}$) can be accounted for by
an application of Construction~$\ell(\vvec)$, and then $\{\vvec\conn\yvec_2\}
\circ\{\zvec\conn\yvec_2\}$ can be accounted for by an application of
Construction~$2^{\sss(0)}_{\vvec}(\yvec_2)$.  The event
$\{\vec\xvec\in\tilde\bC_{\sN}\}$ implies additional connections, accounted for
by an application of Construction~$\ell(\vec\xvec)$.  By \refeq{PN0def}, this
completes the proof of \refeq{Vdiagrbd:N1=0rewr}.

Next, we prove \refeq{Ediagrbd:N1=0rewr}.  Note that, by \refeq{M-def},
\begin{align}\lbeq{Ediagrbd:N1=0rewr-bd1}
\text{LHS of \refeq{Ediagrbd:N1=0rewr}}=\sum_{b_N}p_{b_N}M_{\tb_{N-1},\bb_N;
 \tilde\bC_{N-2}}^{\sss(1)}\bigg(\mP\Big(E'(\tb_{\sN},\bb_{\sN+1};\tilde
 \bC_{\sN-1})\cap\Ecal_{t_{\bb_{N+1}}+\vep}(\tb_{\sN},\yvec_2;\tilde\bC_{\sN
 -1})\cap\{\vec\xvec\in\tilde\bC_{\sN}\}\Big)\bigg).
\end{align}
Using \refeq{Ecaldefrep} and following the argument below \refeq{Vincl-bd}, we
obtain
\begin{align}
&\mP\Big(E'(\tb_{\sN},\bb_{\sss N+1};\tilde\bC_{\sss N-1})\cap\Ecal_{t_{\bb_{N
 +1}}+\vep}(\tb_{\sN},\yvec_2;\tilde\bC_{\sss N-1})\cap\{\vec\xvec\in\tilde
 \bC_{\sN}\}\Big)\nn\\
&\le\mP\Bigg(E'(\tb_{\sN},\bb_{\sss N+1};\tilde\bC_{\sss N-1})\cap\bigcup_{
 \cvec,\wvec\in\tilde\bC_{N-1}}\bigcup_{\substack{\zvec\in\Lambda\\ t_{\zvec}
 >t_{\bb_{N+1}}}}\bigg\{\Big\{\{\tb_{\sN}\conn\zvec\}\circ\{\zvec\conn\wvec\}
 \circ\{\wvec\conn\yvec_2\}\circ\{\zvec\conn\yvec_2\}\Big\}\nn\\
&\hskip7pc\cap\Big\{\{\cvec=\wvec,~\zvec\centernot\conn\wvec_-\}\cup\{\cvec\ne
 \wvec_-,~(\cvec,\wvec)\in\tilde\bC_{\sN-1}\}\Big\}\bigg\}\cap\{\vec\xvec\in
 \tilde\bC_{\sN}\}\Bigg).
\end{align}
Similarly to the above,
$E'(\tb_{\sN},\bb_{\sss N+1};\tilde\bC_{\sss N-1})$ implies the
existence of disjoint connections necessary to obtain the bounding diagram
$P^{\sss(0)}(\tb_{\sss N},\bb_{\sss N+1};\tilde\bC_{\sss N-1})$.  The event
subject to the union over $\zvec$ is accounted for by an application of
Construction~$B(\uvec)$ followed by multiplication of
$\sum_{\wvec:t_{\wvec}>t_{\bb_{N+1}}}S^{\sss(0)}(\uvec,\wvec;\tilde\bC_{\sN-1},
2_{\wvec}^{\sss(0)}(\yvec_2))$, resulting in the bounding diagram
\begin{align}\lbeq{FC-def}
%F(\tb_{\sN},\bb_{\sN+1},\yvec_2;\tilde\bC_{\sN-1})=
\sum_{\substack{\uvec,
 \wvec\\ t_{\wvec}>t_{\bb_{N+1}}}}P^{\sss(0)}\big(\tb_{\sN},\bb_{\sN+1};
 \tilde\bC_{\sN-1},B(\uvec)\big)~S^{\sss(0)}\big(\uvec,\wvec;\tilde\bC_{\sN
 -1},2_{\wvec}^{\sss(0)}(\yvec_2)\big).
\end{align}
The event $\{\vec\xvec\in\tilde\bC_{\sN}\}$ is accounted for by applying
Construction~$\ell(\vec\xvec_I)$ to
$P^{\sss(0)}(\tb_{\sN},\bb_{\sN+1};\tilde\bC_{\sN-1},B(\uvec))$ and Construction~$\ell(\vec\xvec_{J\setminus I})$ to
$S^{\sss(0)}(\uvec,\wvec;\tilde\bC_{\sN-1},2_{\wvec}^{\sss(0)}(\yvec_2))$,
followed by the summation over $I\subset J$.  %We denote the resulting diagram
%by $F(\tb_{\sN},\bb_{\sN+1},\yvec_2;\tilde\bC_{\sN-1},\ell(\vec\xvec))$.
Then, by \refeq{S0C-def} and \refeq{P0Cdef}, we have
\begin{align}\lbeq{Ediagrbd:N1=0rewr-bd2}
\refeq{Ediagrbd:N1=0rewr-bd1}\le\sum_{I\subset J}\sum_{\substack{\avec,\uvec,
 \wvec\\ t_{\wvec}>t_{\bb_{N+1}}}}\sum_{b_N}p_{b_N}&\bigg(M_{\tb_{N-1},\bb_N;
 \tilde\bC_{N-2}}^{\sss(1)}\Big(P^{\sss(0)}\big(\tb_{\sN},\bb_{\sN+1};\tilde
 \bC_{\sN-1},B(\uvec),\ell(\vec\xvec_I)\big)\,\ind{\avec\in\tilde\bC_{N-1}}
 \Big)\nn\\
&\qquad\times S^{\sss(0,0)}\big(\uvec,\wvec;\avec,2_{\wvec}^{\sss(0)}(\yvec_2),
 \ell(\vec\xvec_{J\setminus I})\big)\nn\\[5pt]
&+M_{\tb_{N-1},\bb_N;\tilde\bC_{N-2}}^{\sss(1)}\Big(P^{\sss(0)}\big(\tb_{\sN},
 \bb_{\sN+1};\tilde\bC_{\sN-1},B(\uvec),\ell(\vec\xvec_I)\big)\,\ind{(\avec,
 \wvec)\in\tilde\bC_{N-1}}\Big)\nn\\
&\qquad\times(1-\delta_{\avec,\wvec_-})\,S^{\sss(0,1)}\big(\uvec,\wvec;\avec,
 2_{\wvec}^{\sss(0)}(\yvec_2),\ell(\vec\xvec_{J\setminus I})\big)\bigg).
\end{align}
Note that $P^{\sss(0)}(\tb_{\sN},\bb_{\sN+1};\tilde\bC_{\sN-1},B(\uvec))$ is a
random variable (since $\tilde\bC_{\sN-1}$ is random) which depends only on
bonds in the time interval $[t_{\tb_N},t_{\bb_{N+1}}]$, and that $t_{\avec}\ge
t_{\bb_{\sN+1}}$, which is due to \refeq{S00-def}--\refeq{S01-def} and the
restriction on $t_{\wvec}$. Therefore, by the Markov property (cf.,
\refeq{MN-indbond}) and \refeq{P0cdef},
\begin{align}\lbeq{Ediagrbd:N1=0rewr-bd21}
\refeq{Ediagrbd:N1=0rewr-bd2}\le\sum_{I\subset J}\sum_{\substack{\avec,\uvec\\
 t_{\avec}\ge t_{\bb_{N+1}}}}&\sum_{b_N}p_{b_N}M_{\tb_{N-1},\bb_N;\tilde\bC_{N
 -2}}^{\sss(1)}\Big(P^{\sss(0)}\big(\tb_{\sN},\bb_{\sN+1};\tilde\bC_{\sN-1},
 B(\uvec),\ell(\vec\xvec_I)\big)\,\ind{\avec\in\tilde\bC_{N-1}}\Big)\nn\\
&\times P^{\sss(0)}\big(\uvec,\yvec_2;\avec,\ell(\vec\xvec_{J\setminus I})
 \big).
\end{align}

We need some care to estimate $M_{\tb_{N-1},\bb_N;\tilde\bC_{N-2}}^{\sss(1)}
\big(P^{\sss(0)}(\tb_{\sN},\bb_{\sN+1};\tilde\bC_{\sN-1},B(\uvec),\ell(\vec
\xvec_I))\ind{\avec\in\tilde\bC_{N-1}}\big)$ in \refeq{Ediagrbd:N1=0rewr-bd21}.
First, by \refeq{S0C-def} and $t_{\vvec}\le t_{\bb_{N+1}}\le t_{\avec}$, we
obtain
\begin{align}\lbeq{Ediagrbd:N1=0rewr-bd22}
&M_{\tb_{N-1},\bb_N;\tilde\bC_{N-2}}^{\sss(1)}\Big(P^{\sss(0)}\big(\tb_{\sN},
 \bb_{\sN+1};\tilde\bC_{\sN-1},B(\uvec),\ell(\vec\xvec_I)\big)\,\ind{\avec\in
 \tilde\bC_{N-1}}\Big)\nn\\
&\le\sum_{\substack{\cvec,\vvec\\ t_{\vvec}\le t_{\avec}}}\bigg(M_{\tb_{N-1},
 \bb_N;\tilde\bC_{N-2}}^{\sss(1)}\big(\ind{\cvec,\avec\in\tilde\bC_{N-1}}\big)
 \,S^{\sss(0,0)}\big(\tb_{\sN},\vvec;\cvec,2_{\vvec}^{\sss(0)}(\bb_{\sN+1}),
 B(\uvec),\ell(\vec\xvec_I)\big)\\
&\qquad+M_{\tb_{N-1},\bb_N;\tilde\bC_{N-2}}^{\sss(1)}\big(\ind{(\cvec,\vvec)\in
 \tilde\bC_{N-1}}\ind{\avec\in\tilde\bC_{N-1}}\big)\,(1-\delta_{\cvec,\vvec_-})
 \,S^{\sss(0,1)}\big(\tb_{\sN},\vvec;\cvec,2_{\vvec}^{\sss(0)}(\bb_{\sN+1}),
 B(\uvec),\ell(\vec\xvec_I)\big)\bigg).\nn
\end{align}
Then, by the BK inequality, we have
\begin{align}
&M_{\tb_{N-1},\bb_N;\tilde\bC_{N-2}}^{\sss(1)}\big(\ind{(\cvec,\vvec)\in\tilde
 \bC_{N-1}}\,\ind{\avec\in\tilde\bC_{N-1}}\big)\,(1-\delta_{\cvec,\vvec_-})
 \nn\\
&\le M_{\tb_{N-1},\bb_N;\tilde\bC_{N-2}}^{\sss(1)}\Big(\ind{\cvec\in\tilde
 \bC_{N-1}}\,\big(\ind{(\cvec,\vvec)\text{ occupied}\}\circ\{\avec\in\tilde
 \bC_{N-1}}+\ind{(\cvec,\vvec)\conn\avec}\big)\Big)\,(1-\delta_{\cvec,
 \vvec_-})\nn\\
&\le\Big(M_{\tb_{N-1},\bb_N;\tilde\bC_{N-2}}^{\sss(1)}\big(\ind{\cvec,\avec
 \in\tilde\bC_{N-1}}\big)+M_{\tb_{N-1},\bb_N;\tilde\bC_{N-2}}^{\sss(1)}\big(
 \ind{\cvec\in\tilde\bC_{N-1}}\big)\,\tau(\avec-\vvec)\Big)\,\lamb\vep D(\vvec
 -\cvec).
\end{align}
However, by a version of \refeq{MPIgen}, we have
\begin{align}
M_{\tb_{N-1},\bb_N;\tilde\bC_{N-2}}^{\sss(1)}(\ind{\cvec\in\tilde\bC_{N-1}})
 &\le\sum_\twop P^{\sss(0)}\big(\tb_{\sN-1},\bb_{\sN};\tilde\bC_{\sN-2},
 \ell^\twop(\cvec)\big),\lbeq{Ediagrbd:N1=0rewr-bd22sp1}\\
M_{\tb_{N-1},\bb_N;\tilde\bC_{N-2}}^{\sss(1)}(\ind{\cvec,\avec\in\tilde\bC_{N
 -1}})&\le\sum_\twop P^{\sss(0)}\big(\tb_{\sN-1},\bb_{\sN};\tilde\bC_{\sN-2},
 \ell^\twop(\cvec),\ell(\avec)\big),\lbeq{Ediagrbd:N1=0rewr-bd22sp2}
\end{align}
where $\sum_\twop$ is the sum over the admissible lines of the diagram
$P^{\sss(0)}(\tb_{\sN-1},\bb_{\sN};\tilde\bC_{\sN-2})$.  Therefore, the sum of
the second line on the right-hand side of \refeq{Ediagrbd:N1=0rewr-bd22} is
bounded by
\begin{align}\lbeq{Ediagrbd:N1=0rewr-bd23}
&\sum_\twop\bigg(\sum_{\cvec,\vvec}P^{\sss(0)}\big(\tb_{\sN-1},\bb_{\sN};\tilde
 \bC_{\sN-2},\ell^\twop(\cvec),\ell'(\avec)\big)\,\lamb\vep D(\vvec-\cvec)\,
 S^{\sss(0,1)}\big(\tb_{\sN},\vvec;\cvec,2_{\vvec}^{\sss(0)}(\bb_{\sN+1}),
 B(\uvec),\ell(\vec\xvec_I)\big)\nn\\
&\quad+\sum_{\vvec}\tau(\avec-\vvec)\underbrace{\sum_{\cvec}P^{\sss(0)}\big(
 \tb_{\sN-1},\bb_{\sN};\tilde\bC_{\sN-2},\ell^\twop(\cvec)\big)\,\lamb\vep
 D(\vvec-\cvec)\,S^{\sss(0,1)}\big(\tb_{\sN},\vvec;\cvec,2_{\vvec}^{\sss(0)}
 (\bb_{\sN+1}),B(\uvec),\ell(\vec\xvec_I)\big)}_{\equiv\,\Dcal(\vvec)}\bigg).
\end{align}
By the definition of Construction~$\ell^\twop(\cvec)$, the diagram
function $\Dcal(\vvec)$ can be written as
\begin{align}
\Dcal(\vvec)=\sum_{\cvec'}P^{\sss(0)}\big(\tb_{\sN-1},\bb_{\sN};\tilde\bC_{\sN
 -2},B^\twop(\cvec')\big)~(\tau*\lamb\vep D)(\vvec-\cvec')~S^{\sss(0,1)}\big(
 \tb_{\sN},\vvec;\cvec,2_{\vvec}^{\sss(0)}(\bb_{\sN+1}),B(\uvec),\ell(\vec
 \xvec_I)\big).
\end{align}
Thanks to this identity, the second line of \refeq{Ediagrbd:N1=0rewr-bd23} is
regarded as the result of applying Construction~$B^\twop(\vvec)$ to the diagram
line $(\tau*\lamb\vep D)(\vvec-\cvec')$ \emph{at} $\vvec$, followed by a
multiplication of $\tau(\avec-\vvec)$ and a summation over $\vvec$.  This is
not accounted for in the first line of \refeq{Ediagrbd:N1=0rewr-bd23} and is
the difference between the result of Construction~$\ell(\avec)$ and that of
Construction~$\ell'(\avec)$ in the first line of
\refeq{Ediagrbd:N1=0rewr-bd23}.  By this observation, we obtain
\begin{align}\lbeq{Ediagrbd:N1=0rewr-bd24}
\refeq{Ediagrbd:N1=0rewr-bd23}=\sum_\twop\sum_{\cvec,\vvec}P^{\sss(0)}\big(
 \tb_{\sN-1},\bb_{\sN};\tilde\bC_{\sN-2},\ell^\twop(\cvec),\ell(\avec)\big)
 \,\lamb\vep D(\vvec-\cvec)\,S^{\sss(0,1)}\big(\tb_{\sN},\vvec;\cvec,
 2_{\vvec}^{\sss(0)}(\bb_{\sN+1}),B(\uvec),\ell(\vec\xvec_I)\big).
\end{align}
Therefore, by applying the bounds \refeq{Ediagrbd:N1=0rewr-bd22sp2} and
\refeq{Ediagrbd:N1=0rewr-bd24} to \refeq{Ediagrbd:N1=0rewr-bd22} and using
\refeq{S0-def} and \refeq{P0cdef},
\begin{align}\lbeq{Ediagrbd:N1=0rewr-bd25}
\refeq{Ediagrbd:N1=0rewr-bd21}\le\sum_{I\subset J}\sum_{\substack{\avec,\uvec
 \\ t_{\avec}\ge t_{\bb_{N+1}}}}\bigg(&\sum_\twop\sum_{\cvec}\sum_{b_N}P^{\sss
 (0)}\big(\tb_{\sN-1},\bb_{\sN};\tilde\bC_{\sN-2},\ell^\twop(\cvec),\ell(\avec)
 \big)\,p_{b_N}P^{\sss(0)}\big(\tb_{\sN},\bb_{\sN+1};\cvec,B(\uvec),\ell(\vec
 \xvec_I)\big)\bigg)\nn\\
&\times P^{\sss(0)}\big(\uvec,\yvec_2;\avec,\ell(\vec\xvec_{J\setminus I})
 \big).
\end{align}
Finally, by Lemma~\ref{lem-recPNMthetas} and a version of \refeq{QN-def}, we
obtain
\begin{align}\lbeq{Ediagrbd:N1=0rewr-bd26}
\refeq{Ediagrbd:N1=0rewr-bd25}&\le\sum_{I\subset J}\sum_{\substack{\avec,\uvec
 \\ t_{\avec}\ge t_{\bb_{N+1}}}}P^{\sss(1)}\big(\tb_{\sN-1},\bb_{\sN+1};\tilde
 \bC_{\sN-2},B(\uvec),\ell(\avec),\ell(\vec\xvec_I)\big)~P^{\sss(0)}\big(\uvec,
 \yvec_2;\avec,\ell(\vec\xvec_{J\setminus I})\big)\nn\\
&\le P^{\sss(1)}\big(\tb_{\sN-1},\bb_{\sN+1};\tilde\bC_{\sN-2},\Ecal_{t_{\bb_{
 \sN+1}}}(\yvec_2),\ell(\vec\xvec)\big).
\end{align}
This completes the proof of \refeq{Ediagrbd:N1=0rewr}.
\end{proof}

\begin{proof}[Proof of Lemma~\ref{lem-bdsphi} for $N_1\ge1$]
First we recall that, by \refeq{BdeltaNdef} and \refeq{piPbd},
\begin{align}\lbeq{BdeltaMbd}
B_\delta^{\sss(N_1)}(\tb_{\sN+1},\yvec_1;\tilde{\bC}_{\sN})&\le\sum_{b:\tb=
 \yvec_1}P^{\sss(N_1-1)}(\tb_{\sN+1},\bb;\tilde\bC_{\sN})\,p_b,%=\sum_{b:\tb=
% \yvec_1}S^{\sss(N_1-1)}\big(\tb_{\sN+1},\wvec;\tilde\bC_{\sN},2_{\wvec}^{\sss
% (0)}(\bb)\big)\,p_b,
%&=\sum_{b,e:\tb=\yvec_1}\sum_\twop\sum_{\zvec}P^{\sss(0)}\big(\tb_{\sN+1},
% \eb;\tilde\bC_{\sN},\ell^\twop(\zvec)\big)\,p_eP^{\sss(N_1-2)}(\te,\bb;
% \zvec)\,p_b.
\end{align}
where, by \refeq{PArecrep},
\begin{align}\lbeq{Vdiagrbd-prepr1}
P^{\sss(N_1-1)}(\tb_{\sN+1},\bb;\tilde\bC_{\sN})
 \begin{cases}
 =P^{\sss(0)}(\tb_{\sN+1},\bb;\tilde\bC_{\sN})&(N_1=1),\\[7pt]
 \dpst\le\sum_\twop\sum_{\zvec}\sum_eP^{\sss(0)}\big(\tb_{\sN+1},\eb;\tilde
  \bC_{\sN},\ell^\twop(\zvec)\big)p_eP^{\sss(N_1-2)}(\te,\bb;\zvec)\quad
  &(N_1\ge2).
 \end{cases}
\end{align}
Then, by following the argument between \refeq{Ediagrbd:N1=0rewr-bd21} and
\refeq{Ediagrbd:N1=0rewr-bd26} and using versions of \refeq{Vdiagrbd:N1=0} and
\refeq{recPNMthetas}, we obtain that, for $N_1\ge2$,
\begin{align}\lbeq{Vdiagrbd-prepr2}
&\sum_{b_{N+1}}p_{b_{N+1}}M^{\sss(N+1)}_{\bb_{N+1}}\Big(\indic_{V_{t_{\yvec_1}
 -\vep}(\tb_N,\yvec_2)\,\cap\,\{\vec\xvec\in\tilde\bC_N\}}\,P^{\sss(0)}\big(
 \tb_{\sN+1},\eb;\tilde\bC_{\sN},\ell^\twop(\zvec)\big)\Big)\nn\\
&\le\sum_{b_{N+1}}\sum_{\twop'}\sum_{\cvec}P^{\sss(N)}\big(\bb_{\sN+1};
 \ell^{\twop'}(\cvec),V_{t_{\yvec_1}-\vep}(\yvec_2),\ell(\vec\xvec)\big)\,
 p_{b_{N+1}}P^{\sss(0)}\big(\tb_{\sN+1},\eb;\cvec,\ell^\twop(\zvec)\big)\nn\\
&\le P^{\sss(N+1)}\big(\eb;V_{t_{\yvec_1}-\vep}(\yvec_2),\ell(\vec\xvec),
 \ell^\twop(\zvec)\big)=R^{\sss(N+1)}\big(\eb,\yvec_2;\ell(\vec\xvec),
 \ell^\twop(\zvec)\big).
\end{align}
For $N_1=1$, we simply ignore $\ell^\twop(\zvec)$ and replace $\eb$ by $\bb$,
which immediately yields \refeq{Vdiagrbd}.  For $N_1\ge2$, by a version of
\refeq{recPNMthetas}, we obtain
\begin{align}\lbeq{Vdiagrbd-pr}
\text{LHS of \refeq{Vdiagrbd}}&\le\sum_{b:\tb=\yvec_1}\sum_\twop\sum_{\zvec}
 \sum_eR^{\sss(N+1)}\big(\eb,\yvec_2;\ell(\vec\xvec),\ell^\twop(\zvec)\big)\,
 p_eP^{\sss(N_1-2)}(\te,\bb;\zvec)\,p_b\nn\\
&\le\sum_{b:\tb=\yvec_1}R^{\sss(N+N_1)}\big(\bb,\yvec_2;\ell(\vec\xvec)\big)\,
 p_b,
\end{align}
as required.

The inequality \refeq{Ediagrbd} for $N_1\ge1$ can be proved similarly.  This
completes the proof of Lemma~\ref{lem-bdsphi}.
\end{proof}

\section{Bound on $a(\vec{\xvec}_J)$}\label{ss:ebd}

From now on, we assume $r\equiv|J|+1\ge3$.  The Fourier transform of the
convolution equation \refeq{phidef} is
\begin{align}\lbeq{varphi-fourier-id}
\hat\zetav_{\vec t_J}(\vec k_J)=\hat A_{\vec t_J}(\vec k_J)+\ddsum_{s=
 \vep}^{\underline t}\widehat{(\tau_{s-\vep}*p_\vep)}(k)~\hat a_{\vec
 t_J-s}(\vec k_J),
\end{align}
where $\underline t=\underline t_J\equiv\min_{j\in J}t_j$ and
$k=\sum_{j\in J}k_j$.  We have already shown in Proposition~\ref{prop:BAbds}
and \refeq{treegraph} that
\begin{align}\lbeq{varphi-fourier-bd}
\big|\hat A_{\vec t_J}(\vec k_J)\big|\le\|A_{\vec t_J}\|_1\le\vep O\big((1+\bar
 t)^{r-3}\big),&&&&
\big|\widehat{(\tau_{s-\vep}*p_\vep)}(k)\big|\le\|\tau_{s-\vep}\|_1\,\|p_\vep
 \|_1\le O(1),
\end{align}
where $\bar t\equiv\bar t_J$ is the second-largest element of
$\vec t_J$.  To complete the proof of \refeq{vphibd}, we investigate
the sum $\dsum_{\vep\le s\le\underline t}|\hat a_{\vec t_J-s}(\vec k_J)|$.

First we recall \refeq{edef} and \refeq{aNxI-def} to see that
\begin{align}\lbeq{a1234-id}
a^{\sss(N)}(\vec\xvec_J)=a^{\sss(N)}(\vec\xvec_J;1)+\sum_{\vno\ne I
 \subsetneq J}&\bigg(a^{\sss(N)}(\vec\xvec_{J\setminus I},\vec\xvec_I;2)\nn\\
&+\sum_{\yvec_1}\Big(a^{\sss(N)}(\yvec_1,\vec\xvec_I;3)+a^{\sss(N)}(\yvec_1,
 \vec\xvec_I;4)\Big)\,\tau(\vec\xvec_{J\setminus I}-\yvec_1)\bigg).
\end{align}
Let
\begin{align}\lbeq{Deltat-def}
\Delta_t=
 \begin{cases}
 1&(d>6),\\
 \log(1+t)&(d=6),\\
 (1+t)^{1\wedge\frac{6-d}2}&(d<6).
 \end{cases}
\end{align}
The main estimate on the error terms are the following bounds:

\begin{prop}[\textbf{Bounds on the error terms}]\label{prop-abds}
Let $d>4$ and $\lamb=\lambce$.  For $r\equiv|J|+1\ge3$ and $N\ge0$,
\begin{align}
\bigg|\sum_{\vec x_J}a_{\vec t_J}^{\sss(N)}(\vec x_J;1)\bigg|&\le\bigg(
 \delta_{N,0}\sum_{j\in J}\delta_{t_j,0}+\vep^2\frac{O(\beta)^{1\vee N}}
 {1+\underline t}\bigg)O\big((1+\bar t)^{r-3}\big),\lbeq{a1bdN}\\
\bigg|\sum_{\vec x_J}a_{\vec t_{J\setminus I},\vec t_I}^{\sss(N)}(\vec x_{J
 \setminus I},\vec x_I;2)\bigg|&\le\vep\frac{O(\beta)^N(1+\beta\Delta_{\bar
 t})}{1+\underline t}\,O\big((1+\bar t)^{r-3}\big),\lbeq{a2bdN}\\
\bigg|\sum_{\vec x_J}\sum_{\yvec_1}a^{\sss(N)}(\yvec_1,\vec\xvec_I;3)\,\tau
 (\vec\xvec_{J\setminus I}-\yvec_1)\bigg|&\le\vep\frac{O(\beta)^{N+1}}{1+
 \underline t}\,\Delta_{\bar t}(1+\bar t)^{r-3},\lbeq{a3bdN}\\
\bigg|\sum_{\vec x_J}\sum_{\yvec_1}a^{\sss(N)}(\yvec_1,\vec\xvec_I;4)\,\tau
 (\vec\xvec_{J\setminus I}-\yvec_1)\bigg|&\le\vep\frac{O(\beta)^{1\vee N}}
 {1+\underline t}\,\Delta_{\bar t}(1+\bar t)^{r-3}.\lbeq{a4bdN}
\end{align}
For $d\le4$ and $\lamb=\lamb_\sT$, the same bounds with $\beta$ replaced by
$\hat\beta_\sT\equiv\beta_1T^{-\alphamin}$ hold.
\end{prop}

The bounds \refeq{a1bdN}--\refeq{a4bdN} are proved in
Sections~\ref{sec-aN1bd}--\ref{sec-aN4bd}, respectively.

\begin{proof}[Proof of \refeq{vphibd} assuming Proposition~\ref{prop-abds}]
By \refeq{a1234-id} and \refeq{a1bdN}--\refeq{a4bdN}, we have that, for $d>4$,
\begin{align}
|\hat a_{\vec t_J}(\vec k_J)|\le\sum_{N\ge0}\bigg|\sum_{\vec x_J}a_{\vec t_J}
 ^{\sss(N)}(\vec x_J)\bigg|\le O\big((1+\bar t)^{r-3}\big)\bigg(\sum_{j\in
 J}\delta_{t_j,0}+\vep\frac{1+\beta\Delta_{\bar t}}{1+\underline t}\bigg),
\end{align}
hence, for any $\alphaminn<1\wedge\frac{d-4}2$,
\begin{align}\lbeq{a-fourier-sumbd:d>4}
\ddsum_{s=\vep}^{\underline t}|\hat a_{\vec t_J-s}(\vec k_J)|&\le O\big((1+
 \bar t)^{r-3}\big)\bigg(1+\vep\ddsum_{s=\vep}^{\underline t}\frac{1+\beta
 \Delta_{\bar t}}{1+\underline t-s}\bigg)\nn\\
&\le O\Big((1+\bar t)^{r-3}\log(1+\bar t)\Big)\,(1+\beta\Delta_{\bar t})
 \le O\big((1+\bar t)^{r-2-\alphaminn}\big),
\end{align}
which implies \refeq{vphibd}, due to
\refeq{varphi-fourier-id}--\refeq{varphi-fourier-bd}.

For $d\le4$, $\beta$ in \refeq{a-fourier-sumbd:d>4} is replaced by
$\hat\beta_\sT$, and $\Delta_{\bar t}=1+\bar t$.  Therefore, for any
$\alphaminn<\alphamin$,
\begin{align}\lbeq{a-fourier-sumbd:d<4}
\ddsum_{s=\vep}^{\underline t}|\hat a_{\vec t_J-s}(\vec k_J)|\le O\Big((1+
 \bar t)^{r-2}\log(1+\bar t)\Big)\,\hat\beta_\sT\le O(T^{r-2-\alphaminn}),
 \qquad\text{as }T\uaw\infty.
\end{align}
This completes the proof of \refeq{vphibd}.
\end{proof}

\subsection{Proof of \refeq{a1bdN}}\label{sec-aN1bd}
By the notation $J_j=J\setminus\{j\}$ and the definition \refeq{F'1-def} of
$\Fone(\vvec,\vec\xvec_J;\bC)$, we have
\begin{align}
\Fone(\vvec,\vec\xvec_J;\bC)\subseteq\bigcup_{j\in J}\big\{E'(\vvec,\xvec_j;
\bC)\cap\{\vvec\conn\vec\xvec_{J_j}\}\big\},
\end{align}
which, by \refeq{aN1-def}, intuitively explains why
$a^{\sss(N)}(\vec\xvec_J;1)$ is small (cf., Figure~\ref{fig-E1capP}).

Let $d>4$.  By \refeq{MPIgen}, we obtain that, for $N\ge1$,
\begin{align}\lbeq{a1bd-proof1}
|a^{\sss(N)}(\vec\xvec_J;1)|&\le\sum_{j\in J}\sum_{b_N}p_{b_N}M_{\bb_N}^{\sss
 (N)}\Big(\mP\big(E'(\tb_{\sN},\xvec_j;\tilde\bC_{\sss N-1})\cap\{\tb_{\sN}
 \conn\vec\xvec_{J_j}\}\big)\Big)\nn\\
&=\sum_{j\in J}M^{\sss(N+1)}_{\xvec_j}\big(\ind{\vec\xvec_{J_j}\in\bC_N}\big)
 \le\sum_{j\in J}P^{\sss(N)}(\xvec_j;\ell(\vec\xvec_{J_j})).
\end{align}
The same bound holds for $N=0$, due to \refeq{a01-def}.  By
Lemma~\ref{lem-Ptbds} and repeated applications of Lemma~\ref{lem:constr1}, we
have that, for $d>4$ (cf., \refeq{PN-constr-bd:l1}--\refeq{PN-constr-bd} for
$d>4$ and $N\ge1$),
\begin{align}\lbeq{constr1-sup1-bd4}
\sum_{z,\vec x_I}P^{\sss(N)}\big((z,s);\ell(\vec x_I,\vec t_I)\big)&\le
 \delta_{s,0}\delta_{N,0}O\big((1+\bar t_I)^{|I|-1}\big)+\vep^2\frac{O
 (\beta)^{N\vee1}}{(1+s)^{(d-2)/2}}\,(1+\bar s_{\vec t_I})^{|I|-1},
\end{align}
where $\bar s_{\vec t_I}$ is the second-largest element of $\{s,\vec t_I\}$,
hence
\begin{align}\lbeq{a1bd-proof2}
\sum_{\vec x_J}P^{\sss(N)}\big((x_j,t_j);\ell(\vec x_{J_j},\vec t_{J_j})\big)
 &\le\bigg(\delta_{t_j,0}\delta_{N,0}+\frac{\vep^2O(\beta)^{N\vee1}}{(1+t_j)
 ^{(d-2)/2}}\bigg)O\big((1+\bar t)^{|J_j|-1}\big)\nn\\
&\le\bigg(\delta_{t_j,0}\delta_{N,0}+\frac{\vep^2O(\beta)^{N\vee1}}{1+
 \underline t}\bigg)O\big((1+\bar t)^{r-3}\big).
\end{align}

For $d\le4$, we only need to replace $O(\beta)^{N\vee1}$ in
\refeq{constr1-sup1-bd4} by $O(\beta_\sT)O(\hat\beta_\sT)^{(N-1)\vee0}$ and
use $\beta_\sT(1+t_j)^{(2-d)/2}\le O(\hat\beta_\sT)(1+\underline t)^{-1}$
for $\underline t\le t_j\le T\log T$ to obtain \refeq{a1bd-proof2} with
$O(\beta)^{N\vee1}$ replaced by $O(\hat\beta_\sT)^{N\vee1}$.  This completes
the proof of \refeq{a1bdN}.
\qed

\subsection{Proof of \refeq{a2bdN}}\label{sec-aN2bd}
Let
\begin{align}
\tilde a^{\sss(N,N')}(\vec\xvec_{J\setminus I},\vec\xvec_I;2)=\sum_{b_{N+1}}
 p_{b_{N+1}}M^{\sss(N+1)}_{\bb_{N+1}}\Big(\ind{\tb_N\conn\vec\xvec_I}\,
 A^{\sss(N')}(\tb_{\sss N+1},\vec\xvec_{J\setminus I};\tilde\bC_{\sN})\Big).
\end{align}
Then, by \refeq{aN2-def}, we have
\begin{align}\lbeq{tildea2-bd0}
|a^{\sss(N)}(\vec\xvec_{J\setminus I},\vec\xvec_I;2)|\le\sum_{N'=0}^\infty
 \tilde a^{\sss(N,N')}(\vec\xvec_{J\setminus I},\vec\xvec_I;2).
\end{align}
To prove \refeq{a2bdN}, it thus suffices to show that the sum of
$\tilde a^{\sss(N,N')}(\vec\xvec_{J\setminus I},\vec\xvec_I;2)$ over $N'$
satisfies \refeq{a2bdN}.

We discuss the following three cases separately: (i)~$|J\setminus I|=1$,
(ii)~$|J\setminus I|\ge2$ and $N'=0$, and (iii)~$|J\setminus I|\ge2$ and
$N'\ge1$.  The reason why $a^{\sss(N)}(\vec\xvec_j,\vec\xvec_{J_j};2)$ for
some $j$ is small is the same as that for $a^{\sss(N)}(\vec\xvec_J;1)$
explained in Section~\ref{sec-aN1bd}.  However, as seen in
Figure~\ref{fig-E1capP}, the reason for general
$a^{\sss(N)}(\vec\xvec_{J\setminus I},\vec\xvec_I;2)$ with
$|J\setminus I|\ge2$ to be small is different.  It is because there are at
least \emph{three} disjoint branches coming out of a ``bubble'' started at
$\ovec$.

(i) If $I=J_j$ for some $j$ (hence $|J\setminus I|=1$), then we use
$A^{\sss(N')}(\tb_{\sN+1},\xvec_j;\tilde\bC_{\sN})=M_{\tb_{\sN+1},\xvec_j;
\tilde\bC_{\sN}}^{\sss(N'+1)}\!(1)$ and \refeq{a1bd-proof2} to obtain
\begin{align}\lbeq{tildea2j-bd}
\sum_{\vec x_J}\tilde a^{\sss(N,N')}(\xvec_j,\vec\xvec_{J_j};2)&=\sum_{\vec
 x_J}\sum_{b_{N+1}}p_{b_{N+1}}M^{\sss(N+1)}_{\bb_{N+1}}\Big(\ind{\tb_N\conn\vec
 \xvec_{J_j}}~M_{\tb_{\sN+1},\xvec_j;\tilde\bC_{\sN}}^{\sss(N'+1)}(1)\Big)\nn\\
&=\sum_{\vec x_J}M_{\xvec_j}^{\sss(N+N'+2)}\big(\ind{\tb_N\conn\vec\xvec_{J_j}}
 \big)\nn\\
&\le\sum_{\vec x_J}P^{\sss(N+N'+1)}(\xvec_j;\ell(\vec\xvec_{J_j}))
 \le\vep^2\frac{O(\beta)^{N+N'+1}}{1+\underline t}\,(1+\bar t)^{r-3},
\end{align}
where $\beta$ is replaced by $\hat\beta_\sT$ for $d\le4$.

 (ii) If $|J\setminus I|\ge2$ and $N'=0$, then we use \refeq{APbd} to obtain
\begin{align}\lbeq{tildea20-bd1}
&\tilde a^{\sss(N,0)}(\vec\xvec_{J\setminus I},\vec\xvec_I;2)\nn\\
&\le\sum_{\vno\ne I'\subsetneq {J\setminus I}}\bigg(\sum_{b_{N+1}}
 M^{\sss(N+1)}_{\bb_{N+1}}\big(\ind{\tb_N\conn\{\vec\xvec_I,\tb_{\sss
 N+1}\}}\big)p_{b_{N+1}}\mP\big(\{\tb_{\sss N+1}\conn\vec\xvec_{I'}\}
 \circ\{\tb_{\sss N+1}\conn\vec\xvec_{J\setminus(I\Dot\cup I')}\}\big)
 \nn\\
&\hspace{5pc}+\sum_{\zvec}\sum_{\substack{b_{N+1}:\\ \tb_{N+1}\ne\zvec}}
 p_{b_{N+1}}M^{\sss(N+1)}_{\bb_{N+1}}\Big(\ind{\tb_N\conn\vec\xvec_I}
 P^{\sss(0)}\big(\tb_{\sss N+1},\zvec;\tilde\bC_{\sN},\ell(\vec\xvec_{I'})
 \big)\Big)\,\tau\big(\vec\xvec_{J\setminus(I\Dot\cup I')}-\zvec\big)
 \bigg).
\end{align}
%Moreover, by \refeq{S0-def}--\refeq{S0C-def} and the argument around
%\refeq{MP1pfafirst}--\refeq{MP1pfa}, we obtain that, for any $N'\ge0$
%in general,
%\begin{align}\lbeq{tildea20-bd1:Cl}
%M^{\sss(N+1)}_{\bb_{N+1}}\Big(\ind{\tb_N\conn\vec\xvec_I}P^{\sss(N')}
% \big(\tb_{\sss N+1},\zvec;\tilde\bC_{\sN},\ell(\vec\xvec_{I'})\big)
% \Big)&\le\sum_{\cvec}M^{\sss(N+1)}_{\bb_{N+1}}\big(\ind{\tb_N\conn
% \{\vec\xvec_I,\cvec\}}\big)\,P^{\sss(N')}\big(\tb_{\sss N+1},\zvec;
% \cvec,\ell(\vec\xvec_{I'})\big).
%\end{align}
%However, by \refeq{MPIgen}, we have
%\begin{align}\lbeq{tildea20-bd2}
%M^{\sss(N+1)}_{\bb_{N+1}}\big(\ind{\tb_N\conn\{\vec\xvec_I,\cvec\}}
% \big)\le\sum_\twop P^{\sss(N)}\big(\bb_{\sN+1};\ell^\twop(\cvec),\ell
% (\vec\xvec_I)\big),
%\end{align}
%where $\sum_\twop$ is the sum over the $N^\text{th}$ admissible lines for $P^{\sss(N)}(\bb_{\sN+1})$.
Following the argument between \refeq{Ediagrbd:N1=0rewr-bd21} and
\refeq{Ediagrbd:N1=0rewr-bd26} (see also \refeq{Vdiagrbd-prepr2}), we obtain
\begin{align}
\sum_{\substack{b_{N+1}:\\ \tb_{N+1}\ne\zvec}}p_{b_{N+1}}M^{\sss(N+1)}_{\bb_{N
 +1}}\Big(\ind{\tb_N\conn\vec\xvec_I}P^{\sss(0)}\big(\tb_{\sN+1},\zvec;\tilde
 \bC_{\sN},\ell(\vec\xvec_{I'})\big)\Big)\le P^{\sss(N+1)}\big(\zvec;\ell(\vec
 \xvec_I),\tilde\ell_{\sss\le t_{\zvec}}(\vec\xvec_{I'})\big),
\end{align}
where $\tilde\ell_{\sss\le t_{\zvec}}(\vec\xvec_{I'})$ means that we apply
Construction~$\ell(\vec\xvec_{I'})$ to the lines contained in
$P^{\sss(N+1)}(\zvec;\ell(\vec\xvec_I))$, but at least one of $|I'|$
constructions is applied \emph{before} time $t_{\zvec}$.  This excludes the
possibility that there is a common branch point for $\vec\xvec_{I\Dot\cup I'}$
\emph{after} time $t_{\zvec}$.  Let
\begin{align}
\tilde a^{\sss(N,0)}\big(\vec\xvec_{J\setminus(I\Dot\cup I')},\vec\xvec_I,
 \vec\xvec_{I'};2\big)_1&=\sum_bP^{\sss(N)}\big(\bb;\ell
 (\tb),\ell(\vec\xvec_I)\big)p_b\mP\big(\{\tb\conn\vec\xvec_{I'}\}\circ\{
 \tb\conn\vec\xvec_{J\setminus(I\Dot\cup I')}\}\big),\lbeq{tildea201-def}\\
\tilde a^{\sss(N,0)}\big(\vec\xvec_{J\setminus(I\Dot\cup I')},\vec\xvec_I,
 \vec\xvec_{I'};2\big)_2&=\sum_{\zvec}P^{\sss(N+1)}\big(\zvec;\ell(\vec
 \xvec_I),\tilde\ell_{\sss\le t_{\zvec}}(\vec\xvec_{I'})\big)\,\tau\big(
 \vec\xvec_{J\setminus(I\Dot\cup I')}-\zvec\big).\lbeq{tildea202-def}
%\tilde a^{\sss(N,0)}\big(\vec\xvec_{J\setminus(I\Dot\cup I')},\vec\xvec_I,
% \vec\xvec_{I'};2\big)_2&=\sum_{\zvec}P^{\sss(N+1)}\big(\zvec;\tilde\ell_{
% \sss\le}(\vec\xvec_I),\tilde\ell_{\sss>}(\vec\xvec_{I'})\big)\,\tau\big(
% \vec\xvec_{J\setminus(I\Dot\cup I')}-\zvec\big).\lbeq{tildea202-def}
\end{align}
Then, by \refeq{MPIgen}, we obtain
\begin{align}\lbeq{tildea20-bd3}
\tilde a^{\sss(N,0)}(\vec\xvec_{J\setminus I},\vec\xvec_I;2)\le\sum_{\vno
 \ne I'\subsetneq J\setminus I}\bigg(\tilde a^{\sss(N,0)}\big(\vec\xvec_{J
 \setminus(I\Dot\cup I')},\vec\xvec_I,\vec\xvec_{I'};2\big)_1+\tilde a^{
 \sss(N,0)}\big(\vec\xvec_{J\setminus(I\Dot\cup I')},\vec\xvec_I,\vec
 \xvec_{I'};2\big)_2\bigg).
\end{align}

To estimate the sums of \refeq{tildea201-def}--\refeq{tildea202-def} over $\vec
x_J\in\mZ^{d|J|}$, we use the following extensions of \refeq{constr1-sup1-bd4}:

\begin{lem}\label{lem:constr1-sup}
For $N\ge0$, $s<s'$ and $d\le4$,
\begin{align}
\sup_w\sum_zP^{\sss(N)}\big((z,s);\ell(w,s'),\ell(\vec t_I)\big)&\le\delta_{s,
 0}\delta_{N,0}O\big((1+\bar t_I)^{|I|-1}\big)+\vep^2\frac{O(\beta_\sT)\,O(\hat
 \beta_\sT)^{(N-1)\vee0}}{(1+s)^{(d-2)/2}}\,(1+\bar s_{\vec t_I})^{|I|-1},
 \lbeq{constr1-sup1}\\
\sum_zP^{\sss(N+1)}\big((z,s);\ell(\vec t_I),\tilde\ell_{\sss\le s}(\vec
 t_{I'})\big)&\le\vep^2\frac{O(\beta_\sT)\,O(\hat\beta_\sT)^N}{(1+s)^{(d-2)/2}}
 \Big(1+s\wedge\max_{i\in I\Dot\cup I'}t_i\Big)(1+\bar s_{\vec t_{I\Dot\cup
 I'}})^{|I|+|I'|-2}.\lbeq{constr1-sup2}
\end{align}
For $d>4$, both $\beta_\sT$ and $\hat\beta_\sT$ are replaced by $\beta$.
\end{lem}

We will prove this lemma at the end of this subsection.

Now we assume Lemma~\ref{lem:constr1-sup} and prove \refeq{a2bdN}.  To discuss
both $d>4$ and $d\le4$ simultaneously, we for now interpret $\beta_\sT$
and $\hat\beta_\sT$ below as $\beta$ for $d>4$.  First, by \refeq{treegraph}
and \refeq{PE'vecx-prebd4} and using $\bar t_{J\setminus I}\le\bar t$ for
$|J\setminus I|\ge2$, $\|p_\vep\|_1=O(1)$ and \refeq{constr1-sup1}, we obtain
\begin{align}\lbeq{tildea20-bd4}
\sum_{\vec x_J}\tilde a^{\sss(N,0)}\big(\vec\xvec_{J\setminus(I\Dot\cup I')},
 \vec\xvec_I,\vec\xvec_{I'};2\big)_1&\le\vep\,O\big((1+\bar t_{J\setminus
 I})^{|J\setminus I|-2}\big)\sum_{b:t_{\bb}<\underline t}P^{\sss(N)}
 \big(\bb;\ell(\tb),\ell(\vec t_I)\big)p_b\nn\\
&\le\vep\,O\big((1+\bar t)^{|J\setminus I|-2}\big)\ddsum_{s<\underline t}
 \sup_w\sum_zP^{\sss(N)}\big((z,s);\ell(w,s+\vep),\ell(\vec t_I)
 \big)\nn\\
&\le\vep\,O\big((1+\bar t)^{|J|-3}\big)\,O(\hat\beta_\sT)^N,
\end{align}
where, for $d\le4$, we have used
\begin{align}\lbeq{betaTbd}
\beta_\sT(1+T\log T)^{(4-d)/2}~\Big(\times\log(1+T\log T)~
 \text{ when necessary}\Big)\le O(\hat\beta_\sT).
\end{align}
Moreover, by \refeq{treegraph} and \refeq{constr1-sup2} and using
\refeq{Deltat-def} and \refeq{betaTbd}, we obtain that, if
$J\setminus(I\,\Dot\cup\,I')=\{j\}$ (i.e., $I\,\Dot\cup\,I'=J_j$) and
$t_j=\max_{i\in J}t_i$, then $\max_{i\in J_j}t_i=\bar t$ and thus
\begin{align}\lbeq{tildea20-bd5}
\sum_{\vec x_J}\tilde a^{\sss(N,0)}\big(\xvec_j,\vec\xvec_I,\vec\xvec_{I'};2
 \big)_2&\le\vep\,O(\beta_\sT)\,O(\hat\beta_\sT)^N\bigg(\ddsum_{s\le\bar t}
 \frac{\vep}{(1+s)^{(d-4)/2}}\,(\underbrace{1+\bar s_{\vec t_{J_j}}}_{\le\,1+
 \bar t})^{|J|-3}\nn\\
&\hskip7pc+\ddsum_{\bar t\le s\le t_j}\frac{\vep}{(1+s)^{(d-2)/2}}\,(1+\bar
 t)\,(\underbrace{1+\bar s_{\vec t_{J_j}}}_{=\,1+\bar t})^{|J|-3}\bigg)\nn\\
&\le\vep\,O(\hat\beta_\sT)^{N+1}\Delta_{\bar t}\,(1+\bar t)^{|J|-3},
\end{align}
otherwise we use
$(s\equiv t_{\zvec}\le)~\underline t_{J\setminus(I\Dot\cup I')}\le\bar t=\bar
t_{I\Dot\cup I'}$ and thus $\bar s_{\vec t_{I\Dot\cup I'}}=\bar t$ and simply
bound $s\wedge\max_{i\in I\Dot\cup I'}t_i$ by $s$, so that
\begin{align}\lbeq{tildea20-bd5'}
\sum_{\vec x_J}\tilde a^{\sss(N,0)}\big(\vec\xvec_{J\setminus(I\Dot\cup I')},
 \vec\xvec_I,\vec\xvec_{I'};2\big)_2&\le\ddsum_{s\le\bar t}\vep^2\frac{O
 (\beta_\sT)\,O(\hat\beta_\sT)^N}{(1+s)^{(d-4)/2}}\,(1+\bar t)^{|I|+|I'|+|J
 \setminus(I\Dot\cup I')|-3}\nn\\
&\le\vep\,O(\hat\beta_\sT)^{N+1}\Delta_{\bar t}\,(1+\bar t)^{|J|-3}.
\end{align}
By \refeq{tildea20-bd3} and \refeq{tildea20-bd4}--\refeq{tildea20-bd5'} and
using $(1+\bar t)^{-1}\le(1+\underline t)^{-1}$, we thus obtain
\begin{align}\lbeq{tildea20-bd6}
\sum_{\vec x_J}\tilde a_{\vec t_{J\setminus I},\vec t_I}^{\sss(N,0)}\big(\vec
 x_{J\setminus I},\vec x_I;2\big)\le\vep\frac{O(\hat\beta_\sT)^N(1+\hat
 \beta_\sT\Delta_{\bar t})}{1+\underline t}\,O\big((1+\bar t)^{r-3}\big).
\end{align}
For $d>4$, we only need to replace $\hat\beta_\sT$ by $\beta$, as mentioned
earlier.

(iii) If $|J\setminus I|\ge2$ and $N'\ge1$, then, by \refeq{APbd}, we obtain
\begin{align}\lbeq{tildea2-bd1}
&\tilde a^{\sss(N,N')}(\vec\xvec_{J\setminus I},\vec\xvec_I;2)\nn\\
&\le\sum_{\vno\ne I'\subsetneq {J\setminus I}}\sum_{\zvec}\Bigg(\sum_{b_{N
 +1}}p_{b_{N+1}}M^{\sss(N+1)}_{\bb_{N+1}}\Big(\ind{\tb_N\conn\vec\xvec_I}
 P^{\sss(N')}(\tb_{\sss N+1},\zvec;\tilde\bC_{\sN})\Big)\,\tau(\vec
 \xvec_{I'}-\zvec)\nn\\
&\hspace{6pc}+\sum_{b_{N+1}}p_{b_{N+1}}M^{\sss(N+1)}_{\bb_{N+1}}\Big(
 \ind{\tb_N\conn\vec\xvec_I}P^{\sss(N')}\big(\tb_{\sss N+1},\zvec;\tilde
 \bC_{\sN},\ell(\vec\xvec_{I'})\big)\Big)\Bigg)\,\tau\big(\vec\xvec_{J
 \setminus(I\Dot\cup I')}-\zvec\big).
\end{align}
Let
\begin{align}
\tilde a^{\sss(N,N')}\big(\vec\xvec_{J\setminus(I\Dot\cup I')},\vec\xvec_I,
 \vec\xvec_{I'};2\big)_1&=\sum_{\zvec}P^{\sss(N+N'+1)}\big(\zvec;\ell(\vec
 \xvec_I)\big)\,\tau(\vec\xvec_{I'}-\zvec)\,\tau\big(\vec\xvec_{J\setminus
 (I\Dot\cup I')}-\zvec\big),\lbeq{tildea21-def}\\
\tilde a^{\sss(N,N')}\big(\vec\xvec_{J\setminus(I\Dot\cup I')},\vec\xvec_I,
 \vec\xvec_{I'};2\big)_2&=\sum_{\zvec}P^{\sss(N+N'+1)}\big(\zvec;\ell(\vec
 \xvec_I),\tilde\ell_{\sss\le t_{\zvec}}(\vec\xvec_{I'})\big)\,\tau
 \big(\vec\xvec_{J\setminus(I\Dot\cup I')}-\zvec\big).\lbeq{tildea22-def}
\end{align}
Similarly to the case of $N'=0$ above, we obtain
\begin{align}\lbeq{tildea2-bd2}
\tilde a^{\sss(N,N')}(\vec\xvec_{J\setminus I},\vec\xvec_I;2)\le\sum_{\vno
 \ne I'\subsetneq J\setminus I}\Big(\tilde a^{\sss(N,N')}\big(\vec\xvec_{J
 \setminus(I\Dot\cup I')},\vec\xvec_I,\vec\xvec_{I'};2\big)_1+\tilde
 a^{\sss(N,N')}(\vec\xvec_{J\setminus(I\Dot\cup I')},\vec\xvec_I,\vec
 \xvec_{I'};2\big)_2\Big).
\end{align}
However, by \refeq{treegraph}, \refeq{constr1-sup1-bd4}--\refeq{a1bd-proof2}
and \refeq{betaTbd}, we have
\begin{align}\lbeq{tildea2-bd3}
\sum_{\vec x_J}\tilde a^{\sss(N,N')}\big(\vec\xvec_{J\setminus(I\Dot\cup I')},
 \vec\xvec_I,\vec\xvec_{I'};2\big)_1&\le\ddsum_{s\le\underline t}\vep^2\frac{O
 (\beta_\sT)\,O(\hat\beta_\sT)^{N+N'}}{(1+s)^{(d-2)/2}}\,(1+\bar t)^{|I|+|J
 \setminus I|-3}\nn\\
&\le\vep\,O(\hat\beta_\sT)^{N+N'+1}(1+\bar t)^{|J|-3}.
\end{align}
Moreover, by \refeq{tildea20-bd5}--\refeq{tildea20-bd5'}, we have
\begin{align}\lbeq{tildea2-bd4}
\sum_{\vec x_J}\tilde a^{\sss(N,N')}\big(\vec\xvec_{J\setminus(I\Dot\cup I')},
 \vec\xvec_I,\vec\xvec_{I'};2\big)_2\le\vep\,O(\hat\beta_\sT)^{N+N'+1}
 \Delta_{\bar t}\,(1+\bar t)^{|J|-3}.
\end{align}
Summarising the above results and using
$(1+\bar t)^{-1}\le(1+\underline t)^{-1}$ (since $|J|\ge2$), we obtain that,
for $|J\setminus I|\ge2$ and $N'\ge1$,
\begin{align}\lbeq{tildea2-bd5}
\sum_{\vec x_J}\tilde a_{\vec t_{J\setminus I},\vec t_I}^{\sss(N,N')}\big(
 \vec x_{J\setminus I},\vec x_I;2\big)\le\vep\frac{O(\hat\beta_\sT)^{N+N'+1}
 \Delta_{\bar t}}{1+\underline t}\,(1+\bar t)^{r-3},
\end{align}
for $d\le4$, and the same bound with $\hat\beta_\sT$ replaced by $\beta$ holds
for $d>4$.

The proof of \refeq{a2bdN} is now completed by summing \refeq{tildea2j-bd} over
$N'\ge0$ or summing \refeq{tildea20-bd6} and the sum of \refeq{tildea2-bd5}
over $N'\ge1$, depending on whether $|J\setminus I|=1$ or $|J\setminus I|\ge2$,
respectively.
\qed

\begin{proof}[Proof of Lemma~\ref{lem:constr1-sup}]
First we prove \refeq{constr1-sup1}.  By the definition of
Construction~$\ell(w,s')$, we have (cf., \refeq{constr1-Btau})
\begin{align}\lbeq{constr1-sup1-bd1}
\sup_w\sum_zP^{\sss(N)}\big((z,s);\ell(w,s'),\ell(\vec
 t_I)\big)&\le\sup_w\ddsum_{s''\le(s\vee\bar t_I)\wedge s'}\,\sum_{y,z}
 P^{\sss(N)}\big((z,s);\ell(\vec t_I),B(y,s'')\big)\,\tau_{s'
 -s''}(w-y)\nn\\
&\le\ddsum_{s''\le(s\vee\bar t_I)\wedge s'}\|\tau_{s'-s''}\|_\infty\sum_z
 P^{\sss(N)}\big((z,s);\ell(\vec t_I),B(s'')\big).
\end{align}
Moreover, by Lemma~\ref{lem:constr1},
\begin{align}
\refeq{constr1-sup1-bd1}\le\bigg(\text{Bound on }\sum_zP^{\sss(N)}
 \big((z,s);\ell(\vec t_I)\big)\bigg)\times\ddsum_{s''\le(s\vee\bar
 t_I)\wedge s'}\sum_\twop(\delta_{s'',t_\twop}+\vep C_1)\,\|\tau_{s'-s''}
 \|_\infty,\lbeq{constr1-sup1-bd2}
\end{align}
where we have ignored the fact that $\sum_\twop$ depends on
$P^{\sss(N)}((z,s);\ell(\vec t_I))$ (as $\twop$ runs over all possible
lines in $P^{\sss(N)}((z,s);\ell(\vec t_I))$).  However, since (cf.,
\cite[(4.45)--(4.46)]{hsa04})
\begin{align}
\|\tau_{s''}\|_\infty\le(1-\vep)^{s''/\vep}+(1+s'')^{-d/2}\times
 \begin{cases}
 O(\beta)&(d>4),\\
 O(\beta_\sT)&(d\le4),
 \end{cases}
\end{align}
the sum over $s''$ in \refeq{constr1-sup1-bd2} is bounded in any dimension
(due to the excess power of $\beta_\sT$ when $d\le4$), hence
\refeq{constr1-sup1-bd1} obeys the same bound (modulo a constant) as
$\sum_zP^{\sss(N)}((z,s);\ell(\vec t_I))$, which is given in
\refeq{constr1-sup1-bd4}.  This completes the proof of \refeq{constr1-sup1}.

Next we prove \refeq{constr1-sup2}.  Due to Construction~$\tilde\ell_{\sss\le
s}(\vec t_{I'})$ (see below \refeq{tildea202-def}), the left-hand
side of \refeq{constr1-sup2} is bounded by
\begin{align}\lbeq{constr1-sup1-bd5}
\sum_{j\in I'}\sum_zP^{\sss(N+1)}\big((z,s);\ell(\vec t_{I\Dot\cup I'_j}),
 \tilde\ell_{\sss\le s}(t_j)\big).
\end{align}
By repeated applications of Lemma~\ref{lem:constr1} as in
\refeq{PN-constr-bd:l1}--\refeq{PN-constr-bd}, but bounding $s\wedge t_1$ by
$s\wedge\max_{i\in I\Dot\cup I'}t_i$ instead of by $s$ as in
\refeq{PN-constr-bd:l1}, and then using
$\bar s_{\vec t_{I\Dot\cup I'_j}}\le\bar s_{\vec t_{I\Dot\cup I'}}$, we have
\begin{align}\lbeq{constr1-sup1-bd6}
\sum_zP^{\sss(N+1)}\big((z,s);\ell(\vec t_{I\Dot\cup I'_j})\big)\le\vep^2
 \frac{O(\beta_\sT)\,O(\hat\beta_\sT)^N}{(1+s)^{d/2}}\Big(1+s\wedge\max_{i
 \in I\Dot\cup I'}t_i\Big)(1+\bar s_{\vec t_{I\Dot\cup I'}})^{|I|+|I'|-2},
\end{align}
for $d\le4$, and the same bound with $\beta_\sT$ and $\hat\beta_\sT$ both
replaced by $\beta$ for $d>4$.  Applying Lemma~\ref{lem:constr1} to this
bound, we can estimate \refeq{constr1-sup1-bd5}, similarly to
\refeq{constr1-sup1-bd2}.  However, due to the sum (not the supremum) over
$x_i$ in \refeq{constr1-sup1-bd5}, $\|\tau_{s'-s''}\|_\infty$ in
\refeq{constr1-sup1-bd2} is replaced by $\|\tau_{s'-s''}\|_1\le K$, where
the running variable $s''$ is at most $s$, due to the restriction in
Construction~$\tilde\ell_{\sss\le s}(x_i,t_i)$.  Therefore,
\refeq{constr1-sup1-bd5} is bounded by \refeq{constr1-sup1-bd6} multiplied
by $O(s)$, which reduces the power of the denominator to $(d-2)/2$, as
required. This completes the proof of Lemma~\ref{lem:constr1-sup}.
\end{proof}

\subsection{Proof of \refeq{a3bdN}}\label{sec-aN3bd}
Recall the definition \refeq{a3pm-def} of
$a^{\sss(N)}(\yvec_1,\vec\xvec_I;3)_{\sss\pm}$ and denote by
$a^{\sss(N,N')}(\yvec_1,\vec\xvec_I;3)_{\sss\pm}$ the contribution from
$B_\delta^{\sss(N')}(\tb_{\sN+1},\yvec_1;\bC_{\sN})$ (cf.,
Figure~\ref{fig-2ndexp-error}).  We note that
$a^{\sss(N,N')}(\yvec_1,\vec\xvec_I;3)_{\sss\pm}\ge0$ for every $N,N'\ge0$.
Similarly to the argument around \refeq{Vdiagrbd-prepr2}, we have
\begin{align}\lbeq{a3pm-rewr}
a^{\sss(N,N')}(\yvec_1,\vec\xvec_I;3)_{\sss\pm}\le\sum_{b_{N+1}}\sum_{\cvec,
 \vvec}&\bigg(\text{Diagrammatic bound on }\tilde M_{b_{N+1}}^{\sss(N+1)}\Big(
 \indic_{\Ftwo_{t_{\yvec_1}}(\tb_N,\vec\xvec_I;\bC_\pm)\,\cap\,\{\tb_N\conn
 \cvec\}}\Big)\bigg)\nn\\
&\times p_{b_{N+1}}P^{\sss(N')}(\tb_{\sN+1},\vvec;\cvec)\,p_{\vvec,\yvec_1},
\end{align}
where we recall $\bC_{\sss+}=\{\tb_{\sN}\}$ and $\bC_{\sss-}=\tilde\bC_{\sN-1}$
and define
\begin{align}
p_{\vvec,\yvec_1}=p_\vep(\yvec_1-\vvec).
\end{align}
We discuss the following two cases separately: (i)~$|I|=1$ and (ii)~$|I|\ge2$.

(i) Suppose that $I=\{j\}$ for some $j$.  If $t_j\le t_{\vvec}~(=t_{\yvec_1}
-\vep)$, we use $\Ftwo_{t_{\yvec_1}}(\tb_{\sN},\xvec_j;\bC_\pm)\subseteq
\{\tb_{\sN}\conn\xvec_j\}$.  If $t_j>t_{\vvec}$, the bubble that terminates at
$\xvec_j$ (cf., \refeq{H1incl}--\refeq{Ecaldef}) is cut by
$\Zd\times\{t_{\vvec}\}$ (i.e., $V_{t_{\vvec}}(\tb_{\sN},\xvec_j)$
occurs) or cut by $\bC_\pm=\tilde\bC_{\sss N-1}$ if $N\ge1$ (i.e.,
$\Ecal_{t_{\vvec}+\vep}(\tb_{\sN},\xvec_j;\tilde\bC_{\sss N-1})$ occurs).
Therefore,
\begin{align}\lbeq{a3i1stbd}
&\tilde M_{b_{N+1}}^{\sss(N+1)}\Big(\indic_{\Ftwo_{t_{\yvec_1}}(\tb_N,
 \xvec_j;\bC_\pm)\,\cap\,\{\tb_N\conn\cvec\}}\Big)\\
&\le\begin{cases}
 \tilde M_{b_{N+1}}^{\sss(N+1)}\big(\indic_{\{\tb_N\conn\{\cvec,\xvec_j
  \}\}}\big)&(t_j\le t_{\vvec}),\\[5pt]
 \tilde M_{b_{N+1}}^{\sss(N+1)}\big(\indic_{V_{t_{\vvec}}(\tb_N,\xvec_j)
  \,\cap\,\{\tb_N\conn\cvec\}}\big)+\tilde M_{b_{N+1}}^{\sss(N+1)}\big(
  \indic_{\Ecal_{t_{\vvec}+\vep}(\tb_N,\xvec_j;\tilde\bC_{N-1})\,\cap\,
  \{\tb_N\conn\cvec\}}\big)\indic_{\{N\ge1\}}&(t_j>t_{\vvec}).\nn
 \end{cases}
\end{align}
By Lemma~\ref{lem-remtilde} and the argument around
\refeq{tobeboundedphi1}--\refeq{remtildephi1} and \refeq{remtildephi3} and
using \refeq{Vdiagrbd:N1=0}--\refeq{Ediagrbd:N1=0}, we have
\begin{align}
\tilde M_{b_{N+1}}^{\sss(N+1)}\big(\indic_{\{\tb_N\conn\{\cvec,\xvec_j\}\}}
 \big)&\le M_{\bb_{N+1}}^{\sss(N+1)}\big(\indic_{\{\cvec,\xvec_j\in\tilde
 \bC_N\}}\big)\le\sum_\twop P^{\sss(N)}\big(\bb_{\sN+1};\ell^\twop(\cvec),
 \ell(\xvec_j)\big),\lbeq{a3i1stbd1}\\
\tilde M_{b_{N+1}}^{\sss(N+1)}\Big(\indic_{V_{t_{\vvec}}(\tb_N,\xvec_j)\,\cap
 \,\{\tb_N\conn\cvec\}}\Big)&\le M_{\bb_{N+1}}^{\sss(N+1)}\big(\indic_{V_{t_{
 \vvec}}(\tb_N,\xvec_j)\,\cap\,\{\cvec\in\tilde\bC_N\}}\big)\le\sum_\twop P^{
 \sss(N)}\big(\bb_{\sN+1};V_{t_{\vvec}}(\xvec_j),\ell^\twop(\cvec)\big),
 \lbeq{a3i1stbd2}\\
\tilde M_{b_{N+1}}^{\sss(N+1)}\Big(\indic_{\Ecal_{t_{\vvec}+\vep}(\tb_N,\xvec_j
 ;\tilde\bC_{N-1})\,\cap\,\{\tb_N\conn\cvec\}}\Big)&\le M_{\bb_{N+1}}^{\sss(N+
 1)}\big(\indic_{\Ecal_{t_{\vvec}+\vep}(\tb_N,\xvec_j;\tilde\bC_{N-1})\,\cap\,
 \{\cvec\in\tilde\bC_N\}}\big)\nn\\[5pt]
&\le\sum_\twop P^{\sss(N)}\big(\bb_{\sN+1};\Ecal_{t_{\vvec}}(\xvec_j),
 \ell^\twop(\cvec)\big),\lbeq{a3i1stbd3}
\end{align}
where $\sum_\twop$ is the sum over the $N^\text{th}$ admissible lines of
$P^{\sss(N)}(\bb_{\sN+1})$. Therefore, by \refeq{recPNMthetas} and
\refeq{PN0def}--\refeq{QN-def}, we obtain
\begin{align}\lbeq{a3pm:|I|=1}
a^{\sss(N,N')}(\yvec_1,\xvec_j;3)_{\sss\pm}\le\sum_{\vvec}p_{\vvec,\yvec_1}
 \times\begin{cases}
 P^{\sss(N+N'+1)}\big(\vvec;\ell(\xvec_j)\big)&(t_{\yvec_1}>t_j),\\[5pt]
 R^{\sss(N+N'+1)}(\vvec,\xvec_j)+Q^{\sss(N+N'+1)}(\vvec,\xvec_j)&(t_{\yvec_1}
  \le t_j).
 \end{cases}
\end{align}

We use \refeq{a3pm:|I|=1} to estimate $\sum_{\vec x_J}\sum_{\yvec_1}a^{\sss
(N,N')}(\yvec_1,\xvec_j;3)_{\sss\pm}\,\tau(\vec\xvec_{J_j}-\yvec_1)$.  By
\refeq{treegraph} and \refeq{constr1-sup1-bd4}, the contribution from the
case of $t_{\yvec_1}>t_j$ in \refeq{a3pm:|I|=1} is bounded as
\begin{align}\lbeq{a3pm:|I|=1prebd0p}
\sum_{\vec x_J}\sum_{\substack{\vvec,\yvec_1\\ t_{\yvec_1}>t_j}}P^{\sss(N+N'+
 1)}\big(\vvec;\ell(\xvec_j)\big)p_{\vvec,\yvec_1}\tau(\vec\xvec_{J_j}-\yvec_1)
 &\le O\big((1+\bar t_{J_j})^{|J_j|-1}\big)\ddsum_{s=t_j}^{\underline t_{J_j}}
 \sum_vP^{\sss(N+N'+1)}\big((v,s);\ell(t_j)\big)\nn\\
&\le\vep\,O(\hat\beta_\sT)^{N+N'}\,(1+\bar t_{J_j})^{|J_j|-1}\ddsum_{s=
 t_j}^{\underline t_{J_j}}\vep\frac{O(\beta_\sT)}{(1+s)^{(d-2)/2}}%\nn\\
%&\le\vep\frac{O(\hat\beta_\sT)^{N+N'+1}\Delta_{\bar t}}{1+\underline t}\,
% (1+\bar t)^{r-3},
\end{align}
where $(1+\bar t_{J_j})^{|J_j|-1}~(=(1+\bar t_{J_j})^{r-3})$ can be replaced
by $(1+\bar t)^{r-3}$, since $(1+\bar t_{J_j})^{|J_j|-1}=1$ if $J_j=\{i\}$ and
$t_i=\max_{i'\in J}t_{i'}$.  The sum in \refeq{a3pm:|I|=1prebd0p} is bounded
by $O(\hat\beta_\sT)$ when $d\le4$, and by
\begin{align}
\frac{O(\beta)}{(1+t_j)^{(d-4)/2}}=O(\beta)\,\frac{(1+t_j)^{(6-d)/2}}{1+t_j}
 \le O(\beta)\,\frac{(1+\bar t)^{0\vee(6-d)/2}}{1+\underline t}\le\frac{O
 (\beta)\,\Delta_{\bar t}}{1+\underline t},
\end{align}
when $d>4$.  Therefore, we obtain
\begin{align}\lbeq{a3pm:|I|=1prebd0}
\refeq{a3pm:|I|=1prebd0p}\le\vep\frac{O(\hat\beta_\sT)^{N+N'+1}\Delta_{\bar
 t}}{1+\underline t}\,(1+\bar t)^{r-3},
\end{align}
where $\hat\beta_\sT$ must be interpreted as $\beta$ when $d>4$.

Next we investigate the contribution from the case of $t_{\yvec_1}\le t_j$ in
\refeq{a3pm:|I|=1}.  By \refeq{treegraph} and \refeq{goal2}--\refeq{Qgoal},
we obtain
\begin{align}\lbeq{a3pm:|I|=1prebd1}
&\sum_{\vec x_J}\sum_{\substack{\vvec,\yvec_1\\ t_{\yvec_1}\le t_j}}\Big(
 R^{\sss(N+N'+1)}(\vvec,\xvec_j)+Q^{\sss(N+N'+1)}(\vvec,\xvec_j)\Big)\,
 p_{\vvec,\yvec_1}\tau(\vec\xvec_{J_j}-\yvec_1)\nn\\
&\quad\le O(\hat\beta_\sT)^{N+N'}O\big((1+\bar t_{J_j})^{|J_j|-1}\big)
 \ddsum_{s\le\underline t}\tilde b^{\sss(2)}_{s,t_j}(\delta_{s,t_j}+
 \beta_\sT)\beta_\sT.
\end{align}
We note that $(1+\bar t_{J_j})^{|J_j|-1}$ can be replaced by
$(1+\bar t)^{r-3}$, as explained below \refeq{a3pm:|I|=1prebd0p}.  To bound
the sum over $s$ in \refeq{a3pm:|I|=1prebd1}, we use the following lemma:

\begin{lem}[\textbf{Bounds on sums involving $\tilde b_{s,s'}^{\sss(2)}$}]
 \label{lem-bsumbd}
Let $r\equiv|J|+1\ge3$.  For any $j\in J$ and any $I,I'\subsetneq J$ such that
$\vno\ne I'\subsetneq I$,
\begin{align}
\ddsum_{s\le\underline t}\tilde b_{s,t_j}^{\sss(2)}(\delta_{s,t_j}+\beta_\sT)
 \beta_\sT&\le\vep\frac{O(\hat\beta_\sT)\,\Delta_{\bar t}}{1+\underline t},
 \lbeq{singlesumbd1}\\
%\ddsum_{s\le\underline t}\tilde b_{s,t_j}^{\sss(2)}\beta_\sT^2&\le\vep
% \frac{O(\hat\beta_\sT)^2}{1+\underline t},\lbeq{singlesumbd2}\\
\ddsum_{\substack{s\le\underline t\\ s\le s'\le\underline t_I}}\tilde b_{s,
 s'}^{\sss(2)}(\delta_{s,s'}+\beta_\sT)\beta_\sT&\le\vep\,O(\hat\beta_\sT)\,
 \Delta_{\bar t},\lbeq{doublesumbd1}\\
\ddsum_{\substack{s\le\underline t_{J\setminus I}\\ s\le s'\le\underline
 t_{I\setminus I'}}}\Big(1+s'\wedge\max_{i\in I'}t_i\Big)\,\tilde b_{s,
 s'}^{\sss(2)}\beta_\sT^2&\le\vep\,O(\hat\beta_\sT)^2\Delta_{\bar t}.
 \lbeq{doublesumbd2}
\end{align}
All $\beta_\sT$ and $\hat\beta_\sT$ in the above inequalities must be
interpreted as $\beta$ when $d>4$.
\end{lem}

We postpone the proof of Lemma~\ref{lem-bsumbd} to the end of this subsection.

By \refeq{singlesumbd1}, we immediately conclude that
\refeq{a3pm:|I|=1prebd1} obeys the same bound as \refeq{a3pm:|I|=1prebd0}, and
therefore,
\begin{align}\lbeq{a3pm:|I|=1bd}
\sum_{\vec x_J}\sum_{\yvec_1}a^{\sss(N,N')}(\yvec_1,\xvec_j;3)_{\sss\pm}\tau
 (\vec\xvec_{J_j}-\yvec_1)\le\vep\frac{O(\hat\beta_\sT)^{N+N'+1}\Delta_{\bar
 t}}{1+\underline t}\,(1+\bar t)^{r-3}.
\end{align}
This completes the proof of \refeq{a3bdN} for $|I|=1$.

(ii) Suppose $|I|\ge2$ and that $\Ftwo_{t_{\yvec_1}}(\tb_{\sN},\vec\xvec_I;
\bC_\pm)\cap\{\tb_{\sN}\conn\cvec\}$ occurs.  Then, there are
$\uvec\in\Zd\times\Zp$ and a nonempty $I'\subsetneq I$ such that
$\{\tb_{\sN}\conn\{\cvec,\uvec\}\}\circ\{\uvec\conn\vec\xvec_{I'}\}\circ
\{\uvec\conn\vec\xvec_{I\setminus I'}\}$ occurs.  If such a $\uvec$ does not
exist before or at time $t_{\vvec}$, then $\bC_{\sss\pm}=\tilde\bC_{\sss N-1}$
(hence $N\ge1$) and the event
$\Ecal_{t_{\vvec}+\vep}(\tb_{\sN},\vec\xvec_I;\tilde\bC_{\sss N-1})$ occurs,
where
\begin{align}
\Ecal_{t_{\vvec}+\vep}(\tb_{\sN},\vec\xvec_I;\tilde\bC_{\sss N-1})=
 \bigcup_{\vno\ne I'\subsetneq I}\,\bigcup_{\zvec:t_{\zvec}>t_{\vvec}}\Big\{
 \big\{\Ecal_{t_{\vvec}+\vep}(\tb_{\sN},\zvec;\tilde\bC_{\sss N-1})\cap\{
 \tb_{\sN}\conn\vec\xvec_{I'}\}\big\}\circ\{\zvec\conn\vec\xvec_{I\setminus
 I'}\}\Big\}.
\end{align}
Since
\begin{align}
&\big\{\Ftwo_{t_{\yvec_1}}(\tb_{\sN},\vec\xvec_I;\bC_\pm)\cap\{\tb_{\sN}\conn
 \cvec\}\big\}\setminus\Ecal_{t_{\vvec}+\vep}(\tb_{\sN},\vec\xvec_I;\tilde
 \bC_{\sN-1})\nn\\
&\quad\subset\bigcup_{\vno\ne I'\subsetneq I}\,\bigcup_{\uvec:t_{\uvec}\le
 t_{\vvec}}\Big\{\big\{\tb_{\sN}\conn\{\cvec,\uvec\}\big\}\circ\{\uvec\conn
 \vec\xvec_{I'}\}\circ\{\uvec\conn\vec\xvec_{I\setminus I'}\}\Big\},
\end{align}
we obtain that, by the BK inequality,
\begin{align}\lbeq{M3midbd}
&\tilde M_{b_{N+1}}^{\sss(N+1)}\Big(\indic_{\Ftwo_{t_{\yvec_1}}(\tb_N,\vec
 \xvec_I;\bC_\pm)\,\cap\,\{\tb_N\conn\cvec\}}\Big)\nn\\
&\le\sum_{\vno\ne I'\subsetneq I}\bigg(\sum_{\uvec:t_{\uvec}\le t_{\vvec}}
 \tilde M_{b_{N+1}}^{\sss(N+1)}\big(\indic_{\{\tb_N\conn\{\cvec,\uvec\}\}}
 \big)~\mP\big(\{\uvec\conn\vec\xvec_{I'}\}\circ\{\uvec\conn\vec\xvec_{I
 \setminus I'}\}\big)\nn\\
&\hskip4pc+\indic_{\{N\ge1\}}\sum_{\zvec:t_{\zvec}>t_{\vvec}}\tilde M_{b_{N
 +1}}^{\sss(N+1)}\Big(\indic_{\Ecal_{t_{\vvec}+\vep}(\tb_N,\zvec;\tilde\bC_{N
 -1})\,\cap\,\{\tb_N\conn\{\cvec,\vec\xvec_{I'}\}\}}\Big)\,\tau(\vec\xvec_{I
 \setminus I'}-\zvec)\bigg).
\end{align}

First we investigate the contribution to \refeq{a3bdN} from the sum over
$\uvec$ in \refeq{M3midbd}, which is, by \refeq{a3pm-rewr}, \refeq{a3i1stbd1}
and Lemma~\ref{lem-recPNMthetas},
\begin{align}\lbeq{M3midbd-sumubd1}
\sum_{\vec x_J}\sum_{\substack{\uvec,\vvec,\yvec_1\\ t_{\uvec}\le t_{\vvec}}}
 &\bigg(\underbrace{\sum_\twop\sum_{\cvec}\sum_{b_{N+1}}P^{\sss(N)}\big(
 \bb_{\sN+1};\ell^\twop(\cvec),\ell(\uvec)\big)\,p_{b_{N+1}}P^{\sss(N')}
 (\tb_{\sN+1},\vvec;\cvec)}_{\le\,P^{\sss(N+N'+1)}(\vvec;\ell(\uvec))}\bigg)
 p_{\vvec,\yvec_1}\tau(\vec\xvec_{J\setminus I}-\yvec_1)\nn\\
&\times\mP\big(\{\uvec\conn\vec\xvec_{I'}\}\circ\{\uvec\conn\vec\xvec_{I
 \setminus I'}\}\big).
\end{align}
Note that $|I|\ge2$.
By \refeq{treegraph} and \refeq{PE'vecx-prebd4} and using
$\sum_{\yvec_1}p_{\vvec,\yvec_1}=O(1)$ and
$t_{\vvec}<\underline t_{J\setminus I}$, we can perform the sums over
$\vec x_J$ and $\yvec_1$ to obtain
\begin{align}\lbeq{M3midbd-sumubd2}
\refeq{M3midbd-sumubd1}\le\vep\,O\Big(\underbrace{(1+\bar t_{J\setminus
 I})^{|J\setminus I|-1}(1+\bar t_I)^{|I|-2}}_{\le\,(1+\bar t)^{|J|-3}}\Big)
 \sum_{\substack{\uvec,\vvec\\ t_{\vvec}<\underline t_{J\setminus I}\\
 t_{\uvec}\le t_{\vvec}\wedge\underline t_I}}P^{\sss(N+N'+1)}\big(\vvec;\ell
 (\uvec)\big).
\end{align}
Then, by $(1+\bar t)^{-1}\le(1+\underline t)^{-1}$ for $|J|\ge2$ and using
\refeq{constr1-bd-b}, we obtain
\begin{align}\lbeq{M3midbd-sumubd4p}
\refeq{M3midbd-sumubd2}&\le\vep\frac{O(\hat\beta_{\sT})^{N+N'}}{1+\underline
 t}\,(1+\bar t)^{r-3}\ddsum_{s'<\underline t_{J\setminus I}}\vep\frac{O
 (\beta_{\sT})}{(1+s')^{d/2}}\ddsum_{s\le s'\wedge\underline t_I}\vep\,(1+s)
 \nn\\
&\le\vep\frac{O(\hat\beta_{\sT})^{N+N'}}{1+\underline t}\,(1+\bar t)^{r-3}
 \bigg(~\ddsum_{s'<\underline t}\vep\frac{O(\beta_{\sT})}{(1+s')^{(d-4)/2}}
 +\ddsum_{\underline t_I<s'<\underline t_{J\setminus I}}\vep\frac{O
 (\beta_{\sT})(1+\underline t_I)^2}{(1+s')^{d/2}}\bigg),
\end{align}
where the first sum is readily bounded by $O(\hat\beta_\sT)\Delta_{\bar t}$.
The second sum is bounded as
\begin{align}
\ddsum_{\underline t_I<s'<\underline t_{J\setminus I}}\vep\frac{O(\beta_{\sT})
 (1+\underline t_I)^2}{(1+s')^{d/2}}&\le O(\beta_\sT)\,(1+\underline t_I)^2
 \times
 \begin{cases}
 (1+\underline t_I)^{-(d-2)/2}&(d>2),\\
 \log(1+\underline t_{J\setminus I})&(d=2),\\
 (1+\underline t_{J\setminus I})^{(2-d)/2}&(d<2),
 \end{cases}
\end{align}
which is further bounded by $O(\hat\beta_\sT)\Delta_{\bar t}$, using $|I|\ge2$
and thus $\underline t_I\le\bar t$.  Therefore,
\begin{align}\lbeq{M3midbd-sumubd4}
\refeq{M3midbd-sumubd4p}\le\vep\frac{O(\hat\beta_{\sT})^{N+N'+1}
 \Delta_{\underline t_{J\setminus I}}}{1+\underline t}\,(1+\bar t)^{r-3}.
\end{align}

Next we investigate the contribution to \refeq{a3bdN} from the sum over
$\zvec$ in \refeq{M3midbd}, which is, by \refeq{a3pm-rewr}, a version of
\refeq{a3i1stbd3} and \refeq{Ediagrbd},
\begin{align}\lbeq{M3midbd-sumzbd1}
\sum_{\vec x_J}\sum_{\substack{\vvec,\zvec,\yvec_1\\ t_{\zvec}>t_{\vvec}}}
 &\bigg(\underbrace{\sum_\twop\sum_{\cvec}\sum_{b_{N+1}}P^{\sss(N)}\big(
 \bb_{\sN+1};\Ecal_{t_{\vvec}}(\zvec),\ell^\twop(\cvec),\ell(\vec\xvec_{I'})
 \big)\,p_{b_{N+1}}P^{\sss(N')}(\tb_{\sN+1},\vvec;\cvec)}_{\le\,Q^{\sss(N+N'
 +1)}(\vvec,\zvec;\ell(\vec\xvec_{I'}))}\bigg)\nn\\
&\times p_{\vvec,\yvec_1}\tau(\vec\xvec_{J\setminus I}-\yvec_1)\,\tau(\vec
 \xvec_{I\setminus I'}-\zvec).
\end{align}
By \refeq{treegraph} and $\sum_{\yvec_1}p_{\vvec,\yvec_1}=O(1)$ and using the
fact that $t_{\vvec}<t_{\zvec}\le\underline{t}_{I\setminus I'}$ and
$t_{\vvec}<\underline t_{J\setminus I}$, we can perform the sums over
$\vec x_{J\setminus I'}$ and $\yvec_1$ to obtain
\begin{align}\lbeq{M3midbd-sumzbd2}
\refeq{M3midbd-sumzbd1}&\le O\Big(\underbrace{(1+\bar t_{J\setminus I})^{|J
 \setminus I|-1}(1+\bar t_{I\setminus I'})^{|I\setminus I'|-1}}_{\le\,(1+\bar
 t)^{|J\setminus I'|-2}}\Big)\sum_{\vec x_{I'}}\sum_{\substack{\vvec,\zvec\\
 t_{\vvec}<\underline t_{J\setminus I}\\ t_{\vvec}<t_{\zvec}\le\underline t_{I
 \setminus I'}}}Q^{\sss(N+N'+1)}\big(\vvec,\zvec;\ell(\vec\xvec_{I'})\big).
\end{align}
By repeatedly applying \refeq{constr1-bd-b} to \refeq{Qgoal}, we have
\begin{align}\lbeq{M3midbd-sumzbd4-}
\sum_{v,z}Q_{s,s'}^{\sss(N+N'+1)}(v,z;\ell(\vec t_{I'}))\le O(\beta_{\sT})^2
 O(\hat\beta_{\sT})^{N+N'}\tilde b_{s,s'}^{\sss(2)}\Big(1+s'\wedge\max_{i\in
 I'}t_i\Big)\,(1+\bar s'_{\vec t_{I'}})^{|I'|-1}.
\end{align}
Since $s'\le\underline t_{I\setminus I'}$, we have
$\bar s'_{\vec t_{I'}}\le\bar t$.  Therefore, by \refeq{doublesumbd2},
\begin{align}\lbeq{M3midbd-sumzbd4}
\refeq{M3midbd-sumzbd2}&\le O(\hat\beta_{\sT})^{N+N'}O\big((1+\bar t)^{|J|-3}
 \big)\ddsum_{\substack{s<\underline t_{J\setminus I}\\ s<s'\le\underline t_{I
 \setminus I'}}}\Big(1+s'\wedge\max_{i\in I'}t_i\Big)\,\tilde b_{s,s'}^{\sss
 (2)}\beta_{\sT}^2\nn\\
&\le\vep\frac{O(\hat\beta_{\sT})^{N+N'+2}\Delta_{\underline t_{I\setminus I'}}}
 {1+\underline t}\,(1+\bar t)^{r-3}.
\end{align}
When $d>4$, the above $\hat\beta_\sT$ is replaced by $\beta$.

Summarising \refeq{M3midbd}, \refeq{M3midbd-sumubd4} and
\refeq{M3midbd-sumzbd4}, we now conclude that \refeq{a3bdN} for $|I|\ge2$ also
holds.  This together with \refeq{a3pm:|I|=1bd} completes the proof of
\refeq{a3bdN}. \qed

\begin{proof}[Proof of Lemma~\ref{lem-bsumbd}]
As we have done so far, $\beta_\sT$ and $\hat\beta_\sT$ below are both
replaced by $\beta$ when $d>4$.

First we prove \refeq{singlesumbd1}.  By
$(1+\underline t)^{0\vee(2-d)/2}\le\Delta_{\underline t}$ and
$\underline t\le\bar t$ for $|J|\ge2$ and using \refeq{betaTbd}, we obtain
\begin{align}\lbeq{singlebd-pr1}
\ddsum_{s\le\underline t}\tilde b_{s,t_j}^{\sss(2)}\,\delta_{s,t_j}\beta_\sT
&\le\vep\frac{(1+t_j)^{0\vee(2-d)/2}(\log(1+t_j))^{\delta_{d,2}}}{(1+t_j)^{(d
 -2)/2}}\,\delta_{\underline t,t_j}\beta_\sT\nn\\
&=\vep\frac{(1+\underline t)^{\frac{4-d}2\vee(3-d)}(\log(1+\underline t))^{
 \delta_{d,2}}}{1+\underline t}\,\beta_\sT\nn\\
&\le\vep\frac{(1+\underline t)^{0\vee(2-d)/2}}{1+\underline t}\,O(\hat
 \beta_\sT)\le\vep\frac{\Delta_{\bar t}}{1+\underline t}\,O(\hat\beta_\sT).
\end{align}
For $d>2$, we use
$(1+\underline t)^{-(d-2)/2}\le(1+\underline t)^{-1}(1+t_j)^{0\vee(4-d)/2}$
and \refeq{betaTbd} if $d\in(2,4]$, so that
\begin{align}\lbeq{singlebd-pr2}
\ddsum_{s\le\underline t}\tilde b_{s,t_j}^{\sss(2)}\beta_\sT^2&\le\ddsum_{s\le
 \underline t}\frac{\vep^{2-\delta_{s,t_j}}}{(1+s)^{(d-2)/2}(1+t_j-s)^{(d-2)/
 2}}\,\beta_\sT^2\nn\\
&\le\vep\frac{(1+t_j)^{0\vee(4-d)/2}(\log(1+t_j))^{\delta_{d,4}}}{(1+\underline
 t)^{(d-2)/2}}\,O(\beta_\sT^2)\nn\\
&\le\vep\frac{(1+t_j)^{0\vee(4-d)}(\log(1+t_j))^{\delta_{d,4}}}{1+\underline t}
 \,O(\beta_\sT^2)\le\vep\frac{O(\hat\beta_\sT)^2}{1+\underline t}.
\end{align}
For $d\le2$, on the other hand, we use \refeq{betaTbd} and
$(1+t_j)^{-1}\le(1+\underline t)^{-1}$ to obtain
\begin{align}\lbeq{singlebd-pr3}
\ddsum_{s\le\underline t}\tilde b_{s,t_j}^{\sss(2)}\beta_\sT^2&\le
 (1+t_j)^{(2-d)/2}(\log(1+t_j))^{\delta_{d,2}}\ddsum_{s\le\underline t}\frac{
 \vep^{2-\delta_{s,t_j}}}{(1+s)^{(d-2)/2}}\,\beta_\sT^2\nn\\
&\le\vep\,O(\hat\beta_\sT)\,(1+t_j)^{(2-d)/2}\beta_\sT\nn\\
&\le\vep\frac{O(\hat\beta_\sT)}{1+\underline t}\,(1+t_j)^{(4-d)/2}
 \beta_\sT\le\vep\frac{O(\hat\beta_\sT)^2}{1+\underline t}.
\end{align}
Since $\Delta_{\bar t}\ge1$, this completes the proof of \refeq{singlesumbd1}.

To prove \refeq{doublesumbd1}, we simply use \refeq{betaTbd} and
$\underline t\le\bar t$ to obtain
\begin{align}
\ddsum_{s\le\underline t}\tilde b_{s,s}^{\sss(2)}\beta_\sT&\le\ddsum_{s\le
 \underline t}\vep^{2-\delta_{s,2\vep}}\frac{(1+s)^{0\vee(2-d)/2}(\log(1+
 s))^{\delta_{d,2}}}{(1+s)^{(d-2)/2}}\,\beta_\sT\le\vep\,O(\hat\beta_\sT)\,
 (1+\underline t)^{0\vee(2-d)/2}\le\vep\,O(\hat\beta_\sT)\,\Delta_{\bar t},
\end{align}
and use $\underline t\le\underline t_I\le\bar t$ for $|I|\ge2$ and
use \refeq{betaTbd} twice to obtain
\begin{align}
\ddsum_{\substack{s\le\underline t\\ s\le s'\le\underline t_I}}\tilde b_{s,
 s'}^{\sss(2)}\beta_\sT^2\le\vep\,O(\hat\beta_\sT)\ddsum_{s\le\underline t}
 \frac{\vep^{1-\delta_{s,2\vep}}}{(1+s)^{(d-2)/2}}\,\beta_\sT\le\vep\,O(\hat
 \beta_\sT)^2.
\end{align}
This completes the proof of \refeq{doublesumbd1}.

Finally we prove \refeq{doublesumbd2}, for $d>2$ and $d\le2$ separately (the
latter is easier).  For brevity, we introduce the notation
\begin{align}
T_{I'}=\max_{i\in I'}t_i.
\end{align}
Note that $\underline t_{I\setminus I'}\wedge T_{I'}\le\bar t$ since $I'$ and
$I\setminus I'$ are both nonempty.  Then, for $d>2$,
\begin{align}\lbeq{doublesumbd-d>2}
\ddsum_{\substack{s\le\underline t_{J\setminus I}\\ s\le s'\le\underline t_{I
 \setminus I'}}}(1+s'\wedge T_{I'})\,\tilde b_{s,s'}^{\sss(2)}\beta_\sT^2&=
 \ddsum_{s'\le\underline t_{I\setminus I'}}(1+s'\wedge T_{I'})\ddsum_{s\le s'
 \wedge\underline t_{J\setminus I}}\frac{\vep^{3-\delta_{s,s'}-\delta_{s,2\vep}
 \delta_{s',2\vep}}}{(1+s)^{(d-2)/2}(1+s'-s)^{(d-2)/2}}\,\beta_\sT^2\nn\\
&\le\vep\,O(\hat\beta_\sT)\ddsum_{s'\le\underline t_{I\setminus I'}}\vep^{1-
 \delta_{s',2\vep}}\frac{1+s'\wedge T_{I'}}{(1+s')^{(d-2)/2}}\,\beta_\sT
 \nn\\
&\le\vep\,O(\hat\beta_\sT)\bigg(\ddsum_{s'\le\bar t}\frac{\vep^{1-\delta_{s',2
 \vep}}\beta_\sT}{(1+s')^{(d-4)/2}}+\ddsum_{T_{I'}\le s'\le\underline t_{I
 \setminus I'}}\frac{\vep^{1-\delta_{s',2\vep}}(1+T_{I'})\beta_\sT}{(1+
 s')^{(d-2)/2}}\bigg),
\end{align}
where the second sum in the last line is interpreted as zero if
$T_{I'}>\underline t_{I\setminus I'}$.  The first sum is readily bounded by
$O(\hat\beta_\sT)\Delta_{\bar t}$, whereas the second sum, if it is nonzero
(so that, in particular, $T_{I'}\le\bar t$), is bounded by
\begin{align}
\ddsum_{T_{I'}\le s'\le\underline t_{I\setminus I'}}\frac{\vep^{1-\delta_{s',
 2\vep}}(1+T_{I'})\beta_\sT}{(1+s')^{(d-2)/2}}&\le O(\beta_\sT)\,(1+T_{I'})
 \times
 \begin{cases}
 (1+T_{I'})^{-(d-4)/2}&(d>4)\\
 \log(1+\underline t_{I\setminus I'})&(d=4)\\
 (1+\underline t_{I\setminus I'})^{(4-d)/2}&(d<4)
 \end{cases}\nn\\
&\le O(\hat\beta_\sT)\,\Delta_{\bar t}.
\end{align}
Therefore, the right-hand side of \refeq{doublesumbd-d>2} is bounded by
$\vep O(\hat\beta_\sT)^2\Delta_{\bar t}$, as required.

For $d\le2$, we use \refeq{betaTbd} twice and
$1+\underline t_{I\setminus I'}\wedge T_{I'}\le1+\bar t=\Delta_{\bar t}$
to obtain
\begin{align}
\ddsum_{\substack{s\le\underline t_{J\setminus I}\\ s\le s'\le\underline t_{I
 \setminus I'}}}(1+s'\wedge T_{I'})\,\tilde b_{s,s'}^{\sss(2)}\beta_\sT^2
% &\le(1+\bar t)\ddsum_{s'\le\underline t_{I\setminus I'}}\frac{(\log(1+
% s'))^{\delta_{d,2}}}{(1+s')^{(d-2)/2}}\ddsum_{s\le s'\le\underline t_{I
% \setminus I'}}\frac{\vep^{3-\delta_{s,s'}-\delta_{s,2\vep}\delta_{s',2
% \vep}}}{(1+s)^{(d-2)/2}}\,\beta_\sT^2\nn\\
&\le\vep\,O(\hat\beta_\sT)\,\Delta_{\bar t}\ddsum_{s'\le\underline t_{I
 \setminus I'}}\frac{\vep^{1-\delta_{s',2\vep}}\beta_\sT}{(1+s')^{(d-2)/2}}
 \le\vep\,O(\hat\beta_\sT)^2\Delta_{\bar t}.
\end{align}
This completes the proof of \refeq{doublesumbd2} and hence of
Lemma~\ref{lem-bsumbd}.
\end{proof}

\subsection{Proof of \refeq{a4bdN}}\label{sec-aN4bd}
Recall the definition \refeq{a4pm-def} of
$a^{\sss(N)}(\yvec_1,\vec\xvec_I;4)_{\sss\pm}$ and denote by
$a^{\sss(N,N_1,N_2)}(\yvec_1,\vec\xvec_I;4)_{\sss\pm}$ the contribution from
$B_\delta^{\sss(N_1)}(\tb_{\sN+1},\yvec_1;\bC_{\sN})$ and
$A^{\sss(N_2)}(\te,\vec\xvec_I;\tilde\bC^e_{\sN})$, i.e.,
\begin{align}\lbeq{a4NNN}
&-a^{\sss(N,N_1,N_2)}(\yvec_1,\vec\xvec_I;4)_{\sss\pm}\nn\\
&\quad=\sum_{\substack{b_{N+1},e\\ b_{N+1}\ne e}}p_{b_{N+1}}p_e\,\tilde
 M_{b_{N+1}}^{\sss(N+1)}\Big(\ind{\Ftwo_{t_{\yvec_1}}(\tb_N,\eb;\bC_\pm)
 \text{ in }\tilde\bC^e_N}~B_\delta^{\sss(N_1)}(\tb_{\sss N+1},\yvec_1;
 \bC_{\sN})~A^{\sss(N_2)}(\te,\vec\xvec_I;\tilde\bC^e_{\sN})\Big).
\end{align}
Compare \refeq{a4NNN} with $\phi^{\sss(N,N_1,N_2)}(\yvec_1,\yvec_2)_{\sss\pm}$
in \refeq{phiNNN} and note that the only difference is that
$A^{\sss(N_2)}(\te,\vec\xvec_I;\tilde\bC^e_{\sN})$ in \refeq{a4NNN}
is replaced by $B_\delta^{\sss(N_2)}(\te,\yvec_2;\tilde\bC^e_{\sN})$ in
\refeq{phiNNN} (cf., Figure~\ref{fig-2ndexp-main}).

Similarly to the proof of \refeq{a2bdN} in Section~\ref{sec-aN2bd}, we discuss
the following three cases separately: (i)~$|I|=1$, (ii)~$|I|\ge2$ and
$N_2=0$, and (iii)~$|I|\ge2$ and $N_2\ge1$.

(i) Let $I=\{j\}$ for some $j\in J$.  Then, by the similarity of \refeq{a4NNN}
and \refeq{phiNNN}, we can follow the same proof of Lemma~\ref{lem-phibdslow}
and obtain
\begin{align}
\big|a^{\sss(N,N_1,N_2)}(\yvec_1,\xvec_j;4)_{\sss\pm}\big|&\le\sum_{\uvec_1}
 \Big(R^{\sss(N+N_1,N_2)}(\uvec_1,\xvec_j)+\ind{N\ge1}~Q^{\sss(N+N_1,N_2)}
 (\uvec_1,\xvec_j)\Big)\,p_{\uvec_1,\yvec_1}.
\end{align}
By \refeq{treegraph} and \refeq{goal2}--\refeq{Qgoal}, we obtain
\begin{align}
&\bigg|\sum_{N_1,N_2\ge0}\sum_{\vec x_J}\sum_{\yvec_1}a^{\sss(N,N_1,N_2)}
 (\yvec_1,\xvec_j;4)_{\sss\pm}\,\tau(\vec\xvec_{J_j}-\yvec_1)\bigg|\nn\\
&\le O\big(\underbrace{(1+\bar t_{J_j})^{|J_j|-1}}_{\le\,(1+\bar t)^{|J_j|-1}}
 \big)\sum_{N_1,N_2\ge0}\ddsum_{s\le\underline t}\sum_{u_1,x_j}\Big(R_{s,
 t_j}^{\sss(N+N_1,N_2)}(u_1,x_j)+\ind{N\ge1}~Q_{s,t_j}^{\sss(N+N_1,N_2)}(u_1,
 x_j)\Big)\nn\\
&\le O(\hat\beta_\sT)^{0\vee(N-1)}O\big((1+\bar t)^{r-3}\big)\ddsum_{s\le
 \underline t}\tilde b_{s,t_j}^{\sss(2)}(\delta_{s,t_j}+\beta_\sT)\beta_\sT.
\end{align}
By \refeq{singlesumbd1}, we conclude that, for $I=\{j\}$,
\begin{align}\lbeq{a4NNNibdcompl}
\bigg|\sum_{\vec x_J}\sum_{\yvec_1}a^{\sss(N)}
 (\yvec_1,\xvec_j;4)_{\sss\pm}\,\tau(\vec\xvec_{J_j}-\yvec_1)\bigg|\le\vep
 \frac{O(\hat\beta_\sT)^{1\vee N}\Delta_{\bar t}}{1+\underline t}\,(1+\bar
 t)^{r-3}.
\end{align}

(ii) Let $|I|\ge2$ and $N_2=0$.  Then, by \refeq{APbd} and following the
argument around \refeq{Vdiagrbd-prepr2}, we have
\begin{align}\lbeq{a4NN0ii-bd1}
&\big|a^{\sss(N,N_1,0)}(\yvec_1,\vec\xvec_I;4)_{\sss\pm}\big|\nn\\
&\le\sum_{\vvec}\sum_ep_e
 \sum_{\vno\ne I'\subsetneq I}\bigg(\delta_{\vvec,\te}\;\mP\big(\{\te\conn\vec
 \xvec_{I'}\}\circ\{\te\conn\vec\xvec_{I\setminus I'}\}\big)+\sum_{\zvec\ne\te}
 P^{\sss(0)}\big(\te,\zvec;\vvec,\ell(\vec\xvec_{I'})\big)\,\tau(\vec\xvec_{I
 \setminus I'}-\zvec)\bigg)\nn\\
&\qquad\times\sum_{b_{N+1}\ne e}p_{b_{N+1}}\bigg(\text{Bound on }\tilde M_{b_{N
 +1}}^{\sss(N+1)}\Big(\ind{\Ftwo_{t_{\yvec_1}}(\tb_N,\eb;\bC_\pm)\,\cap\,\{
 \vvec\in\tilde\bC_N\}\text{ in }\tilde\bC^e_N}\,B_\delta^{\sss(N_1)}(\tb_{\sN
 +1},\yvec_1;\bC_{\sN})\Big)\bigg).
\end{align}
By Lemmas~\ref{lem-inclG}--\ref{lem-bdsphi} and following the proof of
Lemma~\ref{lem-phibdslow} for $N_2=0$ in Section~\ref{sec-phibd1}, we obtain
\begin{align}\lbeq{a4NNN-1stlinebd}
&\sum_{b_{N+1}}p_{b_{N+1}}\tilde M_{b_{N+1}}^{\sss(N+1)}\Big(\ind{\Ftwo_{t_{
 \yvec_1}}(\tb_N,\eb;\bC_\pm)\,\cap\,\{\vvec\in\tilde\bC_N\}\text{ in }\tilde
 \bC^e_N}~B_\delta^{\sss(N_1)}(\tb_{\sss N+1},\yvec_1;\bC_{\sN})\Big)\nn\\
&\le\sum_\eta\sum_{\uvec}\Big(R^{\sss(N+N_1)}\big(\uvec,\eb;\ell^\eta
 (\vvec)\big)+\ind{N\ge1}\,Q^{\sss(N+N_1)}\big(\uvec,\eb;\ell^\eta(\vvec)\big)
 \Big)p_{\uvec,\yvec_1}.
\end{align}
Therefore, similarly to \refeq{tildea20-bd3}, we have
\begin{align}
\big|a^{\sss(N,N_1,0)}(\yvec_1,\vec\xvec_I;4)_{\sss\pm}\big|\le\sum_{\vno\ne
 I'\subsetneq I}\Big(\tilde a^{\sss(N,N_1,0)}(\yvec_1,\vec\xvec_{I'},\vec
 \xvec_{I\setminus I'};4)_1+\tilde a^{\sss(N,N_1,0)}(\yvec_1,\vec\xvec_{I'},
 \vec\xvec_{I\setminus I'};4)_2\Big),
\end{align}
where
\begin{align}
\tilde a^{\sss(N,N_1,0)}(\yvec_1,\vec\xvec_{I'},\vec\xvec_{I\setminus I'};4)_1
 =\sum_{\uvec,e}&\Big(R^{\sss(N+N_1)}\big(\uvec,\eb;\ell(\te)\big)+\ind{N\ge1}
 \,Q^{\sss(N+N_1)}\big(\uvec,\eb;\ell(\te)\big)\Big)\nn\\
&\times p_{\uvec,\yvec_1}p_e\,\mP\big(\{\te\conn\vec\xvec_{I'}\}\circ\{\te\conn
 \vec\xvec_{I\setminus I'}\}\big),
\end{align}
and (cf., \refeq{varphiA1rep2}--\refeq{varphiA1rep3})
\begin{align}
\tilde a^{\sss(N,N_1,0)}(\yvec_1,\vec\xvec_{I'},\vec\xvec_{I\setminus I'};4)_2=
 \sum_{\uvec,\vvec}&\Big(R^{\sss(N+N_1,1)}\big(\uvec,\vvec;\ell(\vec\xvec_{I'})
 \big)+\ind{N\ge1}\,Q^{\sss(N+N_1,1)}\big(\uvec,\vvec;\ell(\vec\xvec_{I'})\big)
 \Big)\nn\\
&\times p_{\uvec,\yvec_1}\tau(\vec\xvec_{I\setminus I'}-\vvec),
\end{align}

First, we estimate the contribution to \refeq{a4bdN} from $\tilde
a^{\sss(N,N_1,0)}(\yvec_1,\vec\xvec_{I'},\vec\xvec_{I\setminus I'};4)_1$.  By
\refeq{treegraph} and \refeq{PE'vecx-prebd4} and following the argument around
\refeq{constr1-sup1-bd2}, we obtain
\begin{align}\lbeq{tildea4NN01-bd1}
&\sum_{\vec x_J}\sum_{\yvec_1}\tilde a^{\sss(N,N_1,0)}(\yvec_1,\vec\xvec_{I'},
 \vec\xvec_{I\setminus I'};4)_1~\tau(\vec\xvec_{J\setminus I}-\yvec_1)\nn\\
&\le\ddsum_{\substack{
 s<\underline t\\ s\le s'<\underline t_I}}\sup_w\sum_{u,v}\Big(R_{s,s'}^{
 \sss(N+N_1)}\big(u,v;\ell(w,s'+\vep)\big)+\ind{N\ge1}\,Q_{s,s'}^{\sss(N+
 N_1)}\big(u,v;\ell(w,s'+\vep)\big)\Big)\nn\\[-5pt]
&\qquad\times\vep\,O\Big((1+\bar t_I)^{|I|-2}(1+\bar t_{J\setminus I})^{|J
 \setminus I|-1}\Big)\nn\\[10pt]
&\le\vep O\big((1+\bar t)^{|J|-3}\big)\ddsum_{\substack{s<\underline t\\
 s\le s'<\underline t_I}}\bigg(\text{Bound on }\sum_{u,v}\Big(R_{s,s'}^{\sss(N+
 N_1)}(u,v)+\ind{N\ge1}\,Q_{s,s'}^{\sss(N+N_1)}(u,v)\Big)\bigg),
\end{align}
where we have used $\bar t_I\le\bar t$ for $|I|\ge2$ and
$(1+\bar t_{J\setminus I})^{|J\setminus I|-1}=1$ if $J\setminus I=\{j\}$ and
$t_j=\max_{i\in J}t_i$ (otherwise we use $\bar t_{J\setminus I}\le\bar t$).  By
\refeq{doublesumbd1}, we obtain
\begin{align}\lbeq{tildea4NN01-bd2}
\refeq{tildea4NN01-bd1}\le\vep^2\frac{O(\hat\beta_\sT)^{1\vee(N+N_1)}\,
 \Delta_{\bar t}}{1+\underline t}\,(1+\bar t)^{r-3}.
\end{align}

Next, we estimate the contribution to \refeq{a4bdN} from $\tilde
a^{\sss(N,N_1,0)}(\yvec_1,\vec\xvec_{I'},\vec\xvec_{I\setminus I'};4)_2$.  By
\refeq{treegraph} and repeatedly applying \refeq{constr1-bd-b} to
\refeq{goal2}--\refeq{Qgoal}, we obtain
\begin{align}\lbeq{tildea4NN02-bd1}
&\sum_{\vec x_J}\sum_{\yvec_1}\tilde a^{\sss(N,N_1,0)}(\yvec_1,\vec\xvec_{I'},
 \vec\xvec_{I\setminus I'};4)_2~\tau(\vec\xvec_{J\setminus I}-\yvec_1)\nn\\
&\le\ddsum_{\substack{s<\underline t_{J\setminus I}\\ s\le s'<\underline
 t_{I\setminus I'}}}\sum_{u,v}\Big(R_{s,s'}^{\sss(N+N_1,1)}\big(u,v;\ell
 (\vec t_{I'})\big)+\ind{N\ge1}\,Q_{s,s'}^{\sss(N+N_1,1)}\big(u,v;\ell
 (\vec t_{I'})\big)\Big)\nn\\
&\qquad\times O\Big(\underbrace{(1+\bar t_{I\setminus I'})^{|I\setminus I'|-1}
 (1+\bar t_{J\setminus I})^{|J\setminus I|-1}}_{\le\,(1+\bar t)^{|J\setminus
 I'|-2}}\Big)\nn\\
&\le O(\hat\beta_\sT)^{0\vee(N+N_1-1)}O\big((1+\bar t)^{|J|-3}\big)
 \ddsum_{\substack{s<\underline t_{J\setminus I}\\ s\le s'<\underline t_{I
 \setminus I'}}}\Big(1+s'\wedge\max_{i\in I'}t_i\Big)\,\tilde b_{s,s'}^{\sss
 (2)}\beta_\sT^2.
\end{align}
By \refeq{doublesumbd2}, we arrive at
\begin{align}\lbeq{tildea4NN02-bd2}
\refeq{tildea4NN02-bd1}\le\vep\frac{O(\hat\beta_\sT)^{1\vee(N+N_1)+1}
 \Delta_{\bar t}}{1+\underline t}\,(1+\bar t)^{r-3}.
\end{align}

Summarizing \refeq{tildea4NN01-bd2} and \refeq{tildea4NN02-bd2} yields that,
for $|I|\ge2$ and $N_2=0$,
\begin{align}\lbeq{a4NNNiibdcompl}
\bigg|\sum_{\vec x_J}\sum_{\yvec_1}a^{\sss(N,N_1,0)}(\yvec_1,\vec\xvec_I;
 4)_{\sss\pm}\,\tau(\vec\xvec_{J\setminus I}-\yvec_1)\bigg|\le\vep\frac{O(\hat\beta_\sT)^{1\vee(N+N_1)}\Delta_{\bar t}}
 {1+\underline t}\,(1+\bar t)^{r-3}.
\end{align}

(iii) Let $|I|\ge2$ and $N_2\ge1$.  By \refeq{APbd} and
\refeq{a4NNN-1stlinebd}, we have
\begin{align}\lbeq{a4NNNiii-bd1}
&\big|a^{\sss(N,N_1,N_2)}(\yvec_1,\vec\xvec_I;4)_{\sss\pm}\big|\nn\\
&\le\sum_{\vvec,\zvec}\sum_ep_e\sum_{\vno\ne I'\subsetneq I}\Big(P^{\sss(N_2)}
 (\te,\zvec;\vvec)~\tau(\vec\xvec_{I'}-\zvec)+P^{\sss(N_2)}\big(\te,\zvec;
 \vvec,\ell(\vec\xvec_{I'})\big)\Big)\,\tau(\vec\xvec_{I\setminus I'}-\zvec)
 \nn\\
&\qquad\times\sum_{\substack{b_{N+1},e\\ b_{N+1}\ne e}}p_{b_{N+1}}\bigg(
 \text{Bounds on }\tilde M_{b_{N+1}}^{\sss(N+1)}\Big(\ind{\Ftwo_{t_{\yvec_1}}
 (\tb_N,\eb;\bC_\pm)\,\cap\,\{\vvec\in\tilde\bC_N\}\text{ in }\tilde\bC^e_N}\,
 B_\delta^{\sss(N_1)}(\tb_{\sss N+1},\yvec_1;\bC_{\sN})\Big)\bigg)\nn\\
&\le\sum_{\vno\ne I'\subsetneq I}\sum_{\uvec,\zvec}\bigg(\sum_\eta\sum_{\vvec}
 \sum_e\Big(R^{\sss(N+N_1)}\big(\uvec,\eb;\ell^\eta(\vvec)\big)+\ind{N\ge1}\,
 Q^{\sss(N+N_1)}\big(\uvec,\eb;\ell^\eta(\vvec)\big)\Big)p_{\uvec,\yvec_1}p_e
 \nn\\
&\hskip7pc\times\Big(P^{\sss(N_2)}(\te,\zvec;\vvec)~\tau(\vec\xvec_{I'}-\zvec)
 +P^{\sss(N_2)}\big(\te,\zvec;\vvec,\ell(\vec\xvec_{I'})\big)\Big)\bigg)\,\tau
 (\vec\xvec_{I\setminus I'}-\zvec),
\end{align}
where, by \refeq{P0Cdef}, \refeq{PArecrep} and \refeq{RtilR}--\refeq{QtilQ},
\begin{align}\lbeq{a4NNNiii-bd2}
&\sum_\eta\sum_{\vvec}\sum_e\Big(R^{\sss(N+N_1)}\big(\uvec,\eb;\ell^\eta(\vvec)
 \big)+\ind{N\ge1}\,Q^{\sss(N+N_1)}\big(\uvec,\eb;\ell^\eta(\vvec)\big)\Big)p_e
 \nn\\
&\hskip4pc\times\Big(P^{\sss(N_2)}(\te,\zvec;\vvec)\,\tau(\vec\xvec_{I'}-\zvec)
 +P^{\sss(N_2)}\big(\te,\zvec;\vvec,\ell(\vec\xvec_{I'})\big)\Big)\nn\\
&=\Big(R^{\sss(N+N_1,N_2)}(\uvec,\zvec)+\ind{N\ge1}\,Q^{\sss(N+N_1,N_2)}(\uvec,
 \zvec)\Big)\,\tau(\vec\xvec_{I'}-\zvec)\nn\\
&\quad+R^{\sss(N+N_1,N_2)}\big(\uvec,\zvec;\ell(\vec\xvec_{I'})\big)+\ind{N\ge
 1}\,Q^{\sss(N+N_1,N_2)}\big(\uvec,\zvec;\ell(\vec\xvec_{I'})\big).
\end{align}
Then, by repeatedly applying \refeq{constr1-bd-b} to
\refeq{goal2}--\refeq{Qgoal} and using \refeq{treegraph} and
\refeq{doublesumbd2}, we obtain
\begin{align}\lbeq{a4NNNiiibdcompl}
&\bigg|\sum_{\vec x_J}\sum_{\yvec_1}a^{\sss(N,N_1,N_2)}(\yvec_1,\vec\xvec_I;
 4)_{\sss\pm}\,\tau(\vec\xvec_{J\setminus I}-\yvec_1)\bigg|\nn\\
&\le\sum_{\vno\ne I'\subsetneq I}\bigg(\sum_{\uvec,\zvec}\Big(R^{\sss(N+N_1,
 N_2)}(\uvec,\zvec)+\ind{N\ge1}\,Q^{\sss(N+N_1,N_2)}(\uvec,\zvec)\Big)\,O\big(
 \underbrace{(1+\bar t_{I'})^{|I'|-1}}_{\le\,(1+\bar t)^{|I'|-1}}\big)\nn\\
&\hskip5pc+\sum_{\uvec,\zvec}\Big(R^{\sss(N+N_1,N_2)}\big(\uvec,\zvec;\ell(\vec
 t_{I'})\big)+\ind{N\ge1}\,Q^{\sss(N+N_1,N_2)}\big(\uvec,\zvec;\ell(\vec
 t_{I'})\big)\Big)\bigg)\nn\\
&\hskip3pc\times O\Big(\underbrace{(1+\bar t_{I\setminus I'})^{|I\setminus I'|
 -1}(1+\bar t_{J\setminus I})^{|J\setminus I|-1}}_{\le\,(1+\bar t)^{|J\setminus
 I'|-2}}\Big)\nn\\
&\le O(\hat\beta_\sT)^{1\vee(N+N_1)+N_2-2}O\big((1+\bar t)^{|J|-3}\big)
 \sum_{\vno\ne I'\subsetneq I}~\ddsum_{\substack{s<\underline t_{J\setminus
 I}\\ s\le s'\le\underline t_{I\setminus I'}}}(1+s')\,\tilde b_{s,s'}^{\sss(2)}
 \beta_\sT^2\nn\\
&\le\vep\frac{O(\hat\beta_\sT)^{1\vee(N+N_1)+N_2}\Delta_{\bar t}}{1+\underline
 t}\,(1+\bar t)^{r-3}.
\end{align}

Finally, by summing \refeq{a4NNNiibdcompl} and the sum of
\refeq{a4NNNiiibdcompl} over $N_2\ge1$, we conclude that, for $|I|\ge2$,
\begin{align}
\bigg|\sum_{\vec x_J}\sum_{\yvec_1}a^{\sss(N)}(\yvec_1,\vec\xvec_I;4)_{\sss\pm}
 \,\tau(\vec\xvec_{J\setminus I}-\yvec_1)\bigg|\le\vep\frac{O(\hat\beta_\sT)^{1
 \vee N}\Delta_{\bar t}}{1+\underline t}\,(1+\bar t)^{r-3}.
\end{align}
This together with \refeq{a4NNNibdcompl} completes the proof of \refeq{a4bdN}.
\qed

%See Proposition \ref{prop-phibdslow} to see that $\phi^{\sss(N,M,K)}(\yvec_1,\yvec_2)_{\sss\pm}$
%is bounded by
%    \eq
%    T^{\sss(N,M,K)}(\yvec_1,\yvec_2)=
%    \sum_{\uvec} p(\yvec_1-\uvec)[R^{\sss(N,M,K)}(\uvec,\yvec_2)+
%    Q^{\sss(N,M,K)}(\uvec,\yvec_2)].
%    \en
%The bounding diagram for $A^{\sss(K)}(\te,\vec\xvec_I;\tilde\bC^e_{\sN})$
%in Proposition \ref{prop-ANbd} is explicitly described in terms of
%the bounding diagram for $B^{\sss(K)}(\te,\yvec_2;\tilde\bC^e_{\sN})$.
%Therefore, it follows that
%    \eq
%    |a^{\sss(N,M,K)}(\yvec_1,\vec\xvec_I;4)|
%    \leq \sum_{I'\neq \varnothing, I} \sum_{\yvec_2}
%    T^{\sss(N,M,K)}(\yvec_1,\yvec_2; \ell(\vec\xvec_{I'}))
%    \tau(\vec\xvec_{I\backslash I'}-\yvec_2).
%    \en
%To bound this diagram, we use \refeq{treegraph}, as well as Lemma \ref{lem:constr1}(b)
%to deduce that, with $\yvec_1=(s_1,y_1)$ and $\yvec_2=(s_2,y_2)$,
%in an identical way as in \refeq{a3diagrambound}, to obtain
%    \eq
%    \sum_{\yvec_2}\sum_{y_1}
%    \sum_{\vec x_{I}}T^{\sss(N,M,K)}(\yvec_1,\yvec_2; \ell(\vec\xvec_{I'}))
%    \tau(\vec\xvec_{I\backslash I'}-\yvec_2)
%    \leq  (C\beta)^{N+M+K} \bar t^{|I|-2} \ddsum_{s_2=0}^{\bar t} s_2 b_{s_1,s_2}.
%    \en
%The bound in \refeq{bdDekta} then yields that
%    \eq
%    |a^{\sss(N,M,K)}(\yvec_1,\vec\xvec_I;4)|
%    \leq C_{\sss I} \vep^2 (C\beta)^{N+M+K} \big(\vep\delta_{s_1,\vep}+\vep^2 (s_1+1)^{-(d-2)/2}\big)
%    (s_1+\Delta_{\bar t})
%    \bar t^{|I|-2}.
%    \en
%The proof of \refeq{a4bdN} follows by performing the
%sum over $M,K$.
%\qed

\paragraph{Acknowledgements.}
The work of RvdH was carried out in part at Delft University of Technology and
was supported in part by Netherlands Organisation for Scientific Research
(NWO).  The work of AS was carried out in part at University of British Columbia,
EURANDOM, Eindhoven University of Technology and the University of Bath, and
was supported in part by NSERC of Canada and in part by NWO.  This project was
initiated during an extensive visit of RvdH to the University of British
Columbia, Canada.  We thank Ed Perkins for useful conversations during the
initial stages of the project, and Gordon Slade for various discussions
throughout the project.

\end{document}